\documentclass{article}
\usepackage[utf8]{inputenc}
\usepackage{amsmath}
\usepackage{mathrsfs}
\usepackage{dsfont}
\usepackage{amssymb}
\usepackage{setspace}
\usepackage{algorithm}
\usepackage{algpseudocode}
\usepackage{subcaption}
\usepackage{caption}
\usepackage{graphicx}

\usepackage{xcolor}
\usepackage{amsthm}
\usepackage{hyperref}
\usepackage{mathtools}
\newtheorem{theorem}{Theorem}[section]
\newtheorem{lemma}[theorem]{Lemma}

\newtheorem{proposition}[theorem]{Proposition}

\theoremstyle{remark}
\newtheorem{example}[theorem]{Example}
\newtheorem{remark}[theorem]{Remark}

\usepackage{amsopn}

\newtheorem{assumption}[theorem]{Assumption}

\newcommand\norm[1]{\left\lVert#1\right\rVert}
\renewcommand{\phi}{\varphi}
\usepackage{enumitem}
\usepackage{cleveref}
\usepackage{geometry}
\usepackage{easyReview}

\ifpdf
  \DeclareGraphicsExtensions{.eps,.pdf,.png,.jpg}
\else
  \DeclareGraphicsExtensions{.eps}
\fi

\setlist[enumerate]{leftmargin=.5in}
\setlist[itemize]{leftmargin=.5in}

\title{A continuous scale space of diffeomorphisms\thanks{This work was partially supported by NSF 2309683.}}

\author{Yechen Liu\thanks{Department of Applied Mathematics and Statistics, Johns Hopkins University 
  (yliu381@jhu.edu).}
\and Laurent Younes\thanks{Department of Applied Mathematics and Statistics, Johns Hopkins University 
  (laurent.younes@jhu.edu).}}
\date{}

\begin{document}

\maketitle

\begin{abstract}
    In this paper, we  define and study a nested family of reproducing kernel Hilbert spaces of vector fields that is indexed by a range of scales, from which we construct a reproducing kernel Hilbert space of scale-dependent vector fields. We provide a characterization of the reproducing kernel of that space, with numerical approximations ensuring quick evaluations when this kernel does not have a closed form. We then introduce a multiscale version of the large deformation diffeomorphic metric mapping (LDDMM) problem and prove the existence of solutions.  Finally, we provide numerical experiments performing landmark matching using multiscale LDDMM. 
    
\end{abstract}



\section{Introduction}

Analyzing the characteristics of shapes, and in particular, relating them to  biological or clinical phenotypes has become an essential element of many medical studies, following the development of large databases including various imaging modalities \cite{mueller2005alzheimer,casey2018adolescent,soldan2016hypothetical}. While shapes, or anatomies, are intuitively easy to apprehend, they are complex mathematical objects, as can attested from the variety of representations, such as images, landmarks, curves, surfaces, volumes, vector and tensor fields, or measures,  that have been used to describe them in the literature. Because these representations do not generally align with descriptive shape features, an essential step in most studies involves some form of registration, in which all shapes are placed in a common coordinate system that simplifies subsequent analyses. One of the main focuses of this paper involves
diffeomorphic registration, which optimizes a smooth invertible transformation of a given shape or image to place it in a desired position, a problem that we will formalize in a multi-scale setting.


The scientific literature abounds in mathematical models of shapes, and in algorithms allowing for their analysis and comparison. Among such,  
one can cite methods such as the Gromov-Wasserstein distance approach \cite{memoli2011gromov} which compares shapes as metric spaces, the iterative closest point (ICP) framework \cite{Besl1992, Amberg2007}, Gaussian mixture and variational Bayesian models (GMM) \cite{Jian2011,Hirose2021} or Laplace-Beltrami (LB) eigenmap \cite{Lai2017}. Comparisons based on topology-preserving transformations, such as diffeomorphisms, have been introduced, notably, in the large deformation diffeomorphic metric mapping (LDDMM) framework \cite{Dupuis1998, trouve1998diffeomorphisms, Joshi2000, beg2005computing}, and these techniques are especially relevant when applied to point sets discretizing a continuous curve, surface or volume. LDDMM models  transformations using a geodesic flow of diffeomorphisms associated to a time-dependent vector field. This framework can be applied to point set matching with labeled landmarks \cite{Joshi2000}, or, in the unlabeled case, combined with soft assignments \cite{guo2006diffeomorphic} or with kernel-based distances between measures \cite{glaunes2004diffeomorphic}. A variant to LDDMM in \cite{yang2015diffeomorphic} uses stationary vector fields to improve the computational efficiency, at the expense of removing the metric properties of the  approach. Beside the LDDMM model, diffeomorphic registration methods using quasi-conformal maps have also been proposed, where the diffeomorphic property is guaranteed by controlling the Beltrami coefficient \cite{gardiner2000quasiconformal}. A two-dimensional quasi-conformal landmark registration (QCLR) model was proposed in \cite{lam2014landmark}, and later generalized to higher dimensional spaces in \cite{Lee2014LandmarkMatchingTW} and \cite{zhang2022QC}. Besides, various conformal flattening methods of surfaces, as a special case of shape registration have been developed \cite{Choi2015Disk, choi2015Flash, Jin2008RicciFlow, Mullen2008SCP}.


In all these shape matching problems the notion of ``scale'' plays an essential role, in that it defines the range at which points in space interact during the transformation, of which it determines the rigidity. Scale is typically captured by a parameter in the definition of the smoothness metric that is optimized in the registration algorithm. 

For complex data, or when several modalities are analyzed together, transformations typically occur at several scales. Morphological variations in brain MRI data, for example, may result in global differences in the whole brain, as well as minute shape changes in small subregions. When the observation is combined with histological data, variations may occur at even smaller levels. The importance of understanding these data at all relevant scales, as well as the need to decompose the observed transformations according to these scales has led the authors in \cite{miller2022molecular,miller2020coarse} to introduce a new approach in which registration is estimated for a finite number of scales at once, in a cooperative way. The models described in the present paper can be seen as a generalization of this approach, in which the multiscale morphological analysis operates on a continuum of scales, with various options defining how these scales interact. 
We note also that the concept of multiple scales also appeared in the earlier work in \cite{risser2011simultaneous}, in which the authors propose to use diffeomorphic flows with Eulerian velocities associated with reproducing kernel Hilbert spaces (RKHS) whose kernel is a sum of several others at multiple scales. This latter work, however, still computes a single diffeomorphic registration, rather than a cascade of transformations as in \cite{miller2022molecular, miller2020coarse} and in the present paper. 


Following standard scale-space analysis, it is natural to order transformations from most detailed to most global (small scales to large scales). When these transformations are specified as flows of vector fields, as done here, this leads to defining scale-dependent vector fields, which, as introduced in \cref{sec:scale.dependent.vf}, is achieved by defining families of nested reproducing kernel Hilbert spaces, each of them associated with a specific scale. These spaces are assembled in a single RKHS of scale-dependent vector fields, built as a form of first-order Sobolev space with respect to scale. Important properties of this family of spaces are derived in \cref{sec:scale.dependent.vf}, and a characterization of their reproducing kernels, as well as explicit and computational methods allowing for their computation, are provided in \cref{sec:kernels}. A description of registration algorithms applying the large deformation diffeomorphic metric mapping across multiple scales is given in \cref{sec:mslddmm} and experimental results are provided in \cref{sec:experiments}.

\section{Spaces of scale-dependent vector fields}
\label{sec:scale.dependent.vf}

\subsection{Scale-dependent vector fields}

Let $[s_1,s_2]$ be a range of positive scales (we assume that $s_1>0$ for notational convenience). Our goal in this section is to define and study a family of Hilbert spaces of scale-dependent vector fields, i.e., functions $v: [s_1, s_2] \times \mathbb R^d \to \mathbb R^d$ (so that $v(\lambda, \cdot)$ is the vector field $v$ at scale $\lambda$). We will then be in position to define scale-dependent flows as solutions of 
\begin{equation}
    \label{eq:flow.0}
    \partial_t \phi(t, \lambda, x) = v(t, \lambda, \phi(t, \lambda, x))
\end{equation}
with $\phi(t, \lambda, x) = x$ where $v$ is a time-and-scale-dependent vector field with sufficient regularity properties to ensure the existence of solutions.

We start with the definition of a space $\mathbb V$ of such vector fields that will serve as a container for the ``scale-space'' analysis developed in the next section.
Let $V$ be a reproducing kernel Hilbert space (RKHS) continuously embedded into the space of $C^1$ vector fields $v$ on $\mathbb R^d$ such that $v$ and $Dv$ decay to zero at infinity, denoted as $V\xhookrightarrow{} C^1_0(\mathbb R^d)$. Recall that a Hilbert space of scalar- or vector-valued functions over $\mathbb R^d$ is an RKHS if and only if, for all $x\in \mathbb R^d$, the linear mapping $\delta_x: v \mapsto v(x)$ is bounded. { The continuous embedding of $V$ in  $C^1_0(\mathbb R^d)$ means that there exists some constant $C_V>0$ and a mapping $\mathfrak i: V\to C^1_0(\mathbb R^d)$ such that $\norm{\mathfrak i (v)}_{1,\infty}\leq C_V\norm{v}_{V}$. For the remaining of this paper, $\mathfrak i(v)$ is not distinguished from $v$, and $\norm{\mathfrak i (v)}_{1,\infty}$ is simply denoted as $\norm{v}_{1,\infty}$.} 

Recall that a locally Bochner integrable function $g: (a,b)\to V$ is a weak derivative of a locally Bochner integrable function $v: (a,b)\to V$ if $g$ satisfies
\begin{align*}
    \int_a^b g(\lambda)\varphi(\lambda)d\lambda = -\int_a^b \partial_\lambda\varphi(\lambda) v(\lambda)d\lambda
\end{align*}
for all $C^\infty$ function $\varphi:(a,b)\to\mathbb R$ with compact support. 
One then writes $g=\partial_\lambda v$. 

With this, we define $\mathbb V=H^1((s_1,s_2),V)$, or
\begin{equation}    
\mathbb V=\left\{v:(s_1,s_2)\to V \text{ such that } \int_{s_1}^{s_2}\norm{v(\lambda)}_V^2d\lambda+\int_{s_1}^{s_2}\norm{\partial_{\lambda }v(\lambda)}_V^2d\lambda<\infty\right\},
\end{equation}
the Hilbert space of square integrable functions $v$ with square integrable weak derivatives. For $v\in\mathbb V$, the Hilbert norm is given by \begin{align*}
    \norm{v}_{\mathbb V}^2 = \int_{s_1}^{s_2}\norm{v(\lambda)}_V^2d\lambda+\int_{s_1}^{s_2}\norm{\partial_\lambda v(\lambda)}_V^2d\lambda.
\end{align*}

If $v\in \mathbb V$, one has for almost all $\lambda_1,\lambda_2\in (s_1,s_2)$, \begin{align}
    v(\lambda_2) = v(\lambda_1)+\int_{\lambda_1}^{\lambda_2}\partial_\lambda v(\lambda)d\lambda.
    \label{eqn: IBP}
\end{align}
This shows that elements of $\mathbb V$ are a.e. equal to (uniformly) continuous functions, as for a.e. $\lambda_1,\lambda_2\in(s_1,s_2),$ \begin{align*}
    \norm{v(\lambda_2)-v(\lambda_1)}_V\leq\int_{\lambda_1}^{\lambda_2}\norm{\partial_\lambda v(\lambda)}_Vd\lambda\leq (\lambda_2-\lambda_1)^{1/2}\norm{v}_{\mathbb V}.
\end{align*} 
By convention, $v(\lambda)$ always refers to the continuous version of $v$ evaluated at $\lambda$, so that \cref{eqn: IBP} holds for all $\lambda_1$ and $\lambda_2$. 
This remark allows us to define the operator $\iota_\lambda: \mathbb V\to V$, with $\iota_\lambda v=v(\lambda)$, for all $\lambda\in [s_1,s_2]$. As $v(\lambda)\in C^1(\mathbb R^d)$, the notation $v(\lambda)(x)$ makes sense, and to avoid an excessive use of  parentheses, we will simply write $v(\lambda, x)=v(\lambda)(x)$. 
A similar convention will be  applied to vector fields where time and scale variables are involved. Whenever convenient, we will also denote $v(\lambda, \cdot) = \iota_\lambda(v)$.

For the operator $\iota_\lambda$, we have the following lemma, proved in \cref{sec:lemma2.1}: 
\begin{lemma}
    For all $\lambda\in [s_1,s_2]$, $\iota_\lambda: \mathbb V\to V$ is a linear and bounded operator. 

    \label{lemma: ilambda: V to V}
\end{lemma} 

We now extend the discussion by replacing the first term in $\|\cdot\|_{\mathbb V}$ by an integral with respect to an arbitrary positive measure $\rho$ on $[s_1,s_2]$ (rather than Lebesgue's measure). One can consider the norm $\|\cdot \|_{\mathbb V, \rho}$ given by 
\begin{align*}
    \norm{v}_{\mathbb V, \rho}^2 = \int_{[s_1,s_2]}\norm{v(s)}_V^2d\rho(s)+\int_{s_1}^{s_2}\norm{\partial_\lambda v(s)}_V^2ds. 
\end{align*}
It turns out that we have the following equivalence between these two norms, with the proof provided in \cref{sec:prop.V.equip}: 
\begin{proposition}
\label{prop:V.equiv}
    For any positive measure $\rho$ on $[s_1,s_2]$, $\|\cdot\|_{\mathbb V, \rho}\sim\|\cdot\|_{\mathbb V}$.    
\end{proposition}

\subsection{Scale-space analysis}
\label{sec:scale.space}
We now define our multi-resolution analysis for vector fields. We let below $W_0 = V$ to allow for uniform notation.
                    
We assume given a nested family of RKHS's, $W_\lambda$, with Hilbert norm $\| \cdot \|_{\lambda} $ for $\lambda\in[s_1,s_2]$, satisfying $W_\lambda \xhookrightarrow{} W_\mu \xhookrightarrow{} W_0=V$ for $\lambda, \mu \in \{0\}\cup [s_1, s_2]$, $\mu < \lambda$. For such $\mu$ and $\lambda$, we denote by $\mathfrak i_{\lambda,\mu}$ the embedding of of $W_\lambda$ in $W_\mu$ and let  $c_{\lambda, \mu} = \norm{\mathfrak i_{\lambda, \mu}}$ such that $\norm{\mathfrak i_{\lambda, \mu} v}_\mu \leq c_{\lambda, \mu}\norm{v}_\lambda$ for all $v\in W_{\lambda}$. As above, we will not make the distinction between $\mathfrak i_{\lambda, \mu}(v)$ and $v$ as soon as $\mu< \lambda$. 

Our first assumption is that all such embedding coefficients, including that for $\mu=0$ have a bounded supremum, denoted as $C$:

\begin{assumption}
    C:= $\sup \{c_{\lambda, \mu}:  \mu,\lambda\in \{0\}\cup [s_1,s_2], \mu< \lambda\}<\infty$. 
\end{assumption}




We fix, in the following construction, a positive measure $\rho$ on $[s_1, s_2]$.
Let $\mathbb W$ be the subspace of $\mathbb V$ consisting of functions $v$ such that $\iota_\lambda v$ and $\iota_\lambda \partial_\lambda v$ belong to $ W_\lambda$ for a.e. $\lambda \in[s_1,s_2]$ and such that \begin{align*}
    \norm{v}_{\mathbb W}^2:=\int_{[s_1,s_2]}\norm{v(\lambda)}_\lambda^2d\rho(\lambda) +\int_{s_1}^{s_2}\norm{\partial_\lambda v(\lambda)}_\lambda^2d\lambda<\infty.
\end{align*}

We start with the basic remark:
\begin{proposition}
    \label{prop:w.rkhs}
    The space $\mathbb W$ is a reproducing kernel Hilbert space.
\end{proposition}
\begin{proof}
    The fact that $\mathbb W$ is a Hilbert space, i.e., that it is complete is proved repeating without change the proof that scalar Sobolev spaces are complete, the fact that the norms depend on $\lambda$ bringing only minor disturbance. Given a Cauchy sequence, say $v^n$, in $\mathbb W$, one extracts a subsequence $v^{n_k}$ such that $\|v^{n_{k+1}} - v^{n_k}\|_{\mathbb W} < 2^{-k}$ so that $\sum_{k=1}^\infty \|\partial_\lambda v^{n_{k+1}}(\lambda) - \partial_\lambda v^{n_{k}}(\lambda)\|_\lambda$ is integrable and therefore a.e. finite, showing that 
    \[
    g(\lambda) = v^{n_1}(\lambda) + \sum_{k=1}^\infty (\partial_\lambda v^{n_{k+1}}(\lambda) - \partial_\lambda v^{n_{k}}(\lambda)) 
    \]
    exists and belongs to $W_\lambda$ for almost all $\lambda$. One shows that $\int_{s_1}^{s_2} \|g(\lambda) - \partial_\lambda v^m(\lambda)\|_\lambda d\lambda$ can be made arbitrarily small for large enough $m$ using the fact that this is true with $g$ replaced by $\partial_\lambda v^{n_k}$ for $n_k\geq m$ and $m$ large enough and applying Fatou's lemma to the limit in $k$. Similarly, there exists $v\in \mathbb W$ such that $\int_{s_1}^{s_2} \|v(\lambda) - v^m(\lambda)\|_\lambda d\rho(\lambda)$ tends to 0 when $m$ tends to infinity.

    Because $\|v\|_{\mathbb V, \rho} \leq C \|v\|_{\mathbb W}$, $v^n$ also converges to $v$ in $L^2([s_1, s_2], V)$ and $\partial_\lambda v^n$ to $g$ in $L^2([s_1, s_2], V)$ (recall the equivalence between $\|\cdot\|_{\mathbb V}$ and $\|\cdot\|_{\mathbb V, \rho}$) so that $g = \partial_\lambda v$ and $v^n$ converges to $v$ in $\mathbb W$.

    Now, since $(\mathbb W, \|\cdot\|_{\mathbb W})$ is embedded in $(\mathbb V, \norm{\cdot}_{\mathbb V})$ (calling $\mathfrak i$ the embedding), we have, for any $(\lambda, x) \in [s_1, s_2] \times \mathbb R^d$, the following sequence of continuous linear maps:
    \[
    \delta_{(\lambda, x)}: \mathbb W \stackrel{\mathfrak i}{\longrightarrow} \mathbb V \stackrel{\iota_\lambda}{\longrightarrow} V_\lambda \stackrel{\delta_x}{\longrightarrow} \mathbb R^d, 
    \]
    showing that $\delta_{(\lambda, x)}$ is bounded on $\mathbb W$, which is therefore an RKHS. 
\end{proof}

We note that the first two arrows in the end of the previous proof are equivalent to the statement of the following lemma. 
\begin{lemma}
    The operator $\iota_\lambda: \mathbb W\to V$ is linear and bounded for all $\lambda \in [s_1,s_2]$. \label{lemma: ilambda: W to V} 
\end{lemma}
We now investigate whether this result can be strengthened to $\iota_\lambda$ being bounded from $\mathbb W$ to $W_\lambda$, which is a natural expectation. We will prove this with an additional assumption on $\rho$ is required for $\lambda\in [s_1,s_2)$, and a left-continuity assumption on $W_\lambda$ for the end point.



\begin{proposition}
    If $\lambda\in [s_1, s_2]$ is such that $\rho([\lambda, s_2])>0$, then  $\iota_\lambda(v)\in W_\lambda$ and $w\in \mathbb W$. Moreover, this operator is bounded from  $(\mathbb W, \norm{\cdot}_{\mathbb W})$ to $(W_\lambda, \norm{\cdot}_\lambda)$. 

    Let $\lambda_0\in [s_1, s_2]$ be such that $\rho([\lambda_0, s_2]) = 0$ and $\rho([\lambda, s_2]) > 0$ for all $\lambda < \lambda_0$. 
    If $W_{\lambda_0}=\cap_{\lambda\in [s_1,\lambda_0)}W_\lambda$, then $v$ (extended by continuity to $[s_1, \lambda_0]$) is such that $\iota_{\lambda_0}(v)\in W_{\lambda_0}$. 
    
    Furthermore, if one assumes, in addition, that the mapping $\lambda\in[s_1,\lambda_0]\mapsto \norm{w}_\lambda$ is continuous for all $w\in W_{\lambda_0}$, then $\iota_{\lambda_0}:\mathbb W\to W_{\lambda_0}$ is bounded. 
    
    \label{Prop W lambda}
\end{proposition}
The proof of this proposition can be found in \cref{sec:Prop W lambda}.

\section{Kernel of \texorpdfstring{$\mathbb W$}{}}
\label{sec:kernels}
\subsection{Notation on reproducing kernels}

One of the features of reproducing kernel theory is that many variational problems in RKHS's can be reduced to finite-dimensional problems involving their kernel, and become easy (or at least easier) to solve as soon as these kernels are known in closed form, or easily computable. Such kernels are key components of the implementation of diffeomorphic mapping algorithms such as LDDMM when discretized in terms of point sets or landmarks, and will also be important in our implementation of diffeomorphic matching in the multiscale framework. We therefore dedicate this section to the characterization and computation of the kernel of $\mathbb W$, and we start by fixing some notation that will be used throughout.

If $H$ is a Hilbert space, we let $H^*$ denote its topological dual, and denote as $(\eta\mid h)$ the evaluation of a bounded linear form $\eta\in H^*$ at a vector $h\in H$. This notation should be distinguished from $\langle h, h'\rangle_H$ that denotes the inner product between $h,h'\in H$. The duality operator of $H$, denoted $\mathbf L_H: H \to H^*$ associates to $h\in H$ the linear form $\mathbf L_H h$ such that, for all $h'\in H$, $(\mathbf L_H h\mid h') = \langle h, h'\rangle_H$. Riesz's theorem implies that $\mathbf L_H$ is invertible and we will denote $\mathbf K_H = \mathbf L_H^{-1}: H^* \to H$. 

If $A$ is a continuous linear map between two Hilbert spaces $H$ and $\tilde H$, its dual, denoted $A^*: \tilde H^* \to H^*$ transforms the linear form $\tilde \eta \in \tilde H^* $ into $A^*\tilde \eta \in H^*$ such that $(A^*\tilde \eta \mid h) = (\tilde\eta \mid Ah)$. The transpose of $A$ is the mapping $A^T:\tilde H \to H$ such that, for all $h\in H$ and $\tilde h \in \tilde H$, $\langle h, A^T \tilde h \rangle_{H} = \langle Ah, \tilde h \rangle_{H}$ and can be identified as $A^T = \mathbf K_H A^* \mathbf L_{\tilde H}$.

If $U$ is a fixed set, an RKHS $H$ over $U$ is a Hilbert space of mappings $h:U \to \mathbb R^k$ such that for any $x\in  U$, the evaluation mapping $\delta_x: h \mapsto h(x)$ is continuous from $H$ to $\mathbb R^k$. For $a\in \mathbb R^k$, we denote by $a^T\delta_x$ the mapping $h \mapsto a^Th(x)$, which belongs to $H^*$, so that $\mathbf K_H(a^T\delta_x) \in H$ is a function with values in $\mathbb R^k$ that depends linearly on $a$. The reproducing kernel of $H$ is then a mapping, denoted $K_H$, defined on $U \times U$, taking values in the set of $k\times k$ matrices, such that, for all $a\in \mathbb R^k$, all $x,y\in U$: $K_H(y,x) a = \mathbf K_H(a^T\delta_x)(y)$.

One says that a kernel is scalar if $K_H(x,y) = \kappa_H(x,y) I_k$ (the latter being the $k\times k$ identity matrix), $\kappa_H$ taking real values. If $U= \mathbb R^d$, the kernel is translation-invariant if $K_H(x,y)$ only depends on $x-y$ and it is radial if and only if it only depends on $|x-y|$. Scalar and radial kernels have simpler implementation and good invariance properties, and will be used in our experiments. An important example of such kernels is the Gaussian kernel for which $K_H(x, y) = \exp(-\frac{|x-y|^2}{2\lambda^2})I_k$ for some $\lambda >0$.

Returning to the multiscale problem, our goal in this section is to compute $K_{\mathbb W}$, for which $U = [s_1, s_2] \times \mathbb R^d$. We will denote the kernel and duality operators on $W_\lambda$ by $K_\lambda$, $\mathbf K_\lambda$, $\mathbf L_\lambda$ (rather than $K_{W_\lambda}$, etc.) to lighten the notation. The assumption underlying this discussions is that the kernels $K_\lambda$ are known explicitly and easily computable.

\subsection{Characterization of \texorpdfstring{$K_{\mathbb W}$}{}}
By definition, for any given vector $a\in\mathbb R^d$ and $w\in\mathbb W$, for all $\lambda_0\in[s_1,s_2]$, $x_0\in\mathbb R^d$, one has \begin{align*}
    \langle K_{\mathbb W}(\cdot, (\lambda_0, x_0))a, w\rangle_{\mathbb W} = a^Tw(\lambda_0, x_0).
\end{align*} 
In the computation that follows, we fix $\lambda_0$, $x_0$ and $a$, and let $v(\lambda, x) = K_{\mathbb W}((\lambda, x),(\lambda_0, x_0))a$.

For all $w\in \mathbb W$, we have 
\begin{align}
    \langle v, w\rangle_{\mathbb W}=\int_{[s_1,s_2]}\langle v(\mu), w(\mu)\rangle_\mu d\rho(\mu)+\int_{s_1}^{s_2}\langle \partial_\lambda v(\mu), \partial_\lambda w(\mu)\rangle_\mu d\mu = a^Tw(\lambda_0, x_0). \label{Hilbert: vw}
\end{align}

Fixing $\lambda \in (s_1, s_2]$, take $w(\mu,x)=q_\lambda(\mu)\bar w(x)$ in the equality above, with $q_\lambda(\mu)=(\lambda-\mu)\mathbf 1_{[s_1,\lambda]}(\mu)$, i.e.,
$$q_\lambda(\mu)=\left\{\begin{aligned}
    &\lambda-\mu, &\text{ for }\mu \in [s_1,\lambda],\\
    &0, &\text{ for }\mu \in [\lambda, s_2].
\end{aligned}\right.$$
and $\bar w\in W_\lambda$. Then $w\in \mathbb W$ and
\cref{Hilbert: vw} can be rewritten as 
\begin{align*}
    \int_{[s_1, \lambda]}(\lambda-\mu)\langle v(\mu), \bar w\rangle_\mu d\rho(\mu)-\int_{s_1}^\lambda\langle \partial_\lambda v(\mu), \bar w\rangle_\mu d\mu = q_\lambda(\lambda_0)a^T\bar w(x_0).
\end{align*}
Using the duality operator, the above can be further rewritten as \begin{align}
    \int_{[s_1,\lambda]}(\lambda-\mu)(\mathbf L_\mu v(\mu)|\bar w)d\rho(\mu)-\int_{s_1}^\lambda (\mathbf L_\mu \partial_\lambda v(\mu)|\bar w)d\mu &= (\lambda-\lambda_0)\mathbf{1}_{[s_1,\lambda]}(\lambda_0)a^T(\delta_{x_0}|\bar w) \nonumber\\ \int_{[s_1,\lambda]}(\lambda-\mu)(\mathbf L_\mu v(\mu))d\rho(\mu)-\int_{s_1}^\lambda (\mathbf L_\mu \partial_\lambda v(\mu))d\mu &= (\lambda-\lambda_0)\mathbf{1}_{[s_1,\lambda]}(\lambda_0)a^T\delta_{x_0},
    \label{eqn: kernel 1}
\end{align}
where the left-hand side is understood as an operator on $W_\lambda$, since $\mathbf L_\mu v(\mu)$ and $\mathbf L_\mu \partial_\lambda v(\mu)$ both belong to $W_\mu^* \xhookrightarrow{} W_\lambda^*$. We note that \cref{eqn: kernel 1} can be written as
\[
\int_{s_1}^\lambda \int_{[s_1,\mu]}(\mathbf L_{\mu'}v(\mu'))d\rho(\mu')d\mu-\int_{s_1}^\lambda (\mathbf L_\mu \partial_\lambda v(\mu))d\mu = (\lambda-\lambda_0)\mathbf{1}_{[s_1,\lambda]}(\lambda_0)a^T\delta_{x_0},
\]
which is equivalent to 
\begin{equation}
    \label{eqn: kernel 11}
 \mathbf L_\lambda \partial_\lambda v(\lambda)  = \int_{[s_1,\lambda]}(\mathbf L_{\mu}v(\mu))d\rho(\mu) - \mathbf{1}_{[s_1,\lambda]}(\lambda_0)a^T\delta_{x_0}
\end{equation}
for almost all $\lambda\in [s_1, s_2]$. This can be written as
\begin{equation}
    \label{eqn: kernel 12}
\partial_\lambda v(\lambda)  = \int_{[s_1,\lambda]}(\mathbf K_\lambda \mathbf L_{\mu}v(\mu))d\rho(\mu) - \mathbf{1}_{[s_1,\lambda]}(\lambda_0) K_\lambda(\cdot, x_0)a
\end{equation}
for almost all $\lambda\in [s_1, s_2]$. The r.h.s. of this equation is continuous everywhere except at $\lambda = \lambda_0$ and at the atoms of $\rho$ (scales $s\in [s_1, s_2]$ such that $\rho(\{s\}) > 0$), and there is no loss of generality in requiring that $\partial_\lambda v_\lambda$ is equal to this r.h.s. everywhere except at these points.

On the other hand, if one takes $w(\lambda, x)=\bar w(x)$ with $\bar w\in W_{s_2}$, so that  $\partial_\lambda w = 0$, \cref{Hilbert: vw} becomes
\begin{align}
\label{eqn: kernel 2}
\int_{[s_1,s_2]} (\mathbf  L_\mu v(\mu)) d\rho(\mu) = a^T \delta_{x_0}.
\end{align}

The solution of equations \cref{eqn: kernel 11,eqn: kernel 2} will depend on the choice of $\rho$. We start with a simple case in which this solution is explicit. 

\subsection{First example: Dirac measure}
\label{sec:example.dirac}
Assume that $\rho = \sigma \delta_{s_0}$ for some $s_0\in [s_1, s_2]$ and $\sigma > 0$. Then \cref{eqn: kernel 2} gives 
$\sigma \mathbf  L_{s_0} v(s_0) = a^T \delta_{x_0}$, i.e., $v(s_0) = K_{s_0}(\cdot, x_0) a/\sigma$. \Cref{eqn: kernel 11} gives, for $\lambda\not\in \{s_0, \lambda_0\}$,
\[
 \mathbf L_\lambda \partial_\lambda v(\lambda)  = \frac1\sigma (\mathbf{1}_{[s_0, s_2]} - \mathbf{1}_{[\lambda_0, s_2]})(\lambda)a^T\delta_{x_0}
\]
so that for $\lambda\not\in \{ s_0, \lambda_0\}$,
\[
\partial_\lambda v(\lambda)  = \frac{1}{\sigma}(\mathbf{1}_{[s_0, s_2]} - \mathbf{1}_{[\lambda_0, s_2]})(\lambda) K_\lambda(\cdot, x_0)a.
\]
Bearing in mind that $v= K_{\mathbb W}((\cdot, \cdot), (\lambda_0, x_0))a$, we get
\begin{align}
\nonumber
& K_{\mathbb W}((\lambda, \cdot), (\lambda_0, x_0)) = K_{s_0}(\cdot, x_0)  + \frac{1}{\sigma}\int_{s_0}^\lambda (\mathbf{1}_{[s_0, s_2]} - \mathbf{1}_{[\lambda_0, s_2]})(\mu) K_\mu(\cdot, x_0) d\mu\\
\label{eq:kernel.dirac}
 &= 
K_{s_0}(\cdot, x_0) + \frac{1}{\sigma}\mathrm{sign}(\lambda_0 - s_0)\left\{
\begin{aligned}
& \int_{s_0}^{\min(s_0, \lambda_0)} K_\mu(\cdot, x_0) d\mu: \lambda \leq \min(s_0, \lambda_0)\\ 
& \int_{s_0}^\lambda K_\mu(\cdot, x_0) d\mu: \min(s_0, \lambda_0) \leq \lambda \leq \max(\lambda_0, s_0)\\
& \int_{s_0}^{\max(\lambda_0, s_0)} K_\mu(\cdot, x_0) d\mu: \max(s_0, \lambda_0) \leq \lambda
\end{aligned}
\right.
\end{align}

Taking the two extreme cases for $s_0$, we have, when $s_0 = s_2$:
\[
K_{\mathbb W}((\lambda, \cdot), (\lambda_0, x_0)) = 
K_{s_2}(\cdot, x_0) + \frac{1}{\sigma}\int_{\max(\lambda_0, \lambda)}^{s_2} K_\mu(\cdot, x_0) d\mu,
\]
and when $s_0=s_1$:
\[
K_{\mathbb W}((\lambda, \cdot), (\lambda_0, x_0)) = 
K_{s_1}(\cdot, x_0) + \frac{1}{\sigma}\int_{s_1}^{\min(\lambda_0, \lambda)} K_\mu(\cdot, x_0) d\mu.
\]

\subsubsection*{Piecewise constant kernel}
A simple example of application of the previous formulas can be obtained when $K_\lambda$ is a piecewise constant in scale. Let $\{r_k:k=1,\dots, n+1\}$ be a subdivision of the scale interval with $s_1=r_1<r_2<\dots<r_{n+1}=s_2$. We assume that  $K_{\lambda}$ is  constant on these intervals, with $K_{\lambda}=K_{k}$ for $\lambda\in [r_k, r_{k+1})$. Then, for  $\lambda_1\leq  \lambda_2\in [s_1, s_2]$
\[
\int_{\lambda_1}^{\lambda_2} K_\mu(x, x_0) d\mu = 
(r_{k_1+1} - \lambda_1) K_{k_1}(x, x_0) + \sum_{j= k_1+1}^{k_2-1} (r_{j+1} - r_j)  K_{j}(x, x_0) + (\lambda_2 - r_{k_2}) K_{k_2}(x, x_0)
\]
where $\lambda_i \in [k_i, k_{i+1})$, $i=1,2$. The kernel considered in \cite{miller2022molecular} is equivalent to this kernel, with $\rho= \delta_{s_2}$. 

\subsubsection*{Gaussian kernel}
Consider now the situation in which $K_\lambda$ is defined as a Gaussian kernel, namely
\[
K_{\lambda}(x,y)=\exp{(-\frac{|x-y|^2}{2\lambda^2})}I_d
\]
Then, with $c=\frac{|x-x_0|^2}{2}$, for  $\lambda_1\leq  \lambda_2\in [s_1, s_2]$,
\begin{align*}
     \int_{\lambda_1}^{\lambda_2} K_{\lambda}(x,x_0) d\mu &= \int^{1/\lambda_1}_{1/\lambda_2}\frac{e^{-ct^2}}{t^2}dt=-\frac{e^{-ct^2}}{t}\bigg\vert_{t=1/\lambda_2}^{t=1/\lambda_1}-2c\int_{1/\lambda_2}^{1/\lambda_1}e^{-ct^2}dt\\
     &=\lambda_2 e^{-\frac{c}{\lambda_2^2}}-\lambda_1 e^{-\frac{c}{\lambda_1^2}}+\sqrt{c\pi}\left(\text{erf}(\frac{\sqrt{c}}{\lambda_2})-\text{erf}(\frac{\sqrt{c}}{\lambda_1})\right)
\end{align*}
where $\text{erf}(x)=\frac{2}{\sqrt{\pi}}\int_0^x e^{-t^2}dt$ is the Gauss error function. 



\subsection{Translation-invariant and radial kernels}


While $K_{\mathbb W}$ does not have a closed form for general $\rho$, its computation can be somewhat schematic with additional assumptions on the kernels $K_\lambda$, namely that they are scalar radial kernels. Such a property indeed induces similar invariance on the space component of $K_{\mathbb W}$, as discussed in this section.

Fixing $x\in\mathbb R^d$, 
define the translation operator $\tau_x$ on, say, continuous functions from $\mathbb R^d$ to $\mathbb R^d$ by $(\tau_x w)(y) = w(y-x)$. It is a standard result on reproducing kernels that, if $\tau_x$ maps an $H$ onto itself and is an isometry, then the kernel of $H$ is translation invariant. This result directly extend to the multiscale case, in which we define 
$\tau_x$ on functions from $[s_1, s_2] \times \mathbb R^d$ to $\mathbb R^d$ by $(\tau_x w)(\lambda, y) = w(\lambda, y-x)$. This is stated in the next lemma.


\begin{lemma}
    Suppose that, for all $x\in \mathbb R^d$, $\tau_x$ maps $\mathbb W$ onto itself as an isometry, so that $\norm{\tau_x(w)}_{\mathbb W}=\norm{w}_{\mathbb W}$. Then $K_{\mathbb W}$ is translation invariant in space over scales, in the sense that $K_{\mathbb W}((\lambda, x), (\mu,y))=K_{\mathbb W}((\lambda, 0),(\mu, y-x))$, where $\lambda,\mu\in[s_1,s_2]$ and $x,y,z\in\mathbb R^d$.
\end{lemma}
\begin{proof}
For any $w\in W$, we have
    \begin{align*}
        \langle \tau_{-x}(K_{\mathbb W}((\cdot,\cdot),(\mu, y))a),  w\rangle_{\mathbb W}= 
        \langle K_{\mathbb W}((\cdot,\cdot),(\mu, y))a, \tau_{x} w\rangle_{\mathbb W}=
        a^T (\tau_{x}w)(\mu, y) = a^T (\mu, y-x)
    \end{align*}
    so that 
    \[
    \tau_{-x}(K_{\mathbb W}((\cdot,\cdot),(\mu, y))a) = K_{\mathbb W}((\cdot,\cdot),(\mu, y-x))a
    \]
which is just
\begin{align*}
        K_{\mathbb W}((\lambda, z+x),(\mu, y))=K_{\mathbb W}((\lambda, z),(\mu, y-x))
    \end{align*} 
    for all $\lambda\in [s_1,s_2]$, and $z\in \mathbb R^d$, 
    proving the assertion. 
\end{proof}

Notice that translations are isometries on $\mathbb W$ as soon as they are isometries on $W_\lambda$ for each $\lambda$. So, if $K_\lambda$ is translation invariant for each $\lambda$, $K_{\mathbb W}$ will also be translation invariant.

Scalar kernels can also be characterized by an invariance property. If one requires that the operation $v\mapsto Rv$, where $R$ is any orthogonal matrix, maps $\mathbb W$ onto itself and is an isometry, a similar argument to that made above implies that $K_{\mathbb W}(\cdot, \cdot)$ commutes with orthogonal matrices and is therefore scalar. Such a property is once again inherited from the same invariance holding for each $W_\lambda$, i.e., if each $K_\lambda$ is scalar, then so is $K_{\mathbb W}$.

The same argument applied to the transformation $\eta_R$ such that 
\[
(\eta_R w)(\lambda, x) = w(\lambda, R^Tx)
\]
shows that,
if one assumes that $\eta_R$ is an isometry on $W$ for all orthogonal matrices $R$, then 
\[
K_{\mathbb W}((\lambda, Rx), (\mu, Ry)) = K_{\mathbb W}((\lambda, x), (\mu, y)).
\]
Once again, the property is inherited from it holding for all $K_\lambda$.

Now assuming all three properties, which are true when all $K_\lambda$'s are scalar and radial, we find that $K_{\mathbb W}$ is also scalar and radial, i.e., that it takes the form
\[
K_{\mathbb W}((\lambda, x), (\mu, y)) = \kappa_{\mathbb W}(\lambda, \mu, x-y) I_d =  \gamma_{\mathbb W}(\lambda, \mu, |x-y|) I_d\,,
\]
for some function $\gamma_{\mathbb W}: [0, +\infty)\to \mathbb R$.

Note that the symmetry of the kernel, $K_{\mathbb W}((\lambda, x), (\mu, y))^T = K_{\mathbb W}((\mu, y), (\lambda, x))$, implies that  $\gamma_{\mathbb W}$ is  symmetric with respect to $\lambda$ and $\mu$, while $\kappa_{\mathbb W}(\lambda, \mu, z) = \kappa_{\mathbb W}(\mu, \lambda, -z)$.

Upon suitable integrability assumptions, translation-invariant kernels are characterized by their Fourier transforms. Here, we define the Fourier transform of an integrable function $f: \mathbb R^d \to \mathbb R^k$ as
\[
\hat f(\xi) = \int_{\mathbb R^d} f(x) \exp(-i2\pi x^T\xi) dx
\]
so that the inverse Fourier transform is
\[
\check g(x) = \int_{\mathbb R^d} g(\xi) \exp(i2\pi x^T\xi) d\xi.
\]
Importantly, if $f$ is a radial function, i.e., $f(x) = \alpha(|x|)$ for some function $\alpha$, then so is its Fourier transform, with $\hat f(\xi) = \tilde \alpha(|\xi|)$ and the transformation $\alpha \mapsto \tilde \alpha$ is called the Hankel transform of $\alpha$.

If $f$ or $g$ also depend on one or two scale variables (like $\kappa_{\mathbb W}$ above), we will use the 
same notation, with the understanding that the transform is applied to the space component.

The reason why the Fourier transform is useful for reproducing kernels is because applying the operator $\mathbb K$ for a radial kernel to some function or measure, which is equivalent to taking the convolution with $\kappa$, results in simple products when passing to Fourier transforms. As a consequence, if we assume that for all $\lambda\in[s_1, s_2]$, 
\[
K_{\lambda}(x, y) = \kappa_{\lambda}(x-y) I_d =  \gamma_{\lambda}(|x-y|) I_d\,
\]
then \cref{eqn: kernel 12}, in which we take $x_0=0$ (which is sufficient since $K_{\mathbb W}$ is translation invariant) reads, denoting $\chi_\lambda = 1/\hat \kappa_\lambda$,
\begin{equation}
\label{eqn: kernel 12f}
\chi_\lambda(\xi) \partial_\lambda \hat \kappa_{\mathbb W}(\lambda, \lambda_0, \xi)  = \int_{[s_1,\lambda]} \chi_\mu(\xi) \hat \kappa_{\mathbb W}(\mu,\lambda_0, \xi)d\rho(\mu) - \mathbf{1}_{[\lambda_0, s_2]}(\lambda),
\end{equation}
and \cref{eqn: kernel 2} gives
\begin{equation}
\label{eqn: kernel 2f}
\int_{[s_1,s_2]} \chi_\mu(\xi) \hat\kappa_{\mathbb W}(\mu, \lambda_0, \xi) d\rho(\mu) = 1.
\end{equation}

\subsection{Lebesgue measure}

\subsubsection{Continuous scale kernels}
If $\rho$ is absolutely continuous with respect to Lebes\-gue's measure, say, $\rho = \sigma \mathcal L$ for some continuous function $\sigma$, and $\chi(\lambda, \xi)$ is continuous in $\lambda$,  the r.h.s. of \cref{eqn: kernel 12f} is differentiable in $\lambda$ everywhere except at $\lambda = \lambda_0$. We therefore obtain the second-order differential equation for the Fourier transform 
\[
\partial_\lambda(\chi_\lambda(\xi) \partial_\lambda \hat \kappa_{\mathbb W}(\lambda, \lambda_0, \xi))  = \chi_\lambda (\xi) \hat \kappa_{\mathbb W}(\lambda,\lambda_0, \xi)\sigma(\lambda).
\]
For $\lambda_0 \in (s_1, s_2)$ this equation comes with boundary conditions $\partial_\lambda \hat \kappa_{\mathbb W}(s_1, \lambda_0, \xi) = 0$ (deduced from \cref{eqn: kernel 12f}) and $\partial_\lambda \hat \kappa_{\mathbb W}(s_2, \lambda_0, \xi) = 0$ (deduced from taking the difference between \cref{eqn: kernel 12f} and \cref{eqn: kernel 2f}). These conditions must be combined with the fact that $\hat \kappa_{\mathbb W}(\lambda,\lambda_0, \xi))$ is continuous in $\lambda$ and has a jump in its derivative at $\lambda_0$ given by $-1/\chi_{\lambda_0}(\xi)$ to completely specify the solution.

If $\lambda \in \{s_1, s_2\}$, this equation is satisfied on the whole interval $(s_1, s_2)$ and the boundary conditions are $\hat \kappa_{\mathbb W}(s_1,\lambda_0, \xi) = -1/\chi_{s_1}(\xi)$
and $\partial_\lambda \hat \kappa_{\mathbb W}(s_2, \lambda_0, \xi) = 0$ when $\lambda_0=s_1$, and $\hat \kappa_{\mathbb W}(s_1,\lambda_0, \xi) = 0$
and $\partial_\lambda \hat \kappa_{\mathbb W}(s_2, \lambda_0, \xi) = -1/\chi_{s_2}(\xi)$ for $\lambda = s_2$. 

The solution of this differential equation can be computed numerically for any fixed value of $|\xi|$ and $\lambda_0$, which can be done as a single-run preprocessing step. We have not, however, explored this approach numerically in this paper, but rather focused on an approach in which the scale kernels are modeled (or approximated) to be piecewise constant, as to be described in the next section.

\subsubsection{Piecewise constant kernels}
We now assume that $\rho = \sigma^2 \mathcal L$ where $\sigma$ is a positive constant.
Let $\{r_k:k=1,\dots, n+1\}$ be a subdivision of the scale interval with $s_1=r_1<r_2<\dots<r_{n+1}=s_2$. We assume that  $\kappa_{\lambda}$ is piecewise constant on these intervals, and we will write, with some abuse of notation, $\kappa_{\lambda}=\kappa_{k}$ for $\lambda\in [r_k, r_{k+1})$. We proceed with our computation with the assumption that $\lambda_0\in \{r_k:k=1,\dots, n+1\}$, which comes without loss of generality because one can always add $\lambda_0$ to this list by splitting the interval it falls in into two pieces. So, we let $\lambda_0 = r_{k_0}$.

Let $h(\lambda) = \hat \kappa_{\mathbb W}(\lambda, \lambda_0, \xi)$ in the following computation, with $h_k = h(r_k)$.
On every open interval $(r_k, r_{k+1})$, $h$ satisfies the linear equation $\partial_\lambda^2 h - \sigma^2 h = 0$, as can be obtained by differentiating \cref{eqn: kernel 12f} in $\lambda$. It follows that, for $\lambda \in (r_k, r_{k+1})$, 
\[
h(\lambda) = \frac{\sinh(\sigma(r_{k+1} - \lambda))}{\sinh(\sigma(r_{k+1} - r_k))} h_k + \frac{\sinh(\sigma(\lambda-r_k))}{\sinh(\sigma(r_{k+1} - r_k))} h_{k+1}.
\]

Denote by $\partial_\lambda^-h$ and $\partial_\lambda^+h$ the left and right derivative of $h$ in $\lambda$. \Cref{eqn: kernel 12f} implies 
\[
\chi_{k} \partial_\lambda^+h(r_k) - \chi_{k-1} \partial_\lambda^-h(r_k) = -\delta_{k_0}(k),
\]
the special cases at the boundaries being obtained by letting $\chi_0 = \chi_{n+1} = 0$. If we let $\rho_k = r_{k+1}-r_k$ for $k=1,\dots,n$, computing the left and right derivatives as functions of the $h_k$'s yields, for $k=2, \ldots, n$,
\begin{align*}
\frac{\sigma h_{k+1}\chi_k}{\sinh(\sigma\rho_k)}  - \sigma h_k \left(\coth(\sigma\rho_k)\chi_k + \coth(\sigma\rho_{k-1})\chi_{k-1}\right) + \frac{\sigma h_{k-1}\chi_{k-1}}{\sinh(\sigma\rho_{k-1})} = -\delta_{k_0}(k).
\end{align*}

\begin{remark}
This is a tridiagonal system in $h_1, \ldots, h_{n+1}$, which can be solved independently for each value of $\xi$. It is however numerically preferable to let $g_k = \chi_{k} h_{k}$, $k=1, \ldots, n$, $g_{n+1} = \chi_n h_{n+1}$ and solve the system in terms of $g_1, \ldots, g_{n+1}$, which becomes, with $\psi^v_k = \chi_{k-1}/\chi_{k}$ for $k=1, \ldots, n$, and $\psi^v_{n+1} = 1$:
\begin{align*}
   \frac{\sigma g_{k+1}\psi^v_{k+1}}{\sinh(\sigma\rho_{k})}  - \sigma g_k \left(\coth(\sigma\rho_k) + \coth(\sigma\rho_{k-1})\psi^v_k\right) + \frac{\sigma g_{k-1}}{\sinh(\sigma\rho_{k-1})} = -\delta_{k_0}(k). 
\end{align*}
for $k=2, \ldots, n$ and
\begin{align*}
- \sigma g_{n+1} \coth(\sigma\rho_n) + \frac{\sigma g_{n}}{\sinh(\sigma\rho_n)} &= -\delta_{k_0}(n+1)\\
\frac{\sigma g_{2}\psi^v_{2}}{\sinh(\sigma\rho_1)}  - \sigma g_1 \coth(\sigma\rho_1)  &= -\delta_{k_0}(1).    
\end{align*}

In the Gaussian case, with $\kappa_{r_k}(z) = \exp\left(- \frac{|z|^2}{2r_k^2}\right)$, then 
\[
\chi_k(\xi) = \hat \kappa_{r_k}(\xi)^{-1} = (2\pi r_k^2)^{-d/2} \exp\left(2r_k^2\pi^2|\xi|^2\right),
\]
so that
\[
\psi^v_k(\xi) = \left(\frac{r_{k-1}}{r_k}\right)^{d} \exp\left(-2(r_{k}^2 - r_{k-1}^2)\pi^2|\xi|^2\right).
\]

\end{remark}

\subsection{Sum of Dirac measures}
A similar computation can be made when $\rho$ is a sum of Dirac measures on $[s_1, s_2]$ (for continuous or piecewise constant kernels), and we only consider here the simple where $\rho = \delta_{s_1} + \delta_{s_2}$.
Writing $\hat \kappa_{\mathbb W}(\lambda)$ for $\hat \kappa_{\mathbb W}(\lambda, \lambda_0, \xi)$,
\Cref{eqn: kernel 12f,eqn: kernel 2f} read in this case as:
\[
\partial_\lambda \hat \kappa_{\mathbb W}(\lambda)  = \frac{\chi_{s_1}}{\chi_\lambda} \hat \kappa_{\mathbb W}(s_1) - \frac{\mathbf{1}_{[\lambda_0, s_2]}(\lambda)}{\chi_\lambda}
\]
and
\[
\chi_{s_1} \hat\kappa_{\mathbb W}(s_1) + \chi_{s_2} \hat\kappa_{\mathbb W}(s_2) = 1.
\]
The first equation yields, introducing the notation $X(\lambda) = \int_{s_1}^\lambda(1/\chi_\mu)d\mu$,
\[
\hat \kappa_{\mathbb W}(\lambda)  = \hat \kappa_{\mathbb W}(s_1) + \chi_{s_1}X(\lambda) \hat \kappa_{\mathbb W}(s_1) - X(\max(\lambda, \lambda_0)) + X(\lambda_0),
\]
and replacing $\hat \kappa_{\mathbb W}(s_2)$ in the second one gives
\[
\chi_{s_1} \hat\kappa_{\mathbb W}(s_1) + \chi_{s_2} \hat \kappa_{\mathbb W}(s_1) + \chi_{s_1}\chi_{s_2} X(s_2) \hat \kappa_{\mathbb W}(s_1) - \chi_{s_2}(X(s_2) - X(\lambda_0)) = 1.
\]
This gives
\[
\hat\kappa_{\mathbb W}(s_1) = \frac{1 + \chi_{s_2}(X(s_2) - X(\lambda_0))}{\chi_{s_1} + \chi_{s_2} + \chi_{s_1}\chi_{s_2} X(s_2)}.
\]
Using this in the expression of $\hat \kappa_{\mathbb W}(\lambda)$ yields (after a little algebra)
\[
    \hat \kappa_{\mathbb W}(\lambda)  = \frac{(1+\chi_{s_{2}}(X(s_2) - X(\max(\lambda, \lambda_0))(1+\chi_{s_{1}}X(\min(\lambda, \lambda_0))}{\chi_{s_1} + \chi_{s_2} + \chi_{s_1}\chi_{s_2} X(s_2)}.
\]

\subsection{Inverse Fourier transform}
The previous sections illustrated situations in which the functions $\kappa_{\mathbb W}(\lambda, \lambda_0, \cdot)$ were computable (for fixed $\lambda$ and $\lambda_0$) through their Fourier transform. Because this kernel needs to be called many times in the registration procedure, it is important to ensure that this function can be computed as efficiently as possible. We now restrict to the case of radial kernels, which implies that the same property holds for the Fourier transform. This restriction simplifies the approximation since it will only require working with a scalar function of a single scalar variable (for each pair of scales).

We will approximate $\hat \kappa_{\mathbb W}(\lambda_1, \lambda_2, \xi)$ in the form
\begin{equation}
\label{eq:approx.ft}
\hat \kappa_{\mathbb W}(\lambda_1,\lambda_2, \xi) \simeq \sum_{q=1}^Q \beta_q(\lambda_1, \lambda_2) \hat h_q(|\xi|)
\end{equation}
where $\hat h_1, \ldots, \hat h_Q: [0, +\infty) \to [0, +\infty)$ are fixed functions and the coefficients $\beta_q(\lambda_1, \lambda_2)$ can be precomputed and stored for all scales $\lambda_1, \lambda_2$ of interest (we will see in \cref{sec:mslddmm} that only a finite number of them are typically needed for registration). The functions $\hat h_1, \ldots, \hat h_q$ should be designed to be Hankel transforms of known functions $h_1, \ldots, h_Q$, yielding the direct approximation
\begin{equation}
\label{eq:approx.ft.2}
\kappa_{\mathbb W}(\lambda_1,\lambda_2, \xi) \simeq \sum_{q=1}^Q \beta_q(\lambda_1, \lambda_2) h_q(|\xi|).
\end{equation}

It is however desirable to ensure that the approximation still corresponds to a positive kernel. The following extension of Bochner's theorem, whose proof is given in appendix, describes the required condition.


\begin{theorem}
    \label{th:bochner.plus}
    Let $S$ be any set and $\Gamma: S\times S \times \mathbb R^d\to \mathbb R$. Then the followings are equivalent:
    \begin{enumerate}[label=(\arabic*)]
        \item Far all $s_1, s_2\in S$, $z\mapsto \Gamma(s_1, s_2, z) $ is continuous and the function
    \[
    K: ((s_1, x_1), (s_2, x_2)) \mapsto \Gamma(s_1, s_2, x_1-x_2)
    \]
    is a positive semi-definite kernel. 
    \item There exists a family of finite complex Radon measures $\mathcal M = (\mu_{s_1,s_2}, s_1, s_2 \in S)$ such that, for all $s_1, s_2\in S$ 
    \[
    \Gamma(s_1, s_2, z) = \int_{\mathbb R^d} e^{-i2\pi \xi^Tz} d\mu_{s_1, s_2}(\xi),
    \]
    and $\mathcal M$ is a measure-valued positive kernel in the sense that, for all $s_1, \ldots, s_n\in S$ and all $a_1, \ldots, a_n\in \mathbb C$, 
    \[
    \sum_{k,l=1}^n a_k\bar a_l \mu_{s_k, s_l}
    \]
    is a positive measure.
  \end{enumerate}
  \end{theorem}

Assume that one only wants to approximate the kernel for a finite family of target scales $s_1, \ldots, s_n$. Assume also that Fourier transforms are computed on a finite family of frequencies $\xi_j$, $j=1, \ldots, J$. Then  \cref{th:bochner.plus} requires that 
for all $j$, the matrix with entries
\[
\sum_{q=1}^Q \beta_q(s_k, s_l) \hat h_q(|\xi_j|)
\]
is positive, where each of these entries is an approximation of $\hat \kappa_{\mathbb W}(s_k,s_l, \xi_j)$. This leads to a semi-definite program involving as many matrices as discrete frequencies $\xi_j$. In the experiments presented here, we only ensure positive-definiteness for pairs of scales, using a greedy approach organized as follows:
\begin{enumerate}[label=$\bullet$,wide]
\item First, for every $k=1, \ldots, n$, we estimate  $\beta_q(s_k, s_k)$ by minimizing
\[
\max_{j=1, \ldots, J} \left|\hat \kappa_{\mathbb W}(s_k,s_k, \xi_j) - \sum_{q=1}^Q \beta_q(s_k, s_k) \hat h_q(|\xi_j|)\right|
\]
subject to 
\[
\sum_{q=1}^Q \beta_q(s_k, s_k) \hat h_q(|\xi_j|) \geq 0
\]
for all $j= 1, \ldots, J$. 

\item Then, for all $k\neq l$, we define
\[
c_j = \sqrt{\left(\sum_{q=1}^Q \beta_q(s_k, s_k) \hat h_q(|\xi_j|)\right)\left(\sum_{q=1}^Q \beta_q(s_l, s_l) \hat h_q(|\xi_j|)\right)}
\]
and minimize
\[
\max_{j=1, \ldots, J} \left|\hat \kappa_{\mathbb W}(s_k,s_l, \xi_j) - \sum_{q=1}^Q \beta_q(s_k, s_l) \hat h_q(|\xi_j|)\right|
\]
subject to
\[
\left | \sum_{q=1}^Q \beta_q(s_k, s_k) \hat h_q(|\xi_j|)\right| \leq c_j.
\]
\end{enumerate}
This computation, which can be implemented as a series of linear programs, ensures the positivity of the kernel restricted to any pair of target scales.

\section{Multi-scale LDDMM}
\label{sec:mslddmm}

\subsection{Optimal control and Pontryagin Maximum Principle (PMP)}

\subsubsection{Single-scale problem}
The LDDMM algorithm is formulated as an optimal control problem whose state space is a group of $d$-dimensional diffeomorphisms and the control space an RKHS $V$ of $C^1$ vector fields on $\mathbb R^d$, with objective to minimize a cost function taking the form
\begin{subequations}
\begin{equation}
    \label{eq:lddmm.base.a}
\frac{1}{2} \int_0^1 \|v(t)\|^2_V dt + G(\phi(1))
\end{equation}
subject to 
\begin{equation}
    \label{eq:lddmm.base.b}
    \varphi(t, x) = x+\int_{0}^t v(\tau, \varphi(\tau, x))d\tau  .
\end{equation}
\end{subequations}

Let $C^p_0(\mathbb R^d, \mathbb R^d)$ denote the Banach space of $C^p$ functions $w$ such that $w$, $dw$, \ldots, $d^pw$ tend to zero at infinity with norm $\|w\|_{p, \infty}$ equal to the maximum norm taken by $w(\cdot)$ and its first two derivatives. \Cref{eq:lddmm.base.b} has a unique solution, that we will denote $\phi(\cdot, \cdot;v)$ as soon as $v\in L^1([0,1], C^p_0(\mathbb R^d, \mathbb R^d))$ with $p\geq 1$ and $\phi(t, \cdot;v)$ is, for all $t \in [0,1]$, a $C^p$ diffeomorphism (see, e.g., \cite{sontag2013mathematical,Younes_2019}).

The group of diffeomorphisms used as state space is the set of $C^1$ diffeomorphisms that tend to the identity at infinity, which  is an open subset of the affine space $\mathrm{id}_{\mathbb R^d} + C^1_0(\mathbb R^d, \mathbb R^d)$. We will denote it as $\mathit{Diff}^1_0(\mathbb R^d)$.

\subsubsection{Multi-scale problem}
We  will assume that $V \hookrightarrow C^p_0(\mathbb R^d, \mathbb R^d)$ for $p$ large enough (with $p\geq 1$). This implies that $L^2([0,1], V) \hookrightarrow L^1([0,1], C^p_0(\mathbb R^d, \mathbb R^d))$, so that 
transformations $\phi(t)$ associated with finite-cost controls in \cref{eq:lddmm.base.a,eq:lddmm.base.b} are $C^p$ diffeomorphisms and therefore belong to the state space at all times as soon as $p\geq 1$.

In the multiscale framework, we consider time-and-scale-dependent vector fields $v\in L^2([0,1], \mathbb W)$. 
We introduce a new notation for fixing scales in this space, defining, for $\lambda\in [s_1, s_2]$, $\mathbb I_\lambda: L^2\left([0,1], \mathbb W\right) \to L^2([0,1], V)$ by 
$(\mathbb I_\lambda v) (\cdot, \cdot) = v(\cdot,\lambda,\cdot)$. One can use almost the same proofs as in the previous section to show that $\mathbb I_\lambda$ is linear and bounded.

Given $v\in L^2([0,1], \mathbb W)$, we define scale-dependent flows of diffeomorphisms as solutions of the equation
\begin{align}
    \varphi(t, \lambda, x) = x+\int_{0}^t v(\tau, \lambda, \varphi(\tau, \lambda, x))d\tau  .
    \label{eqn: ms diffeo}
\end{align} 

\Cref{eqn: ms diffeo} describes  the flow on $[s_1, s_2]\times \mathbb R^d$ associated with the time-dependent vector field
\[
\bar v(t, \lambda, x) = \begin{pmatrix}
    0\\
    v(t, \lambda, x)
\end{pmatrix}
\]
so that the existence of a unique solution is a  consequence of the fact that the Lispchitz constant of $\hat v(t, \cdot, \cdot)$ (which is the same as that of $v(t, \cdot, \cdot)$) is time integrable as soon as $v\in L^1([0,1], \mathbb W)$. This solution is continuous in $(\lambda, x)$ and, using the fact that $\mathbb I_\lambda$ is continuous to analyze the solution for fixed $\lambda$, it is $C^p$ with respect to the space variable.

The MS-LDDMM optimal control problem  minimizes
\begin{equation}
\label{eq:ms.lddmm.obj}
\frac{1}{2} \int_0^1 \|v(t, \cdot, \cdot)\|^2_{\mathbb W} dt + G(\phi(1, \cdot, \cdot))
\end{equation}
subject to \cref{eqn: ms diffeo}. 

We state without proof the following result, which can be proved using, almost without change, the proof of Theorem 7.9 in \cite{Younes_2019}. 

\begin{theorem}
\label{th:weak.cvg}
If $(v_n)$ is a bounded sequence  in $L^2\left([0,1], \mathbb W\right)$ which converges weakly to $v \in L^2\left([0,1], \mathbb W\right)$, $\phi(\cdot, \cdot, \cdot; v_n)$ converges, uniformly on $[0, 1]\times [s_1, s_2] \times K$, to $\phi(\cdot, \cdot, \cdot; v)$, for  every compact subset $K\subset \mathbb{R}^d$. 
\end{theorem}

Since minimizing sequences for \cref{eq:ms.lddmm.obj} are bounded and have subsequences that converge weakly, 
\cref{th:weak.cvg} implies that minimizers always exist, as soon as $\Gamma$ is lower semi-continuous for the uniform convergence on compact subsets.

\subsubsection{Base scales}
\label{sec:base.scales}
To simplify the discussion, and because this is a natural setting for applications, we now assume that $\Gamma(\psi(\cdot, \cdot))$ only depends on $\psi$ evaluated at a finite number of scales, so that there exists $\lambda_1, \ldots, \lambda_m$ and a function $g$ such that
\[
G(\psi(\cdot, \cdot)) = g(\psi(\lambda_1, \cdot), \ldots, \psi(\lambda_m, \cdot)).
\]
We will refer to $\lambda_1,\ldots, \lambda_m$ as the base scales.

In this form, the state space of the problem can be reduced to collections of $m$ diffeomorphisms, included in $(\mathrm{id}_{\mathbb R^d} + C^1_0(\mathbb R^d, \mathbb R^d))^m$, letting $\phi_k(t) = \phi(t, \lambda_k)$. We assume that $g$ is differentiable in its variables. We also assume that $V\hookrightarrow C^p_0(\mathbb R^d, \mathbb R^d)$ with $p\geq 2$. This ensures that the r.h.s. of the state equation $\partial_t \phi(t, \lambda, \cdot) = v(t, \lambda, \phi(t, \lambda, \cdot))$ is differentiable with respect to the state variable, and that the Pontryagin maximum principle (PMP) is satisfied. This principle implies that there exists a co-state $p: [0,1] \to (C^1_0(\mathbb R^d, \mathbb R^d)^*)^m$ such that, at almost every time $t$, the optimal control maximizes the Hamiltonian
\[
\sum_{k=1}^m (p_k \mid w(\lambda_k, \phi(t, \lambda_k, \cdot))) - \frac{1}{2} \|w\|_{\mathbb W}^2 = \sum_{k=1}^m (p_k \mid R_{\phi(t, \lambda_k, \cdot)} \iota_{\lambda_k} w) - \frac{1}{2} \|w\|_{\mathbb W}^2
\]
with respect to $w\in \mathbb W$, where $R_{\varphi(t,\lambda_k,\cdot)}$ is the right composition operator such that for all $k$, one has $ R_{\varphi(t,\lambda_k, \cdot)}\iota_{\lambda_k w}= w(\lambda_k, \varphi(t,\lambda_k, \cdot))$.
This yields 
\[
v(t, \cdot, \cdot) = \sum_{k=1}^m \mathbf K_{\mathbb W} (\iota_{\lambda_k}^*  R_{\phi(t, \lambda_k, \cdot)}^* p_k)
\]
for optimal controls $v$.
In this form, we have 
\[
\|v(t)\|_{\mathbb W}^2 = \sum_{k, l=1}^m (p_k\mid R_{\phi(t, \lambda_k, \cdot)}\iota_{\lambda_k} \mathbf K_{\mathbb W} (\iota_{\lambda_l}^*  R_{\phi(t, \lambda_l, \cdot)}^* p_l)).
\]
The boundary conditions of the PMP implies that $p_k(1) = - \partial_{k}g(\phi_1(t), \ldots, \phi_m(1))$ and the co-state equation still expresses the fact that, for all $h\in C^1_0(\mathbb R^d, \mathbb R^d)$, 
\[
\partial_t (p_k(t) \mid h) = - (p_k(t) \mid \partial_x v(t, \lambda_k, \phi(t, \lambda_k, \cdot)) h),
\]
which implies $(p_k(t) \mid \partial_x \phi(t, \lambda_k, \cdot)h) = (p_k(0)\mid h)$.

 Assume now that the end-cost $g$ is discrete, in the sense that it only depends on the values of $\phi_k(1, \cdot)$ on finite point sets, say, $(x_{ki}^0, i=1, \ldots, N_k)$, so that
 \[
 g(\psi^v_1, \ldots, \psi^v_m) = q(\psi^v_k(x_{ki}^0), i=1, \ldots, N_k, k=1, \ldots, m),
 \]
 for some function $q: \mathbb R^{N_1} \times \cdots \times \mathbb R^{N_m} \to \mathbb R$. 
 Then the costate takes the form
\[
p_k(t) = \sum_{j=1}^{N_k} a_{kj}(t)^T \delta_{x_{kj}(0)}
\]
with $a_{kj}(t) = d\phi(t, \lambda_k, x_{kj}(0))^T a_{kj}(0)$. One then has
\begin{equation}
\label{eq:base.scale.vf}
v(t, \lambda, x) = \sum_{k=1}^m \sum_{j=1}^{N_k} K_{\mathbb W}((\lambda, x), (\lambda_k, x_{kj}(t))) a_{kj}(t)
\end{equation}
with $\phi(t, \lambda_k, x_{kj}(0))$. In this form, we have
\[
\|v(t)\|_{\mathbb W}^2 = \sum_{k,l=1}^m \sum_{i=1}^{N_k} \sum_{j=1}^{N_l} a_{ki}(t)^T K_{\mathbb W}((\lambda_k, x_{ki}(t)), (\lambda_l, x_{lj}(t))) a_{lj}(t).
\]

Similar to other point-set schemes used in LDDMM, this formulation leads to a reduction of the original problem to an optimal control formulation with finite-dimensional state (the collection $(x_{kj}(\cdot), k=1, \ldots, m, j=1, \ldots, N_k)$) and control (the collection $(a_{kj}(\cdot), k=1, \ldots, m, j=1, \ldots, N_k)$). The experiments in \cref{sec:experiments} use this formulation with standard optimal control optimization where the gradient is computed using the ``adjoint'' (or ``backpropagation'') method. In this setting, the differential of the objective function (call it $F$) with respect to the control is given by, assuming a scalar, translation-invariant kernel
\[
\partial_{a_{ki}(t)}  F = \sum_{l=1}^m \sum_{j=1}^{N_k} \kappa_{\mathbb W}((\lambda_k, \lambda_l, x_{ki}(t) - x_{lj}(t)) p_{lj}(t)
\]
where $p_{ki}(\cdot)$ satisfies the dynamical system
\begin{multline*}
\partial_t p_{ki}(t) \\
= - \sum_{l=1}^m\sum_{j=1}^{N_l} \nabla_x \kappa_{\mathbb W}(\lambda_k, \lambda_l, x_{ki}(t) -x_{lj}(t)) (p_{ki}(t)^T a_{lj}(t) + p_{lj}(t)^T a_{ki}(t) - 2 a_{ki}(t)^T a_{lj}(t))
\end{multline*}
with $p_{ki}(1) = - \partial_{k,i} q_{|{x(1)}}$.

\section{Alternative scale-space kernels}

The previous discussion provides examples of what can be called ``scale-space kernels,'' namely, reproducing kernels on $[s_1, s_2]\times \mathbb R^d$, and we here discuss possible alternative constructions of such kernels. Arguably, the first-order operators used to ensure embeddings in $W^{1,2}([s_1, s_2], W_0)$, are among the simplest  choices that could be made, and one may consider higher order operators in scale, ensuing higher regularity instead (provided they lead to feasible computation). 

Moreover, given that reproducing kernels can be combined in multiple ways to create new kernels, the examples provided in the previous section easily lead to new ones. For example, one can consider the integral with respect to $s_0$, of the kernels computed in  \cref{sec:example.dirac}, yielding, taking $\sigma = s_2-s_1$:
\begin{multline*}
K((\lambda, x), (\lambda_0, x_0)) \\
= \int_{s_1}^{s_2} \left(1+ \mathbf 1_{[s_1, \mathrm{min}(\lambda, \lambda_0)]}(\mu) \frac{\mu-s_1}{s_2-s_1} + \mathbf 1_{[\mathrm{max}(\lambda, \lambda_0), s_2]}(\mu) \frac{s_2-\mu}{s_2-s_1}\right) K_\mu(x, x_0) d\mu.
\end{multline*}

Beyond these extensions of the main construction in this paper, one may try to directly build a reproducing kernels on $[s_1, s_2]\times \mathbb R^d$, which raises the question of what kind of properties such scale-space kernels should satisfy to provide a multi-scale analysis.

So, let us start $K$ with a reproducing kernel on $[s_1, s_2]\times \mathbb R^d$ with values on the set of $d\times d$ matrices. We assume that this kernel corresponds to a Hilbert space of continuous functions $v: [s_1, s_2]\times \mathbb R^d \to \mathbb R^d$, denoted  $\mathbb W^K$. Our first task is to define induced spaces at fixed $\lambda\in [s_1, s_2]$, and a natural choice is to let $W^K_\lambda$ be the set of all vector fields $w: \mathbb R^d \to \mathbb R^d$ such that there exists $v\in \mathbb W^K $ with $v(\lambda, \cdot) = w$. One can equip $W^K_\lambda$ with the norm
\[
\|w\|_\lambda = \inf\{\|v\|_{\mathbb W^K}: v(\lambda, \cdot) = w\},
\]
which provides $W^K_\lambda$ with a Hilbert norm isometric to that on the quotient space $\mathbb W^K/H_\lambda$ with
\[
H_\lambda = \{\tilde v \in \mathbb W^K: \tilde v(\lambda, \cdot) = 0\}.
\]
The isometry indeed works by identifying $w$ with the set $I_\lambda(w) = \{v\in \mathbb W^K: v(\lambda, \cdot) = w\}\in \mathbb W^K/H_\lambda$, the norm being equal to the quotient norm by definition. Note that this norm is equal to $\|\pi_{H_\lambda^\perp}(v)\|_{\mathbb W^K}$ for any $v\in I_\lambda(w)$, where $\pi_{H_\lambda^\perp}$ is the orthogonal projection onto $H_\lambda^\perp$.

One can show that $W^K_\lambda$ is an RKHS of vector fields, with associated kernel simply given by $K'_\lambda(x,y) = K((\lambda, x), (\lambda, y))$. Indeed, for any $a, y\in \mathbb R^d$, $w:= K'_\lambda(\cdot, y)a\in W^K_\lambda$ and $v:=K((\cdot, \cdot), (\lambda, y))a$ belongs to $I_\lambda(w)$ by definition. Moreover, $v\in H_\lambda^\perp$ since for any $\tilde v\in H_\lambda$,
\[
\left\langle \tilde v, v\right\rangle_{\mathbb W^K} = a^T\tilde v(\lambda, y) = 0. 
\]
This shows that, for any $\tilde w\in \mathbb W_\lambda$ such that $\tilde w = \tilde v(\lambda, \cdot)$,
\begin{align*}
\left\langle\tilde w, w\right\rangle_{W^K_\lambda} &= \left\langle \pi_{H_\lambda^\perp}(\tilde v), \pi_{H_\lambda^\perp}(v) \right\rangle_{\mathbb W^K}&\\
&= \left\langle \pi_{H_\lambda^\perp}(\tilde v), v \right\rangle_{\mathbb W^K} & (v\in H_\lambda^\perp)\\
&= \left\langle \tilde v, v \right\rangle_{\mathbb W^K} & (\text{property of the orthogonal projection})\\
&= \tilde v(\lambda, y) = w(y). &
\end{align*}

The spaces $W^K_\lambda$ can be defined for any kernel $K$ on $[s_1, s_2]\times \mathbb R^d$, but in order for this kernel to provide a multi-scale analysis, it is natural to require as a minimal condition that these spaces satisfy the nesting property that $W^K_\lambda \hookrightarrow W^K_\mu$ if $\mu\leq \lambda$. 

As an example, one may consider kernels of the form
\begin{equation}
\label{eq:explicit.K}
K((\lambda, x), (\mu, y)) = K_{\mathit{scale}}(\lambda, \mu) K_{\mathit{space}}(h(\lambda, x), h(\mu, y))
\end{equation}
for some function $h: [s_1, s_2] \times \mathbb R^d \to \mathbb R^d$, where $K_{\mathit{scale}}$ and $K_{\mathit{space}}$ are respectively kernels on $[s_1, s_2]$ and $\mathbb R^d$. Using the fact that the entry-wise product of two positive semi-definite matrices is 
positive semi-definite, one easily shows that this provides a positive kernel. To ensure positive-definiteness, we need to require that $x \mapsto h(\lambda, x)$ is one-to-one for all $\lambda$. Making this assumption, let $V$ be an RKHS with kernel $K_{\mathit{space}}$ and let $V_\lambda$ be the subspace of $V$ generated by the mappings $K_{\mathit{space}}(\cdot, h(\lambda, y))a$, for $a, y\in \mathbb R^d$. Then, one can see that the mapping $\mathcal J_\lambda: v \mapsto v\circ h(\lambda, \cdot)$ is, up to a multiplicative factor, an isometry between $V_\lambda$ and $\mathbb W^K_\lambda$. Indeed, $\mathcal J_\lambda K_{\mathit{space}}(\cdot, h(\lambda, y))a = K_{\mathit{space}}(h(\lambda, \cdot), h(\lambda, y))a$ with 
\[
\langle K_{\mathit{space}}(\cdot, h(\lambda, y))a, K_{\mathit{space}}(\cdot, h(\lambda, y'))a' \rangle_{V_\lambda} = a^T K_{\mathit{space}}(h(\lambda, y), h(\lambda, y')) a'
\]
while
\begin{multline*}
\langle K_{\mathit{space}}(h(\lambda, \cdot), h(\lambda, y)a), K_{\mathit{space}}(h(\lambda, \cdot), h(\lambda, y')a')\rangle_{{\mathbb W}_\lambda^K} \\
= K_{\mathit{scale}}(\lambda, \lambda)^{-1} a^T K_{\mathit{space}}(h(\lambda, y), h(\lambda, y')) a'.
\end{multline*}
Given that $K_{\mathit{space}}(\cdot, h(\lambda, y))a$ (resp. $K_{\mathit{space}}(h(\lambda, \cdot), h(\lambda, y))a$), for $a,y\in \mathbb  R^d$ generate $V_\lambda$ (resp. $\mathbb W_\lambda^K$) and that both families are related by $\mathcal J_\lambda$, we obtain the fact that the latter is an isometry, up to the factor $K_{\mathit{scale}}(\lambda, \lambda)$.

If $h(\lambda, \cdot)$ is onto (in addition to being one-to-one), then obviously $V_\lambda = V$ and all spaces $\mathbb W_\lambda^K$ are isometric, which obviously implies the nesting property. This is the case, for example, if $h(\lambda, x) = x$, which give in this case a product kernel
\[
K((\lambda, x), (\mu, y)) = K_{\mathit{scale}}(\lambda, \mu) K_{\mathit{space}}(x,y)
\]
with the quite  uninteresting property that all $\mathbb W_\lambda^K$ are identical. This is also the case for $h(\lambda, x) = x/\lambda$, in which case the isometry is just rescaling.

\begin{remark}
It is important to point out that, the construction discussed in this section is not reverse of that introduced in \cref{sec:scale.space}. If one takes 
$K = K_{\mathbb W}$ (which is discussed in \cref{sec:kernels}), where $\mathbb W$ is defined in  \cref{sec:scale.space}, then the spaces $W_\lambda$ and  $W^K_\lambda$ both correspond to fixing the scale coordinate, but their norms (or reproducing kernels) differ. 

Another important remark is that kernels in the form of equation \cref{eq:explicit.K} do not satisfy, in general, the translation invariance property discussed in \cref{th:bochner.plus}.

\end{remark}

\section{Experimental results}
\label{sec:experiments}

We focus, in our illustration, on two-dimensional point set matching examples, in which landmark correspondences are extrapolated to diffeomorphic transformation of the whole space. We provide synthetic examples rather than real-data applications of diffeomorphic mapping, because they are easier to analyze and interpret, and provide sufficient insight to assess the performance of the new approach. Arguably, some of the examples provided below are most challenging than those encountered in typical applications. Of specific interest when analyzing our results, is how registration, based on one or two point sets, is interpolated across scales, and how transformation residuals from one scale to another can be visualized.


In all our experiments, we set $s_1 = 0.1$ and $s_2 = 2$ and let $K_\lambda$ be Gaussian kernel, piecewise constant on intervals $[r_n, r_{n+1}]$ with $r_n = n/10$ for $n=1,\dots, 20$. We will use  the finest and coarsest scales ($\lambda=0.1$ and $\lambda=2$) as base scales (see \cref{sec:base.scales}). Denoting, for short, $\psi^v_\lambda(x) = \phi^v(1, \lambda, x)$,  this yields the objective function 
\begin{multline*}
F(v) = \frac{1}{2}\int_0^1\norm{v(t,\cdot, \cdot)}_{\mathbb W}^2dt 
+ w\left(\sum_{j=1}^{N_1}|\psi^v_{r_1}(x_{1,j})-T_{1, j}|^2+\sum_{j=1}^{N_2}|\psi^v_{r_{20}}(x_{20,j})-T_{20, j}|^2\right),
\end{multline*}
where $w\in \mathbb R$ and $T_{i, j}$ denotes the target of $j$-th landmark $x_{i, j}$ for scale $r_i$ and $\varphi^v(1, \lambda, \cdot)$ denotes the diffeomorphism at time $t=1$ and scale $\lambda$ associated to vector field $v$. In our implementation, the integration over time is realized by numerical integration using a Euler scheme,  over 20 time steps, and we took $w=1$. We will also use the notation, for $k = 1, \ldots, 20$, $\rho_{r_k}^v = \psi^v_{r_k} \circ \psi^v_{r_{k-1}}$ with the convention $\psi^v_{r_{k-1}} = \mathrm{id}$ to represent the residual transformations from one scale to another.

\begin{example}
We consider a schematic human, represented by a head, two arms and the body. We change the positions of both arms from raising horizontally to raising above, and change the shape of the head from a circle to an ellipse, as illustrated in  \cref{fig: human template_target}. The same landmark matching is enforced at both $r_1$ and $r_{20}$ (i.e., $x_{1,j} = x_{20,j}$ and $T_{1,j} = T_{20,j}$).

\begin{figure}
\centering
\begin{subfigure}{.45\textwidth}
\centering
    \includegraphics[width=.25\textwidth]{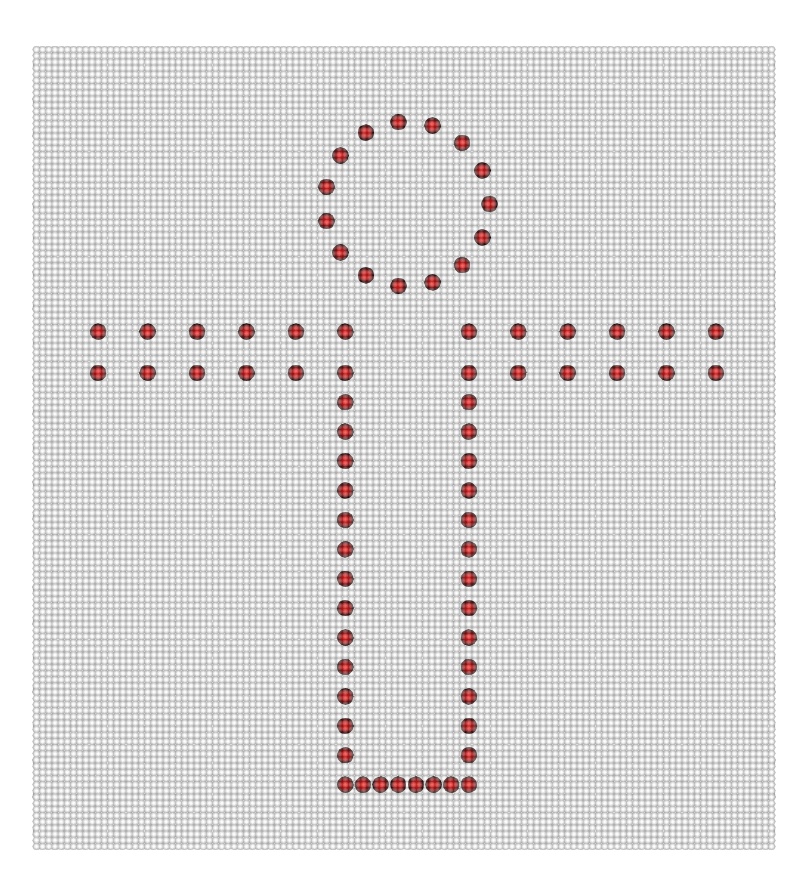}
\subcaption{template}
\end{subfigure}
\begin{subfigure}{.45\textwidth}
\centering
    \includegraphics[width=.25\textwidth]{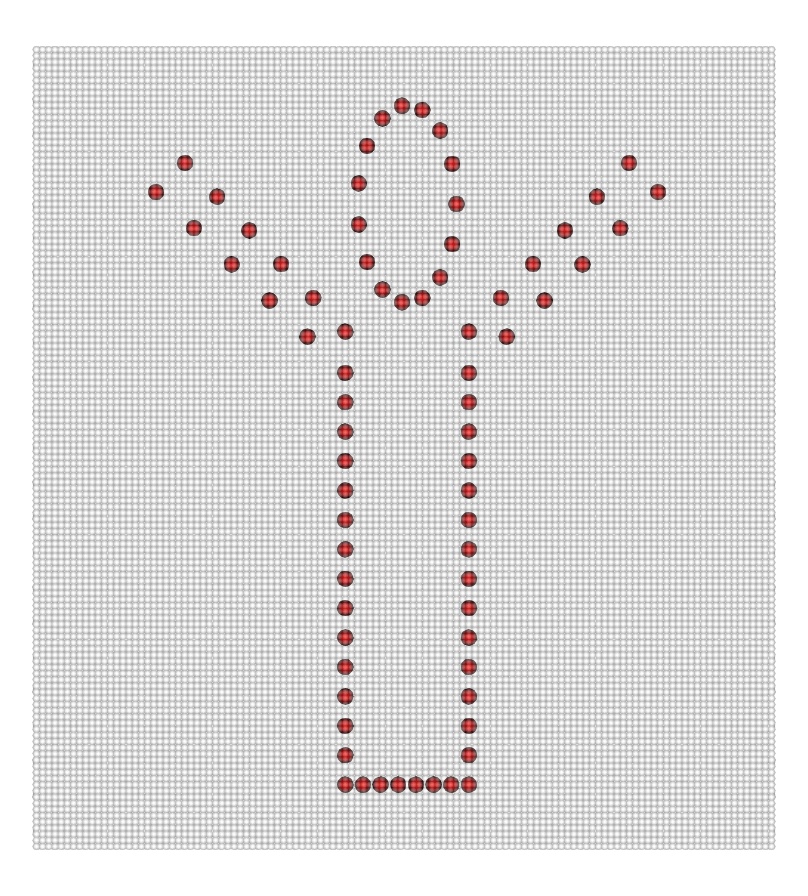}
\subcaption{target}
\end{subfigure}
\caption{Schematic human: template and target.}
\label{fig: human template_target}
\end{figure}

\Cref{fig: human ms results} shows the deformation $\psi^v_{r_k}$ at different scales. One can see that the deformed templates at $r_1$ and $r_{20}$ are very close to the target. The blue color stands for a shrinkage and the red color implies the presence of an expansion at a point. For a more intuitive view, \cref{fig: human ms residual} (first row) shows the determinant of the Jacobian of $\psi^v_{r_k}$. We will simply write ``Jacobian'' for the determinant of Jacobian $\det(d \psi^v_{r_k})$, and ``log Jacobian'' for $\log \det(d\psi^v_{r_k})$ in the description of figures. 

\begin{figure}
\centering
\begin{subfigure}{.16\linewidth}
\centering
\includegraphics[width=.7\linewidth]{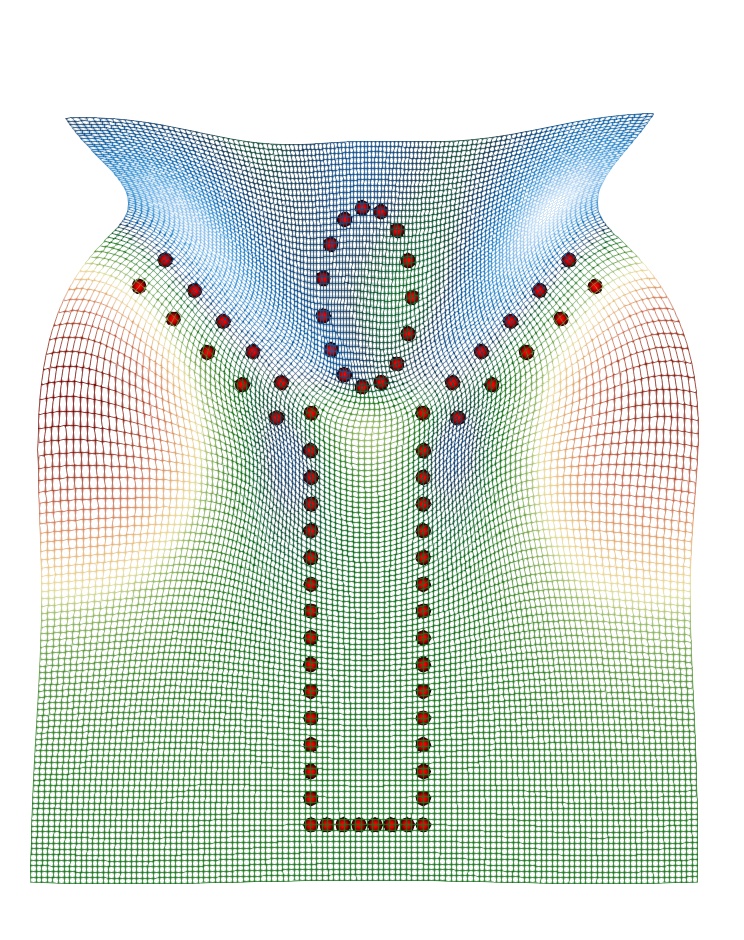}
\subcaption{$r_{1}$}
\end{subfigure}
\begin{subfigure}{.16\linewidth}
\centering
\includegraphics[width=.7\linewidth]{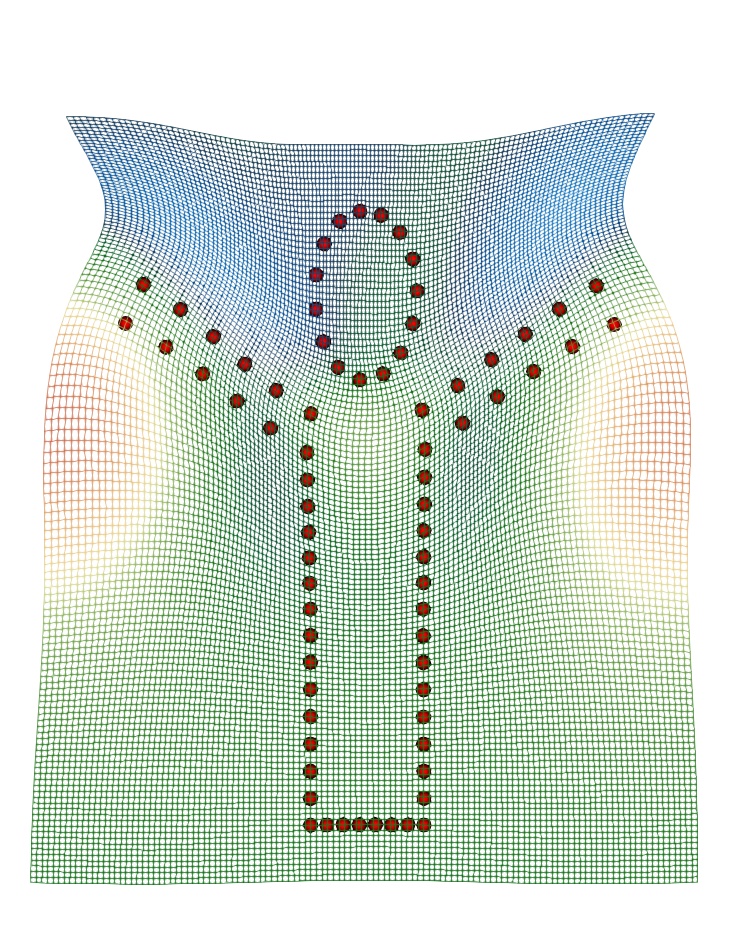}
\subcaption{$r_{5}$}

\end{subfigure}
\begin{subfigure}{.16\linewidth}
\centering
\includegraphics[width=.7\linewidth]{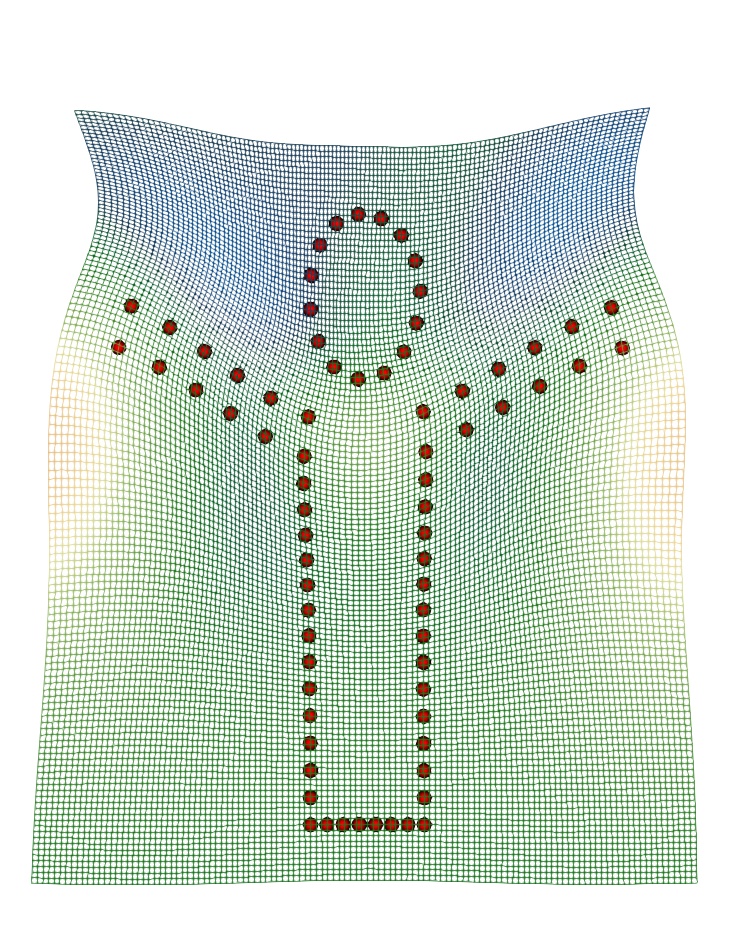}
\subcaption{$r_{10}$}

\end{subfigure}
\begin{subfigure}{.16\linewidth}
\centering
\includegraphics[width=.7\linewidth]{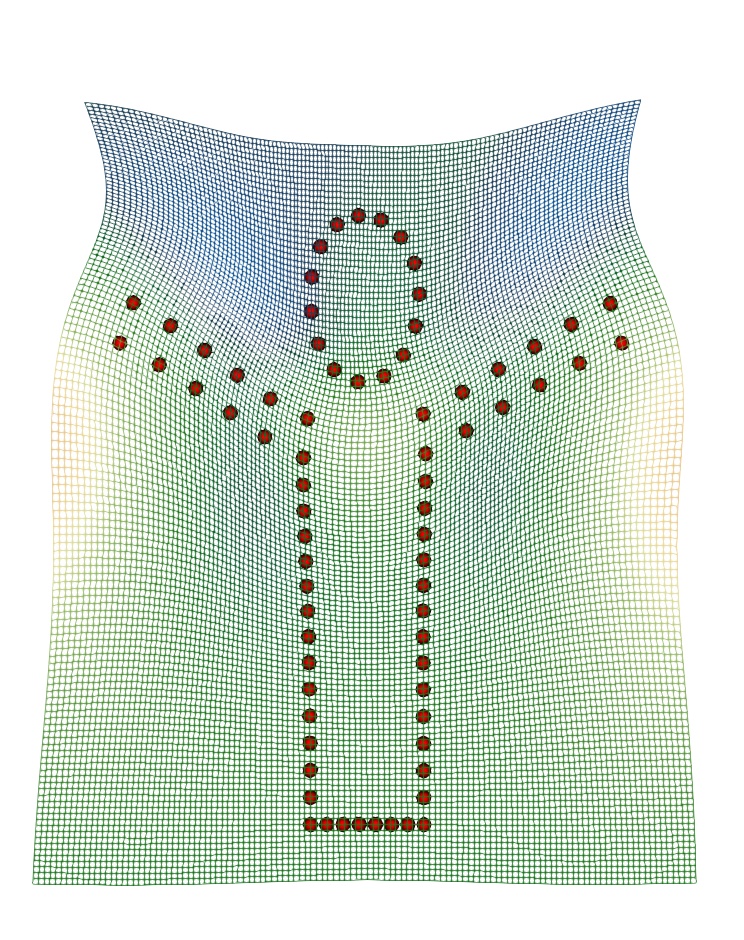}
\subcaption{$r_{13}$}
\end{subfigure}
\begin{subfigure}{.16\linewidth}
\centering
\includegraphics[width=.7\linewidth]{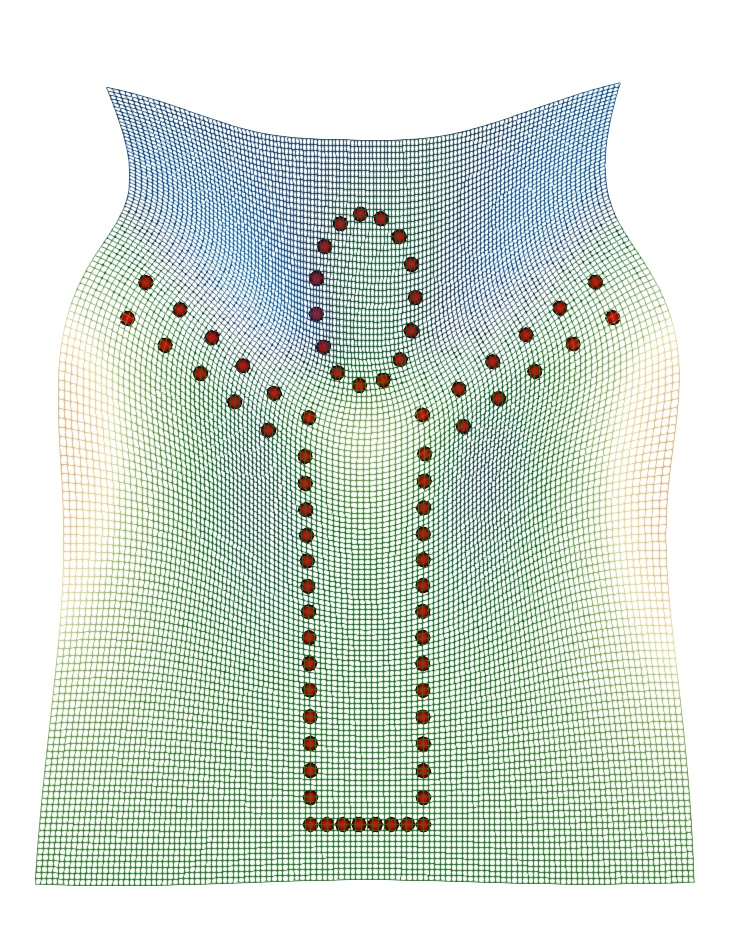}
\subcaption{$r_{17}$}
\end{subfigure}
\begin{subfigure}{.16\linewidth}
\centering
\includegraphics[width=.7\linewidth]{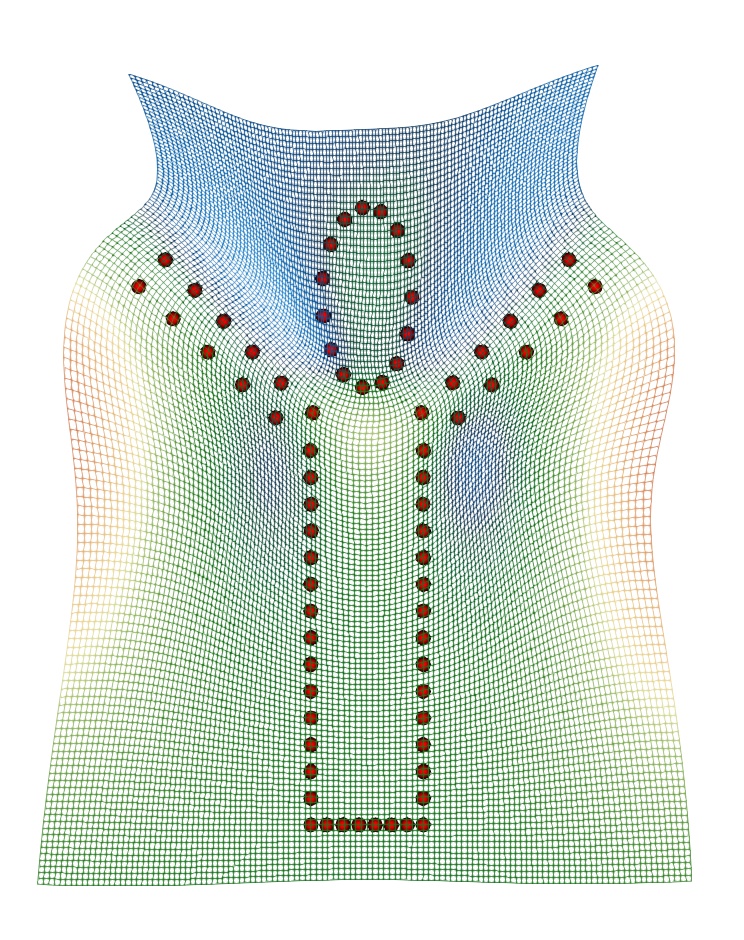}
\subcaption{$r_{20}$}
\end{subfigure}
\caption{Schematic human: deformation at different scales. Base scales: $r_1$, $r_{20}$.}
\label{fig: human ms results}
\end{figure}

\begin{figure}
\centering
\begin{subfigure}{.16\linewidth}
\centering
\includegraphics[width=.8\linewidth]{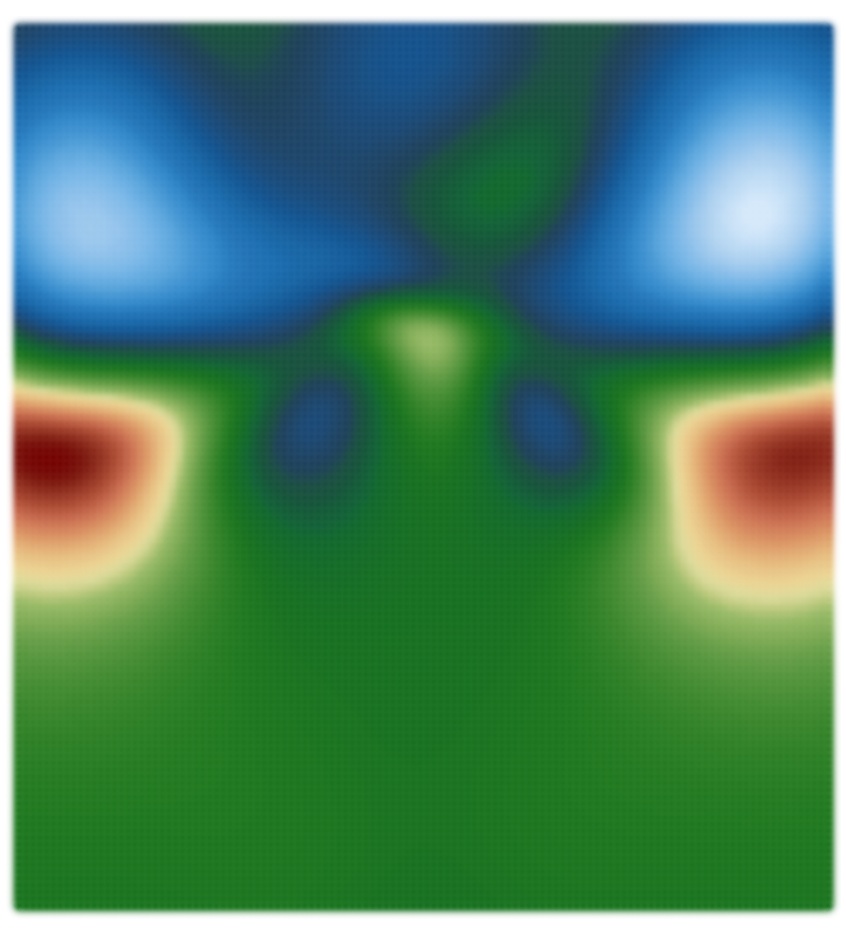}
\subcaption{$r_1$}
\end{subfigure}
\begin{subfigure}{.16\linewidth}
\centering
\includegraphics[width=.8\linewidth]{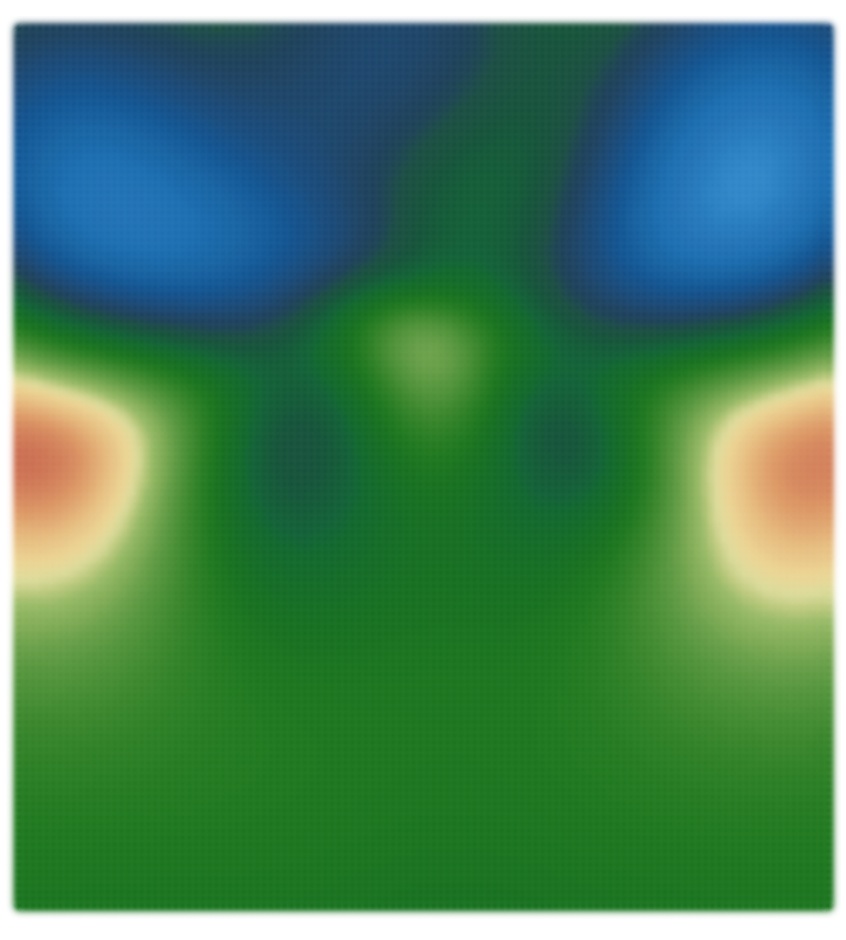}
\subcaption{$r_5$}
\end{subfigure}
\begin{subfigure}{.16\linewidth}
\centering
\includegraphics[width=.8\linewidth]{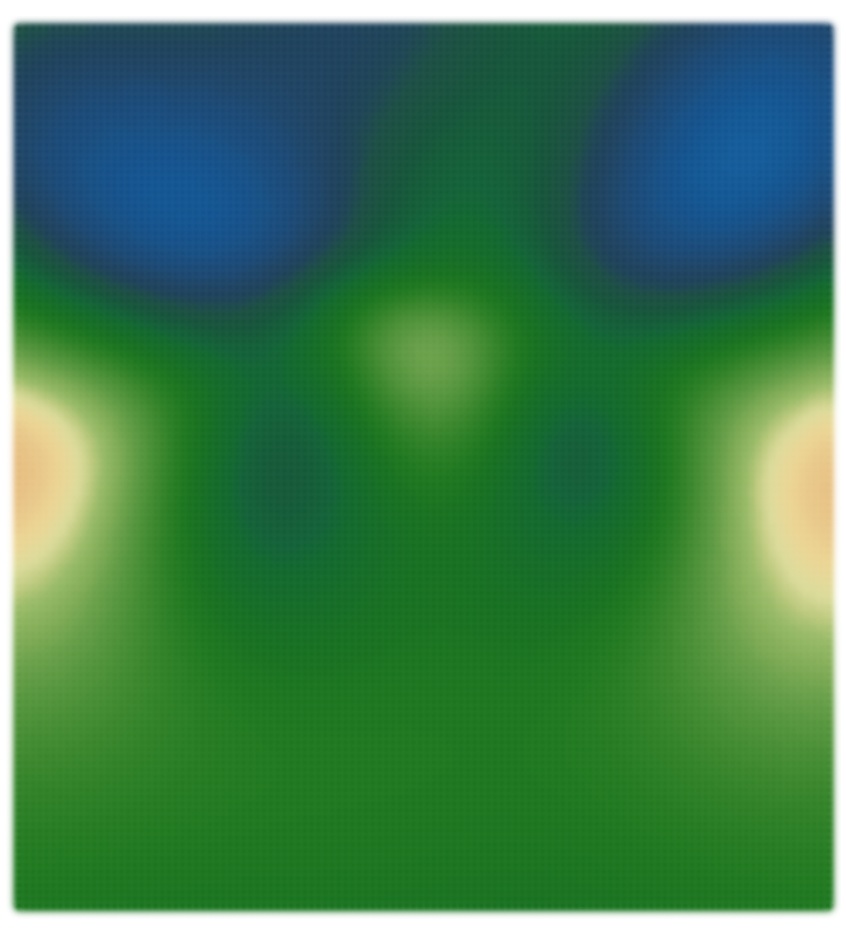}
\subcaption{$r_{10}$}
\end{subfigure}
\begin{subfigure}{.16\linewidth}
\centering
\includegraphics[width=.8\linewidth]{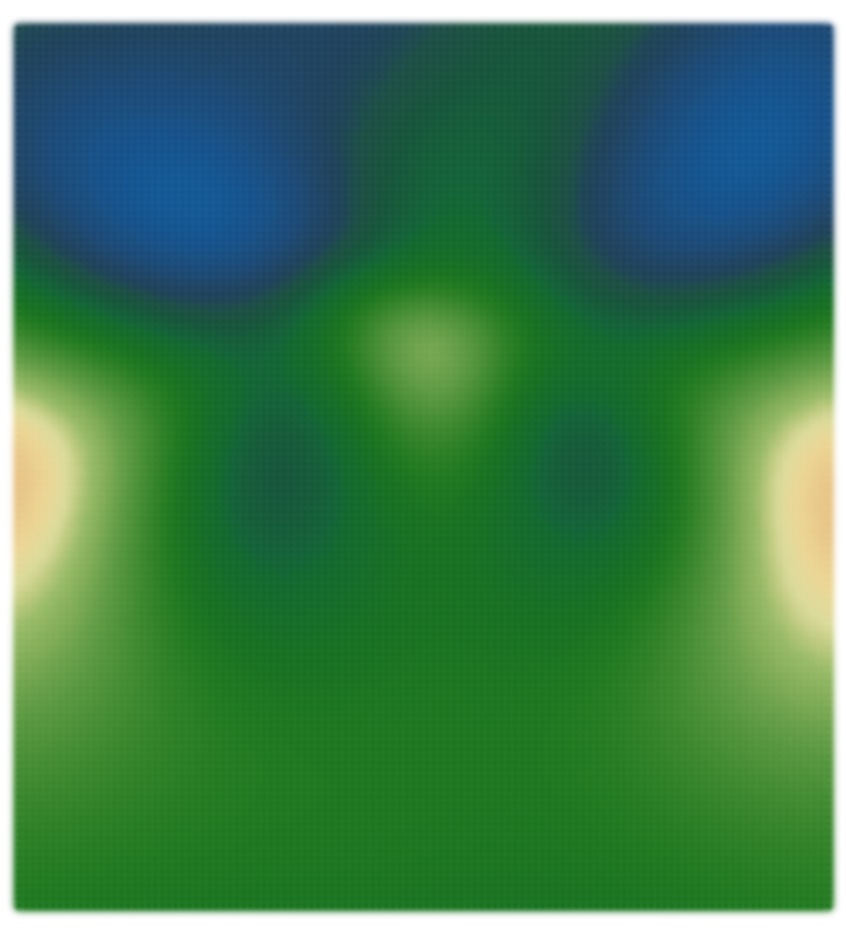}
\subcaption{$r_{13}$}
\end{subfigure}
\begin{subfigure}{.16\linewidth}
\centering
\includegraphics[width=.8\linewidth]{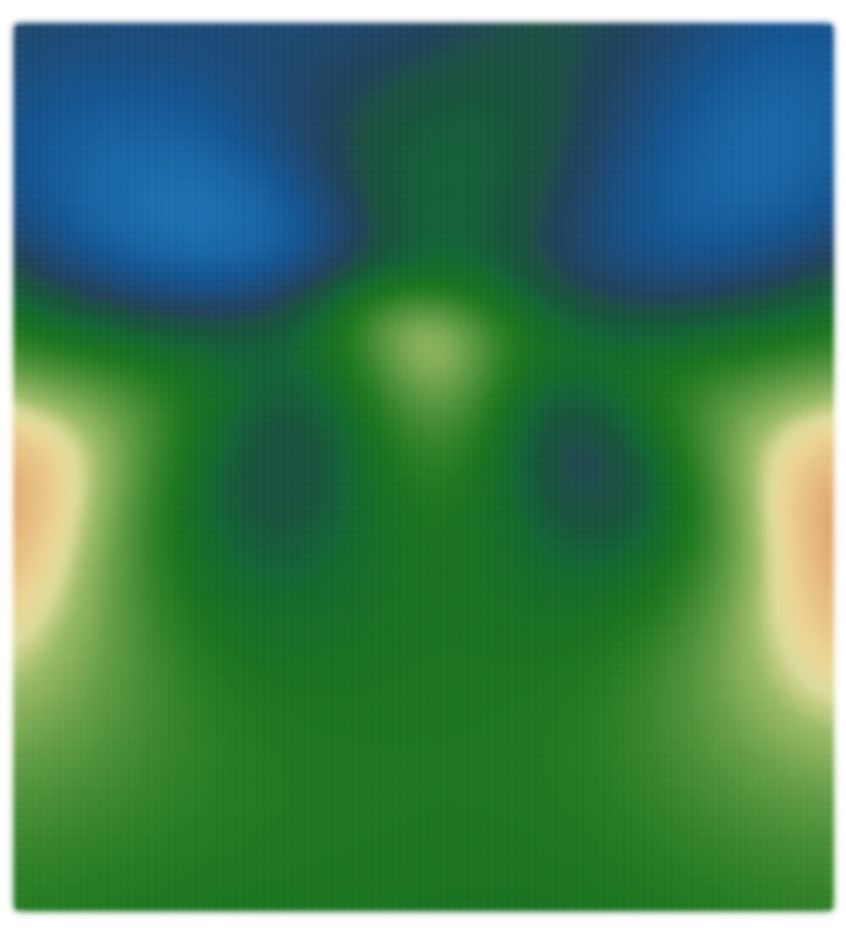}
\subcaption{$r_{17}$}
\end{subfigure}
\begin{subfigure}{.16\linewidth}
\centering
\includegraphics[width=.8\linewidth]{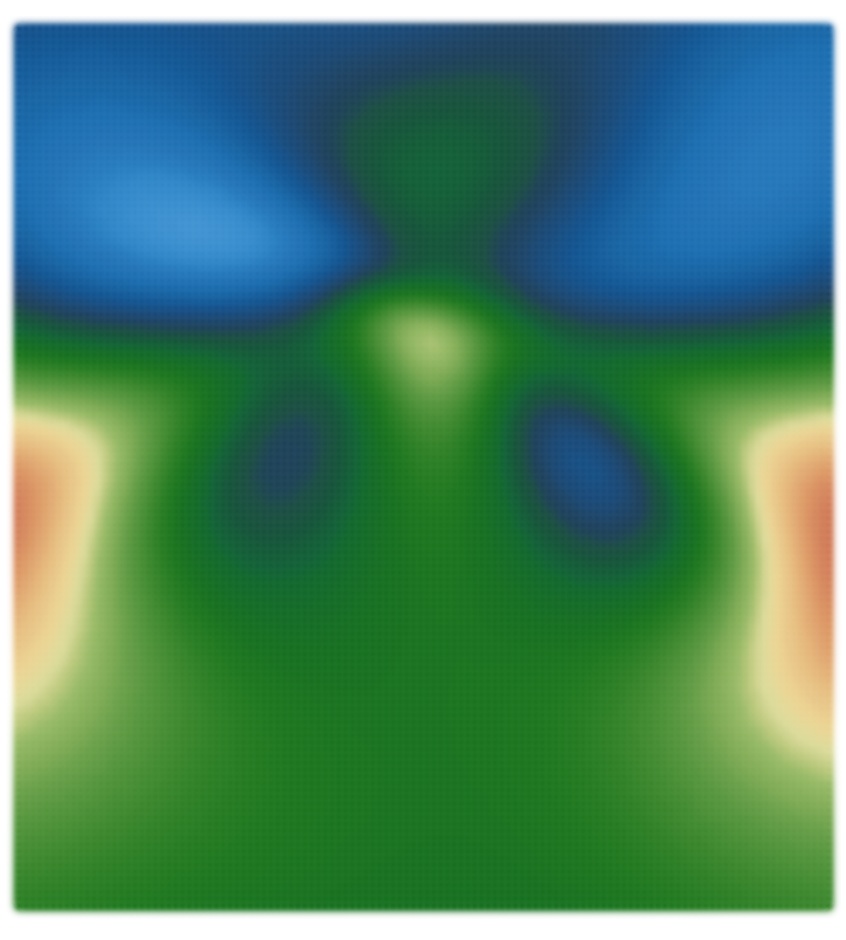}
\subcaption{$r_{20}$}
\end{subfigure}

\setcounter{subfigure}{0}
\begin{subfigure}{.135\textwidth}
\centering
\includegraphics[width=\textwidth]{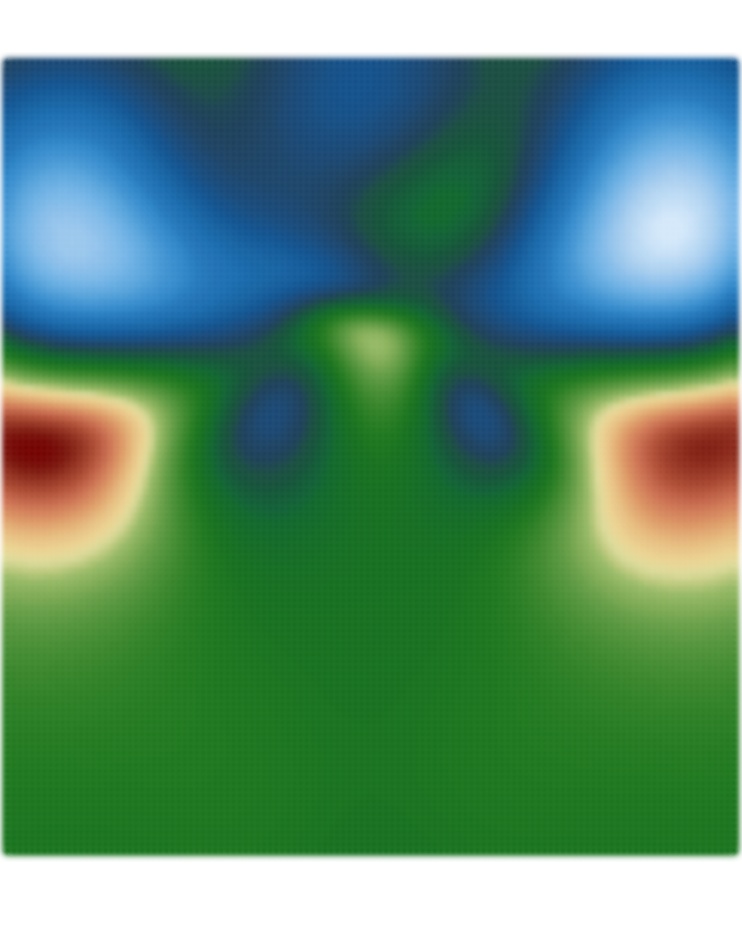}
\subcaption{$\rho^v_{r_1}$}
\end{subfigure}
\begin{subfigure}{.135\textwidth}
\centering
\includegraphics[width=\textwidth]{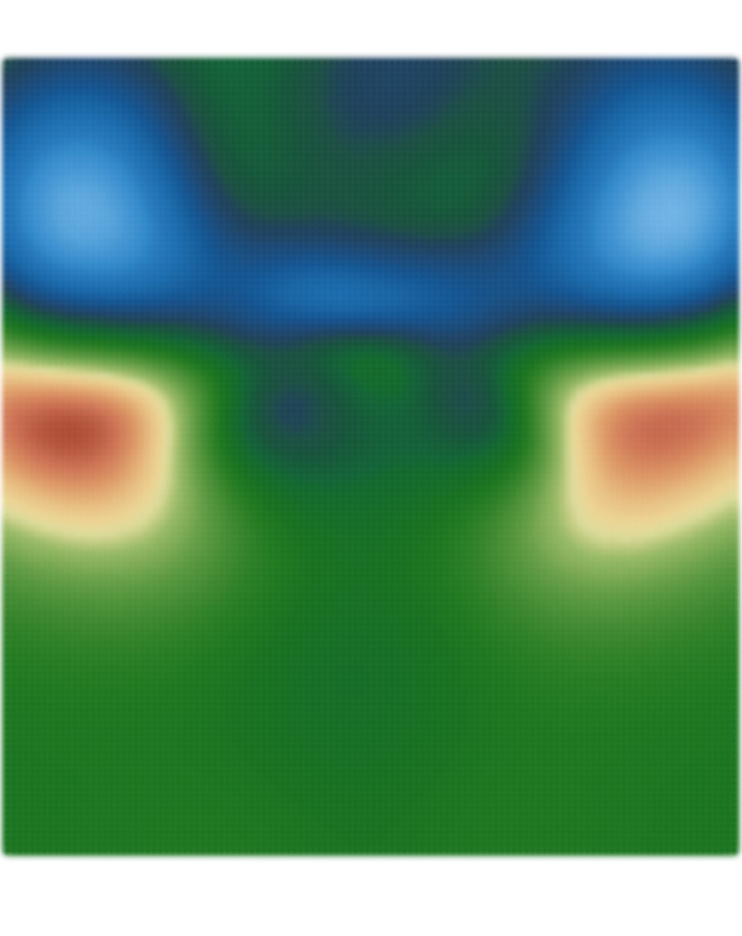}
\subcaption{$\rho^v_{r_2}$}
\end{subfigure}
\begin{subfigure}{.135\textwidth}
\centering
\includegraphics[width=\textwidth]{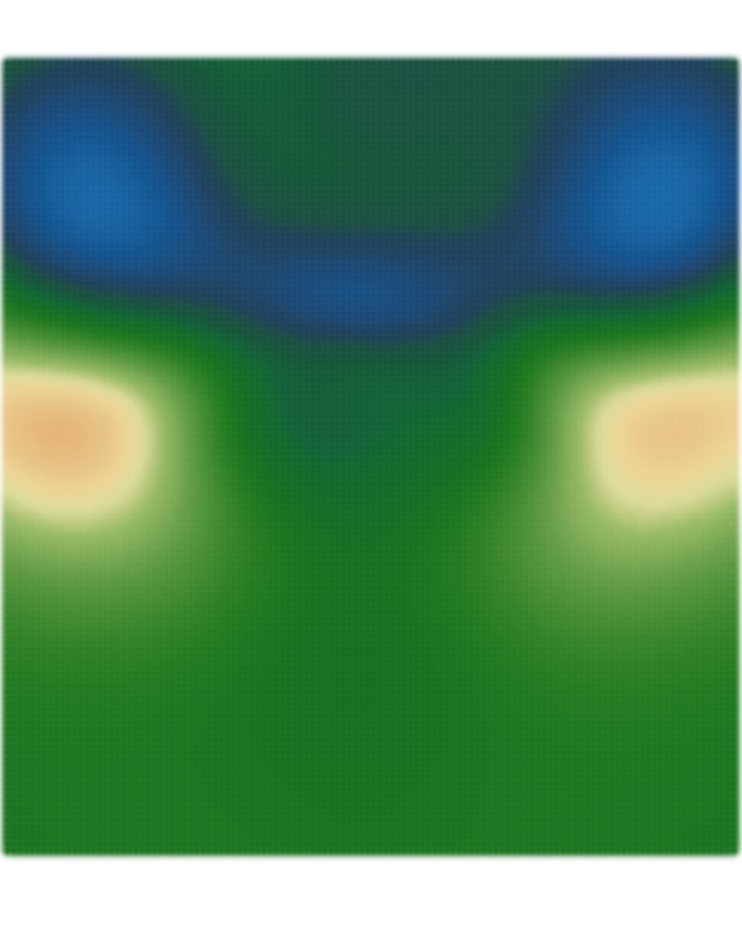}
\subcaption{$\rho^v_{r_{5}}$}
\end{subfigure}
\begin{subfigure}{.135\textwidth}
\centering
\includegraphics[width=\textwidth]{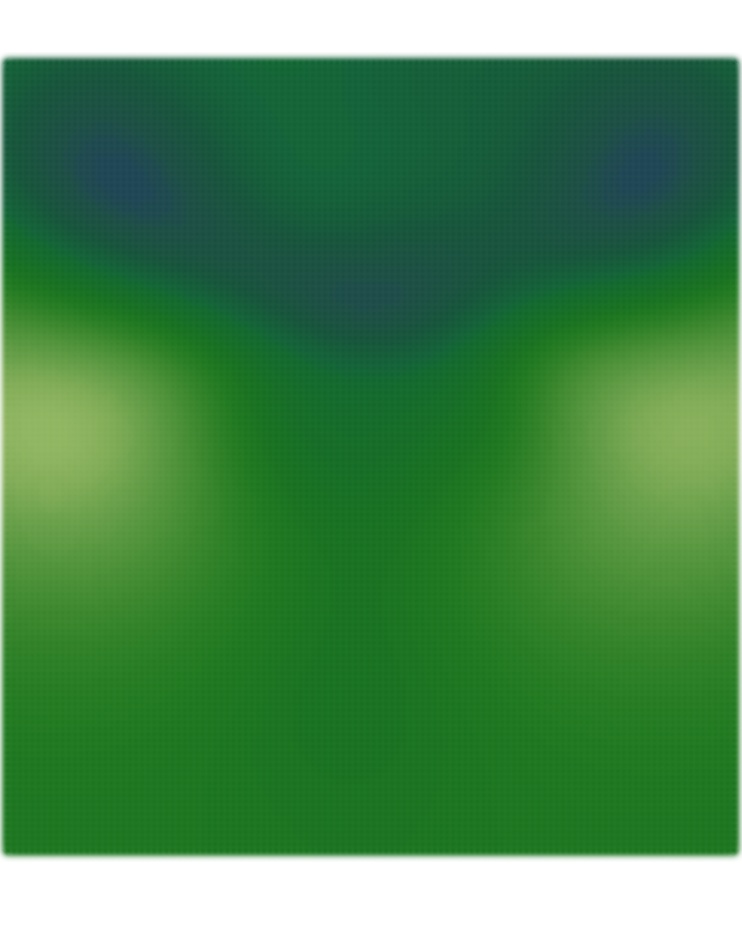}
\subcaption{$\rho^v_{r_{10}}$}
\end{subfigure}
\begin{subfigure}{.135\textwidth}
\centering
\includegraphics[width=\textwidth]{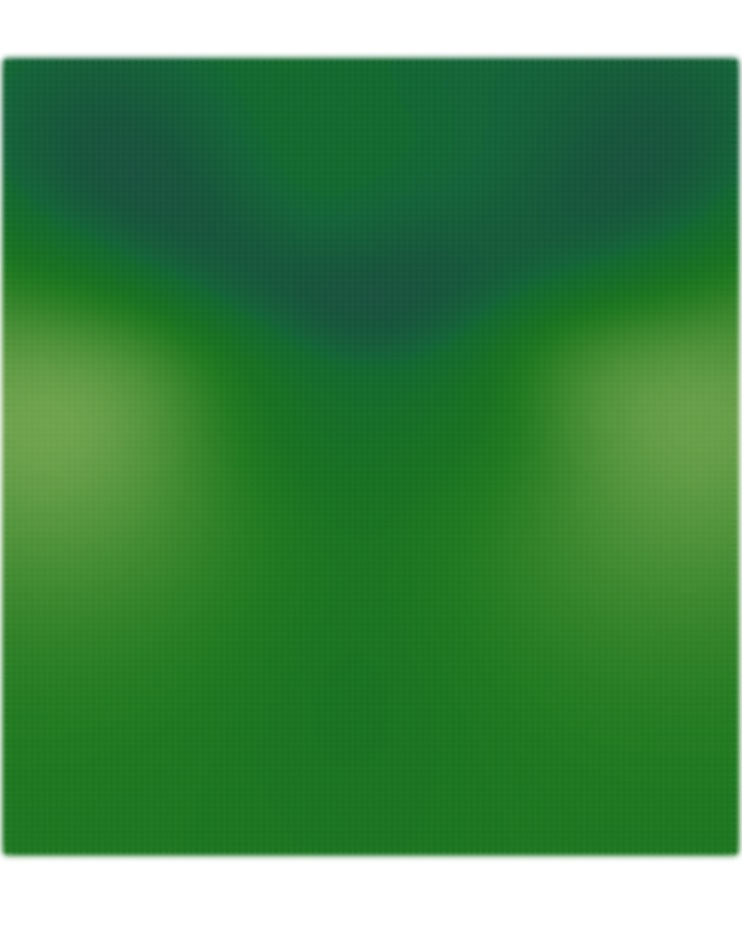}
\subcaption{$\rho^v_{r_{13}}$}
\end{subfigure}
\begin{subfigure}{.135\textwidth}
\centering
\includegraphics[width=\textwidth]{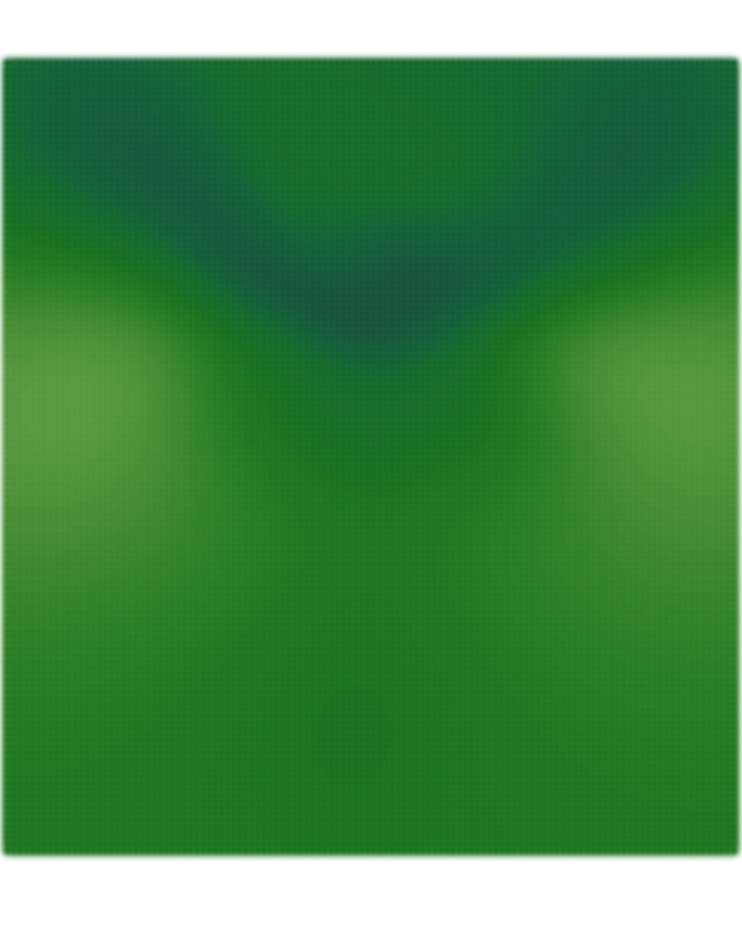}
\subcaption{$\rho^v_{r_{17}}$}
\end{subfigure}
\begin{subfigure}{.135\textwidth}
\centering
\includegraphics[width=\textwidth]{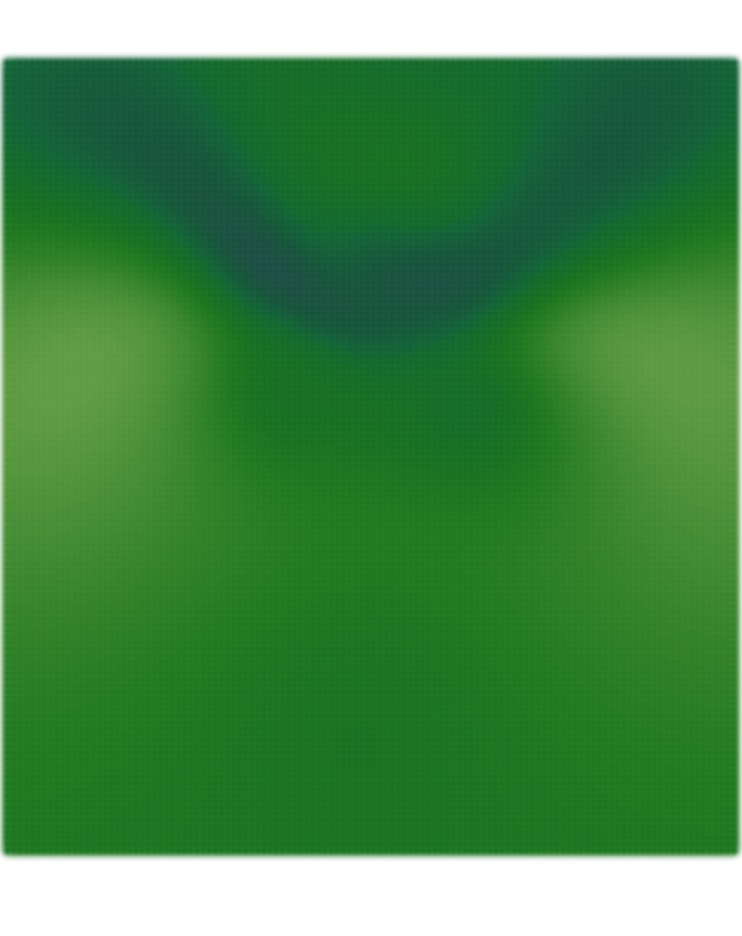}
\subcaption{$\rho^v_{r_{20}}$}
\end{subfigure}
\caption{Schematic human: First row: log Jacobian; Second row: log Jacobian of residuals $\rho^v_{r_i} = \psi^v_{r_{i}}\circ (\psi^v_{r_{i-1}})^{-1}$. Base scales: $r_1, r_{20}$.}
\label{fig: human ms residual}
\end{figure}

We plot the residuals between consecutive discretized scales $\psi^v_{r_{k+1}}\circ (\psi^v_{r_i})^{-1}$ for different $i$ on the original grid. As shown in \cref{fig: human ms residual} (second row), we can observe the colors are deeper as we are approaching the base scales. But the magnitude is clearly larger for the scales closer to the finer base scale, which indicates a sharper deformation as the scale varies from the finer base scale than the coarser one. As 
\[
\psi^v_{r_{k}}=\rho^v_{r_{k}}\circ\dots \circ \rho^v_{r_{1}},
\]
the residuals provide decompositions of the deformation at  coarser scales. 

As a comparison, we also applied the MS-LDDMM model, with matching enforced at a single base scale, $r_1$ and $r_{20}$ respectively. For each base scale, we compute the multiscale deformations as well as the Jacobians, which can be seen from  \cref{fig: human fine scale,fig: human coarse scale logJacobian} for the case of $r_1$, and  \cref{fig: human coarse scale,fig: human coarse scale logJacobian} for the case of $r_{20}$. In all example, we observe a dampening of the deformation (which gets closer to the identity) as scales  become further away from the base scales, with a stronger effect with a single base scale than with two.

\begin{figure}
\centering
\begin{subfigure}{.16\linewidth}
\centering
\includegraphics[width=.7\textwidth]{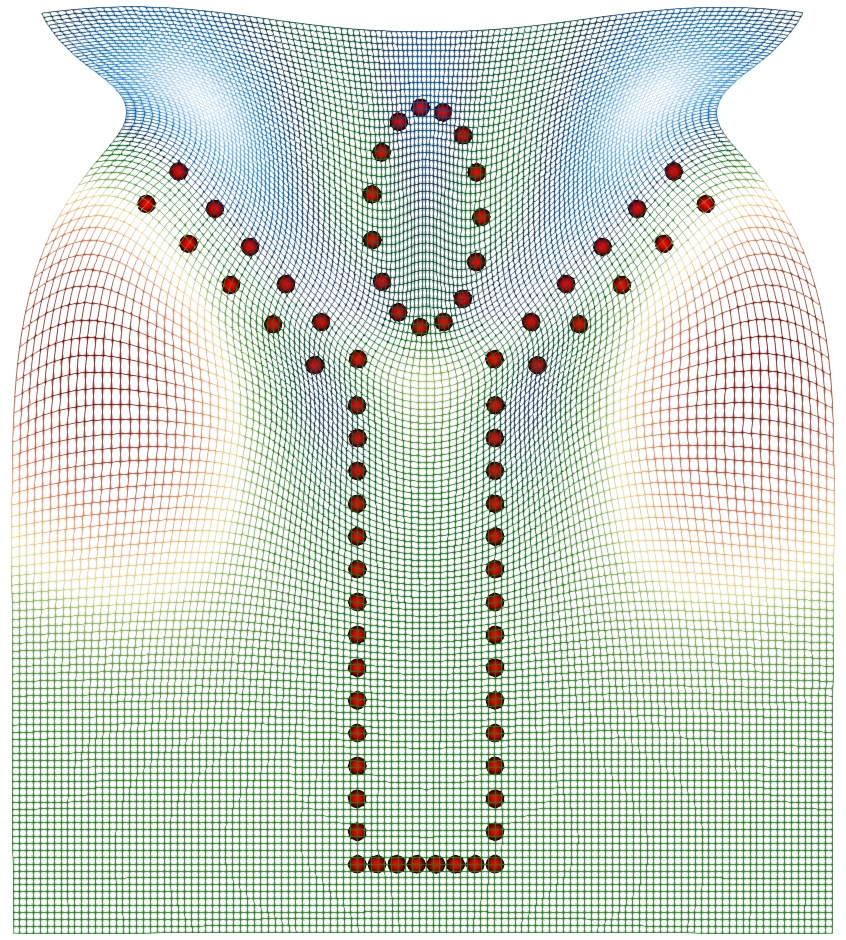}
\subcaption{$r_{1}$}
\end{subfigure}
\begin{subfigure}{.16\linewidth}
\centering
\includegraphics[width=.7\textwidth]{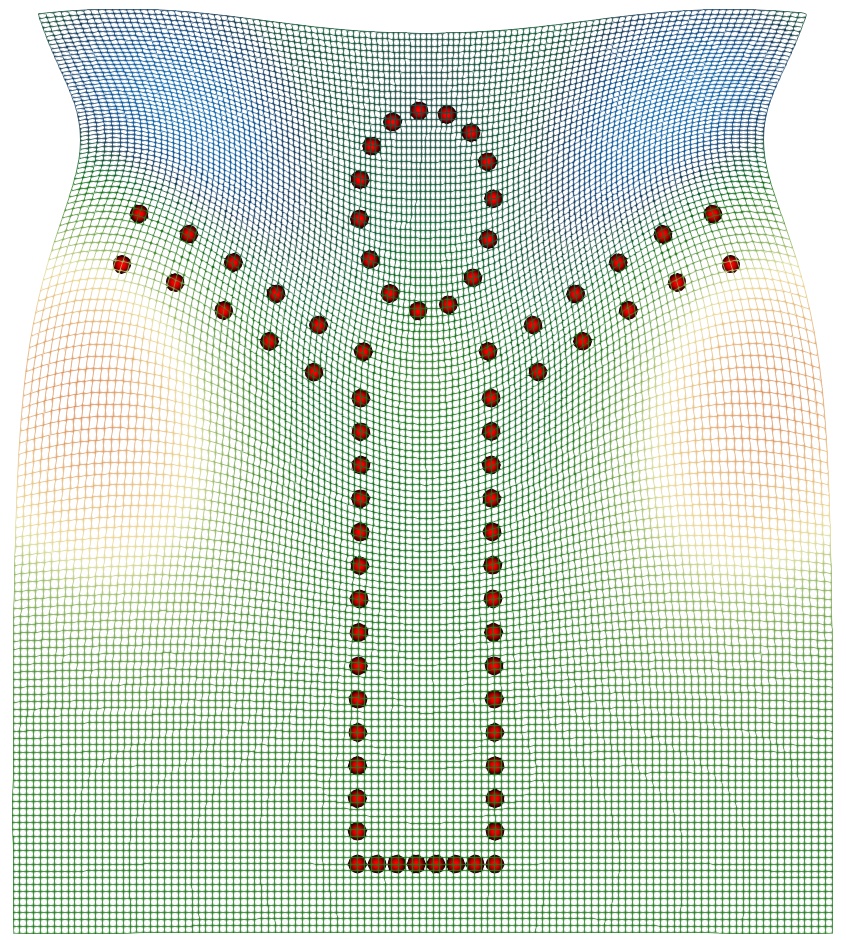}
\subcaption{$r_{5}$}
\end{subfigure}
\begin{subfigure}{.16\linewidth}
\centering
\includegraphics[width=.7\textwidth]{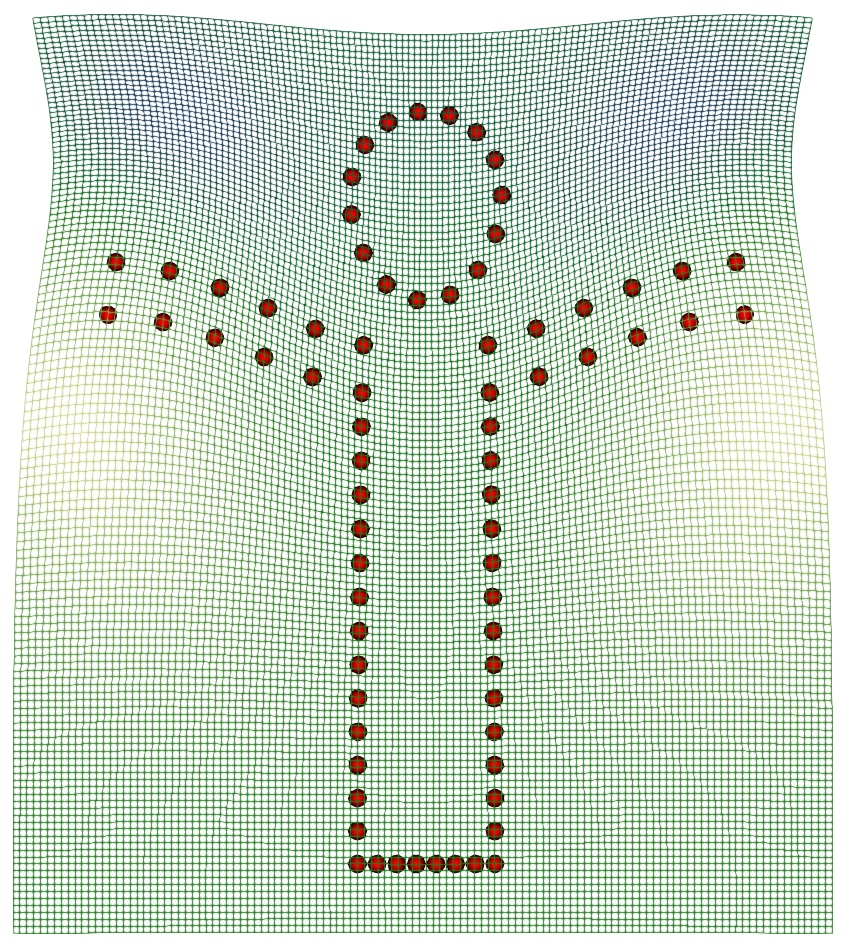}
\subcaption{$r_{9}$}
\end{subfigure}
\begin{subfigure}{.16\linewidth}
\centering
\includegraphics[width=.7\textwidth]{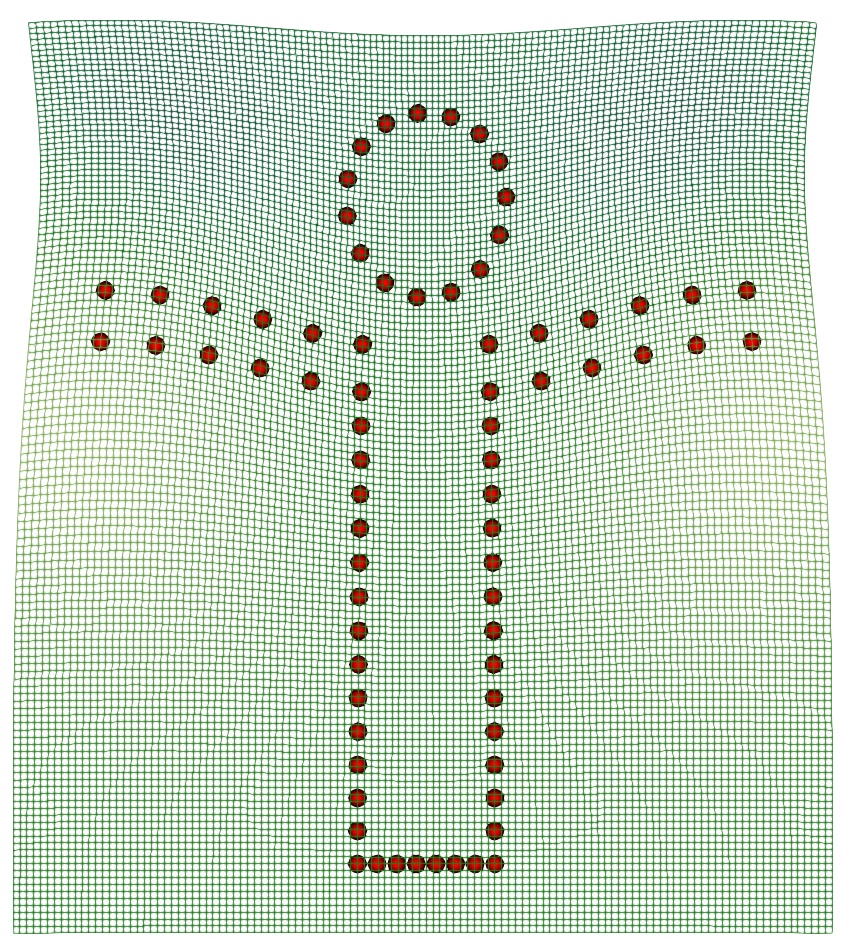}
\subcaption{$r_{13}$}
\end{subfigure}
\begin{subfigure}{.16\linewidth}
\centering
\includegraphics[width=.7\textwidth]{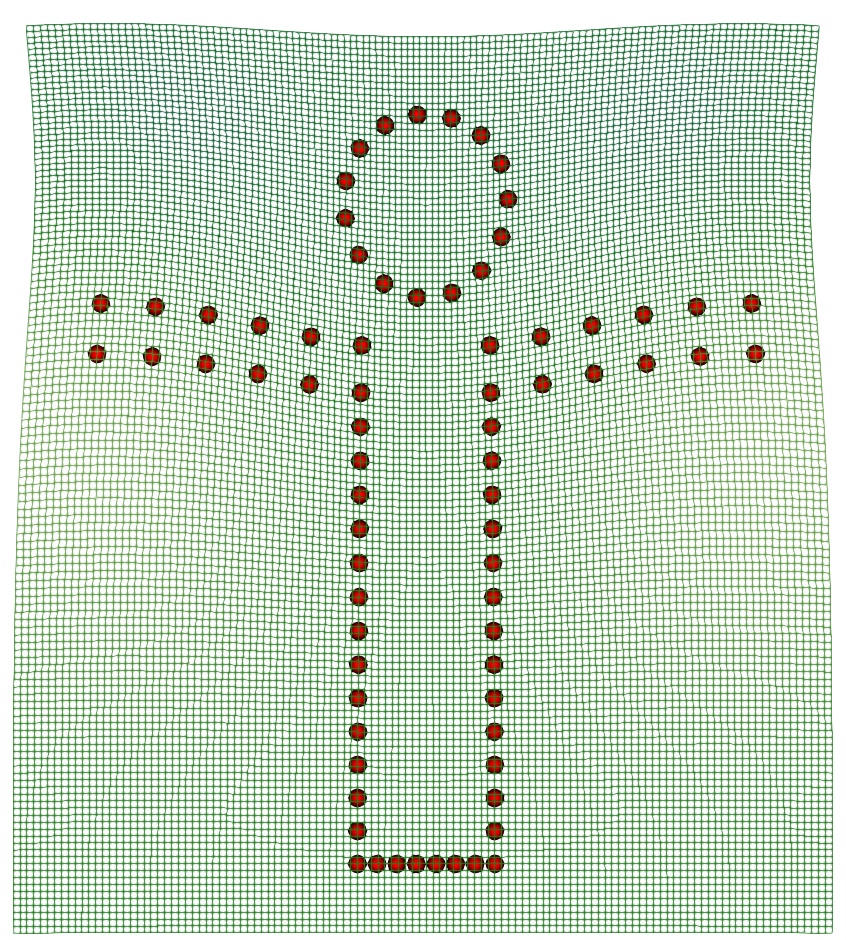}
\subcaption{$r_{17}$}
\end{subfigure}
\begin{subfigure}{.16\linewidth}
\centering
\includegraphics[width=.7\textwidth]{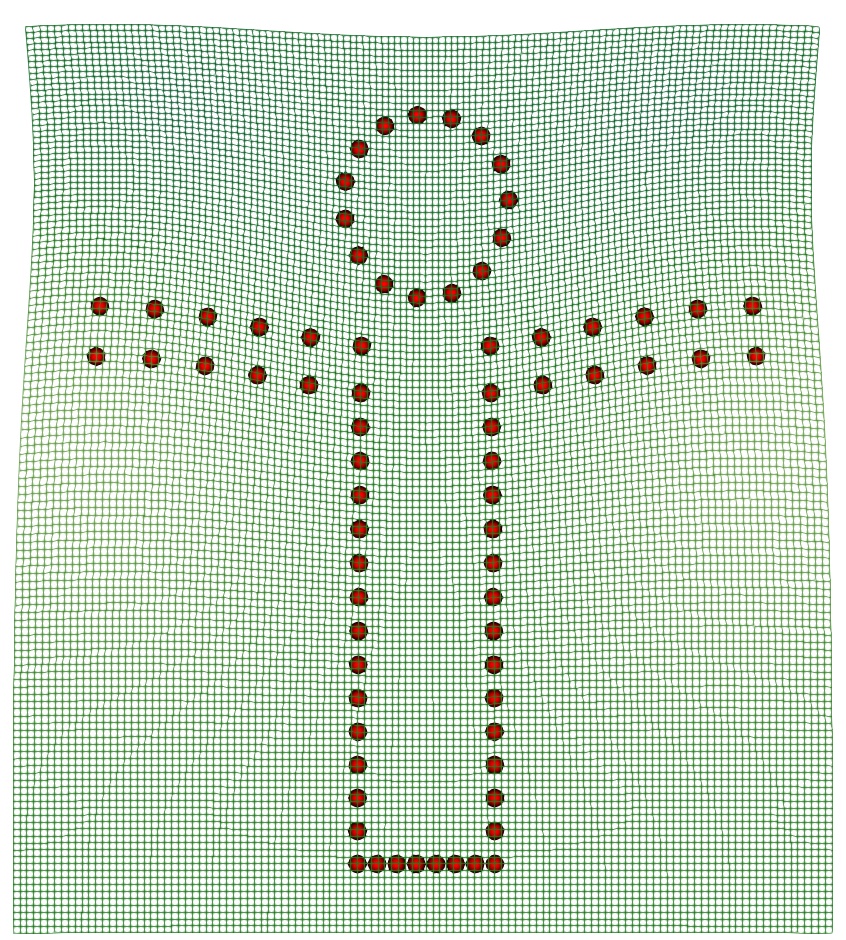}
\subcaption{$r_{20}$}
\end{subfigure}
\caption{Schematic human: deformations. Base scale: $r_1$.}
\label{fig: human fine scale}
\end{figure}

\begin{figure}
\centering
\begin{subfigure}{.16\linewidth}
\centering
\includegraphics[width=.7\textwidth]{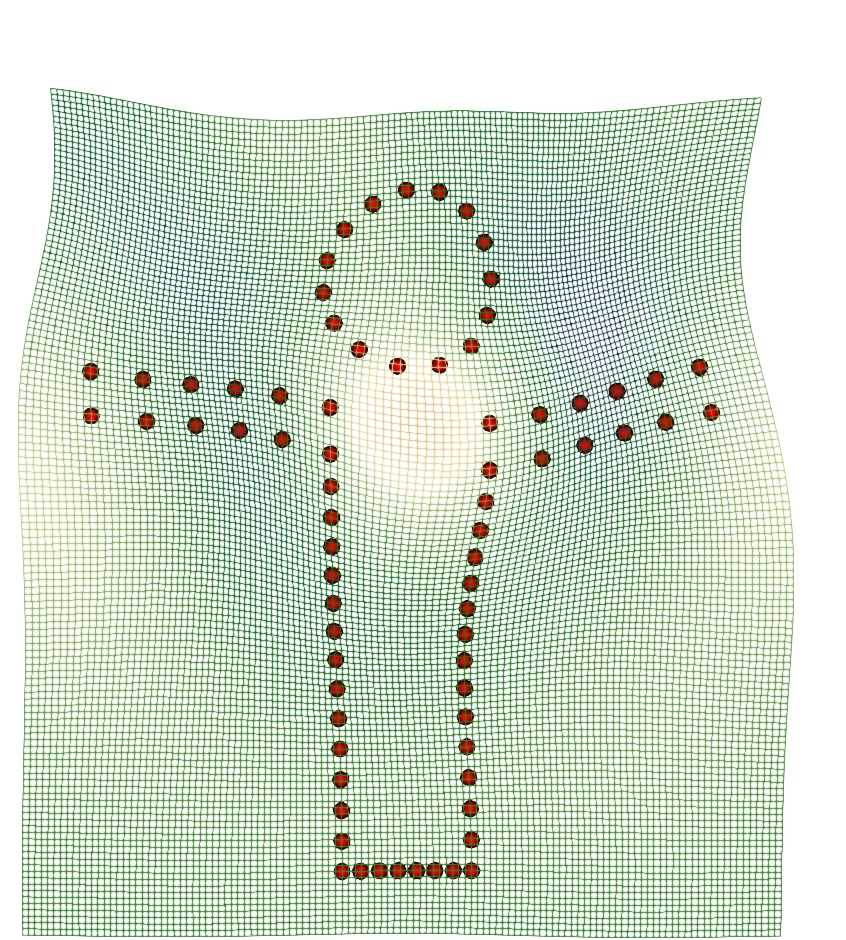}
\subcaption{$r_{1}$}
\end{subfigure}
\begin{subfigure}{.16\linewidth}
\centering
\includegraphics[width=.7\textwidth]{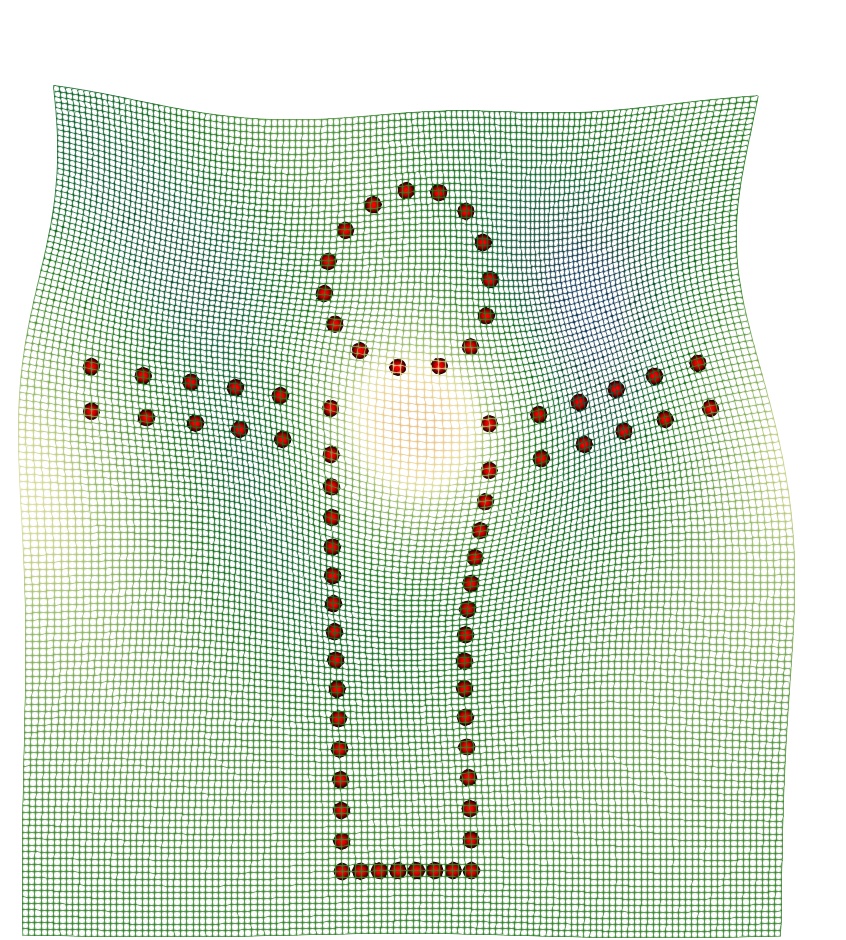}
\subcaption{$r_{5}$}
\end{subfigure}
\begin{subfigure}{.16\linewidth}
\centering
\includegraphics[width=.7\textwidth]{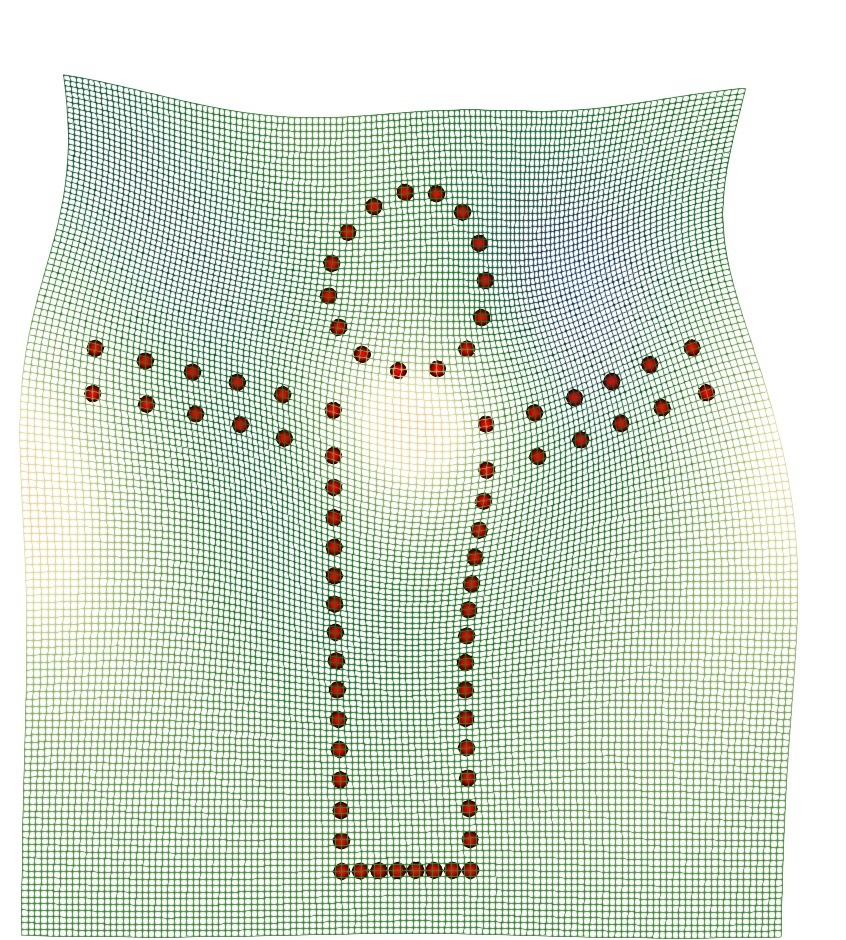}
\subcaption{$r_{9}$}
\end{subfigure}
\begin{subfigure}{.16\linewidth}
\centering
\includegraphics[width=.7\textwidth]{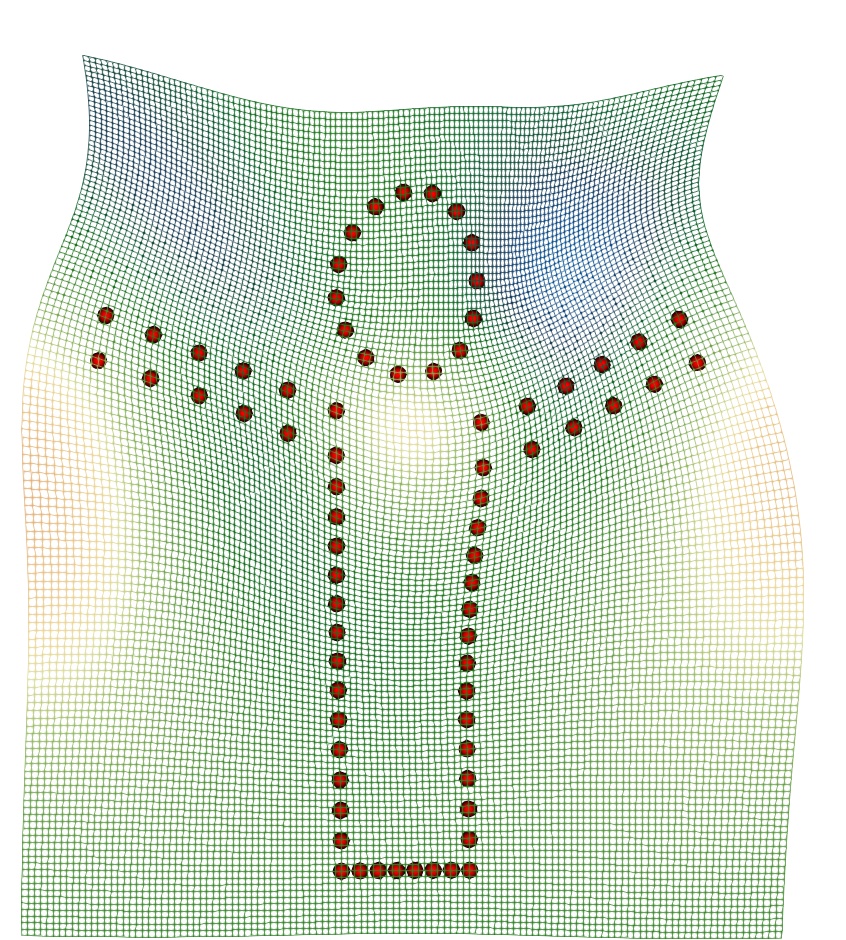}
\subcaption{$r_{13}$}
\end{subfigure}
\begin{subfigure}{.16\linewidth}
\centering
\includegraphics[width=.7\textwidth]{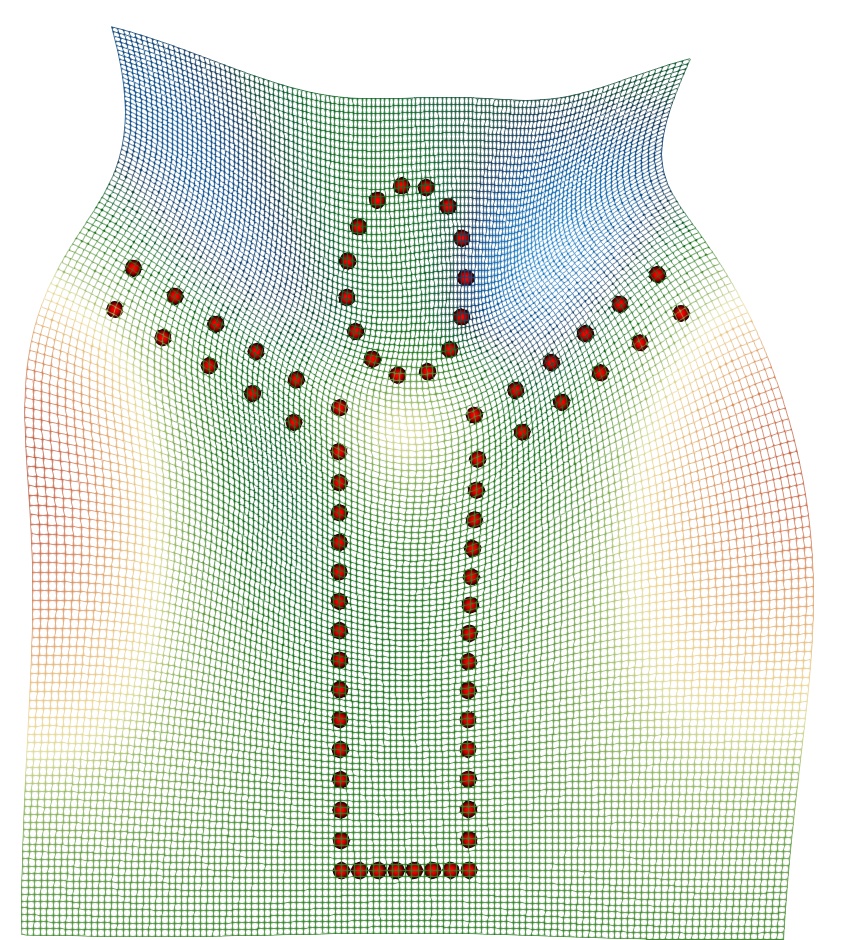}
\subcaption{$r_{17}$}
\end{subfigure}
\begin{subfigure}{.16\linewidth}
\centering
\includegraphics[width=.7\textwidth]{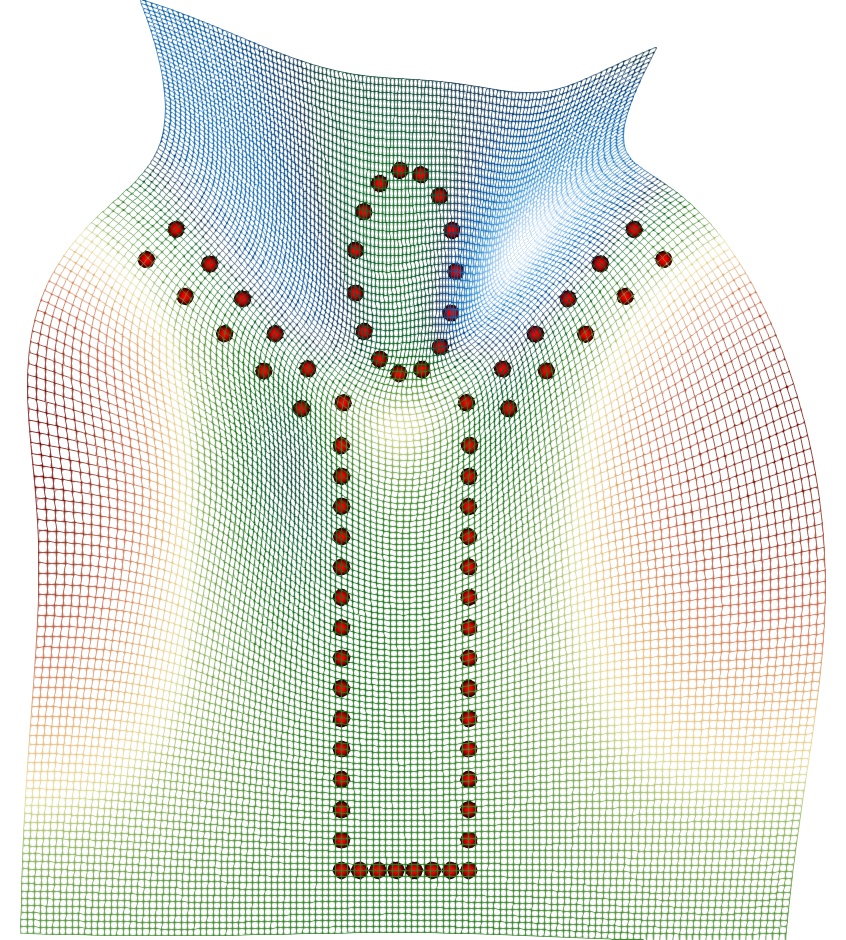}
\subcaption{$r_{20}$}
\end{subfigure}
\caption{Schematic human: deformations. Base scale: $r_{20}$.}
\label{fig: human coarse scale}
\end{figure}

\begin{figure}
\centering
\begin{subfigure}{.16\linewidth}
\centering
\includegraphics[width=.7\textwidth]{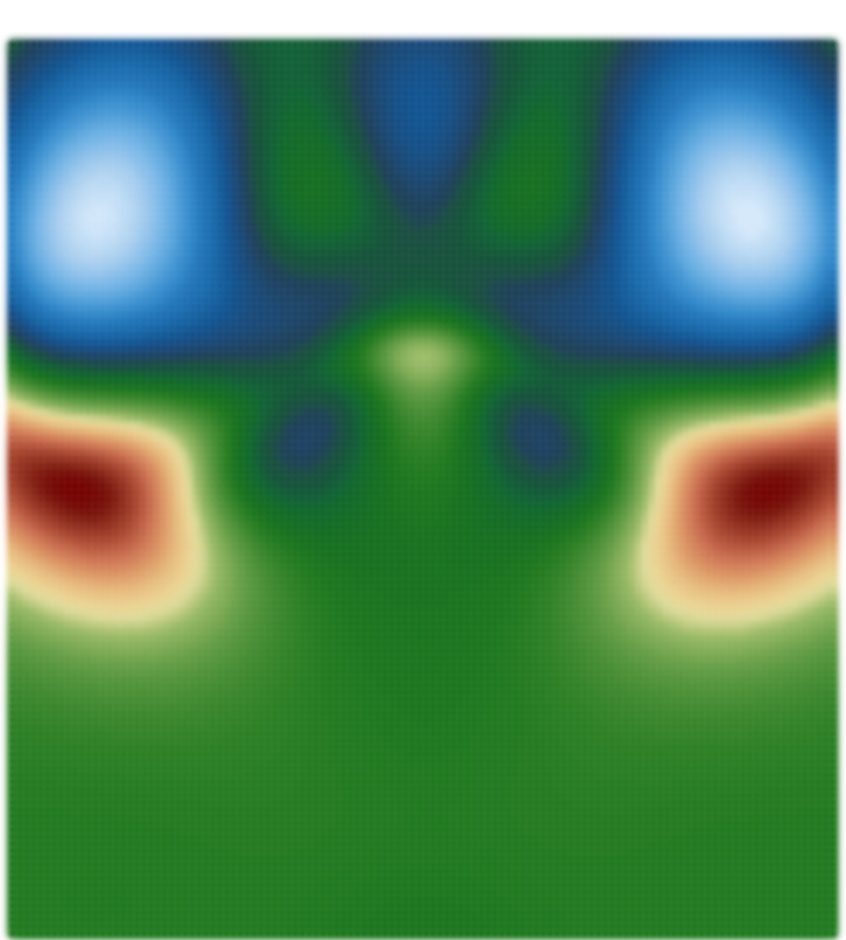}
\subcaption{$r_{1}$}
\end{subfigure}
\begin{subfigure}{.16\linewidth}
\centering
\includegraphics[width=.7\textwidth]{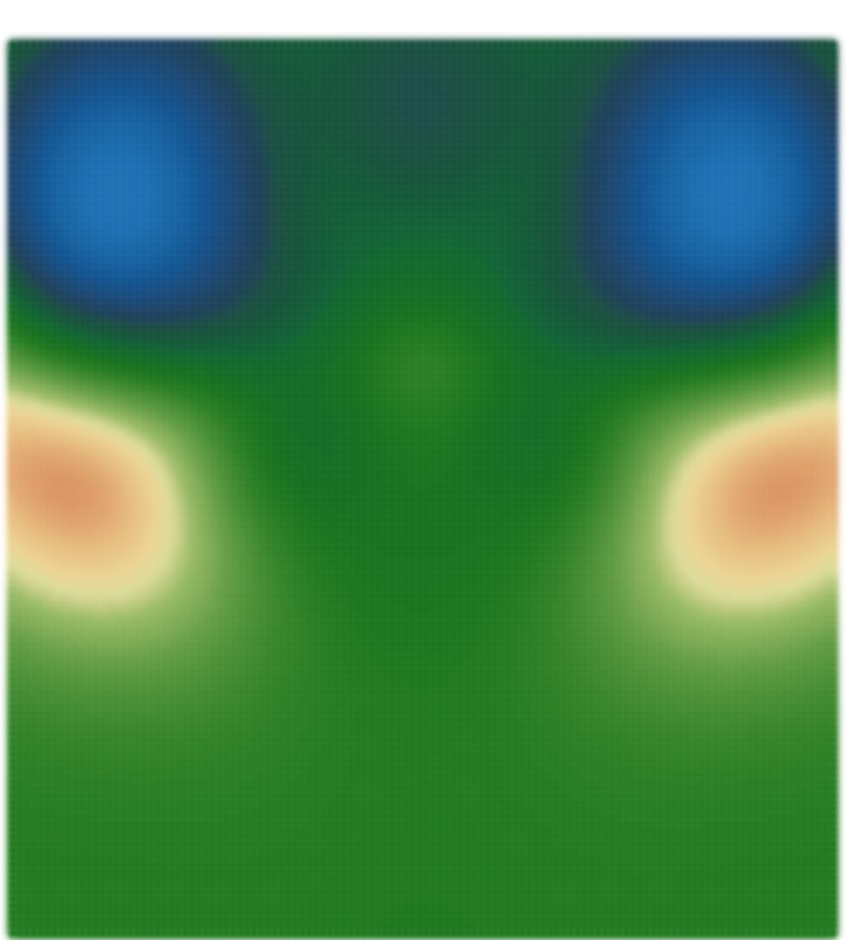}
\subcaption{$r_{5}$}
\end{subfigure}
\begin{subfigure}{.16\linewidth}
\centering
\includegraphics[width=.7\textwidth]{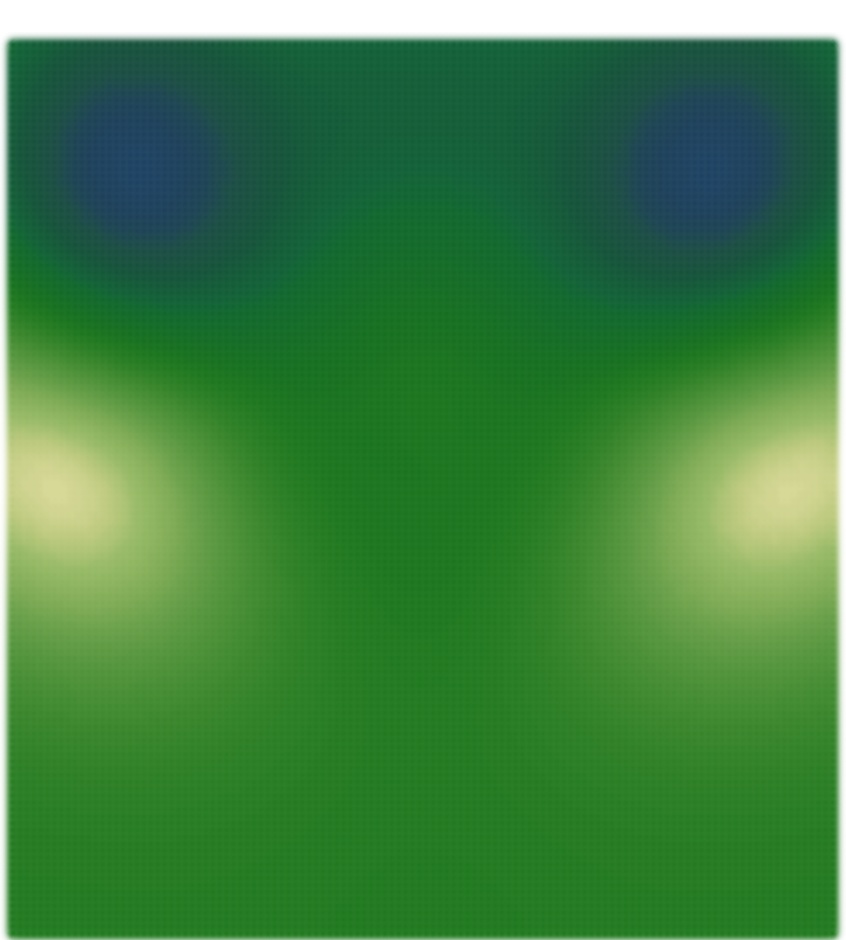}
\subcaption{$r_{9}$}
\end{subfigure}
\begin{subfigure}{.16\linewidth}
\centering
\includegraphics[width=.7\textwidth]{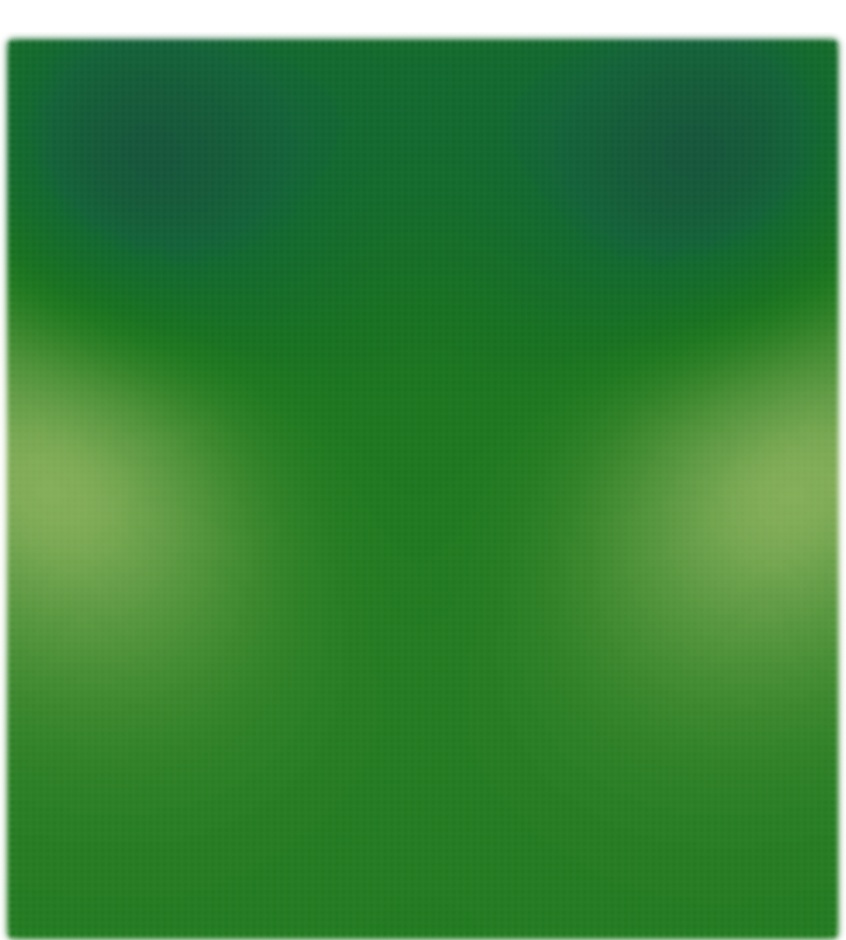}
\subcaption{$r_{13}$}
\end{subfigure}
\begin{subfigure}{.16\linewidth}
\centering
\includegraphics[width=.7\textwidth]{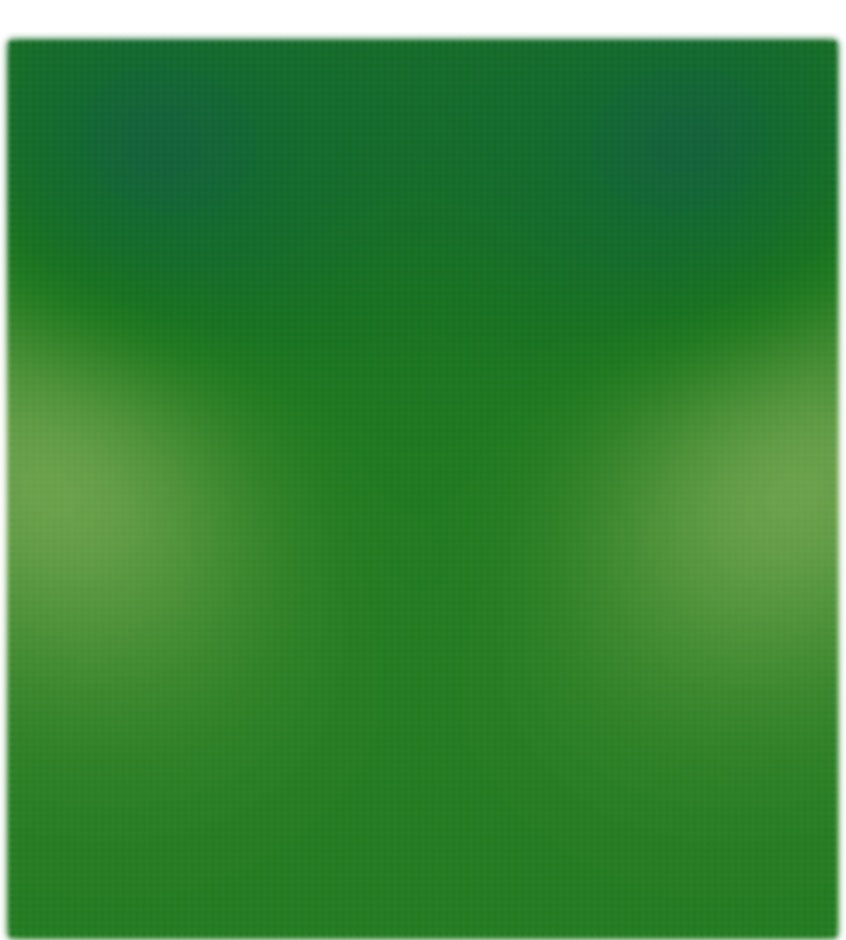}
\subcaption{$r_{17}$}
\end{subfigure}
\begin{subfigure}{.16\linewidth}
\centering
\includegraphics[width=.7\textwidth]{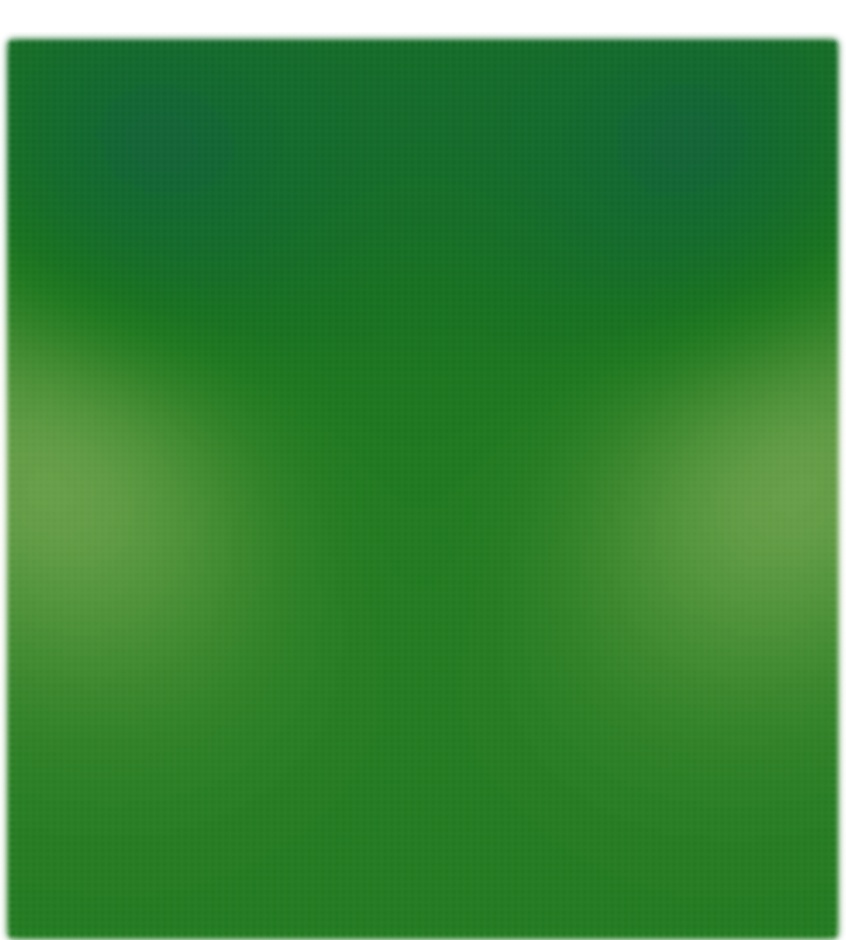}
\subcaption{$r_{20}$}
\end{subfigure}

\setcounter{subfigure}{0}
\begin{subfigure}{.16\linewidth}
\centering
\includegraphics[width=.7\textwidth]{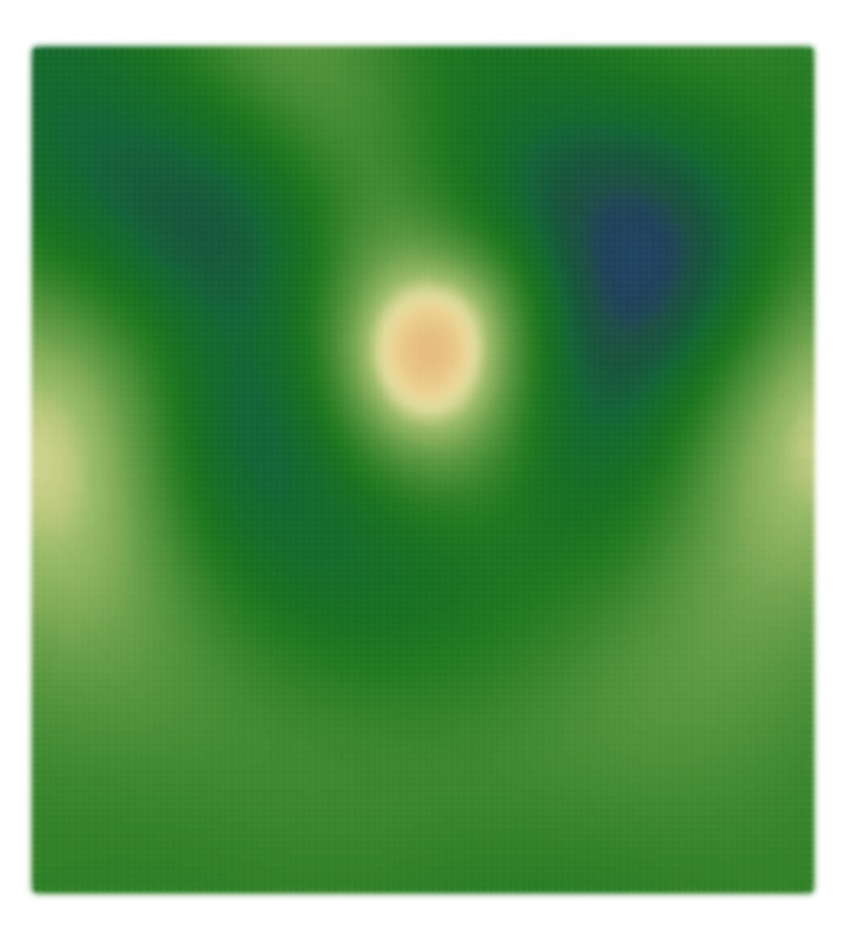}
\subcaption{$r_{1}$}
\end{subfigure}
\begin{subfigure}{.16\linewidth}
\centering
\includegraphics[width=.7\textwidth]{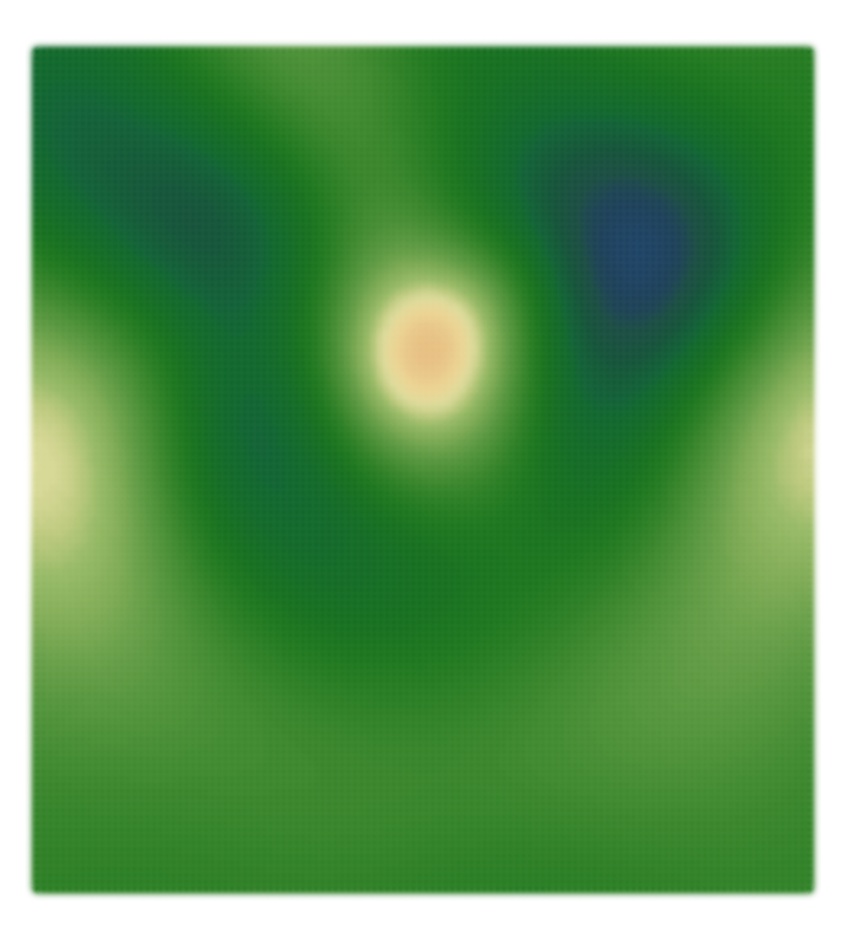}
\subcaption{$r_{5}$}
\end{subfigure}
\begin{subfigure}{.16\linewidth}
\centering
\includegraphics[width=.7\textwidth]{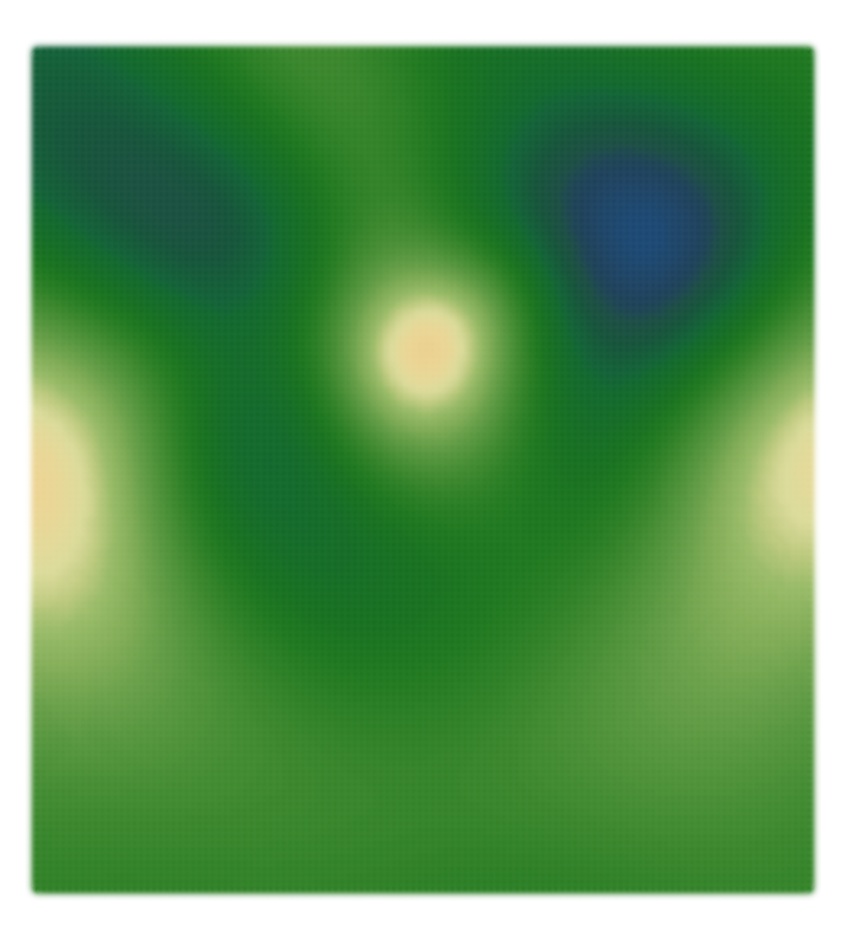}
\subcaption{$r_{9}$}
\end{subfigure}
\begin{subfigure}{.16\linewidth}
\centering
\includegraphics[width=.7\textwidth]{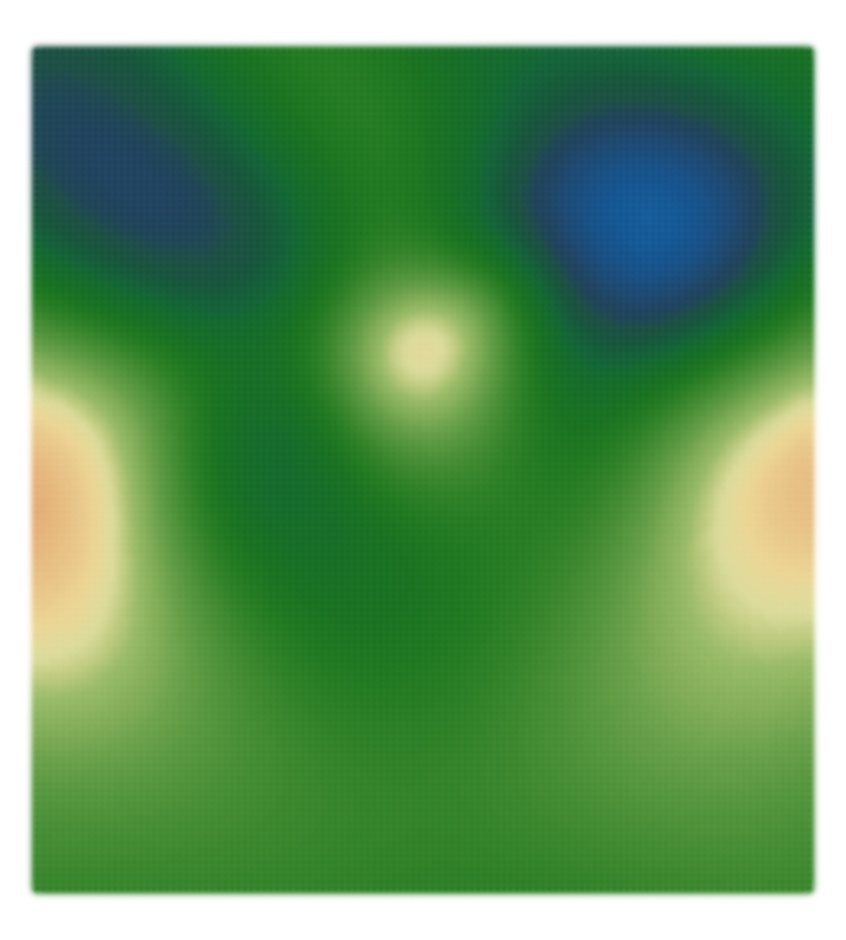}
\subcaption{ $r_{13}$}
\end{subfigure}
\begin{subfigure}{.16\linewidth}
\centering
\includegraphics[width=.7\textwidth]{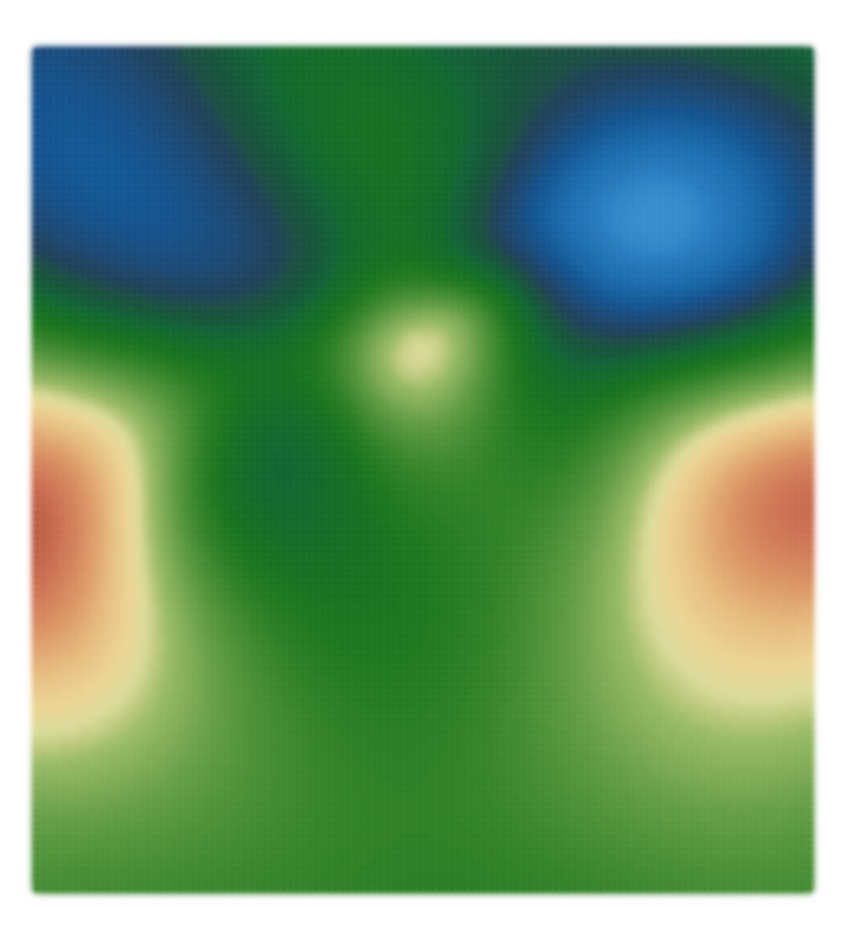}
\subcaption{$r_{17}$}
\end{subfigure}
\begin{subfigure}{.16\linewidth}
\centering
\includegraphics[width=.7\textwidth]{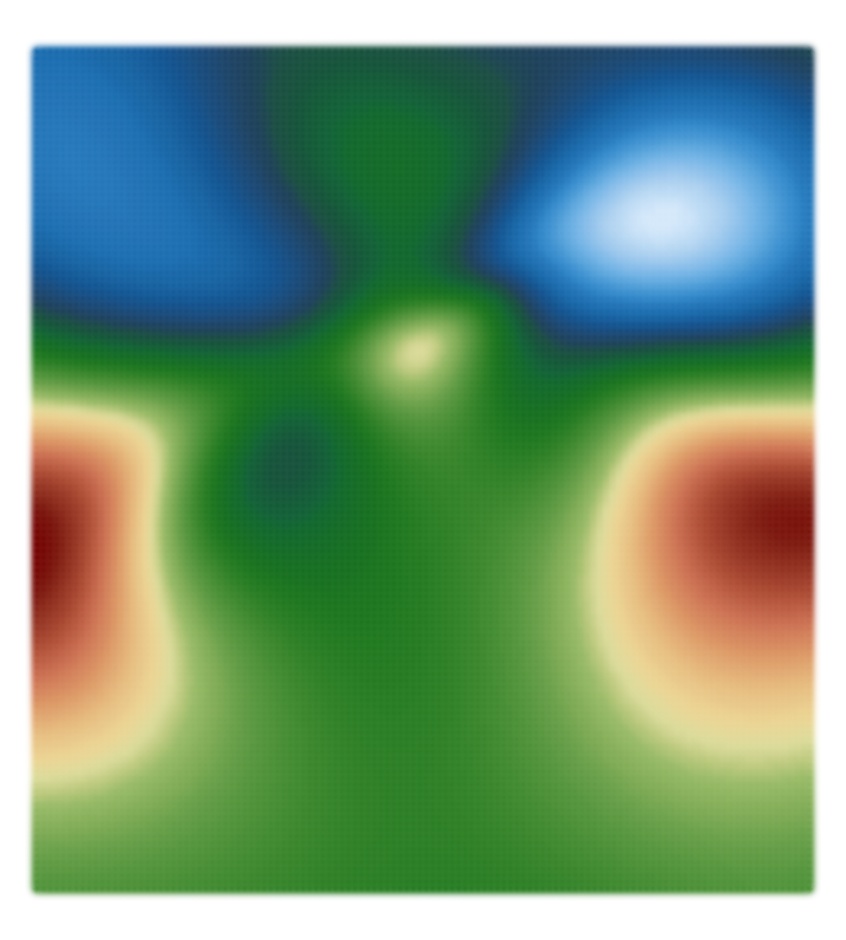}
\subcaption{ $r_{20}$}
\end{subfigure}
\caption{Schematic human: log Jacobian. First row: Base scale: $r_1$; Second row: Base scale: $r_{20}$.}
\label{fig: human coarse scale logJacobian}
\end{figure}
\end{example}


\begin{example}
    
For the second example, we map a circle to a ``bumpy ellipse,'' which is obtained by applying a sinusoidal normal displacement along an ellipse. One can refer to  \cref{fig: bumpy template and target} for an intuitive illustration. 
\begin{figure}
\centering
\begin{subfigure}{.45\textwidth}
\centering
\includegraphics[width=.7\textwidth]{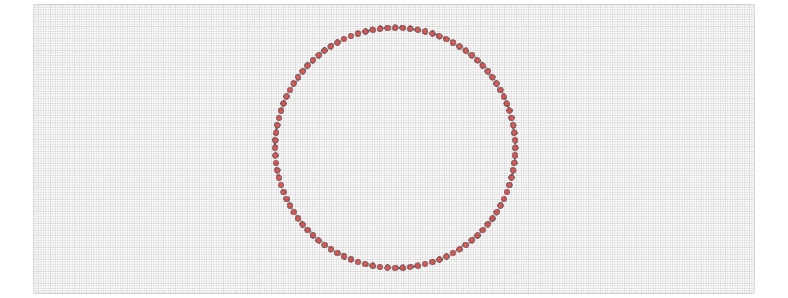}
\subcaption{template}
\end{subfigure}
\begin{subfigure}{.45\textwidth}
\centering
\includegraphics[width=.7\textwidth]{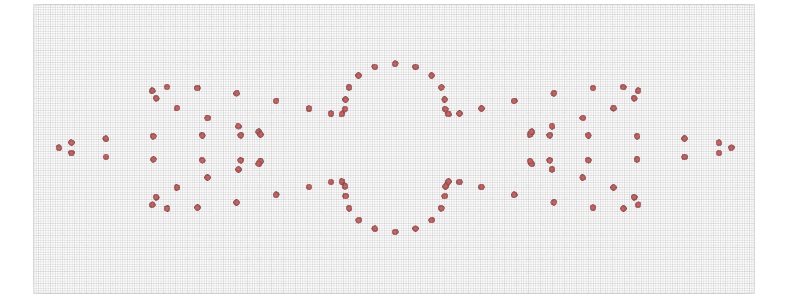}
\subcaption{target}
\end{subfigure}
    \caption{Bumpy ellipse: template and target.}
    \label{fig: bumpy template and target}
\end{figure}
 \Cref{fig: bumpy ms deformation} shows the obtained multiscale diffeomorphisms across scales, and  \cref{fig: bumpy ms logJacobian} provides the log Jacobian of the diffeomorphisms on the original grid. One can see that as the scale gets further away from either base scale, the deformation becomes smaller, as in the previous example. Moreover, one can observe that the deformation around the landmarks are bigger for scales closer to the coarser base scale. This can be seen from the brighter colors in both \cref{fig: bumpy ms deformation,fig: bumpy ms logJacobian} near the landmarks at $r_{17}$ and $r_{20}$. \Cref{fig: bumpy ms residual} shows the different transition patterns or the residuals when the scale variable approaches the base scales. There are subtler deformations near the finer base scale, in which we can observe sawtooth boundaries in colors, while there are more regular shapes with smoother boundaries near the coarser base scale. 

\begin{figure}
\centering
\begin{subfigure}{.25\textwidth}
\centering
\includegraphics[width=1.\textwidth]{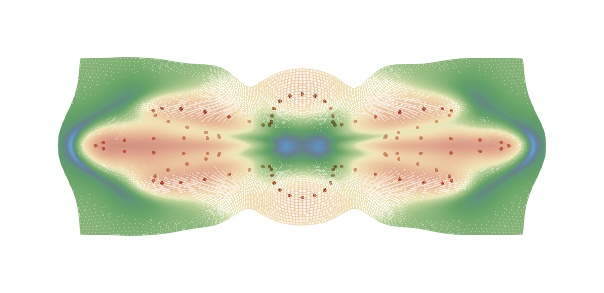}
\subcaption{$r_1$}
\end{subfigure}
\begin{subfigure}{.25\textwidth}
\centering
\includegraphics[width=1.\textwidth]{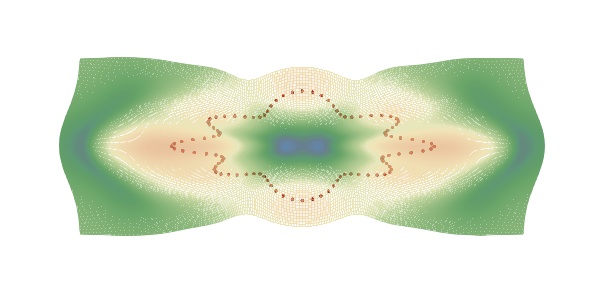}
\subcaption{$r_5$}
\end{subfigure}
\begin{subfigure}{.25\textwidth}
\centering
\includegraphics[width=1.\textwidth]{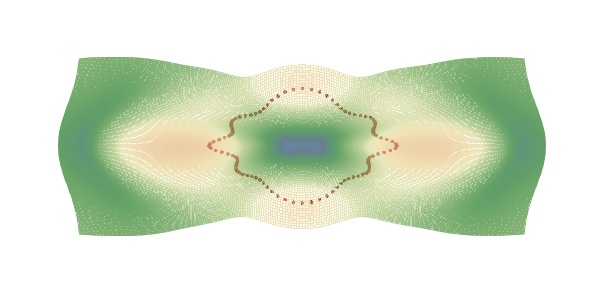}
\subcaption{$r_9$}
\end{subfigure}
\begin{subfigure}{.25\textwidth}
\centering
\includegraphics[width=1.\textwidth]{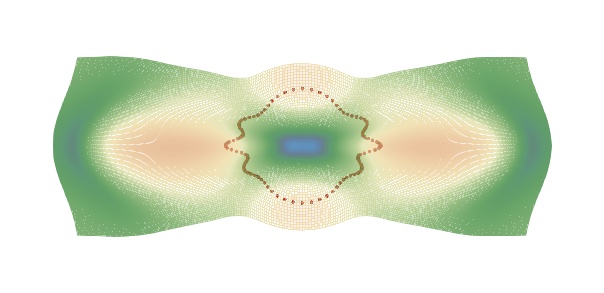}
\subcaption{$r_{13}$}
\end{subfigure}
\begin{subfigure}{.25\textwidth}
\centering
\includegraphics[width=1.\textwidth]{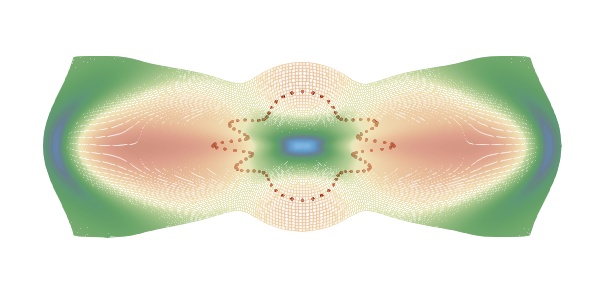}
\subcaption{$r_{17}$}
\end{subfigure}
\begin{subfigure}{.25\textwidth}
\centering
\includegraphics[width=1.\textwidth]{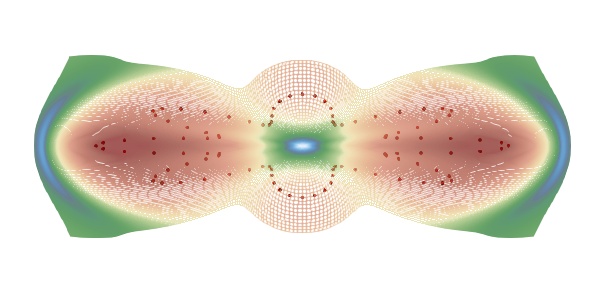}
\subcaption{$r_{20}$}
\end{subfigure}
\caption{Bumpy ellipse: deformation at different scales. Base scales: $r_1$, $r_{20}$.}
\label{fig: bumpy ms deformation}
\end{figure}

\begin{figure}
\centering
\begin{subfigure}{.2\textwidth}
\centering
\includegraphics[width=1.\textwidth]{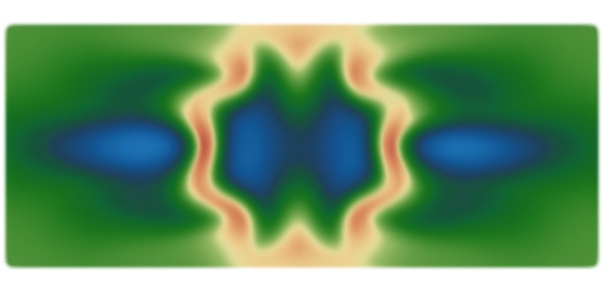}
\subcaption{$r_1$}
\end{subfigure}
\begin{subfigure}{.2\textwidth}
\centering
\includegraphics[width=1.\textwidth]{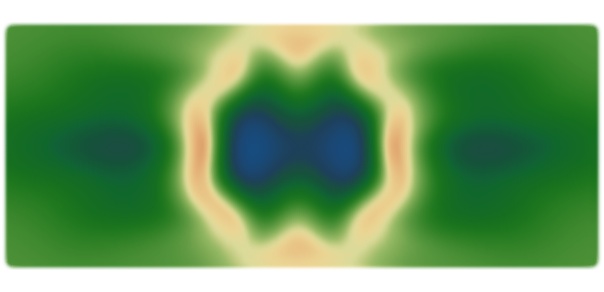}
\subcaption{$r_5$}
\end{subfigure}
\begin{subfigure}{.2\textwidth}
\centering
\includegraphics[width=1.\textwidth]{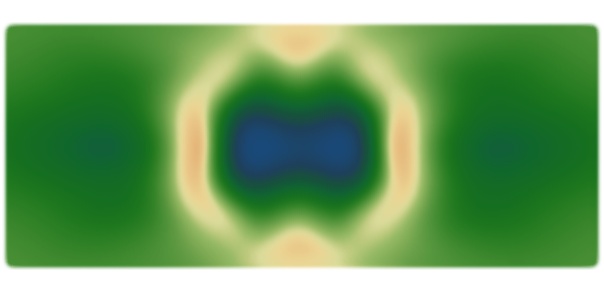}
\subcaption{$r_9$}
\end{subfigure}
\begin{subfigure}{.2\textwidth}
\centering
\includegraphics[width=1.\textwidth]{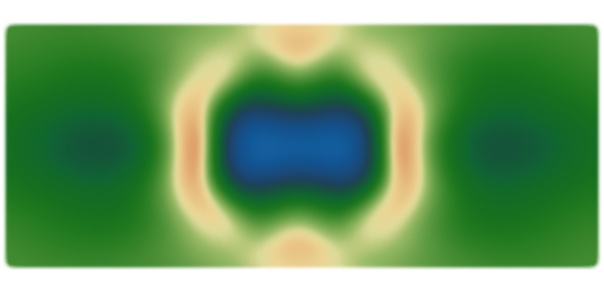}
\subcaption{$r_{13}$}
\end{subfigure}
\begin{subfigure}{.2\textwidth}
\centering
\includegraphics[width=1.\textwidth]{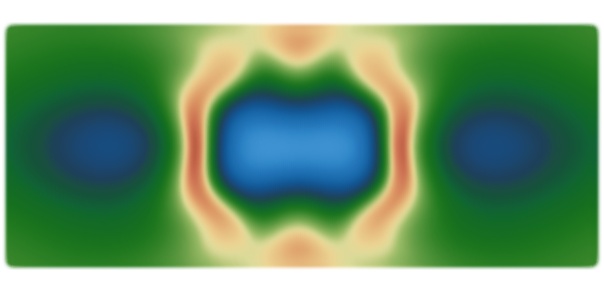}
\subcaption{$r_{17}$}
\end{subfigure}
\begin{subfigure}{.2\textwidth}
\centering
\includegraphics[width=1.\textwidth]{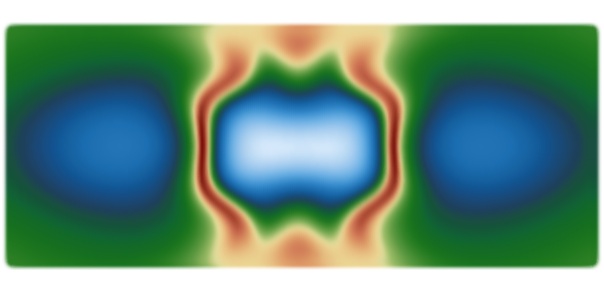}
\subcaption{$r_{20}$}
\end{subfigure}
\caption{Bumpy ellipse: log Jacobian. Base scales: $r_1$, $r_{20}$.}
\label{fig: bumpy ms logJacobian}
\end{figure}

\begin{figure}
\centering
\begin{subfigure}{.2\textwidth}
\centering
\includegraphics[width=\textwidth]{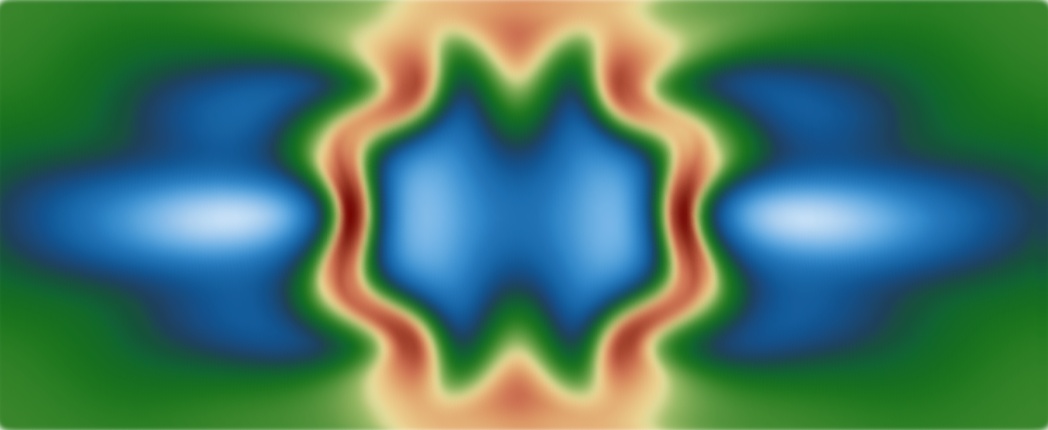}
\subcaption{$\rho^v_{r_{1}}$}
\end{subfigure}
\begin{subfigure}{.2\textwidth}
\centering
\includegraphics[width=\textwidth]{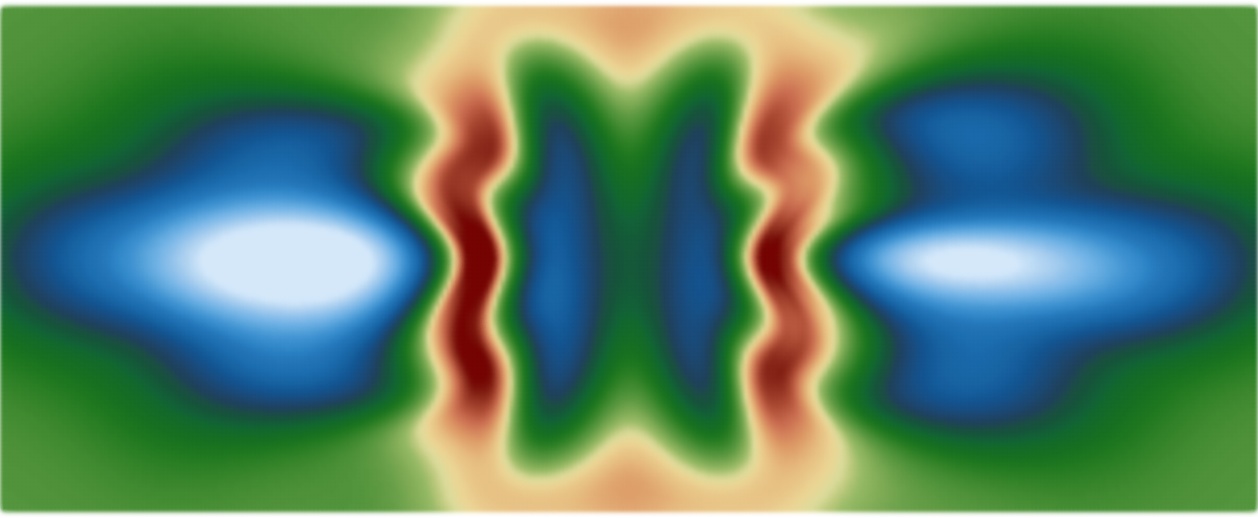}
\subcaption{$\rho^v_{r_{2}}$}
\end{subfigure}
\begin{subfigure}{.2\textwidth}
\centering
\includegraphics[width=\textwidth]{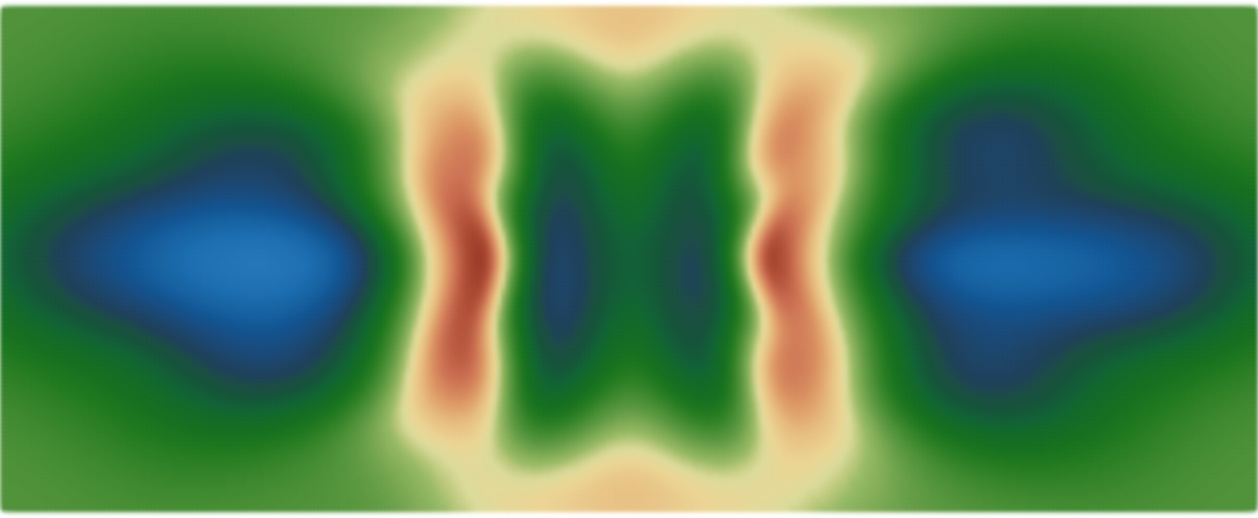}
\subcaption{$\rho^v_{r_{5}}$}
\end{subfigure}
\begin{subfigure}{.2\textwidth}
\centering
\includegraphics[width=\textwidth]{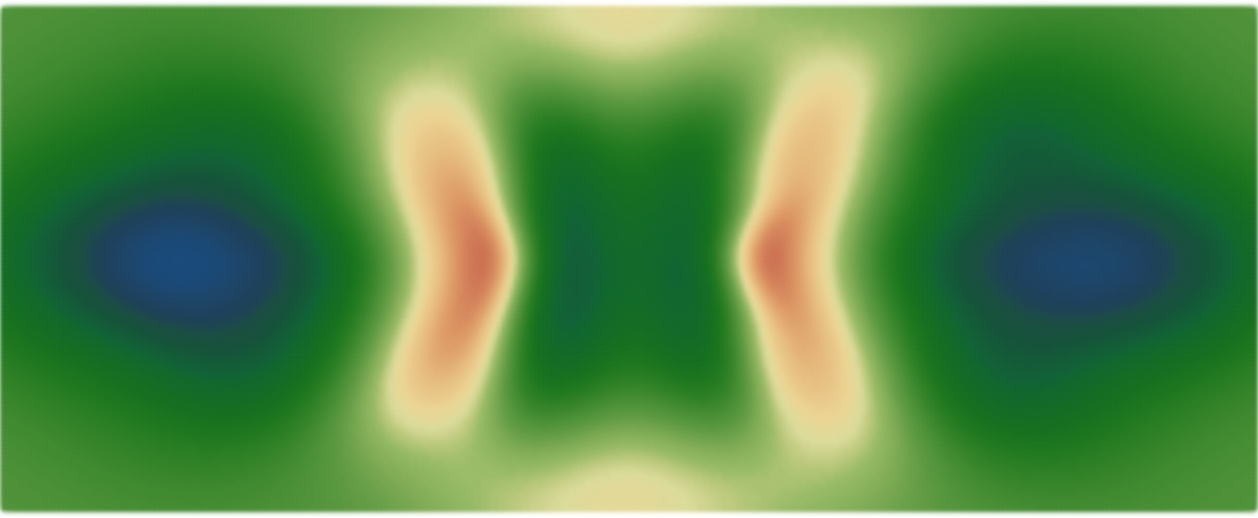}
\subcaption{$\rho^v_{r_{9}}$}
\end{subfigure}
\begin{subfigure}{.2\textwidth}
\centering
\includegraphics[width=\textwidth]{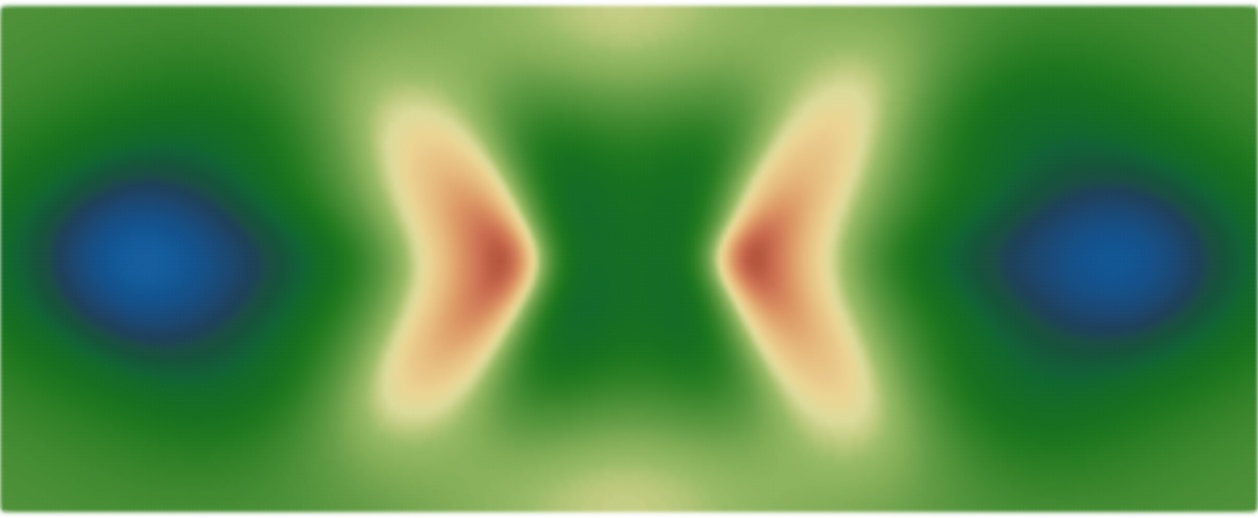}
\subcaption{$\rho^v_{r_{13}}$}
\end{subfigure}
\begin{subfigure}{.2\textwidth}
\centering
\includegraphics[width=\textwidth]{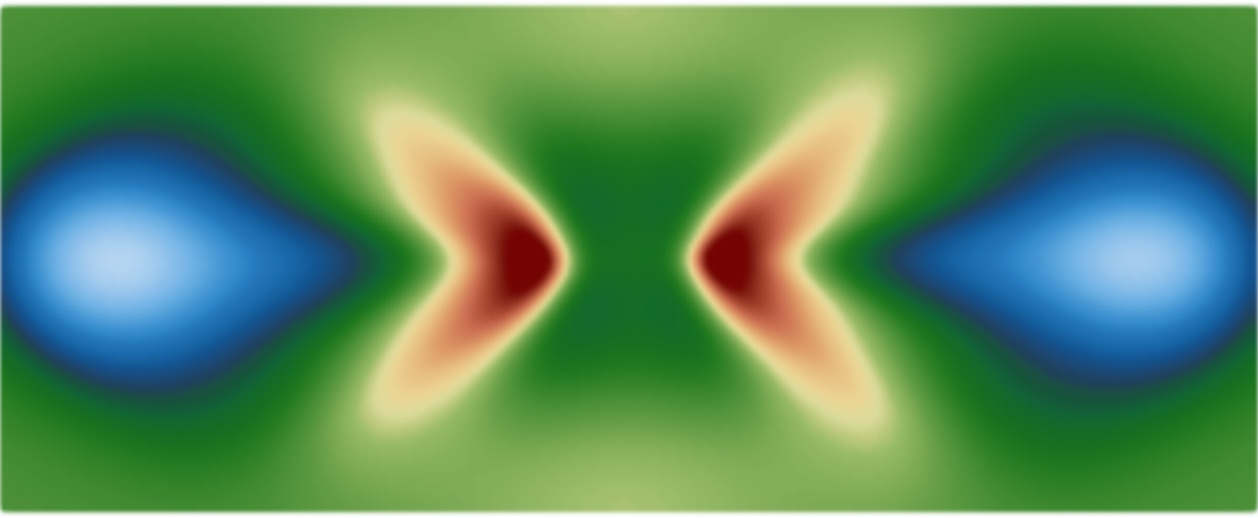}
\subcaption{$\rho^v_{r_{17}}$}
\end{subfigure}
\begin{subfigure}{.2\textwidth}
\centering
\includegraphics[width=\textwidth]{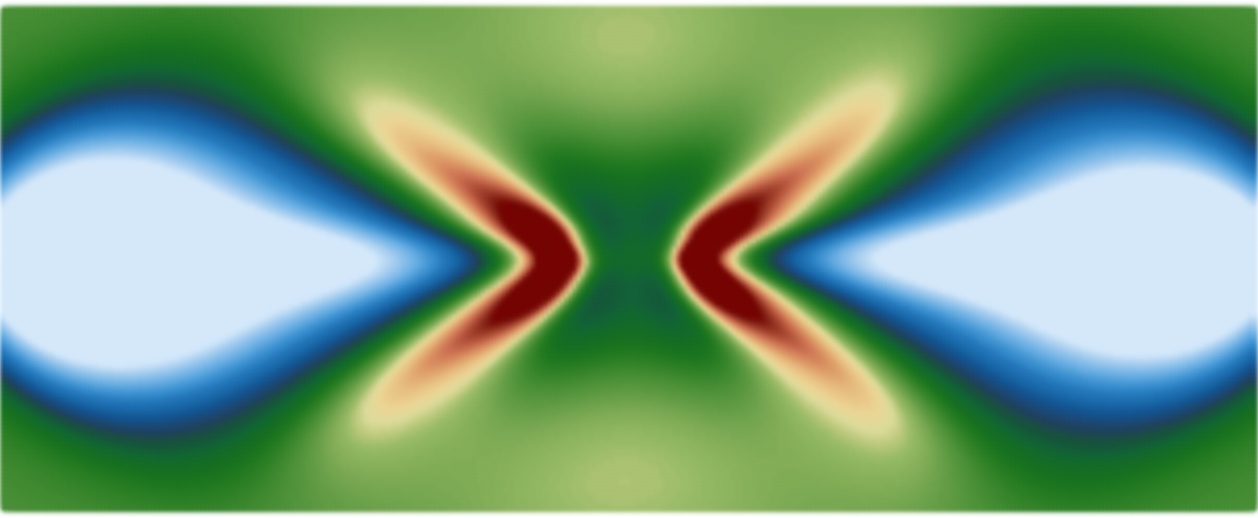}
\subcaption{$\rho^v_{r_{20}}$}
\end{subfigure}
\caption{Bumpy ellipse: residuals. Base scales: $r_1$, $r_{20}$.}
\label{fig: bumpy ms residual}
\end{figure}

\Cref{fig: bumpy fine scale deformation,fig: bumpy coarse scale deformation} show the LDDMM result with $r_1$ (resp. $r_{20}$) being the base scale. \Cref{fig: bumpy coarse scale logJacobian} shows the log Jacobian of these transformations. 

\begin{figure}
\centering
\begin{subfigure}{.25\textwidth}
\centering
\includegraphics[width=\textwidth]{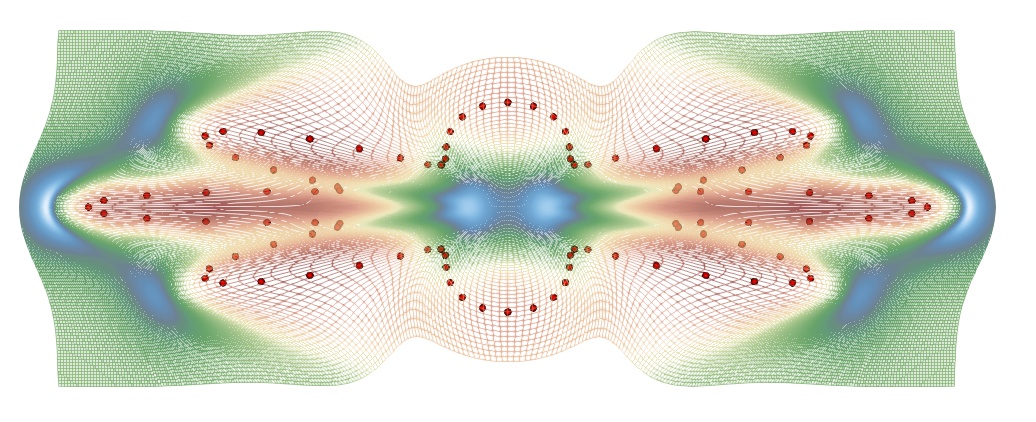}
\subcaption{$r_{1}$}
\end{subfigure}
\begin{subfigure}{.25\textwidth}
\centering
\includegraphics[width=\textwidth]{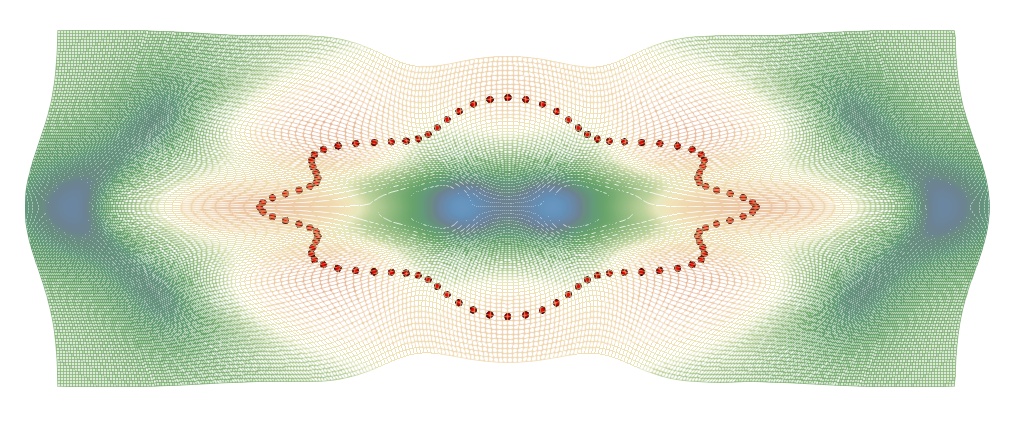}
\subcaption{$r_{5}$}
\end{subfigure}
\begin{subfigure}{.25\textwidth}
\centering
\includegraphics[width=\textwidth]{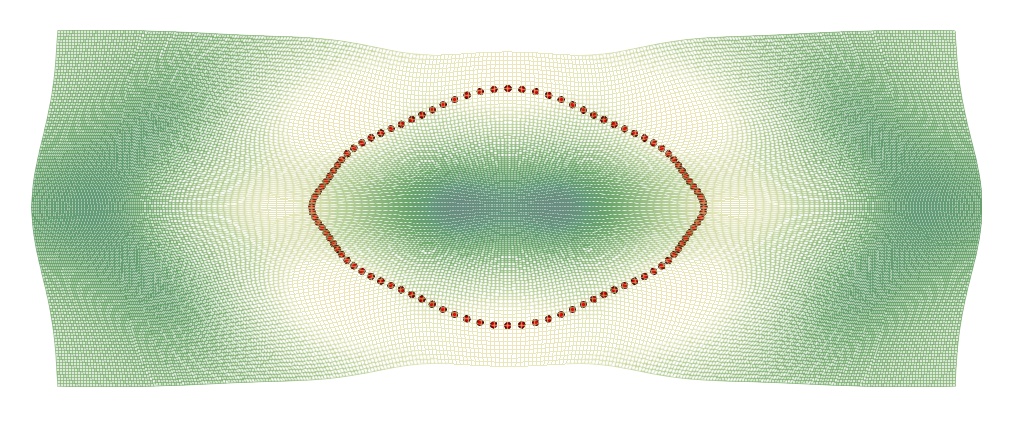}
\subcaption{$r_{9}$}
\end{subfigure}
\begin{subfigure}{.25\textwidth}
\centering
\includegraphics[width=\textwidth]{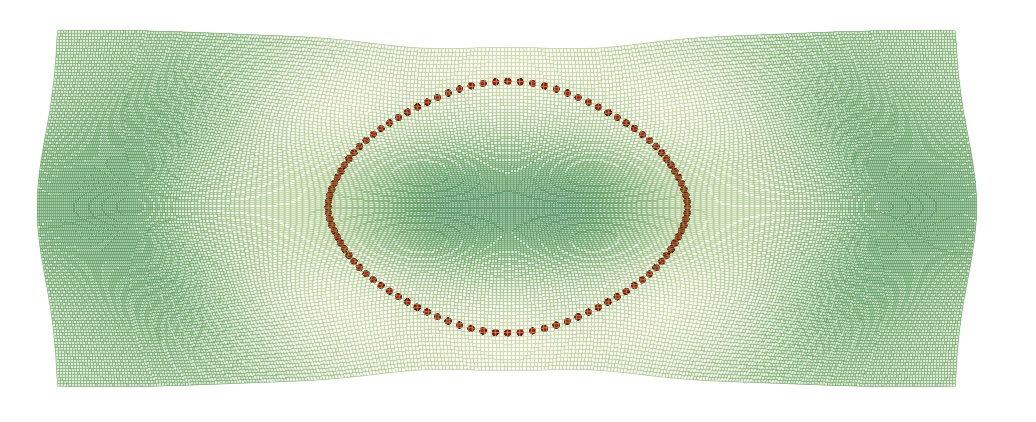}
\subcaption{$r_{13}$}
\end{subfigure}
\begin{subfigure}{.25\textwidth}
\centering
\includegraphics[width=\textwidth]{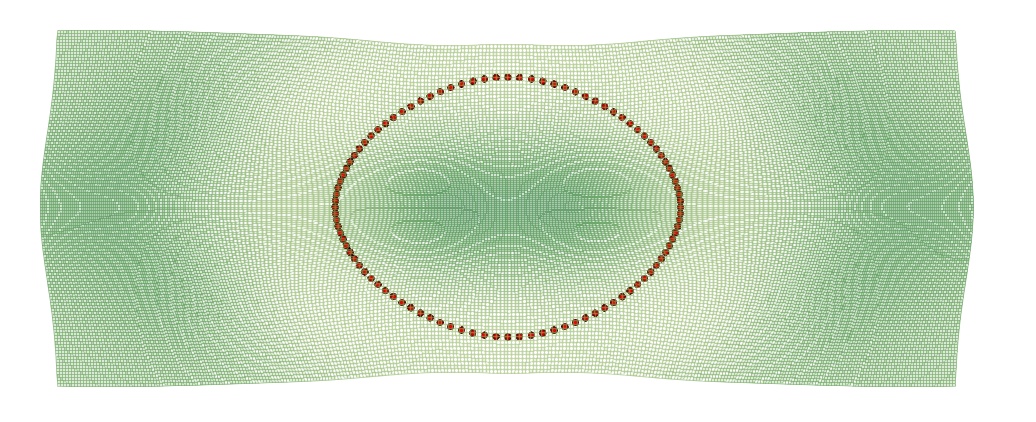}
\subcaption{$r_{17}$}
\end{subfigure}
\begin{subfigure}{.25\textwidth}
\centering
\includegraphics[width=\textwidth]{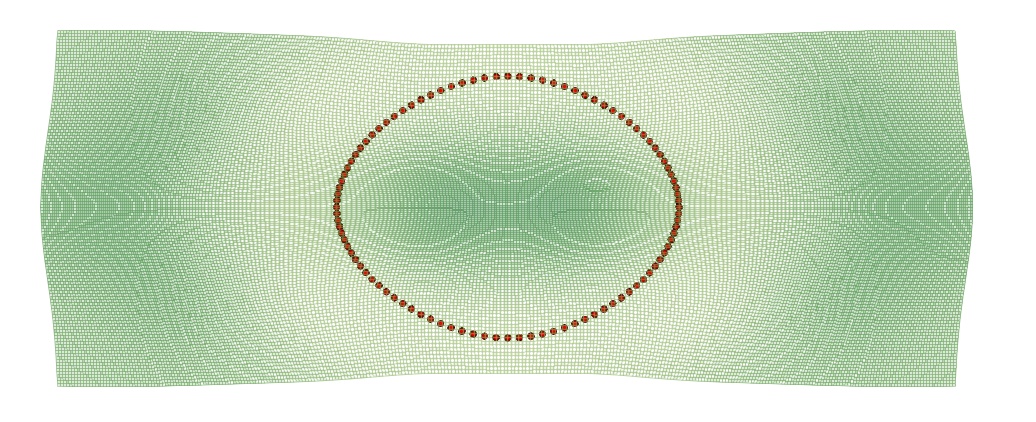}
\subcaption{$r_{20}$}
\end{subfigure}
\caption{Bumpy ellipse: deformations. Base scale: $r_1$.}
\label{fig: bumpy fine scale deformation}
\end{figure}

\begin{figure}
\centering
\begin{subfigure}{.25\textwidth}
\centering
\includegraphics[width=\textwidth]{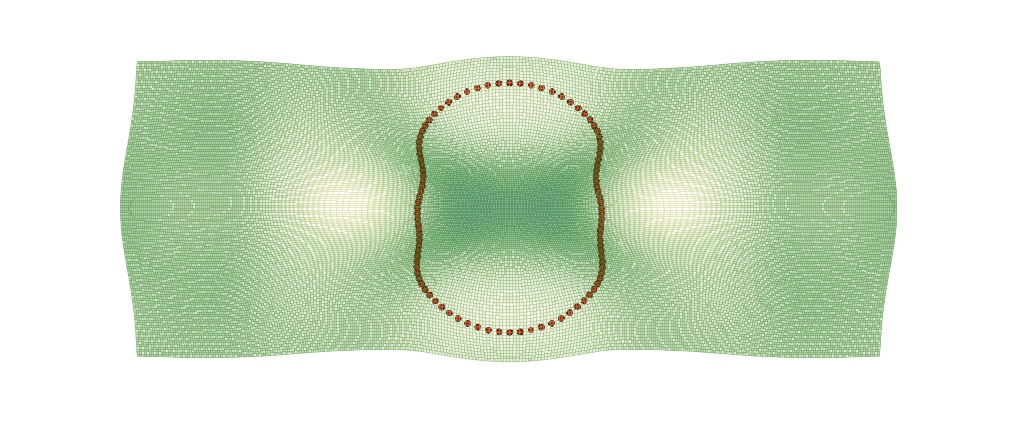}
\subcaption{$r_{1}$}
\end{subfigure}
\begin{subfigure}{.25\textwidth}
\centering
\includegraphics[width=\textwidth]{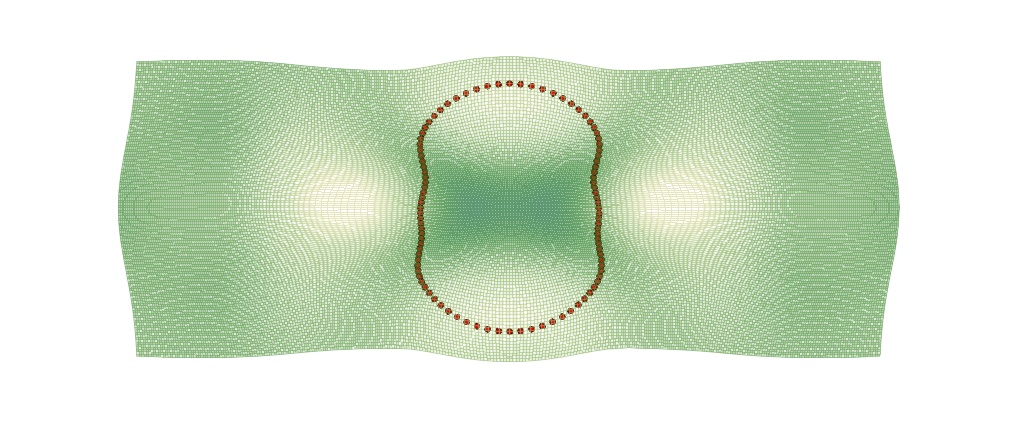}
\subcaption{$r_{5}$}
\end{subfigure}
\begin{subfigure}{.25\textwidth}
\centering
\includegraphics[width=\textwidth]{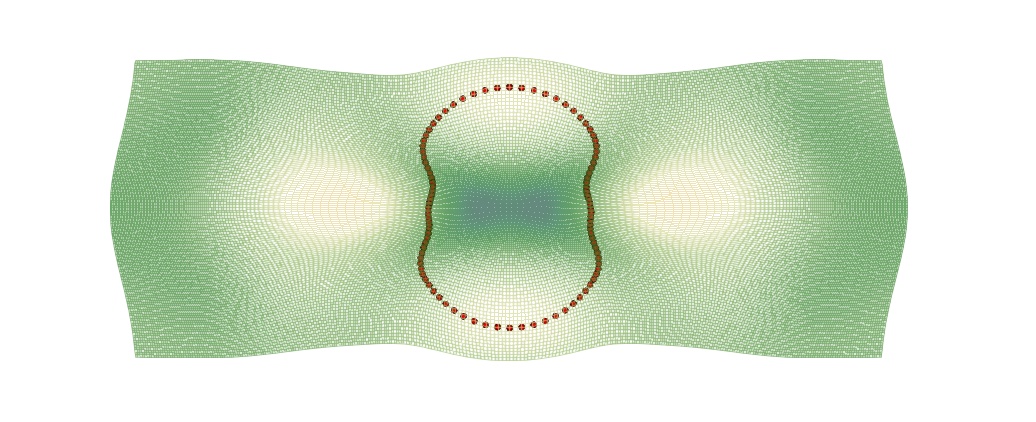}
\subcaption{$r_{9}$}
\end{subfigure}
\begin{subfigure}{.25\textwidth}
\centering
\includegraphics[width=\textwidth]{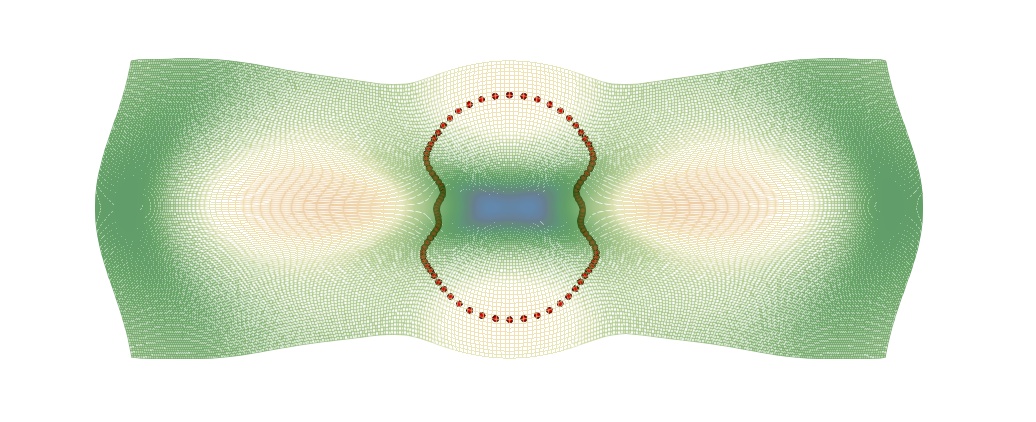}
\subcaption{$r_{13}$}
\end{subfigure}
\begin{subfigure}{.25\textwidth}
\centering
\includegraphics[width=\textwidth]{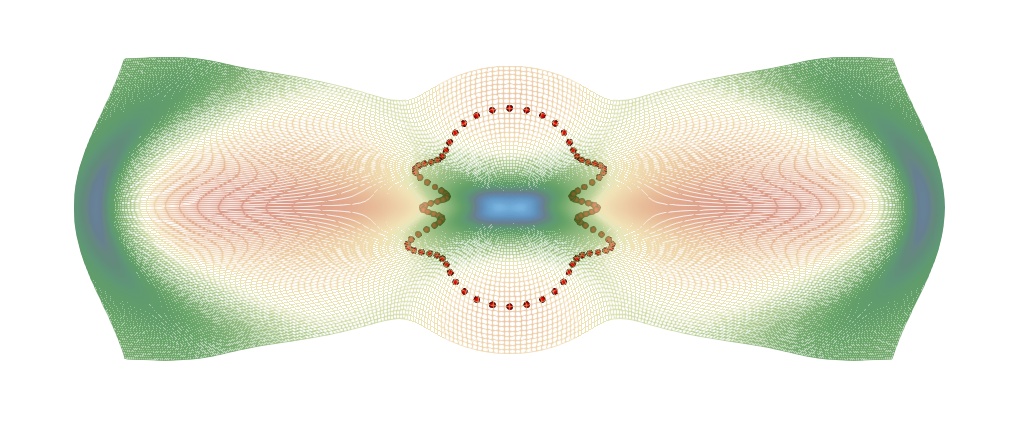}
\subcaption{$r_{17}$}
\end{subfigure}
\begin{subfigure}{.25\textwidth}
\centering
\includegraphics[width=\textwidth]{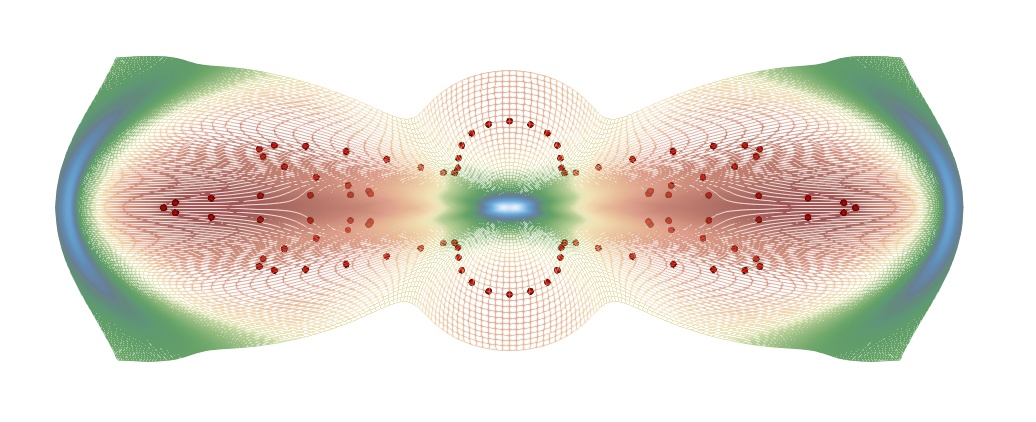}
\subcaption{$r_{20}$}
\end{subfigure}
\caption{Bumpy ellipse: Deformation. Base scale: $r_{20}$.}
\label{fig: bumpy coarse scale deformation}
\end{figure}

\begin{figure}
\centering
\begin{subfigure}{.15\textwidth}
\centering
\includegraphics[width=\textwidth]{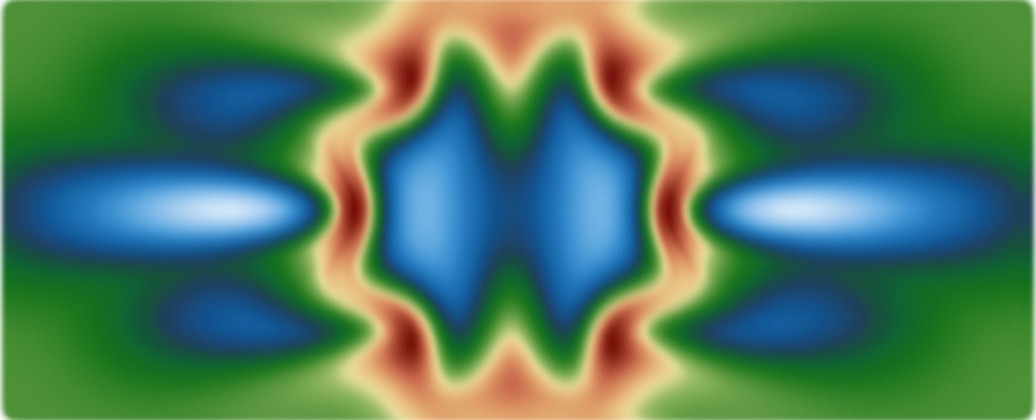}
\subcaption{$r_1$}
\end{subfigure}
\begin{subfigure}{.15\textwidth}
\centering
\includegraphics[width=\textwidth]{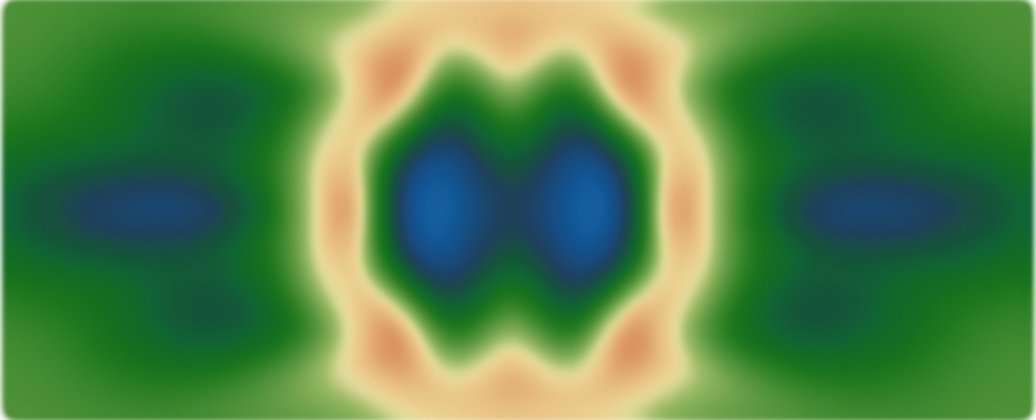}
\subcaption{$r_5$}
\end{subfigure}
\begin{subfigure}{.15\textwidth}
\centering
\includegraphics[width=\textwidth]{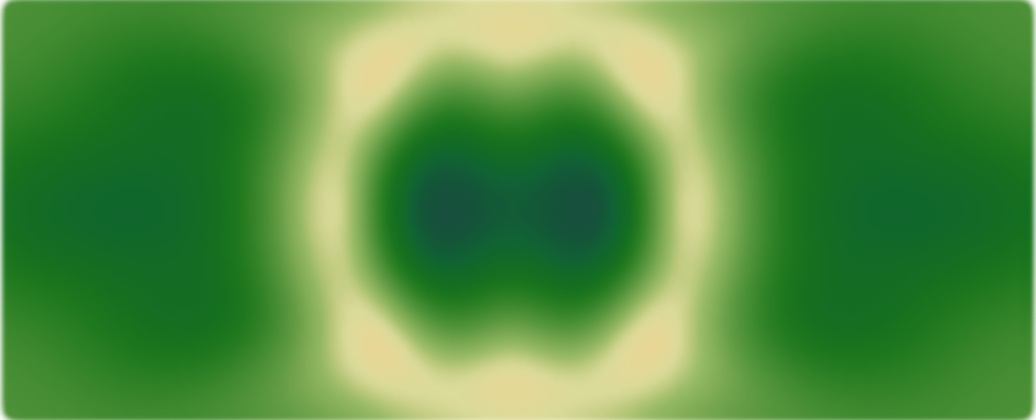}
\subcaption{$r_9$}
\end{subfigure}
\begin{subfigure}{.15\textwidth}
\centering
\includegraphics[width=\textwidth]{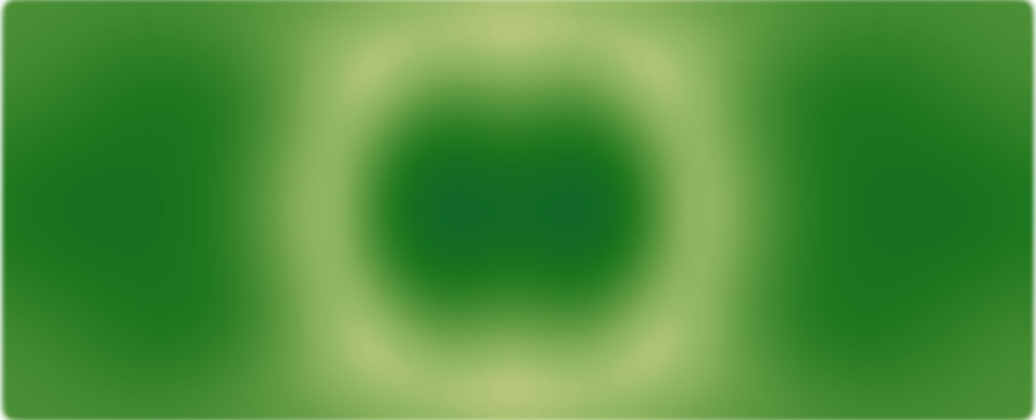}
\subcaption{$r_{13}$}
\end{subfigure}
\begin{subfigure}{.15\textwidth}
\centering
\includegraphics[width=\textwidth]{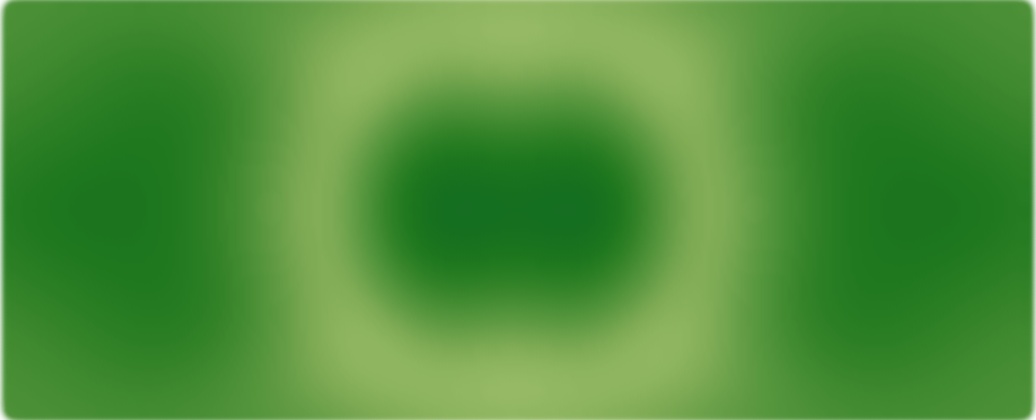}
\subcaption{$r_{17}$}
\end{subfigure}
\begin{subfigure}{.15\textwidth}
\centering
\includegraphics[width=\textwidth]{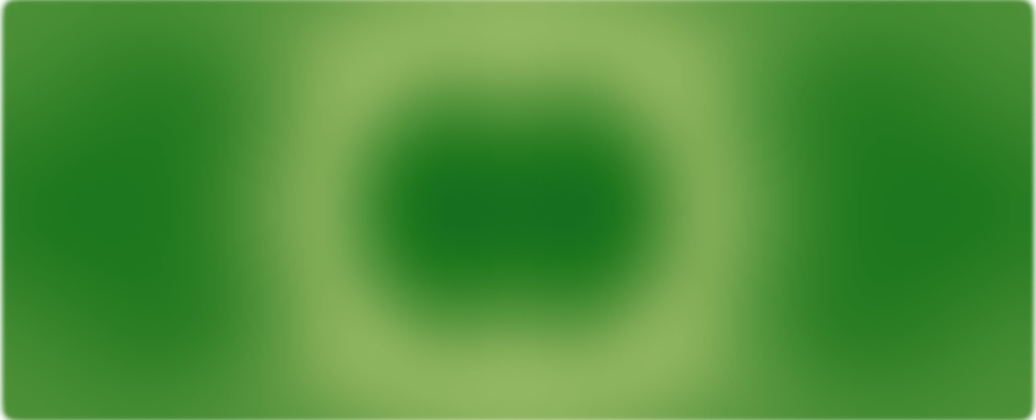}
\subcaption{$r_{20}$}
\end{subfigure}

\setcounter{subfigure}{0}
\begin{subfigure}{.15\textwidth}
\centering
\includegraphics[width=\textwidth]{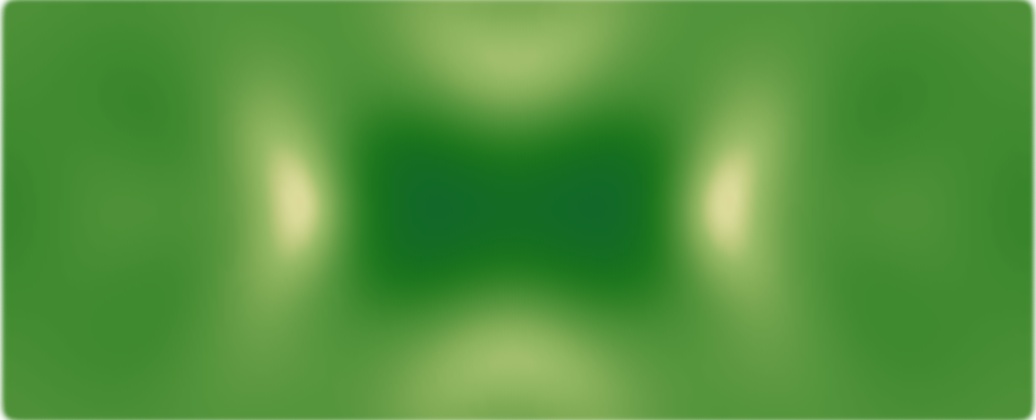}
\subcaption{$r_{1}$}
\end{subfigure}
\begin{subfigure}{.15\textwidth}
\centering
\includegraphics[width=\textwidth]{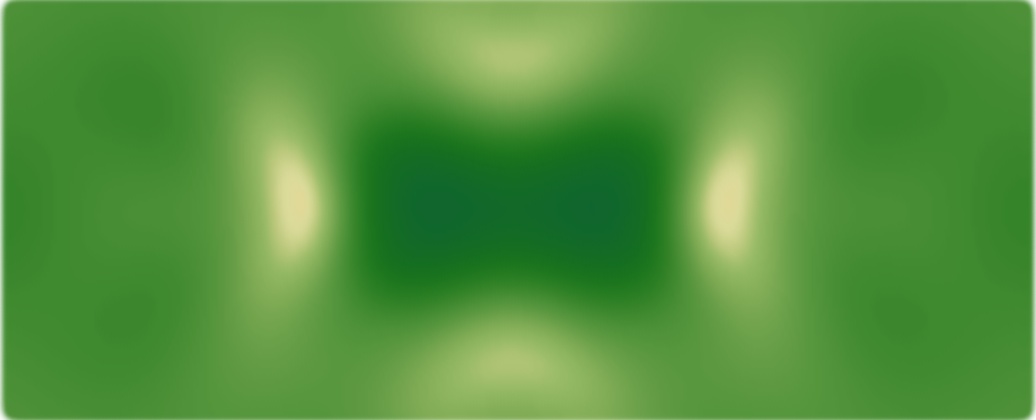}
\subcaption{$r_{5}$}
\end{subfigure}
\begin{subfigure}{.15\textwidth}
\centering
\includegraphics[width=\textwidth]{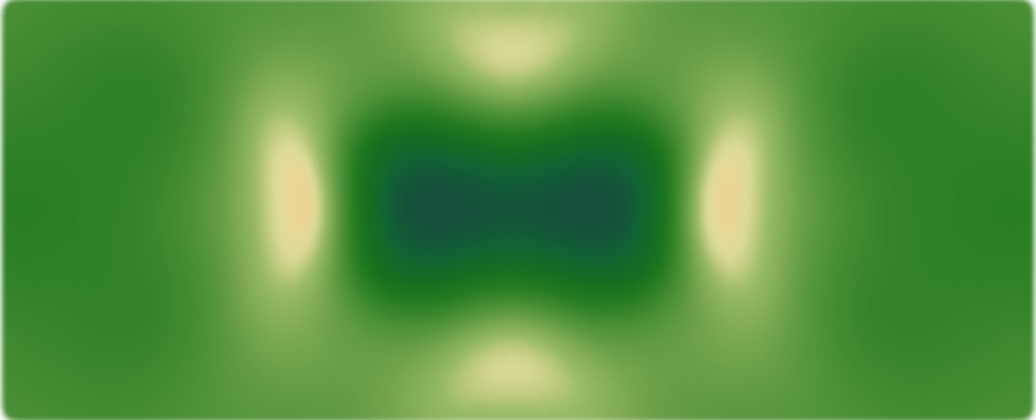}
\subcaption{$r_{9}$}
\end{subfigure}
\begin{subfigure}{.15\textwidth}
\centering
\includegraphics[width=\textwidth]{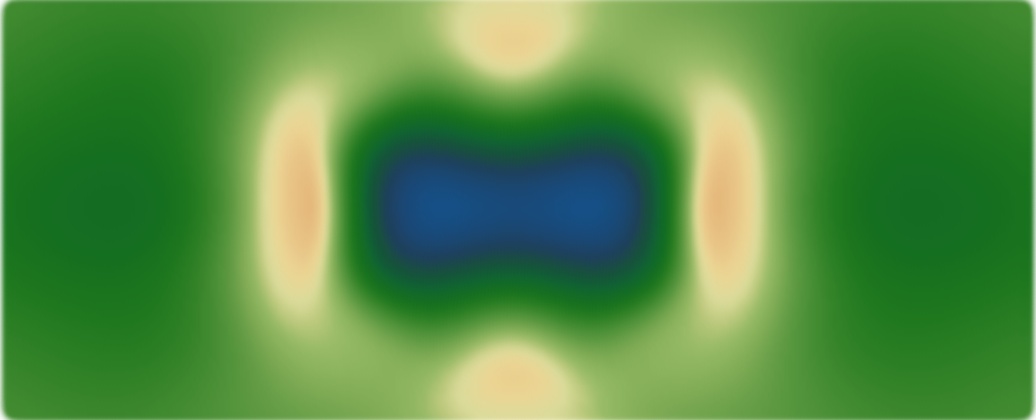}
\subcaption{$r_{13}$}
\end{subfigure}
\begin{subfigure}{.15\textwidth}
\centering
\includegraphics[width=\textwidth]{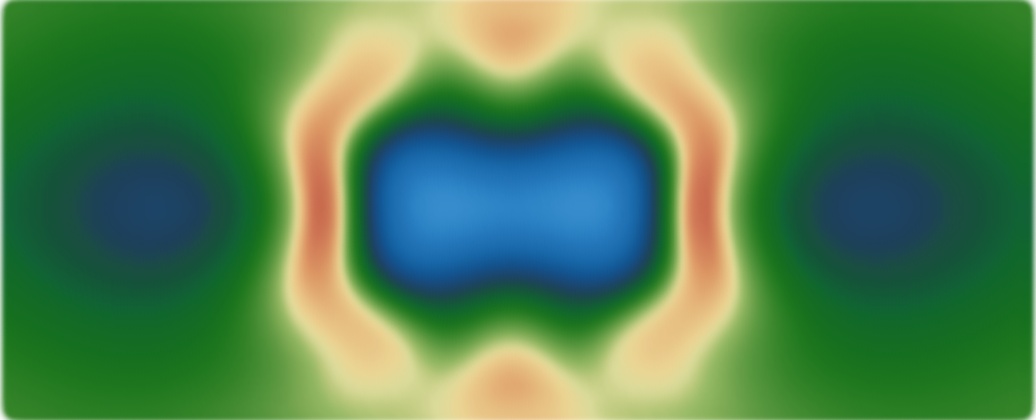}
\subcaption{$r_{17}$}
\end{subfigure}
\begin{subfigure}{.15\textwidth}
\centering
\includegraphics[width=\textwidth]{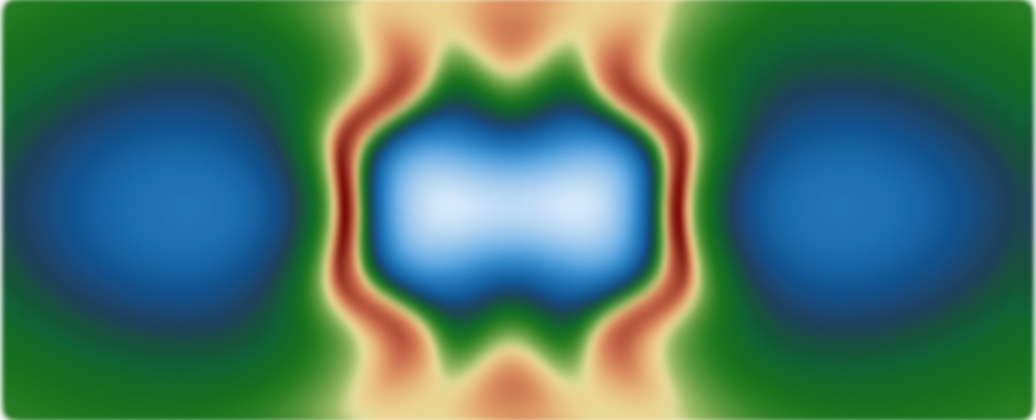}
\subcaption{$r_{20}$}
\end{subfigure}
\caption{Bumpy ellipse: log Jacobian. First row: Base scale $r_1$. Second row: Base scale: $r_{20}.$}
\label{fig: bumpy coarse scale logJacobian}
\end{figure}
\end{example}

\begin{example}
The previous examples have the same landmark matching template as well as target for different base scales. We now consider situations in which they vary across scales. 

This example maps a circle to the same flower shape at the two base scales, albeit with a slight difference in point location. This can be observed in \cref{fig: flower rotate template and target}, especially near the centers.

\begin{figure}
\centering
\begin{subfigure}{.3\textwidth}
\centering
\includegraphics[width=.8\textwidth]{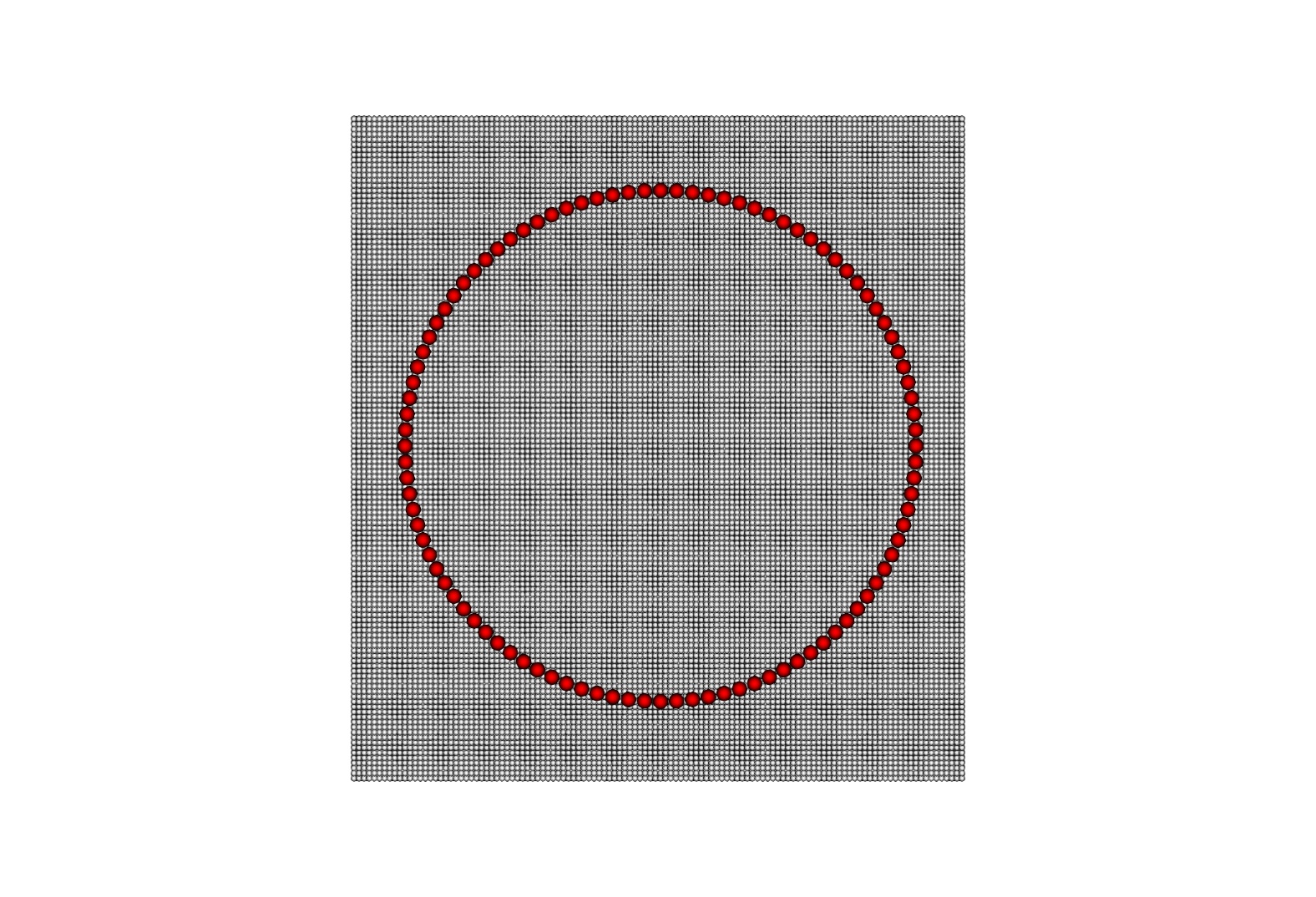}
\subcaption{template}
\end{subfigure}
\begin{subfigure}{.3\textwidth}
\centering
\includegraphics[width=.8\textwidth]{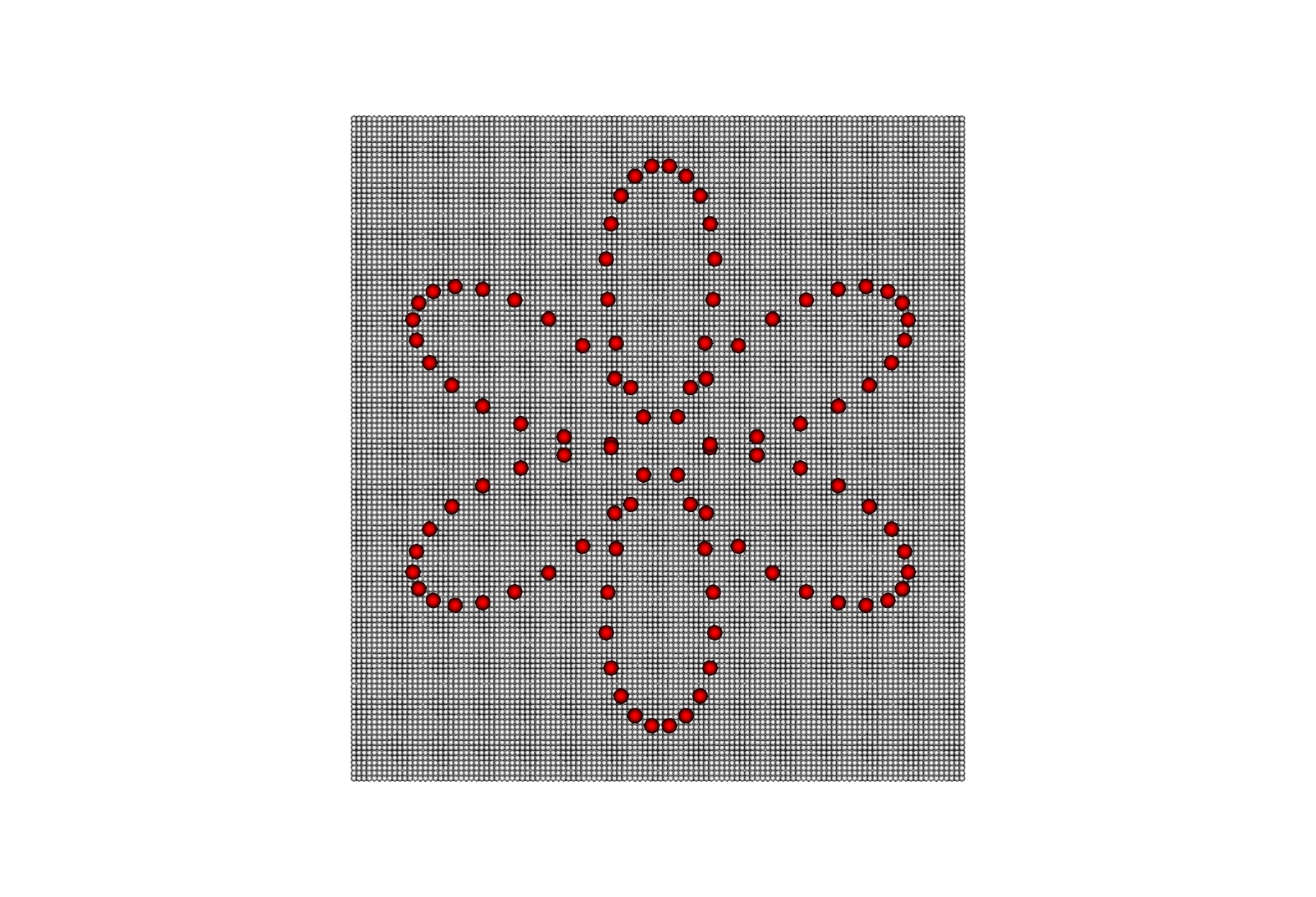}
\subcaption{target at $r_1$}
\end{subfigure}
\begin{subfigure}{.3\textwidth}
\centering
\includegraphics[width=.8\textwidth]{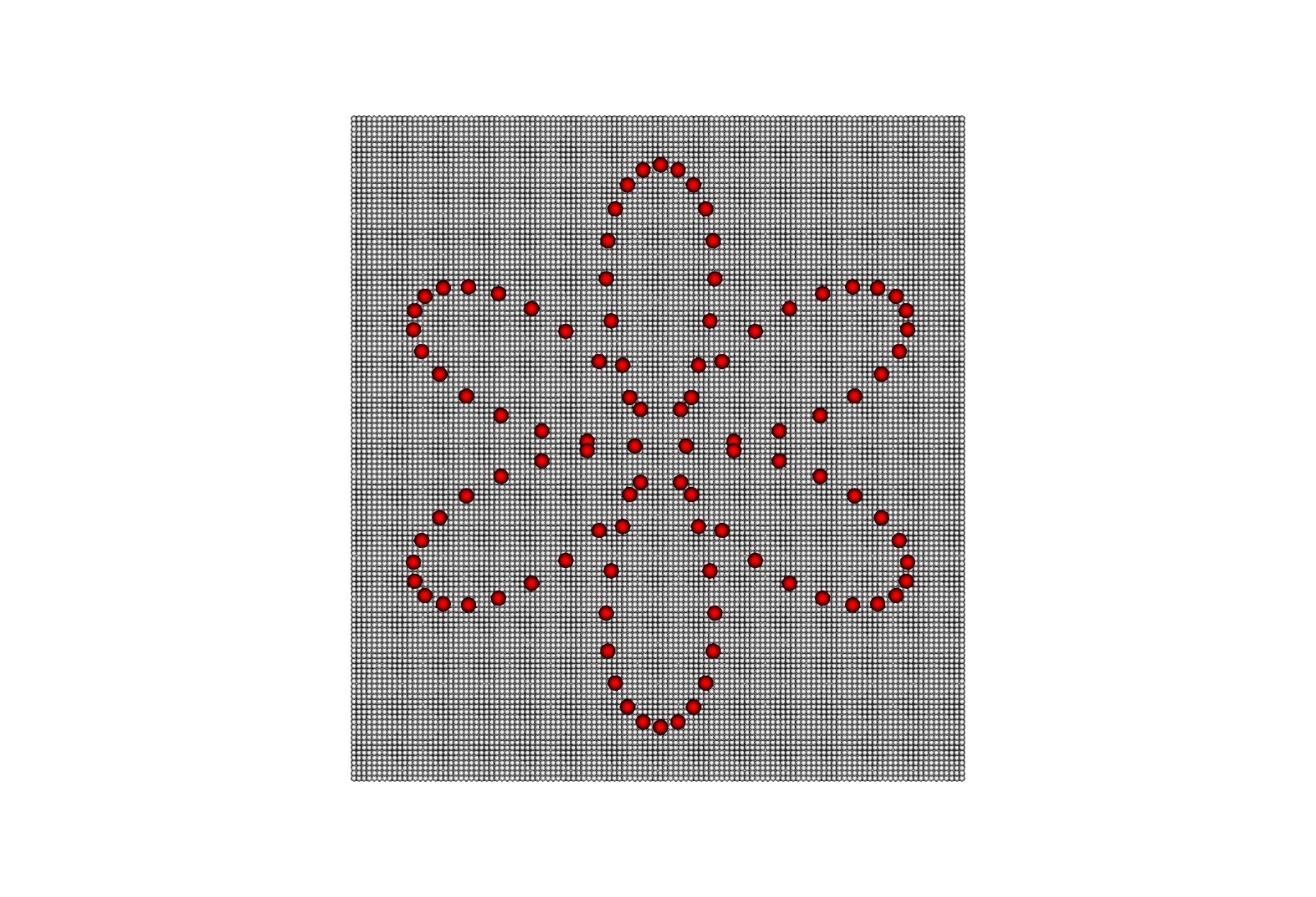}
\subcaption{target at $r_{20}$}
\end{subfigure}
\caption{Flowers: template and target.}
\label{fig: flower rotate template and target}
\end{figure}
 \Cref{fig: flower rotate deformation} demonstrates the deformed template at different scales. One can roughly see that as scale varies from $r_1$ to $r_{20}$, at the intermediate scales the flower folds and expands, with the landmarks shifting to reach the target at the other base scale.  \Cref{fig: flower rotate residual} shows the corresponding log Jacobian plots on the original grid together with the log Jacobian of the residual transformations. As in the previous examples, the changes in the finer scale are larger than those in the coarser scales, indicated by obvious deeper colors. 
\begin{figure}
\centering
\begin{subfigure}{.12\textwidth}
\centering
\includegraphics[width=1.\textwidth]{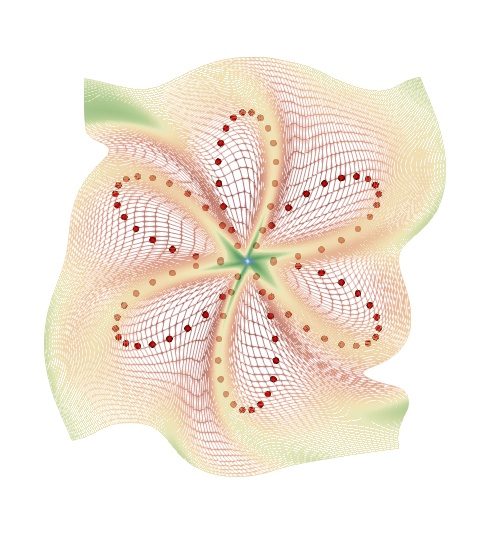}
\subcaption{$r_{1}$}
\end{subfigure}
\begin{subfigure}{.12\textwidth}
\centering
\includegraphics[width=1.\textwidth]{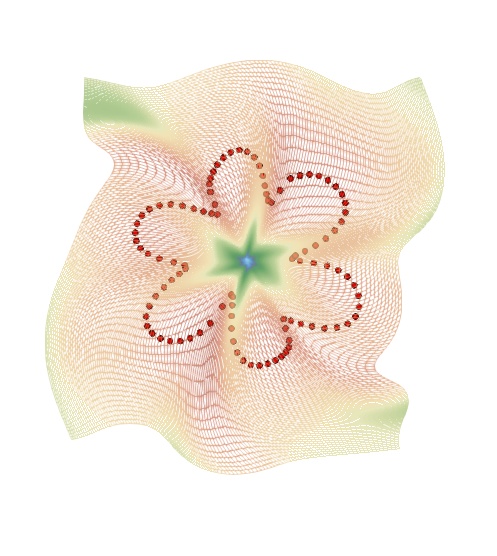}
\subcaption{$r_{3}$}
\end{subfigure}
\begin{subfigure}{.12\textwidth}
\centering
\includegraphics[width=1.\textwidth]{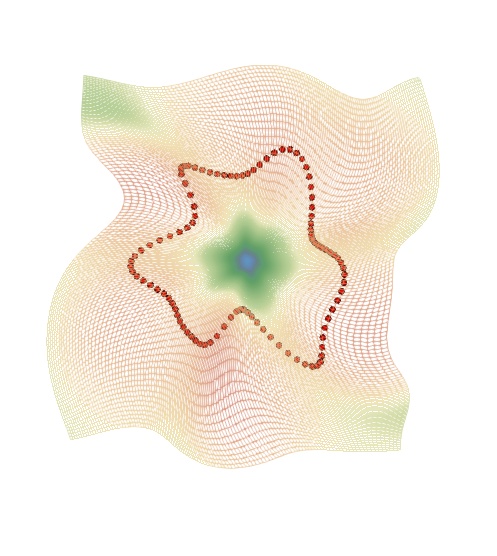}
\subcaption{$r_{6}$}
\end{subfigure}
\begin{subfigure}{.12\textwidth}
\centering
\includegraphics[width=1.\textwidth]{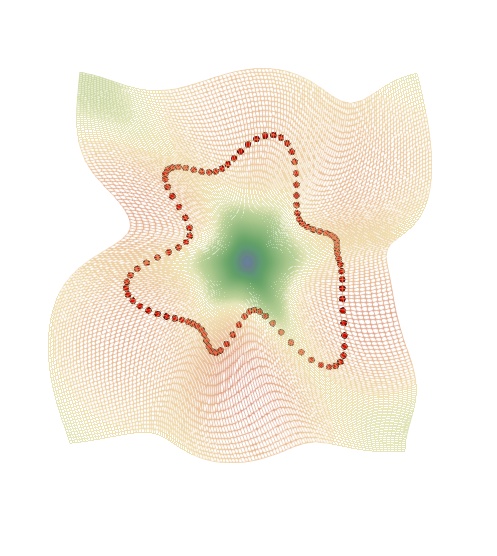}
\subcaption{$r_{10}$}
\end{subfigure}
\begin{subfigure}{.12\textwidth}
\centering
\includegraphics[width=1.\textwidth]{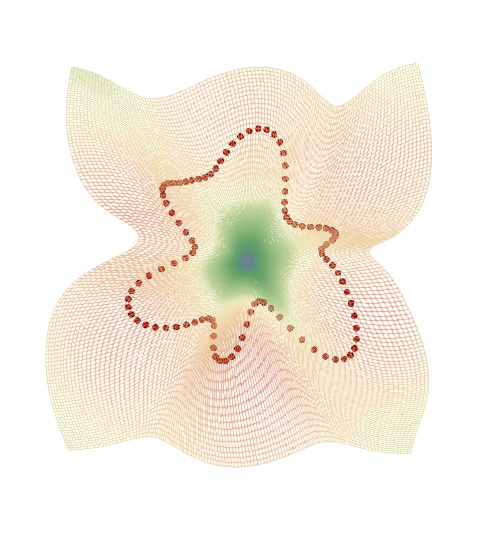}
\subcaption{$r_{14}$}
\end{subfigure}
\begin{subfigure}{.12\textwidth}
\centering
\includegraphics[width=1.\textwidth]{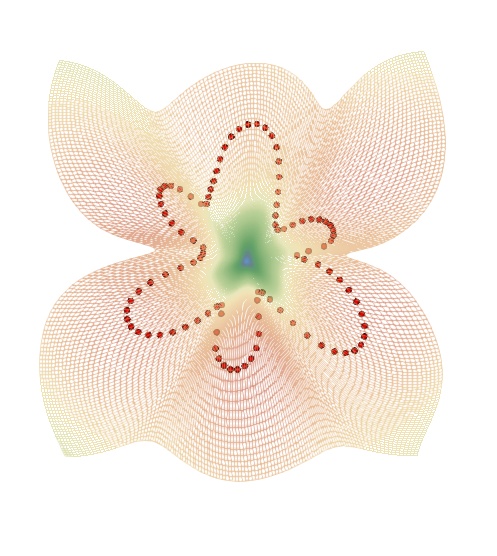}
\subcaption{$r_{17}$}
\end{subfigure}
\begin{subfigure}{.12\textwidth}
\centering
\includegraphics[width=1.\textwidth]{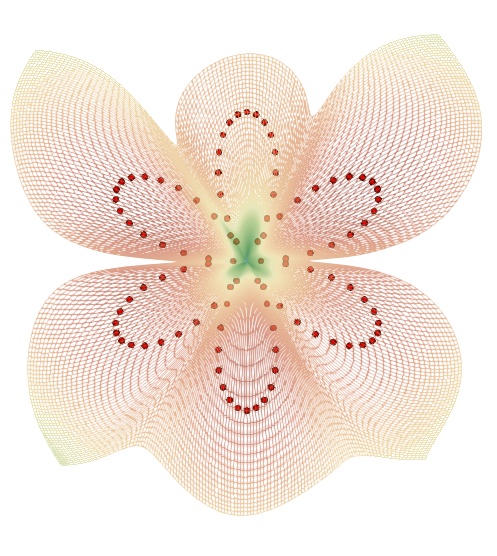}
\subcaption{$r_{20}$}
\end{subfigure}
\caption{Flowers: deformation. Base scales: $r_1$, $r_{20}$.}
\label{fig: flower rotate deformation}
\end{figure}

\begin{figure}
\centering
\begin{subfigure}{.11\textwidth}
\centering
\includegraphics[width=.9\textwidth]{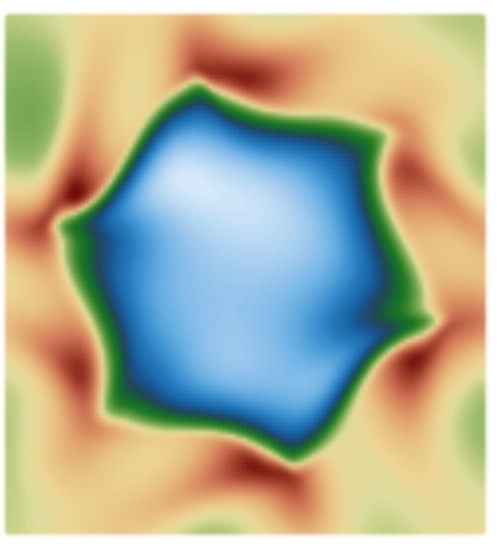}
\subcaption{$r_{1}$}
\end{subfigure}
\begin{subfigure}{.11\textwidth}
\centering
\includegraphics[width=.9\textwidth]{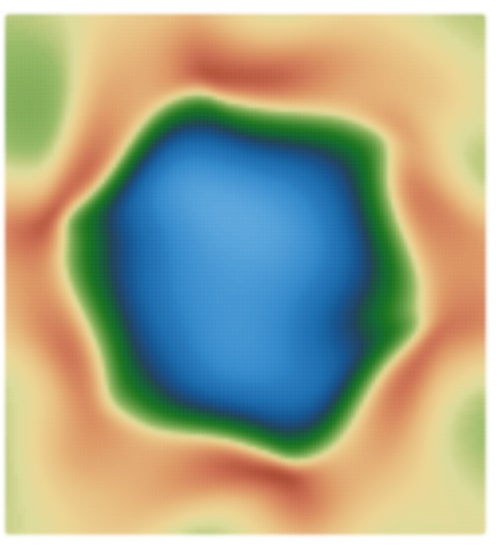}
\subcaption{$r_{3}$}
\end{subfigure}
\begin{subfigure}{.11\textwidth}
\centering
\includegraphics[width=.9\textwidth]{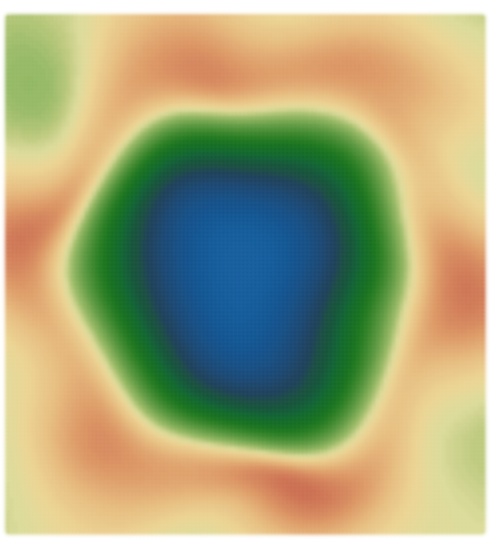}
\subcaption{$r_{6}$}
\end{subfigure}
\begin{subfigure}{.11\textwidth}
\centering
\includegraphics[width=.9\textwidth]{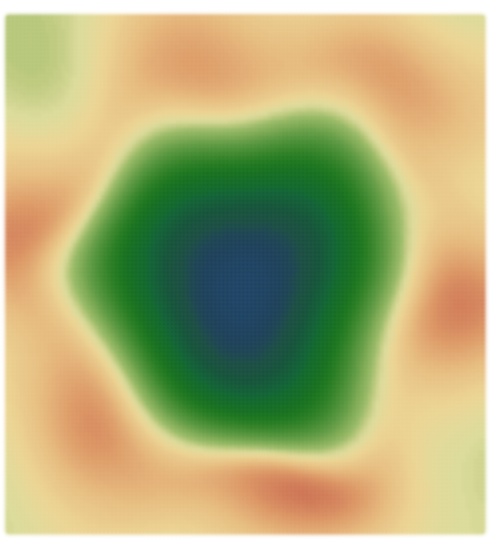}
\subcaption{$r_{10}$}
\end{subfigure}
\begin{subfigure}{.11\textwidth}
\centering
\includegraphics[width=.9\textwidth]{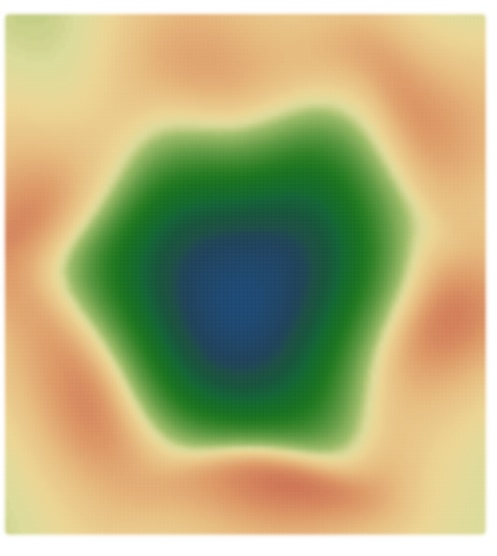}
\subcaption{$r_{14}$}
\end{subfigure}
\begin{subfigure}{.11\textwidth}
\centering
\includegraphics[width=.9\textwidth]{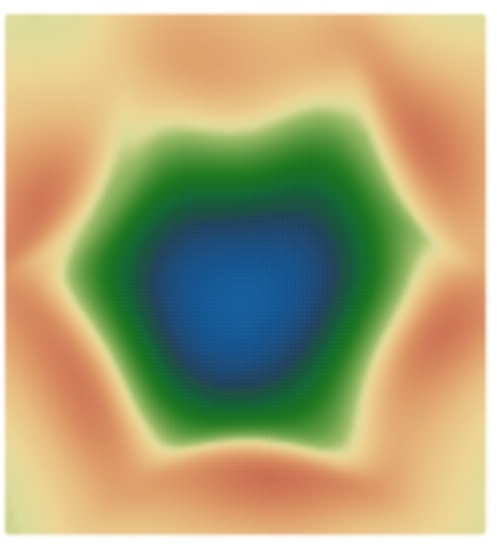}
\subcaption{$r_{17}$}
\end{subfigure}
\begin{subfigure}{.11\textwidth}
\centering
\includegraphics[width=.9\textwidth]{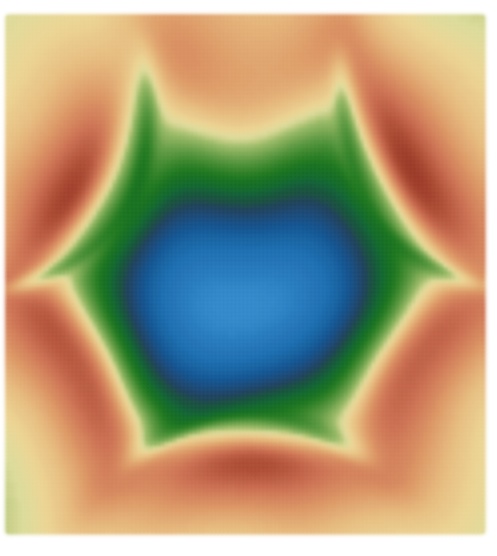}
\subcaption{$r_{20}$}
\end{subfigure}

\setcounter{subfigure}{0}
\begin{subfigure}{.11\textwidth}
\centering
\includegraphics[width=.9\linewidth]{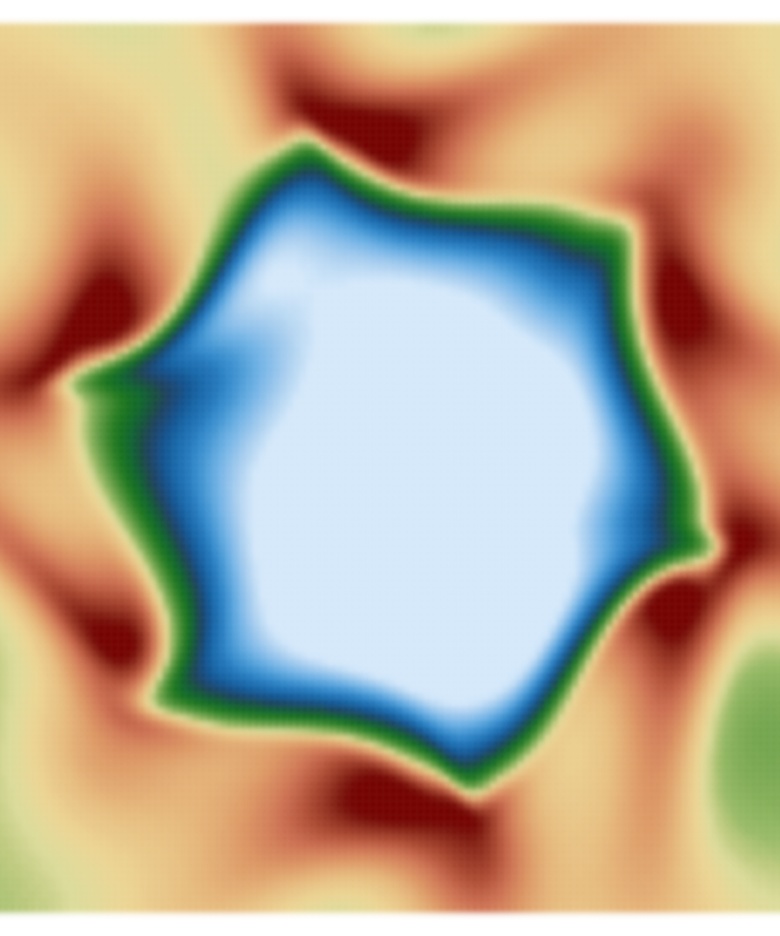}
\subcaption{$\rho^v_{r_1}$}
\end{subfigure}
\begin{subfigure}{.11\textwidth}
\centering
\includegraphics[width=.9\linewidth]{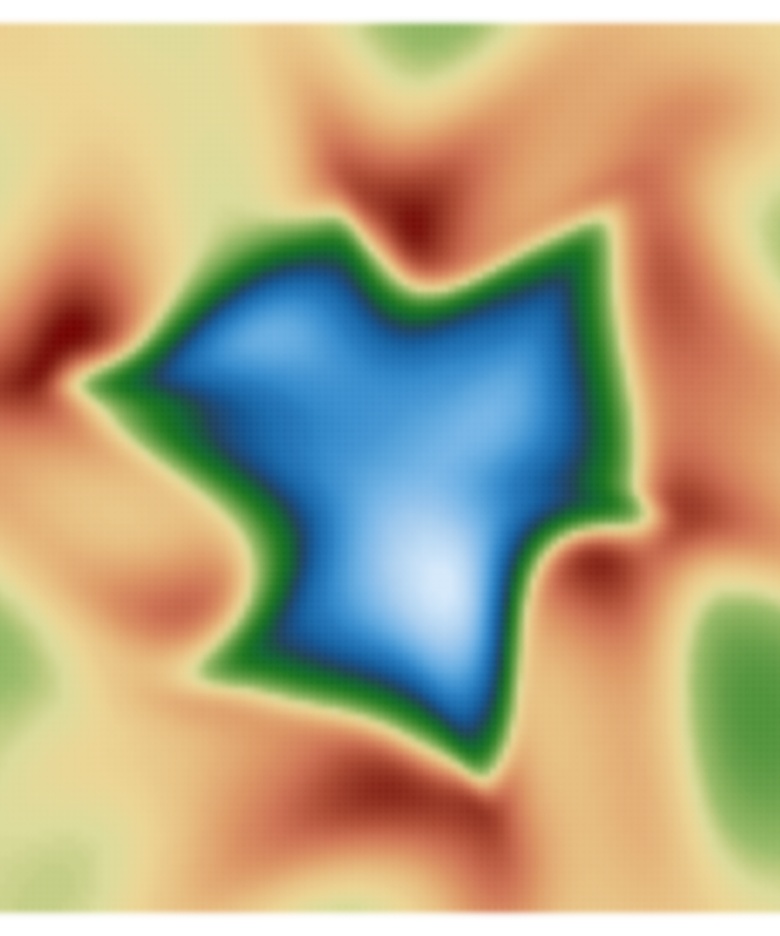}
\subcaption{$\rho^v_{r_2}$}
\end{subfigure}
\begin{subfigure}{.11\textwidth}
\centering
\includegraphics[width=.9\linewidth]{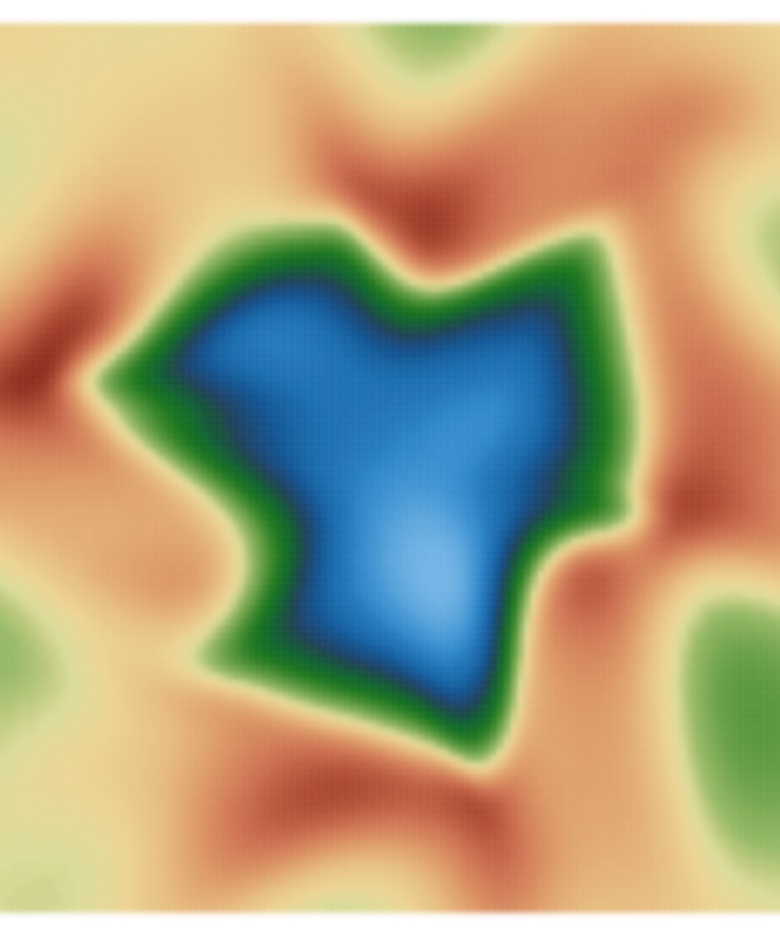}
\subcaption{$\rho^v_{r_3}$}
\end{subfigure}
\begin{subfigure}{.11\textwidth}
\centering
\includegraphics[width=.9\linewidth]{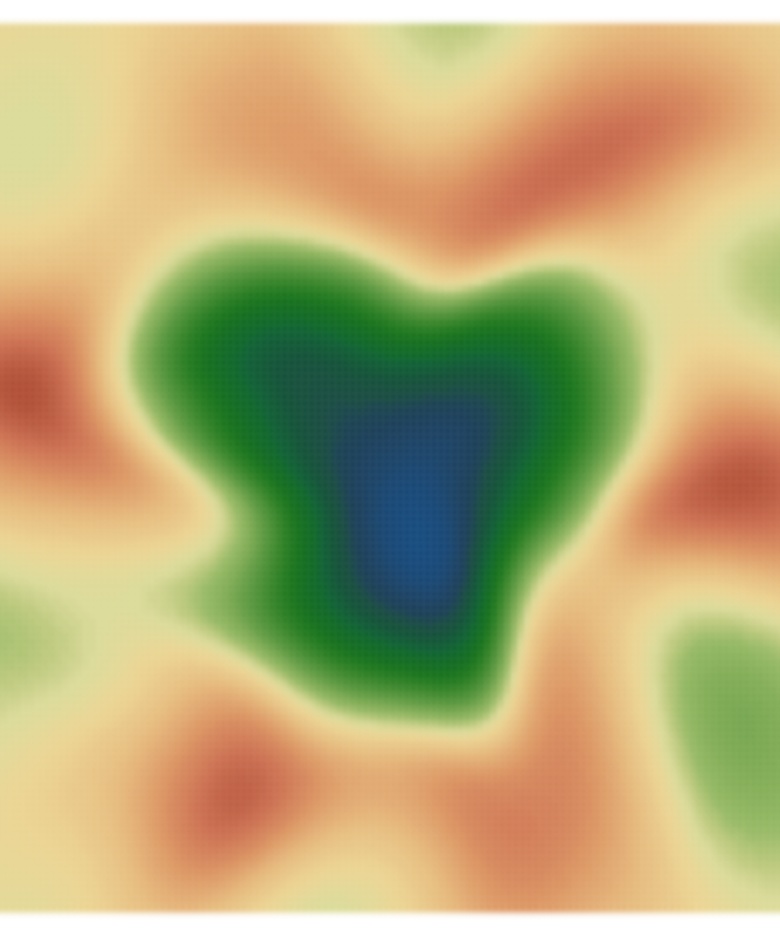}
\subcaption{$\rho^v_{r_6}$}
\end{subfigure}
\begin{subfigure}{.11\textwidth}
\centering
\includegraphics[width=.9\linewidth]{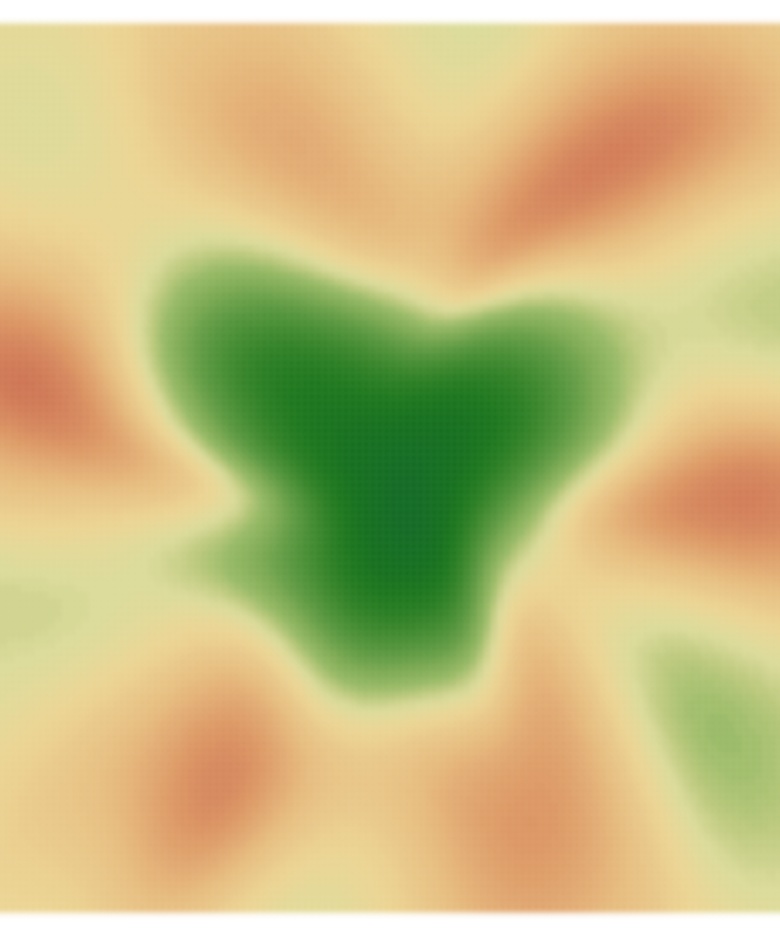}
\subcaption{$\rho^v_{r_{10}}$}
\end{subfigure}
\begin{subfigure}{.11\textwidth}
\centering
\includegraphics[width=.9\linewidth]{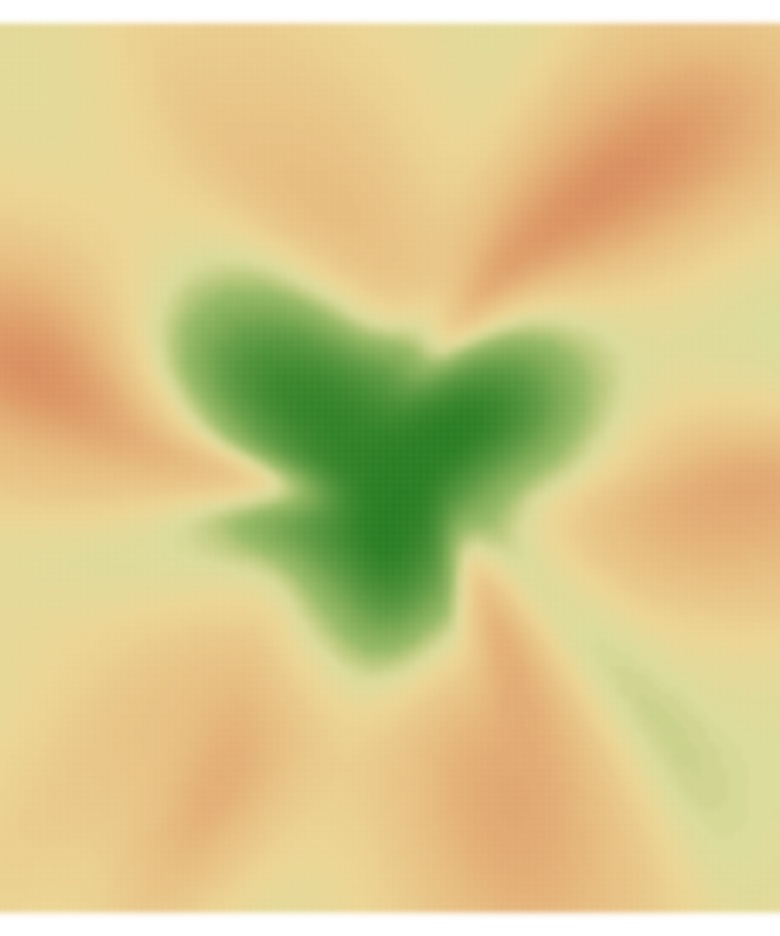}
\subcaption{$\rho^v_{r_{15}}$}
\end{subfigure}
\begin{subfigure}{.11\textwidth}
\centering
\includegraphics[width=.9\linewidth]{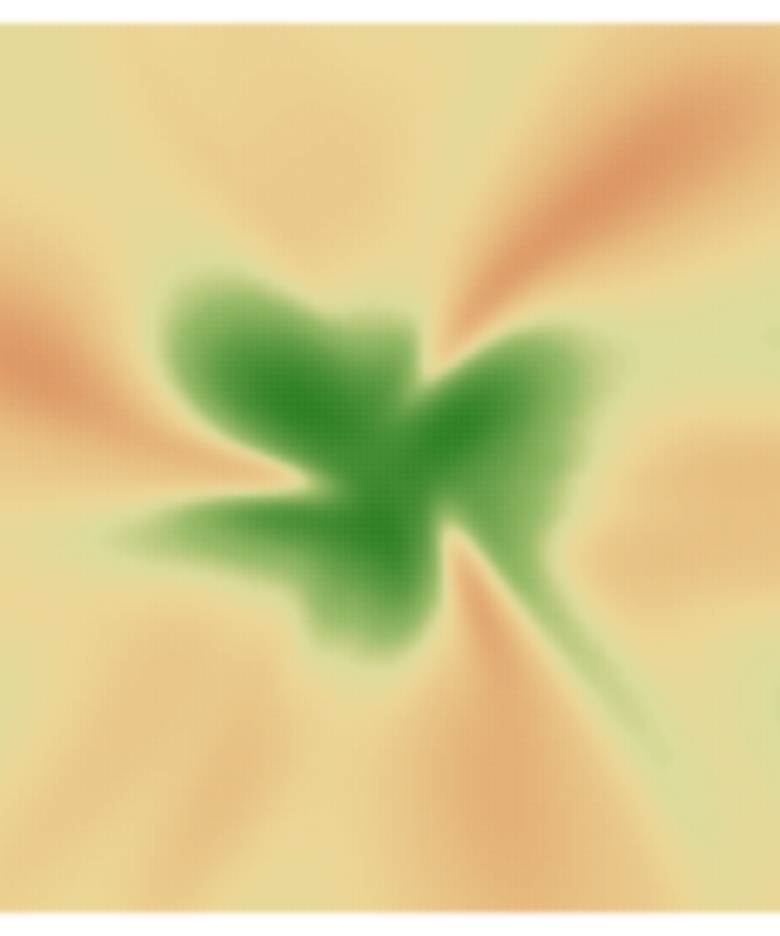}
\subcaption{$\rho^v_{r_{17}}$}
\end{subfigure}
\begin{subfigure}{.11\textwidth}
\centering
\includegraphics[width=.9\linewidth]{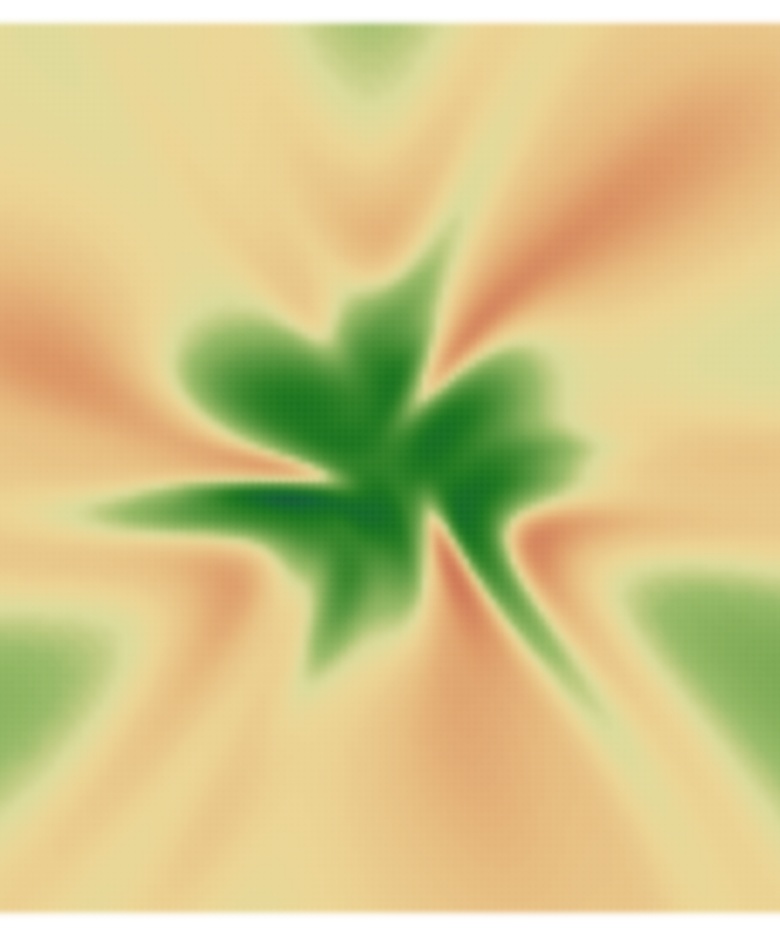}
\subcaption{$\rho^v_{r_{20}}$}
\end{subfigure}
\caption{Flowers: First row: log Jacobian; Second row: log Jacobian of residuals. Base scales: $r_{1}$, $r_{20}$.}
\label{fig: flower rotate residual}
\end{figure}

From  \cref{fig: flower rotate deformation} we see that the interior of the circle template shrinks dramatically, while the exterior of the circle stretches mostly. One might notice that at $r_{20}$ of  \cref{fig: flower rotate deformation}, the center of the flower has some folding. This is simply because our discretization with 20 time steps is not fine enough for the explicit Euler's method used to compute the diffeomorphism from the flow equation. (This singularity disappears when using 40 time steps, but we have kept the 20-step result here to be consistent with the other examples.)
\end{example}
\begin{example} This example maps a circle to a flower at the finest scale $r_1$ and to the same circle at the coarsest scale $r_{20}$.  \Cref{fig: flower and circle deformation} shows the deformed template at different scales. We can see that the flower gradually deforms to the circle. We can observe from the first row of  \cref{fig: flower and circle residual} at $r_1$ that the center of the circle is dramatically compressed, and then gradually stretch out to small values as the scale varies towards $r_{20}$, the second base scale where the landmarks are expected to be steady. As the deformation at the coarsest scale $r_{20}$ is approximately identity, the residual plots (second row of \cref{fig: flower and circle residual}) demonstrate an approximate decomposition of the identity map as compositions of transition transformations. 

\begin{figure}
\centering
\begin{subfigure}{.3\textwidth}
\centering
\includegraphics[width=.4\linewidth]{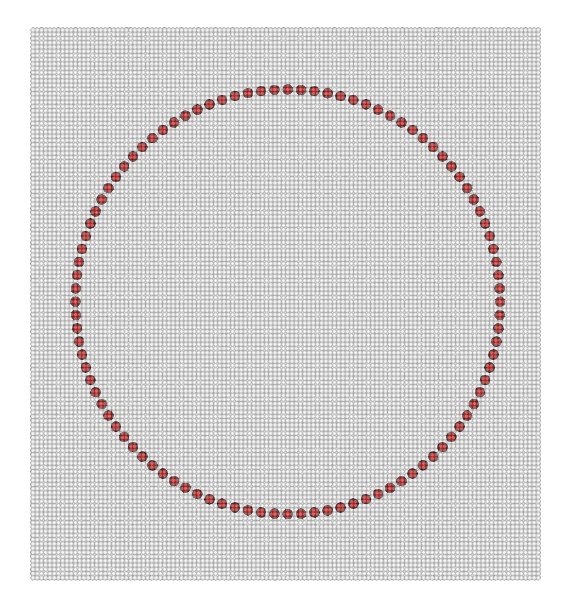}
\subcaption{template}
\end{subfigure}
\begin{subfigure}{.3\textwidth}
\centering
\includegraphics[width=.4\linewidth]{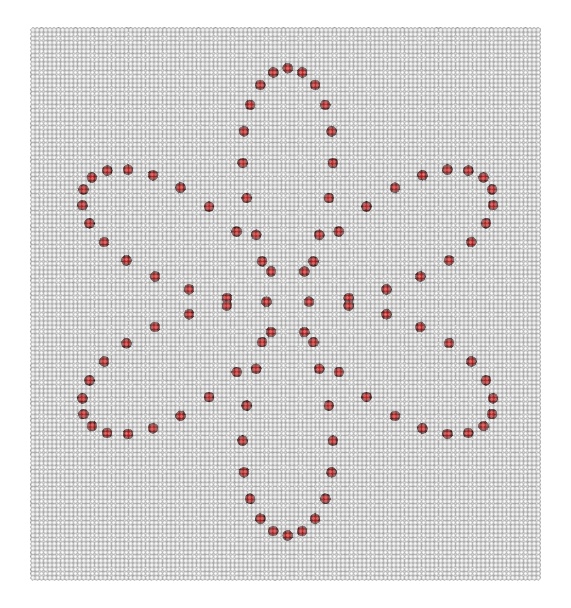}
\subcaption{target at $r_1$}
\end{subfigure}
\begin{subfigure}{.3\textwidth}
\centering
\includegraphics[width=.4\linewidth]{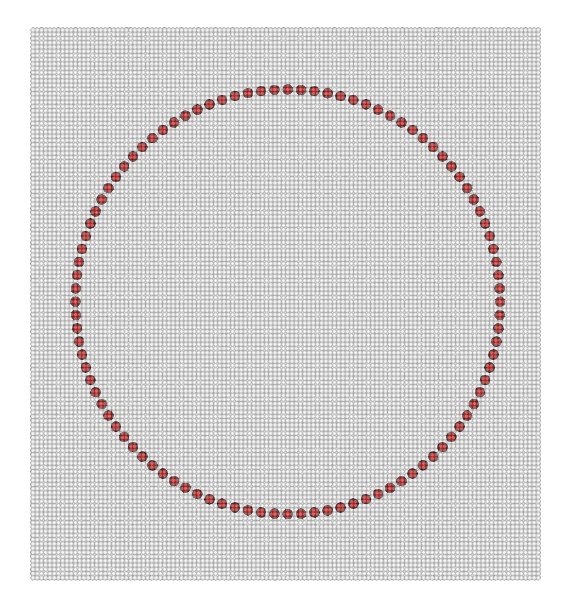}
\subcaption{target at $r_{20}$}
\end{subfigure}
\caption{Flower and circle example: template and targets.}
\end{figure}

\begin{figure}
\centering
\begin{subfigure}{.16\textwidth}
\centering
\includegraphics[width=.8\linewidth]{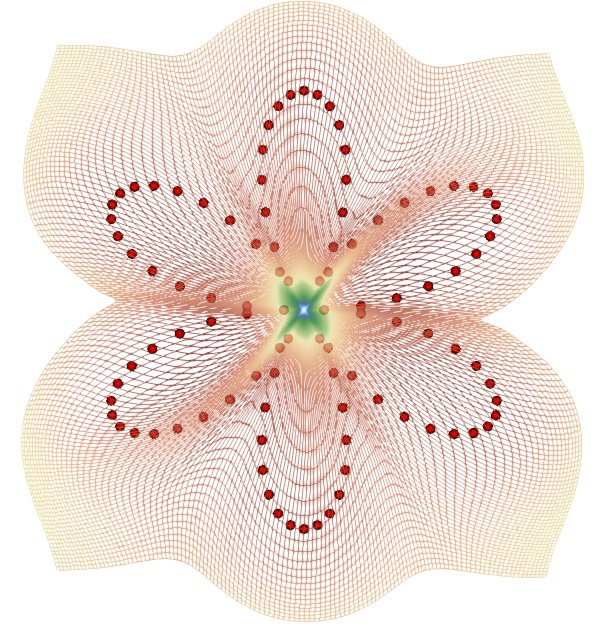}
\subcaption{$r_{1}$}
\end{subfigure}
\begin{subfigure}{.16\textwidth}
\centering
\includegraphics[width=.8\linewidth]{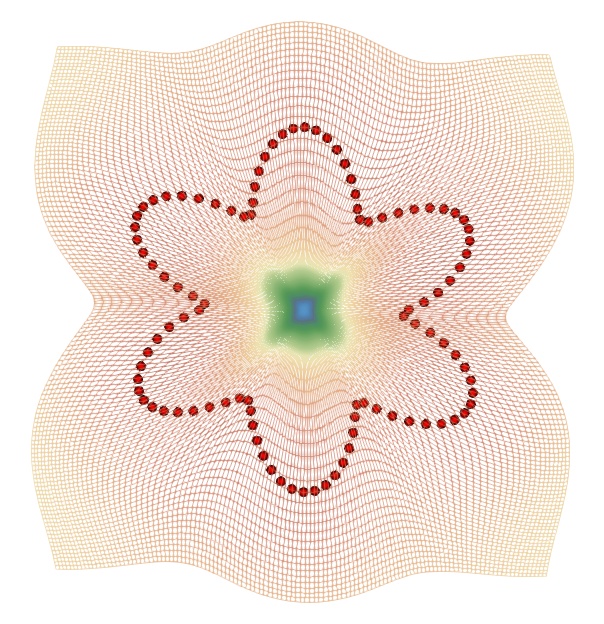}
\subcaption{$r_{4}$}
\end{subfigure}
\begin{subfigure}{.16\textwidth}
\centering
\includegraphics[width=.8\linewidth]{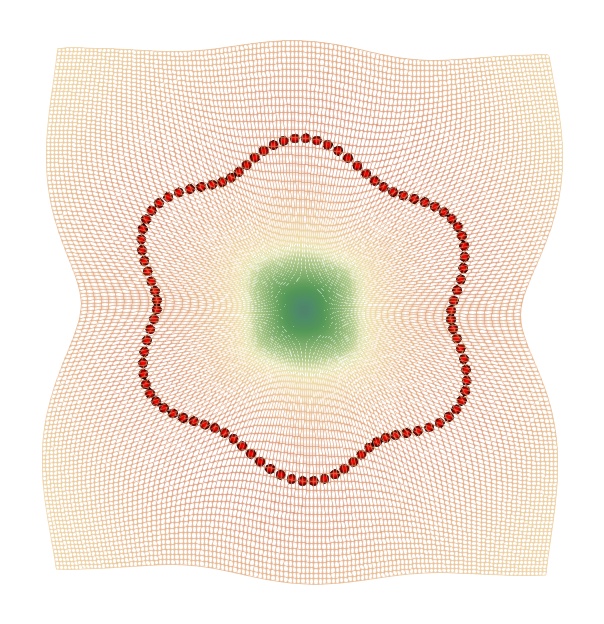}
\subcaption{$r_{7}$}
\end{subfigure}
\begin{subfigure}{.16\textwidth}
\centering
\includegraphics[width=.8\linewidth]{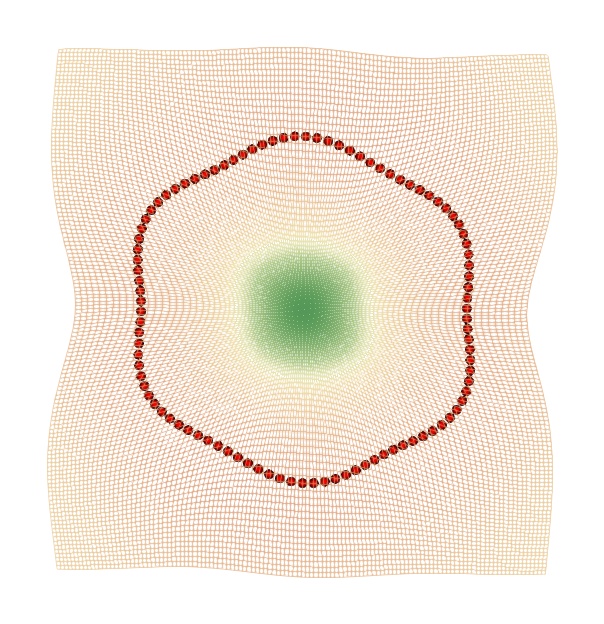}
\subcaption{$r_{9}$}
\end{subfigure}
\begin{subfigure}{.16\textwidth}
\centering
\includegraphics[width=.8\linewidth]{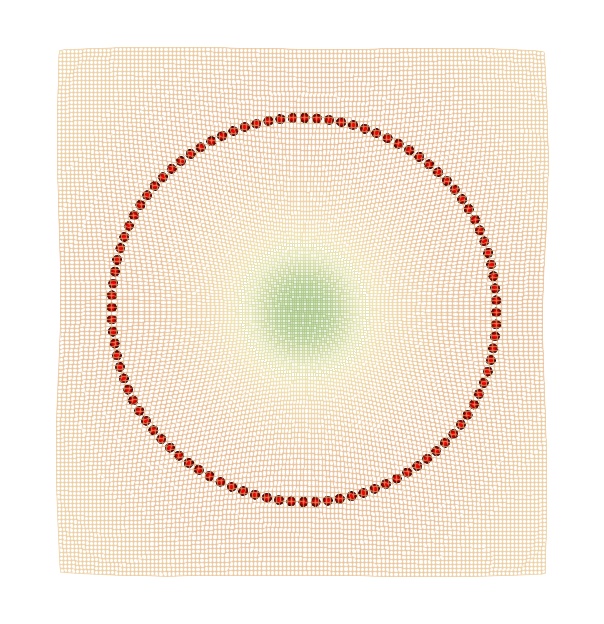}
\subcaption{$r_{17}$}
\end{subfigure}
\begin{subfigure}{.16\textwidth}
\centering
\includegraphics[width=.8\linewidth]{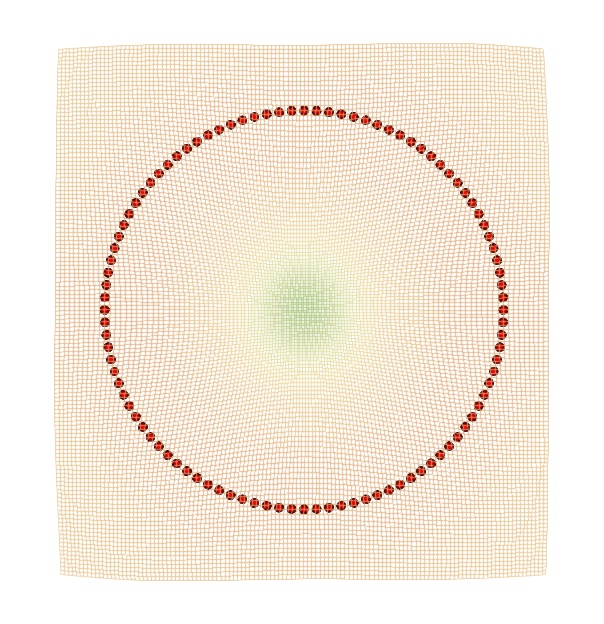}
\subcaption{$r_{20}$}
\end{subfigure}
\caption{Flower and circle: deformation. Base scales: $r_{1}$, $r_{20}$.}
\label{fig: flower and circle deformation}
\end{figure}

\begin{figure}
\centering
\begin{subfigure}{.12\textwidth}
\centering
\includegraphics[width=.8\linewidth]{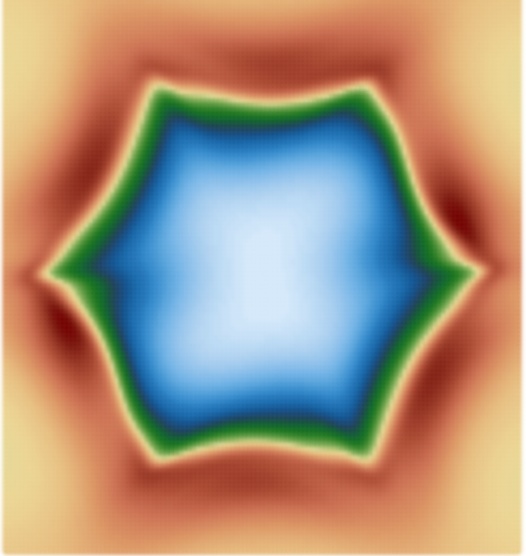}
\subcaption{$r_{1}$}
\end{subfigure}
\begin{subfigure}{.12\textwidth}
\centering
\includegraphics[width=.8\linewidth]{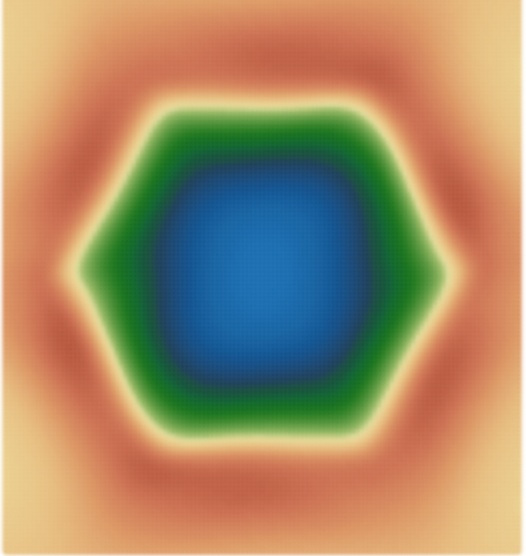}
\subcaption{$r_{4}$}
\end{subfigure}
\begin{subfigure}{.12\textwidth}
\centering
\includegraphics[width=.8\linewidth]{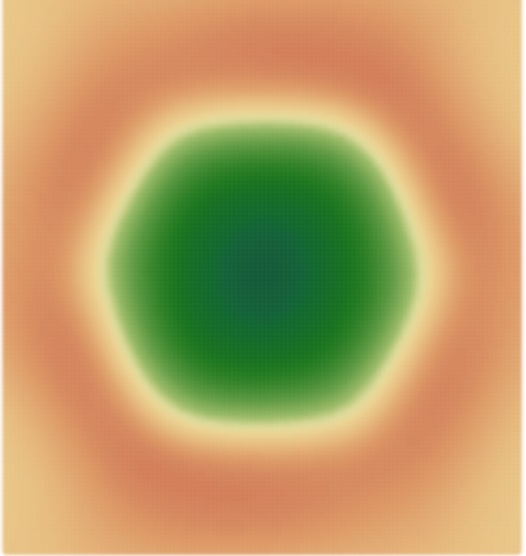}
\subcaption{$r_{7}$}
\end{subfigure}
\begin{subfigure}{.12\textwidth}
\centering
\includegraphics[width=.8\linewidth]{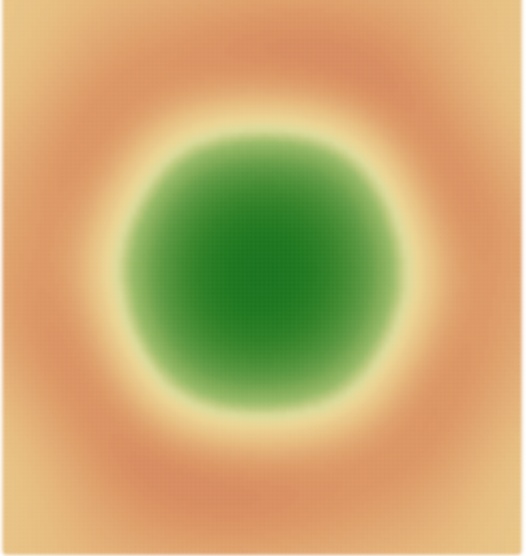}
\subcaption{$r_{9}$}
\end{subfigure}
\begin{subfigure}{.12\textwidth}
\centering
\includegraphics[width=.8\linewidth]{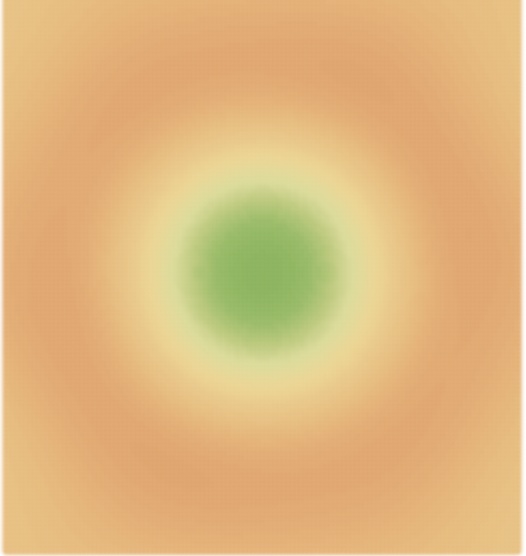}
\subcaption{$r_{17}$}
\end{subfigure}
\begin{subfigure}{.12\textwidth}
\centering
\includegraphics[width=.8\linewidth]{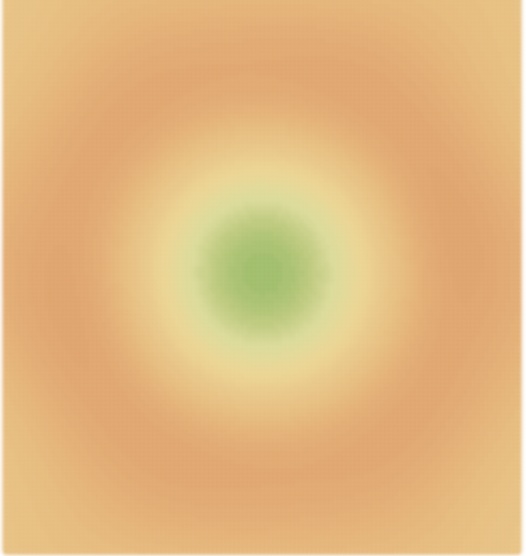}
\subcaption{$r_{20}$}
\end{subfigure}

\setcounter{subfigure}{0}
\begin{subfigure}{.12\textwidth}
\centering
\includegraphics[width=.8\linewidth]{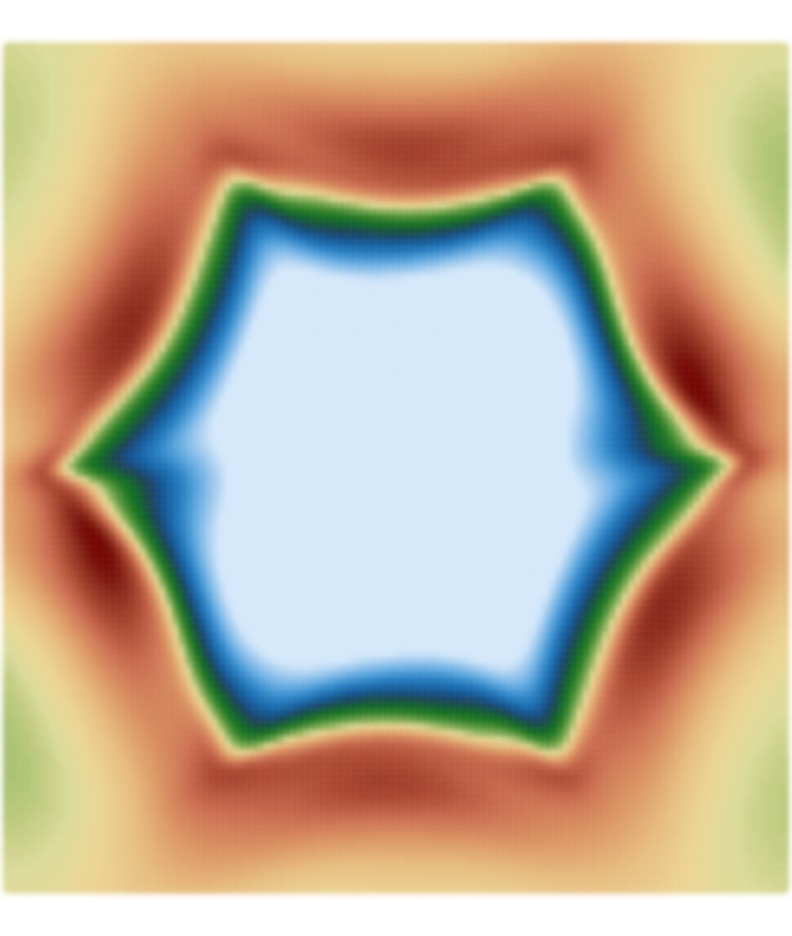}
\subcaption{$\rho^v_{r_{1}}$}
\end{subfigure}
\begin{subfigure}{.12\textwidth}
\centering
\includegraphics[width=.8\linewidth]{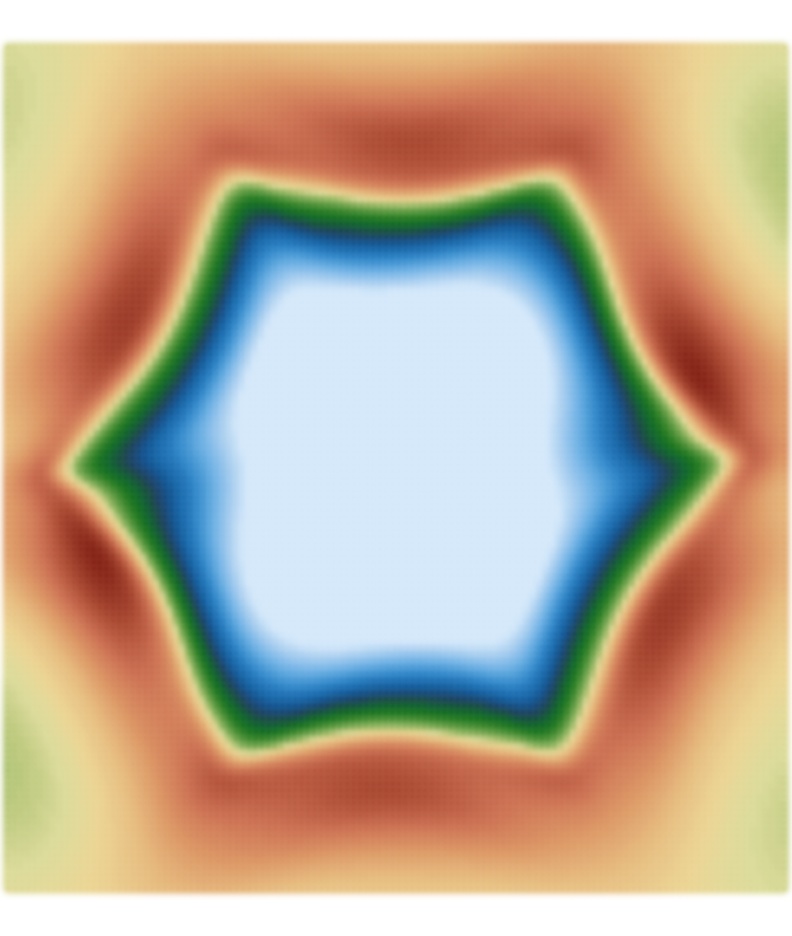}
\subcaption{$\rho^v_{r_{2}}$}
\end{subfigure}
\begin{subfigure}{.12\textwidth}
\centering
\includegraphics[width=.8\linewidth]{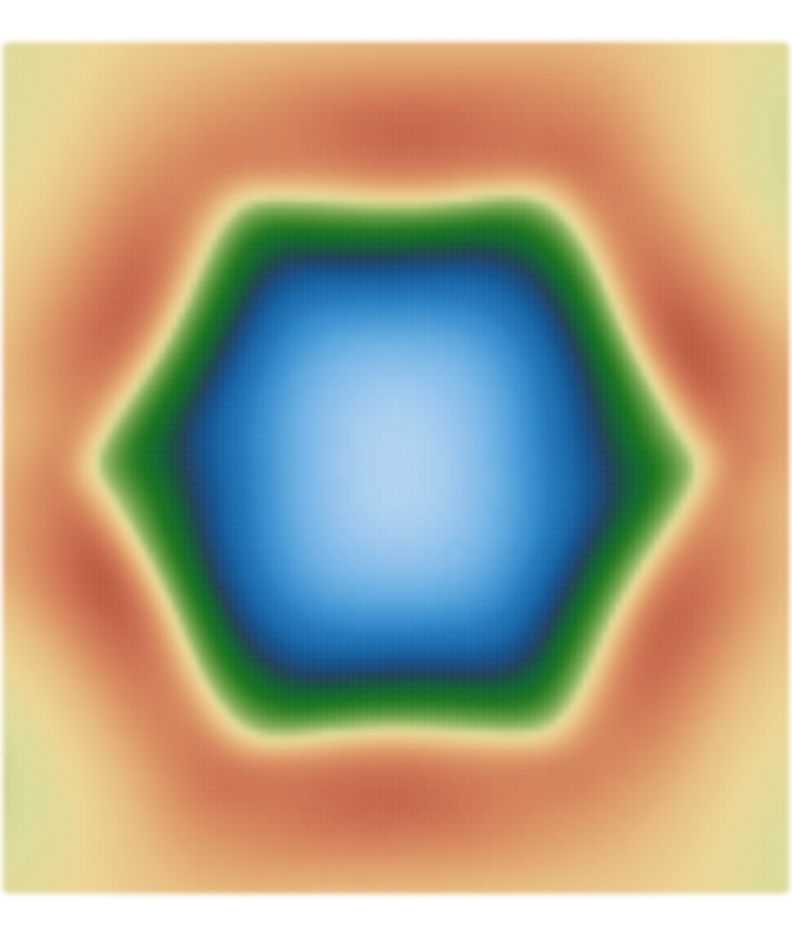}
\subcaption{$\rho^v_{r_{4}}$}
\end{subfigure}
\begin{subfigure}{.12\textwidth}
\centering
\includegraphics[width=.8\linewidth]{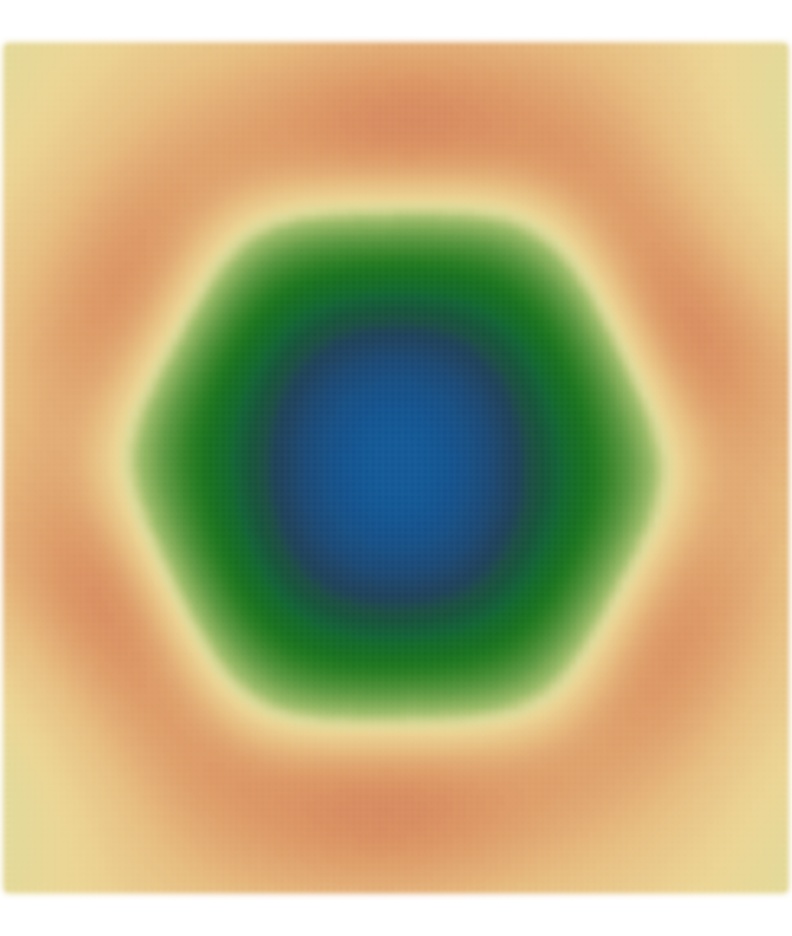}
\subcaption{$\rho^v_{r_{7}}$}
\end{subfigure}
\begin{subfigure}{.12\textwidth}
\centering
\includegraphics[width=.8\linewidth]{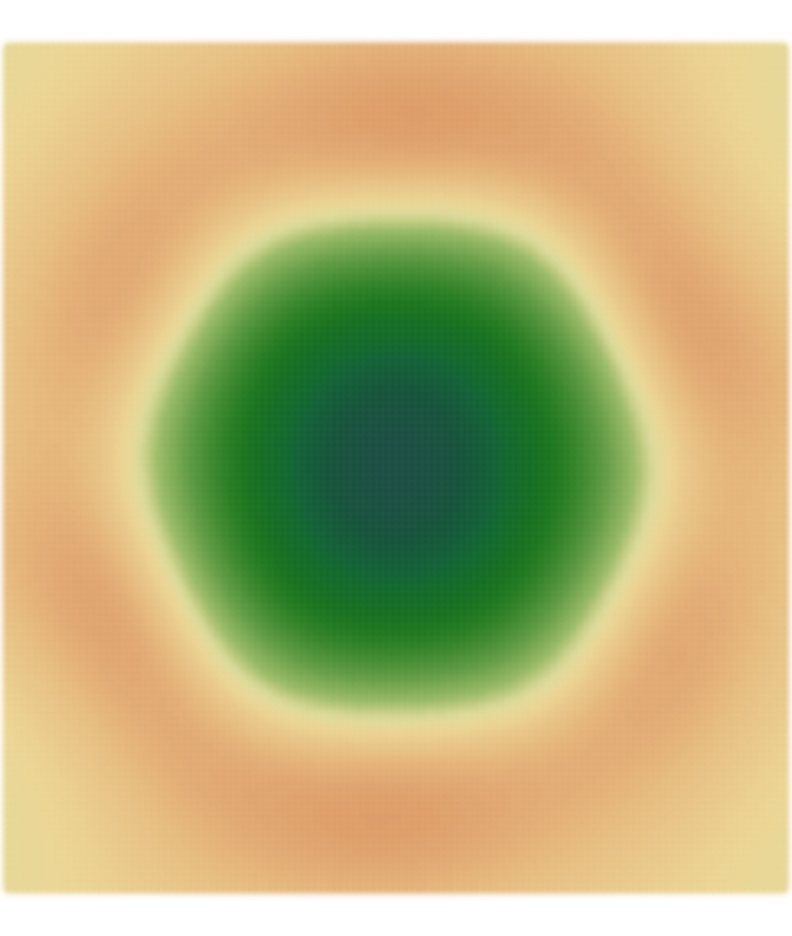}
\subcaption{$\rho^v_{r_{9}}$}
\end{subfigure}
\begin{subfigure}{.12\textwidth}
\centering
\includegraphics[width=.8\linewidth]{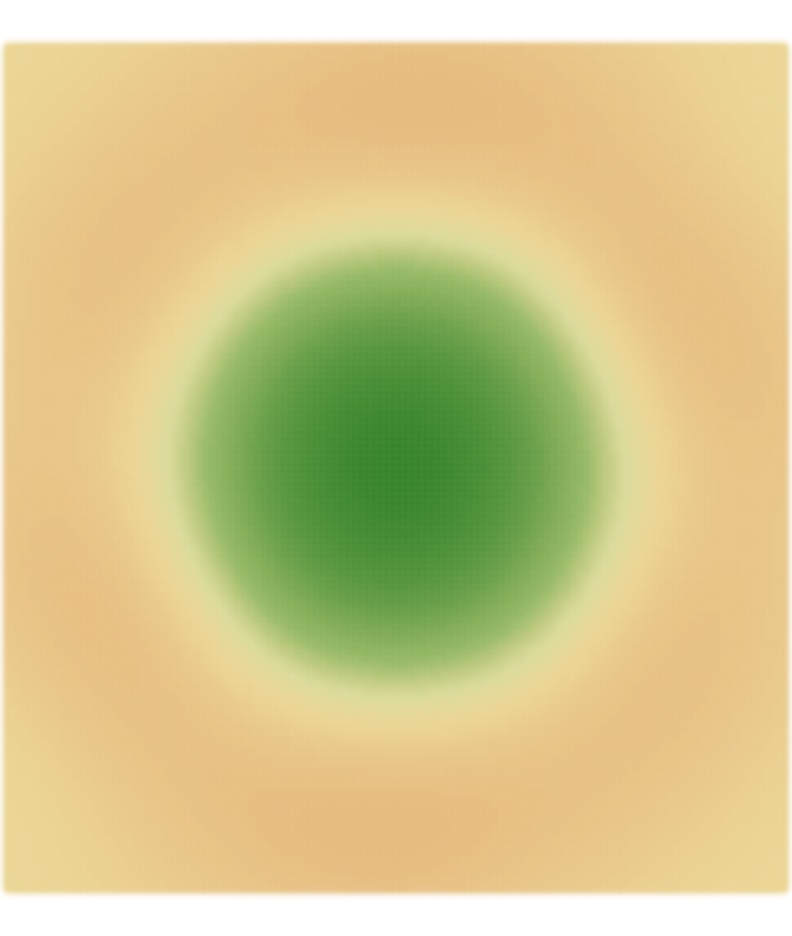}
\subcaption{$\rho^v_{r_{17}}$}
\end{subfigure}
\begin{subfigure}{.12\textwidth}
\centering
\includegraphics[width=.8\linewidth]{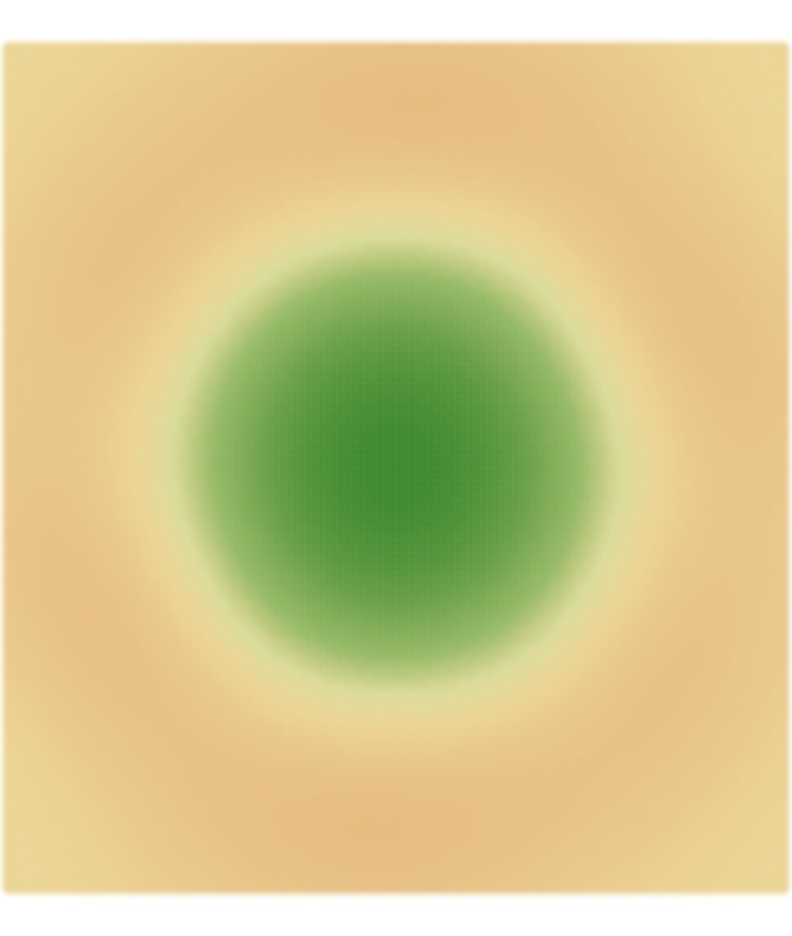}
\subcaption{$\rho^v_{r_{20}}$}
\end{subfigure}
\caption{Flower and circle: First row: log Jacobian; Second row: log Jacobian of residuals. Base scale: $r_1$, $r_{20}$.}
\label{fig: flower and circle residual}
\end{figure}
\end{example}

\begin{example}
For this example, we consider an other schematic human example in which two templates being mapped to two different targets at two scales respectively. As is shown in  \cref{fig: waving template and target}, the template in (a) is mapped to the target in (b) at scale $r_1$, and template in (c) is mapped to the target in (d) at $r_{20}$.

\begin{figure}
\centering
\begin{subfigure}{.24\textwidth}
\centering
\includegraphics[width=\linewidth]{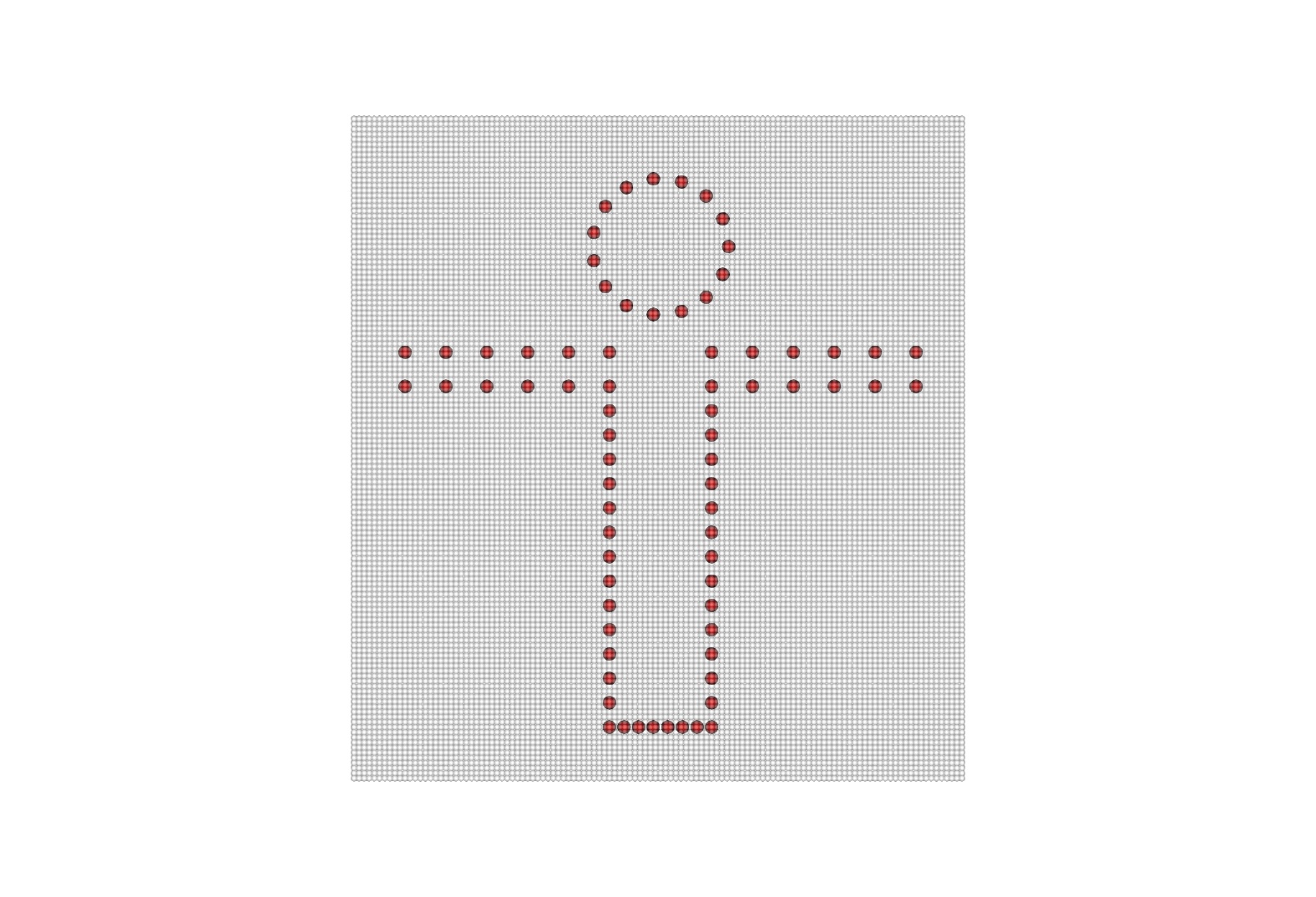}
\subcaption{template at $r_1$}
\end{subfigure}
\begin{subfigure}{.24\textwidth}
\centering
\includegraphics[width=\linewidth]{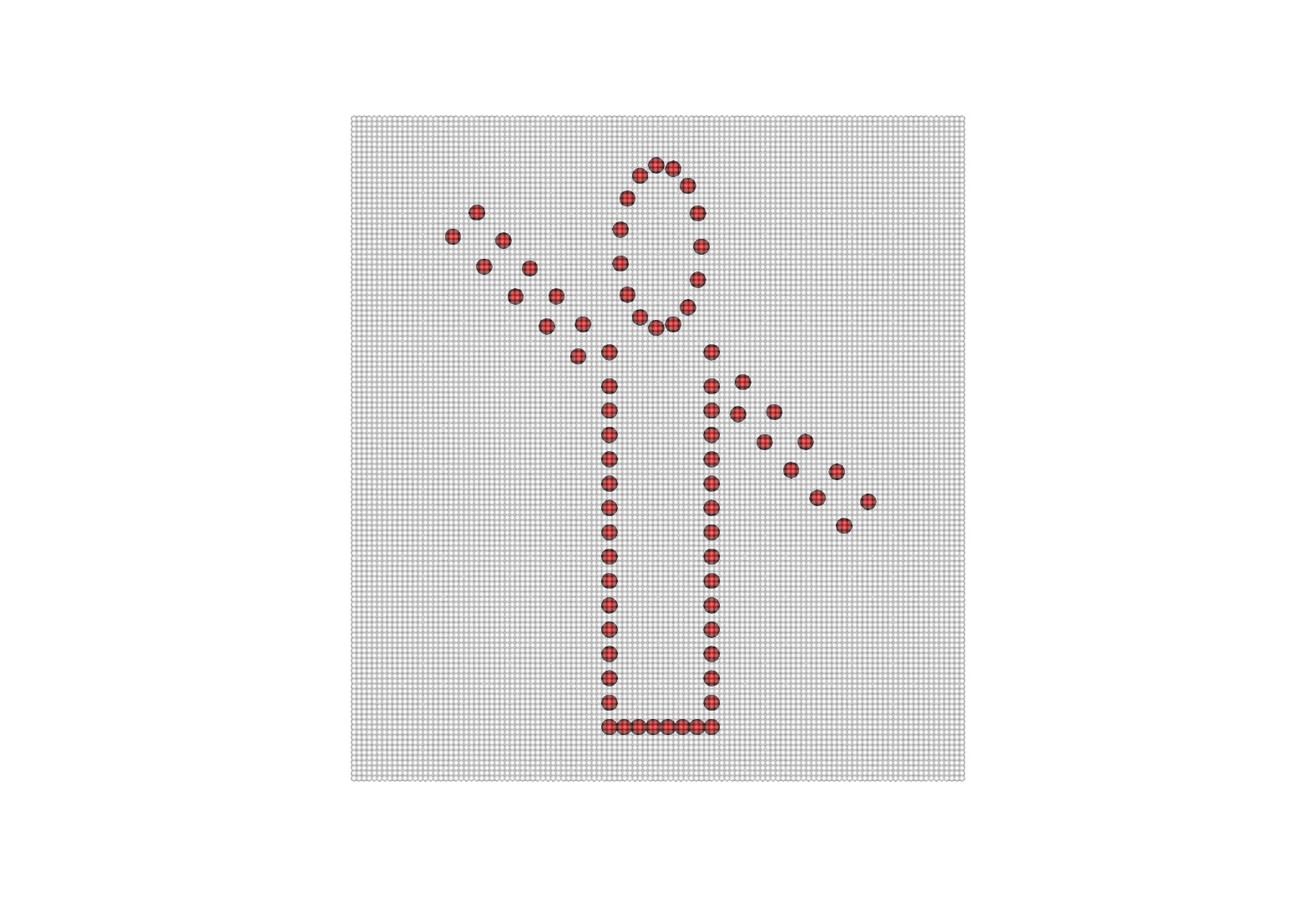}
\subcaption{target at $r_1$}
\end{subfigure}
\begin{subfigure}{.24\textwidth}
\centering
\includegraphics[width=\linewidth]{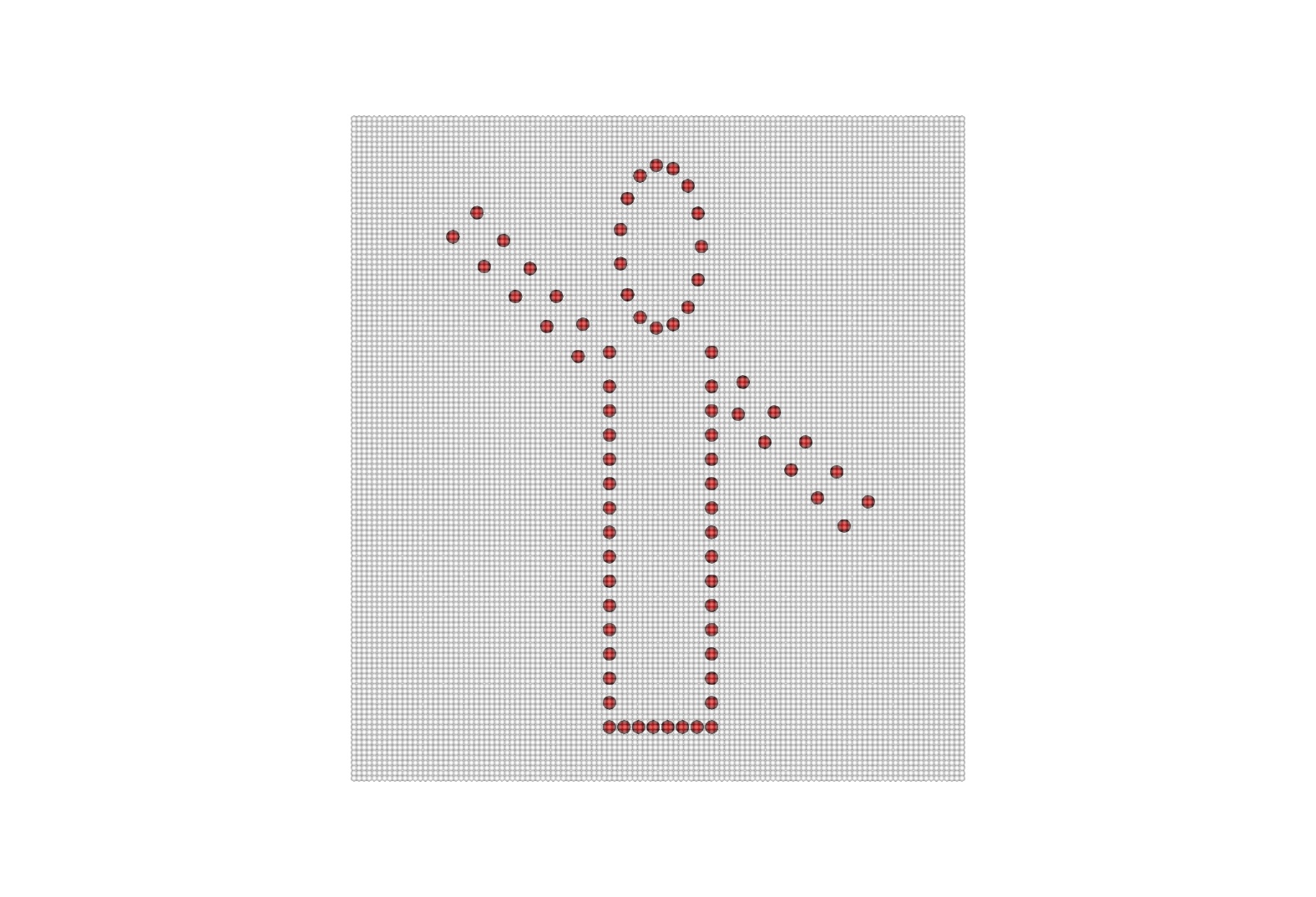}
\subcaption{template at $r_{20}$}
\end{subfigure}
\begin{subfigure}{.24\textwidth}
\centering
\includegraphics[width=\linewidth]{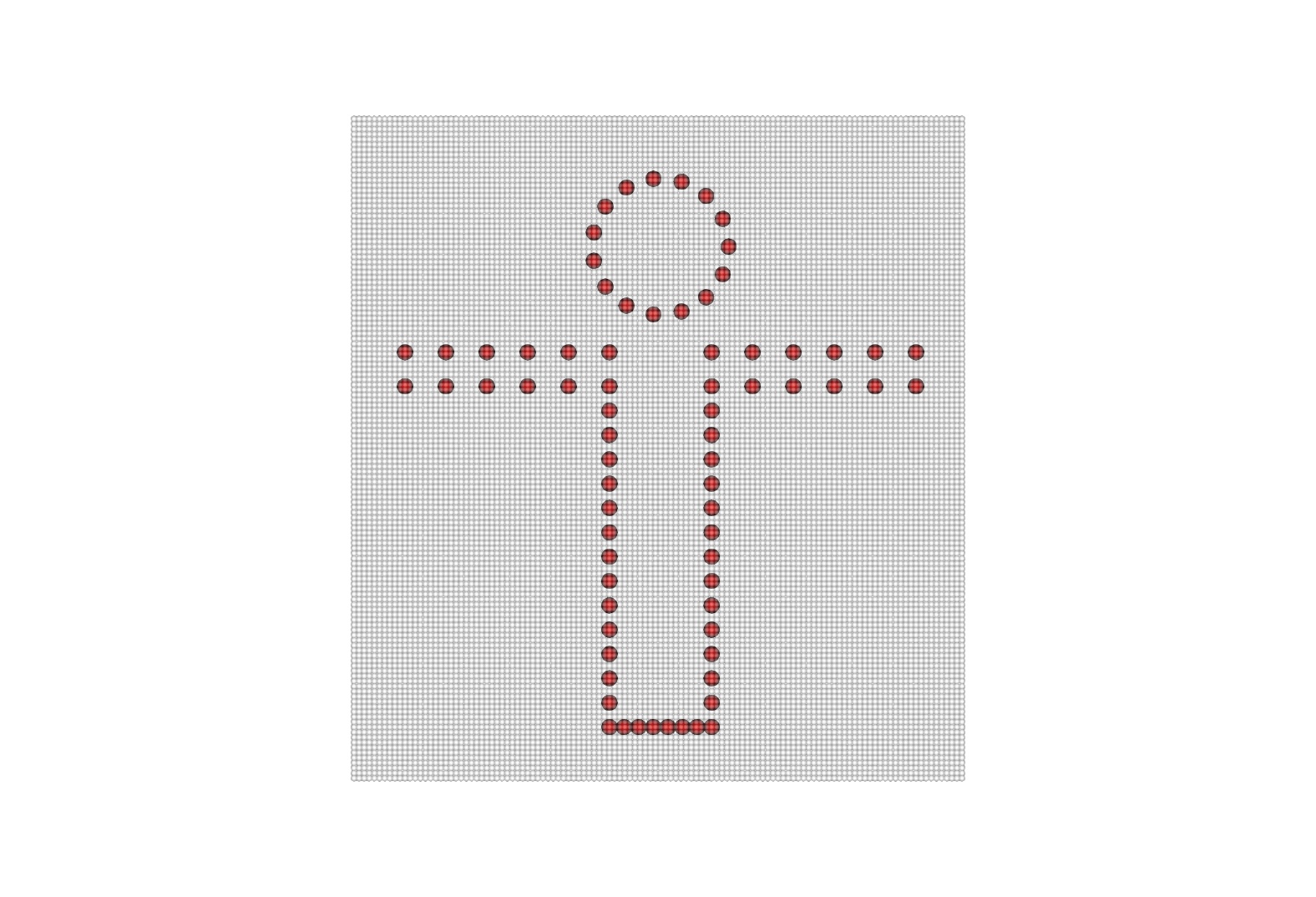}
\subcaption{target at $r_{20}$}
\end{subfigure}
\caption{Schematic human example 2: templates and targets.}
\label{fig: waving template and target}
\end{figure}

Since there are two sets of templates, there are two set of landmarks we need to keep track of.  The motion of these landmarks at different scales are shows in the two rows of \cref{fig: lmkset2 deformation}.
\Cref{fig: waving residual} shows the log Jacobian  of the transformations and  of their the residuals  on the original grid. Although the multiscale result presents interesting inter-scale behavior, the existence of multiple templates prevents further interpretation. 

\begin{figure}
\centering
\begin{subfigure}{.15\textwidth}
\centering
\includegraphics[width=\linewidth]{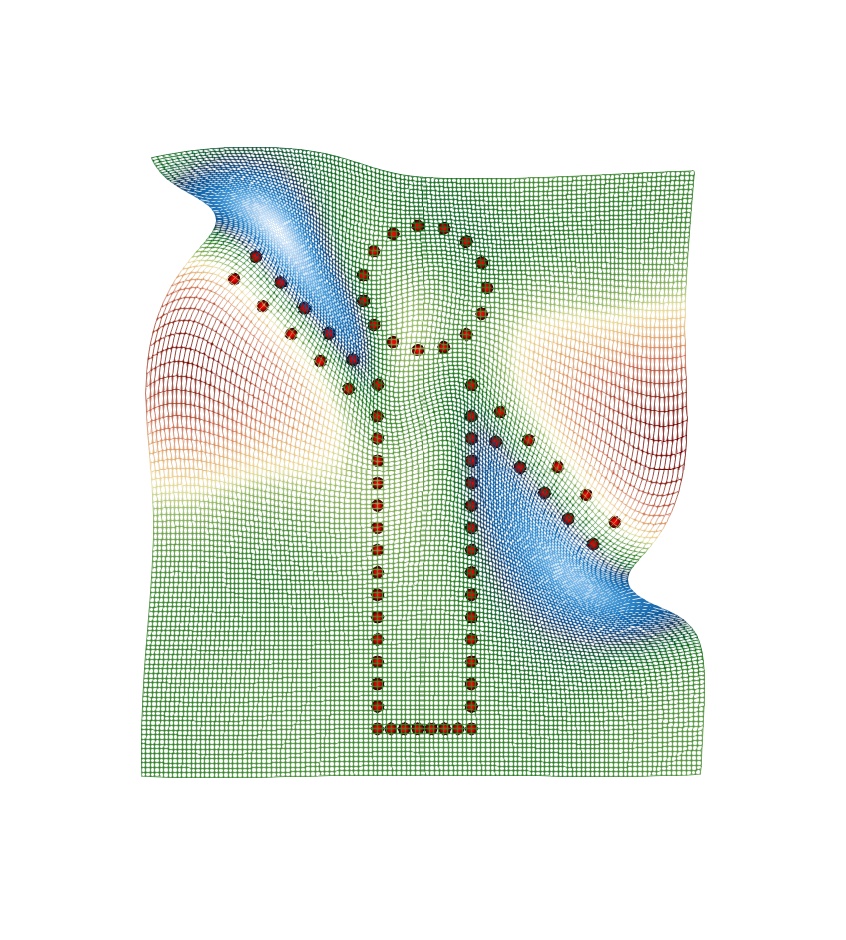}
\subcaption{$r_{1}$}
\end{subfigure}
\begin{subfigure}{.15\textwidth}
\centering
\includegraphics[width=\linewidth]{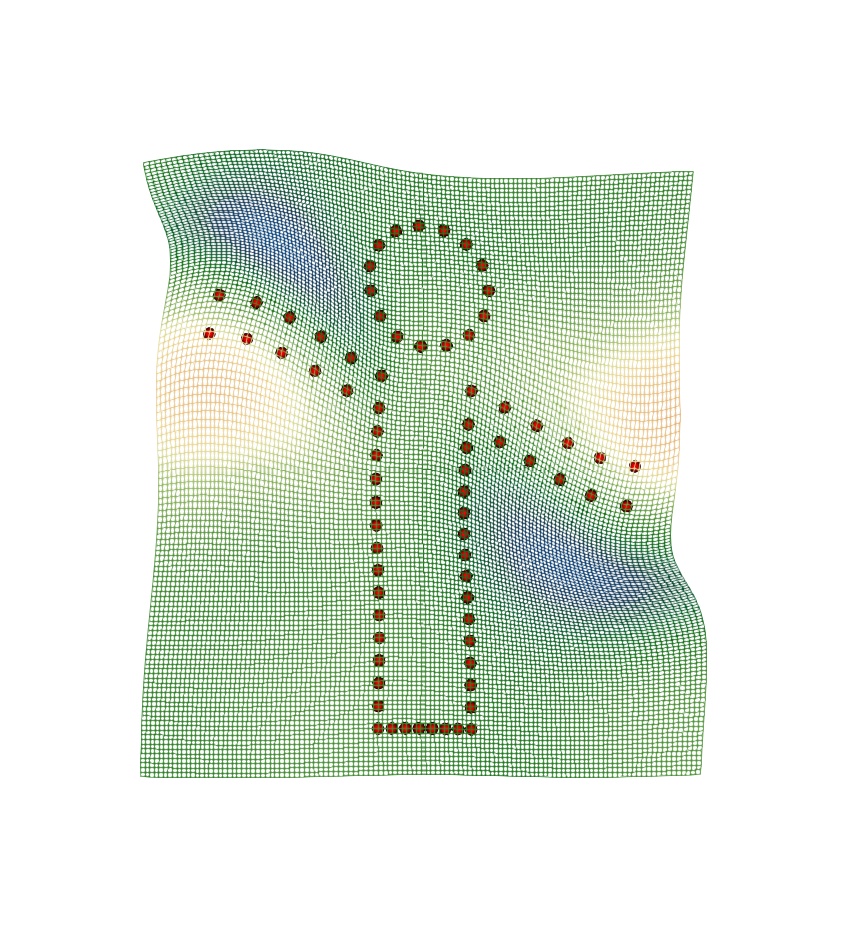}
\subcaption{$r_{5}$}
\end{subfigure}
\begin{subfigure}{.15\textwidth}
\centering
\includegraphics[width=\linewidth]{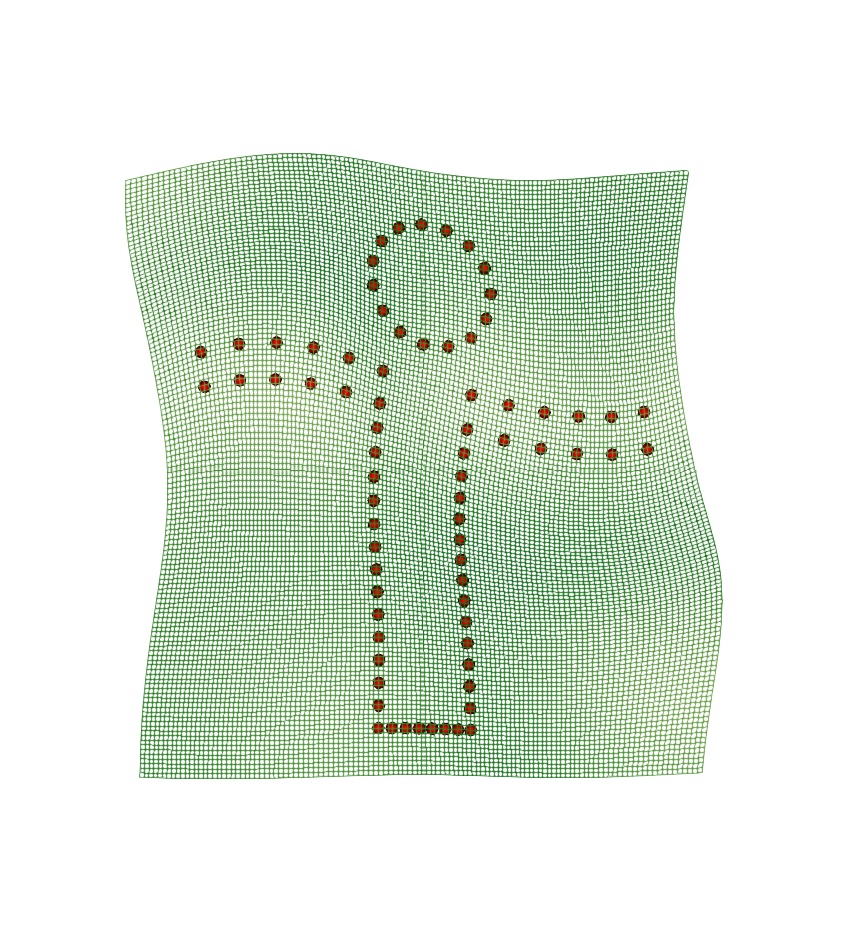}
\subcaption{$r_{9}$}
\end{subfigure}
\begin{subfigure}{.15\textwidth}
\centering
\includegraphics[width=\linewidth]{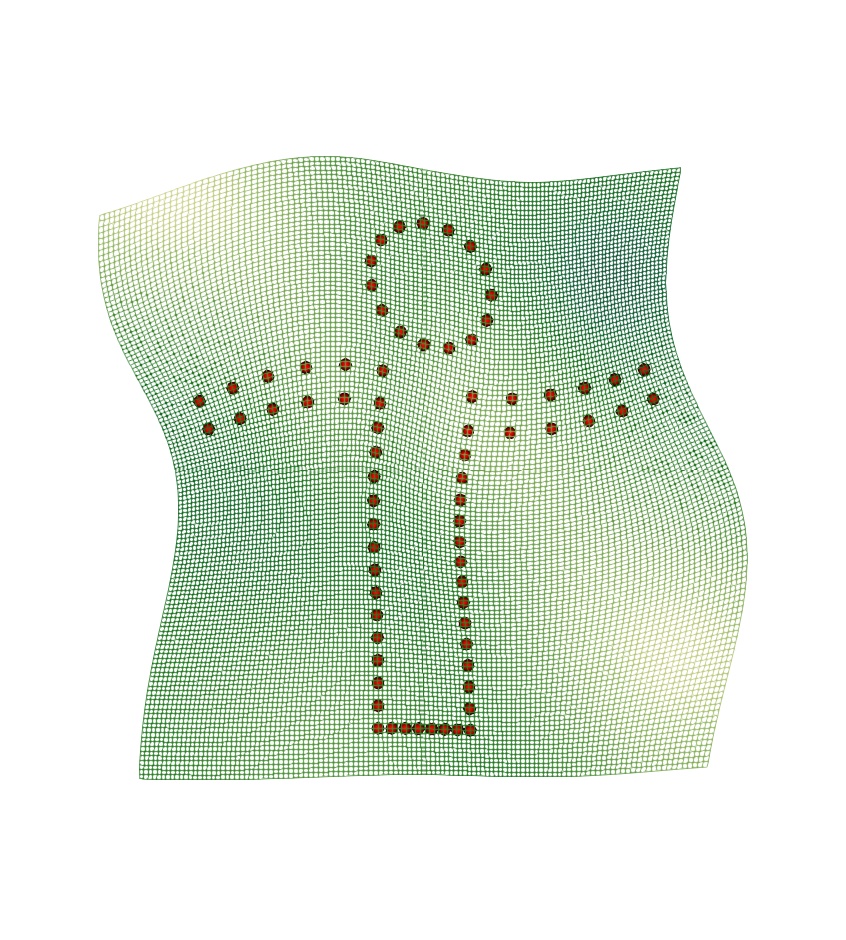}
\subcaption{$r_{13}$}
\end{subfigure}
\begin{subfigure}{.15\textwidth}
\centering
\includegraphics[width=\linewidth]{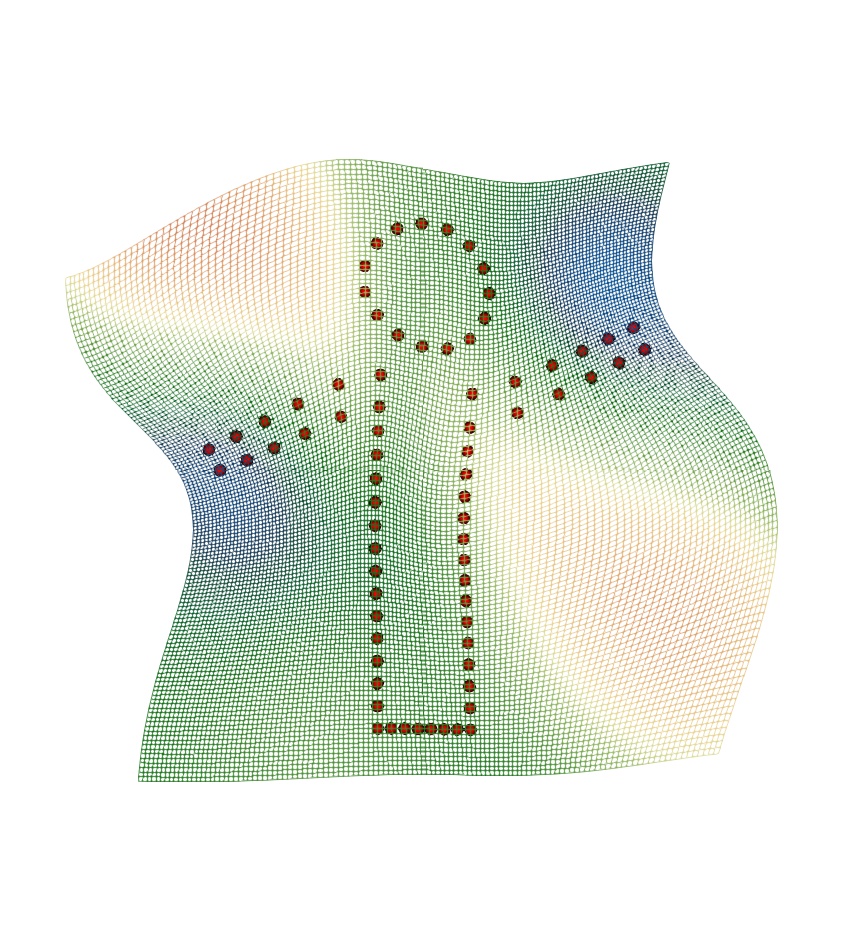}
\subcaption{$r_{17}$}
\end{subfigure}
\begin{subfigure}{.15\textwidth}
\centering
\includegraphics[width=\linewidth]{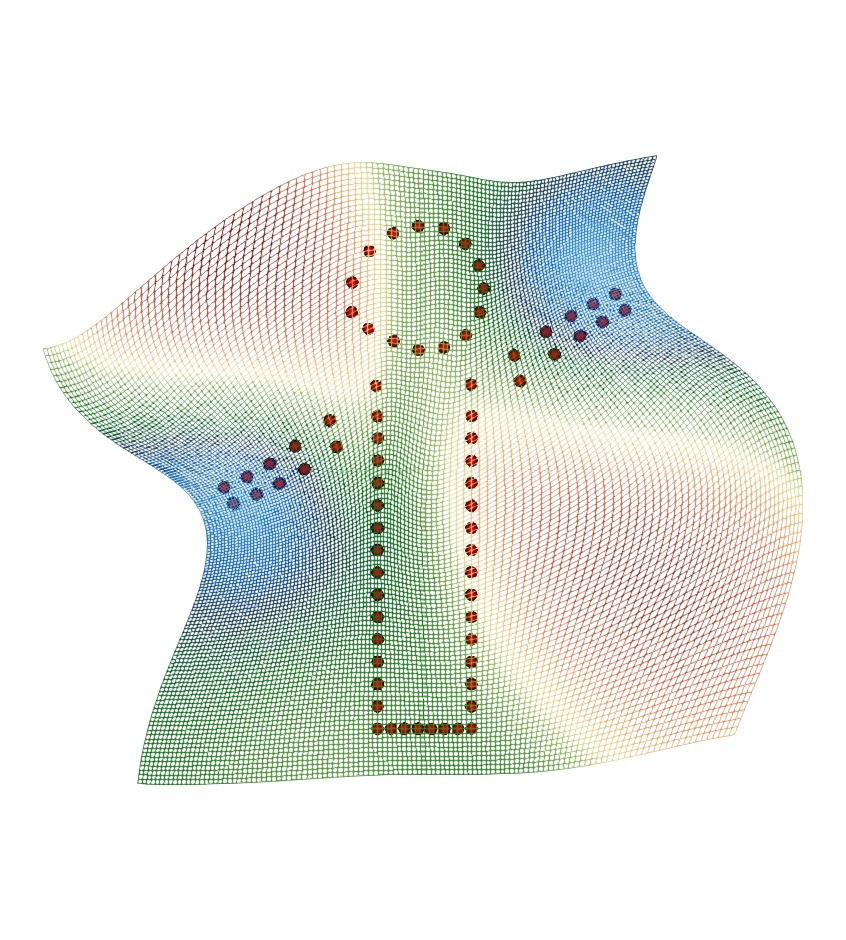}
\subcaption{$r_{20}$}
\end{subfigure}

\setcounter{subfigure}{0}
\centering
\begin{subfigure}{.15\textwidth}
\centering
\includegraphics[width=\linewidth]{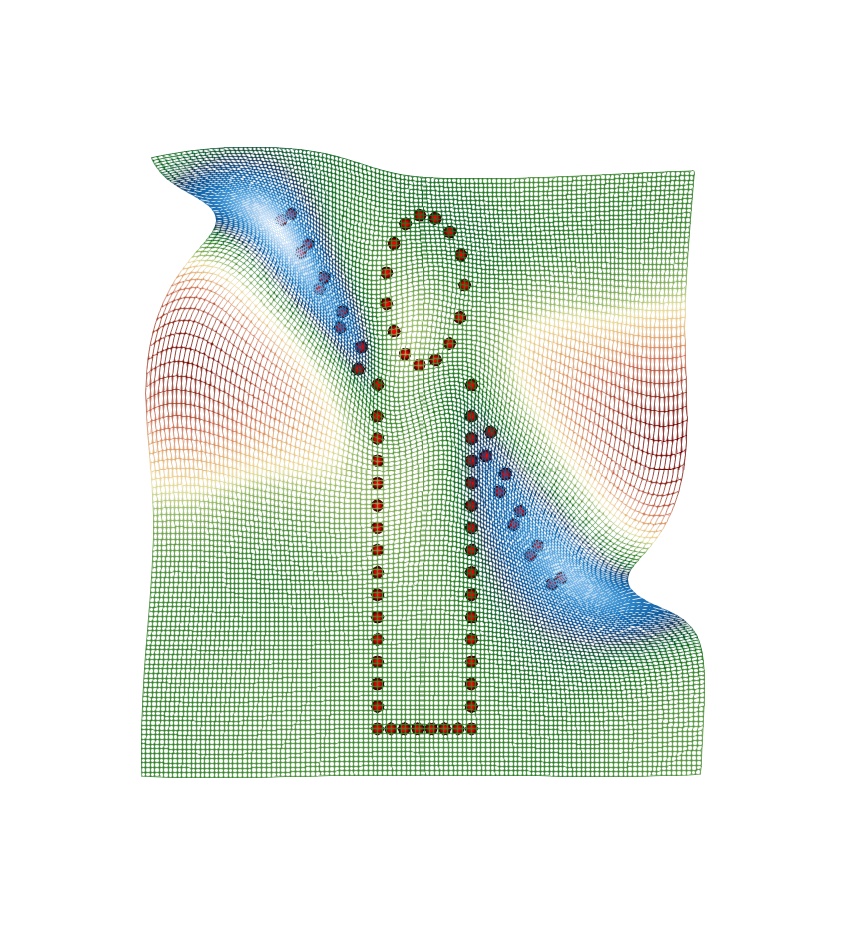}
\subcaption{$r_{1}$}
\end{subfigure}
\begin{subfigure}{.15\textwidth}
\centering
\includegraphics[width=\linewidth]{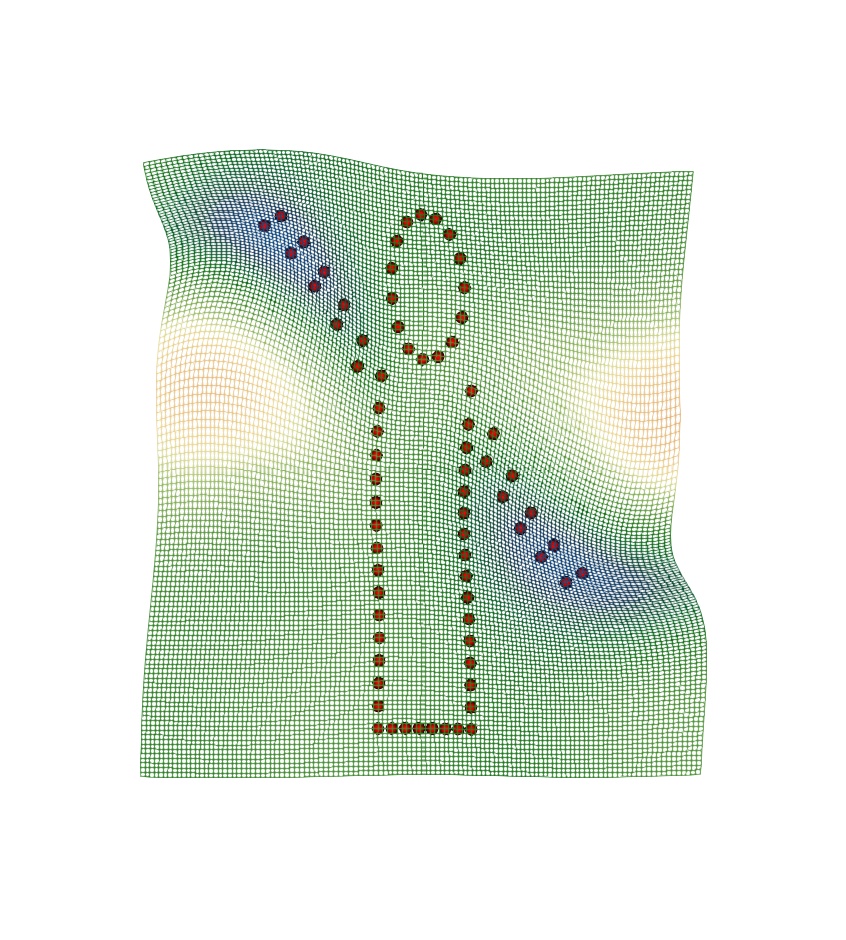}
\subcaption{$r_{5}$}
\end{subfigure}
\begin{subfigure}{.15\textwidth}
\centering
\includegraphics[width=\linewidth]{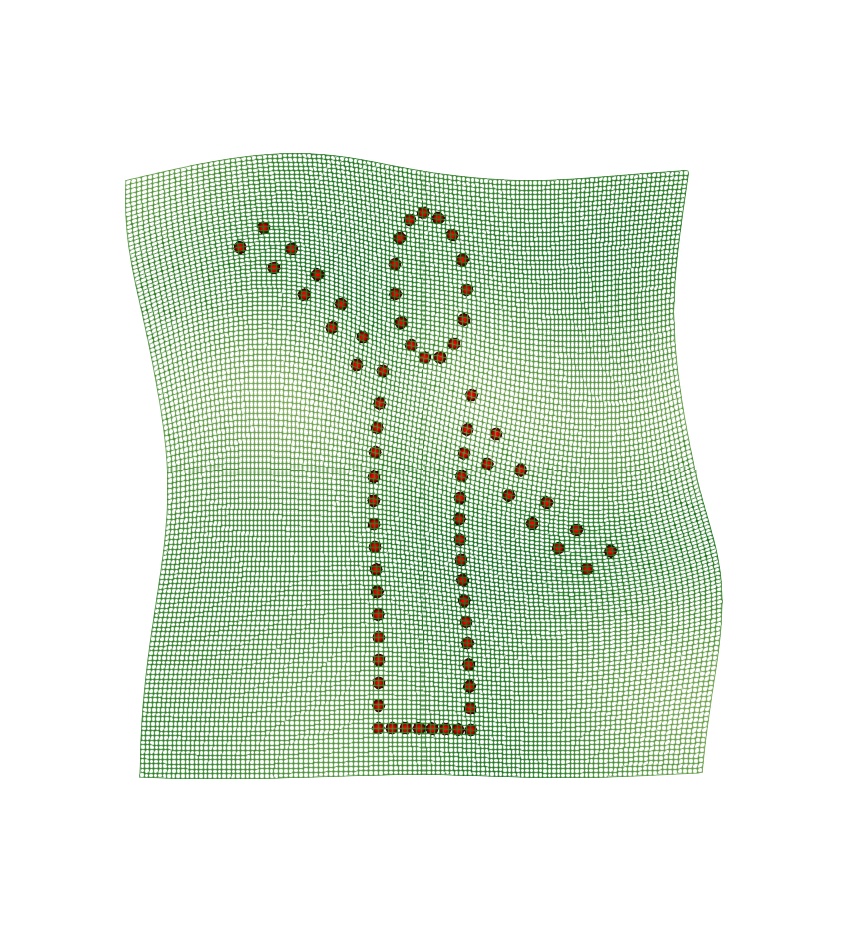}
\subcaption{$r_{9}$}
\end{subfigure}
\begin{subfigure}{.15\textwidth}
\centering
\includegraphics[width=\linewidth]{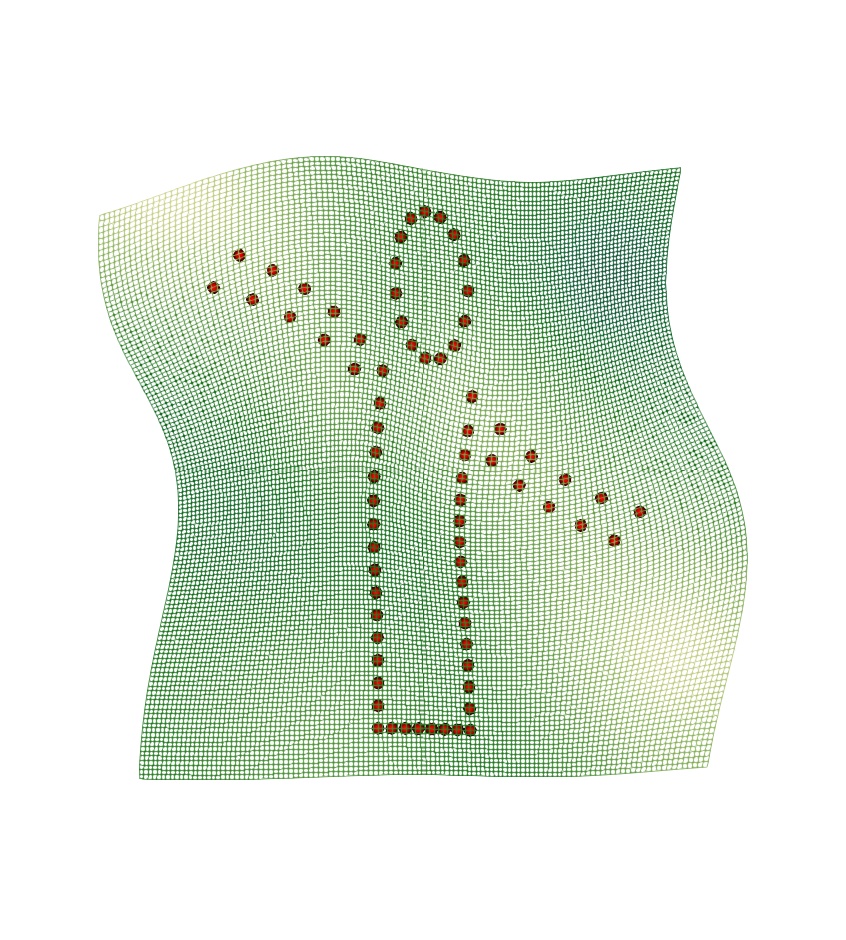}
\subcaption{$r_{13}$}
\end{subfigure}
\begin{subfigure}{.15\textwidth}
\centering
\includegraphics[width=\linewidth]{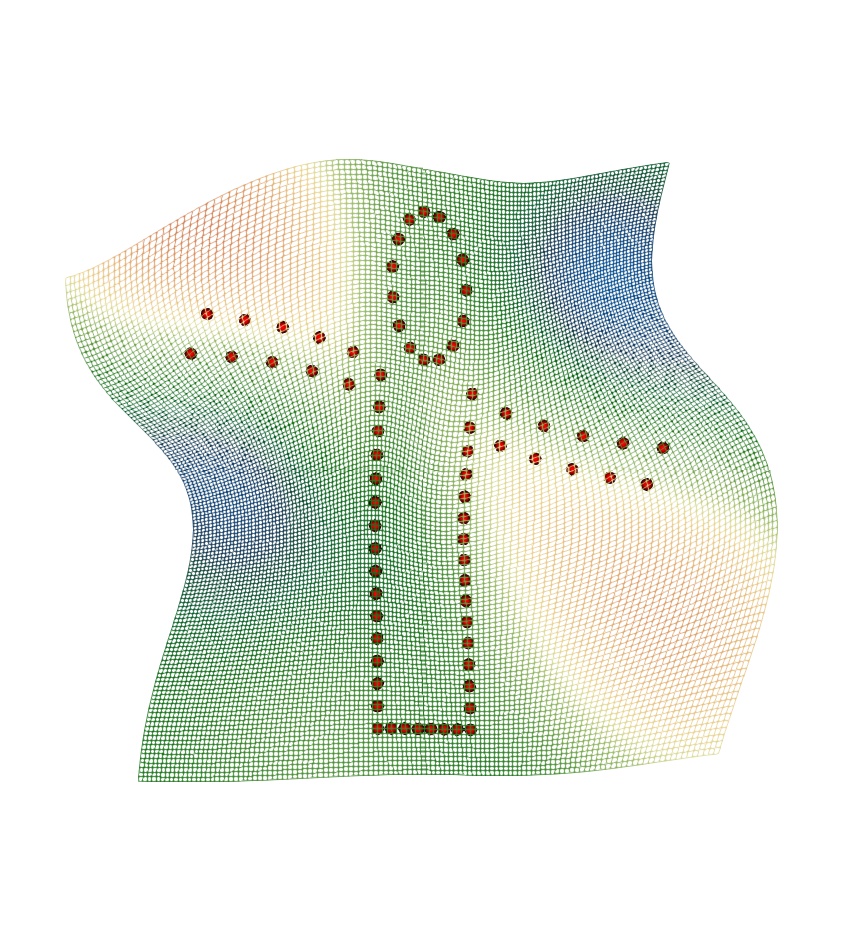}
\subcaption{$r_{17}$}
\end{subfigure}
\begin{subfigure}{.15\textwidth}
\centering
\includegraphics[width=\linewidth]{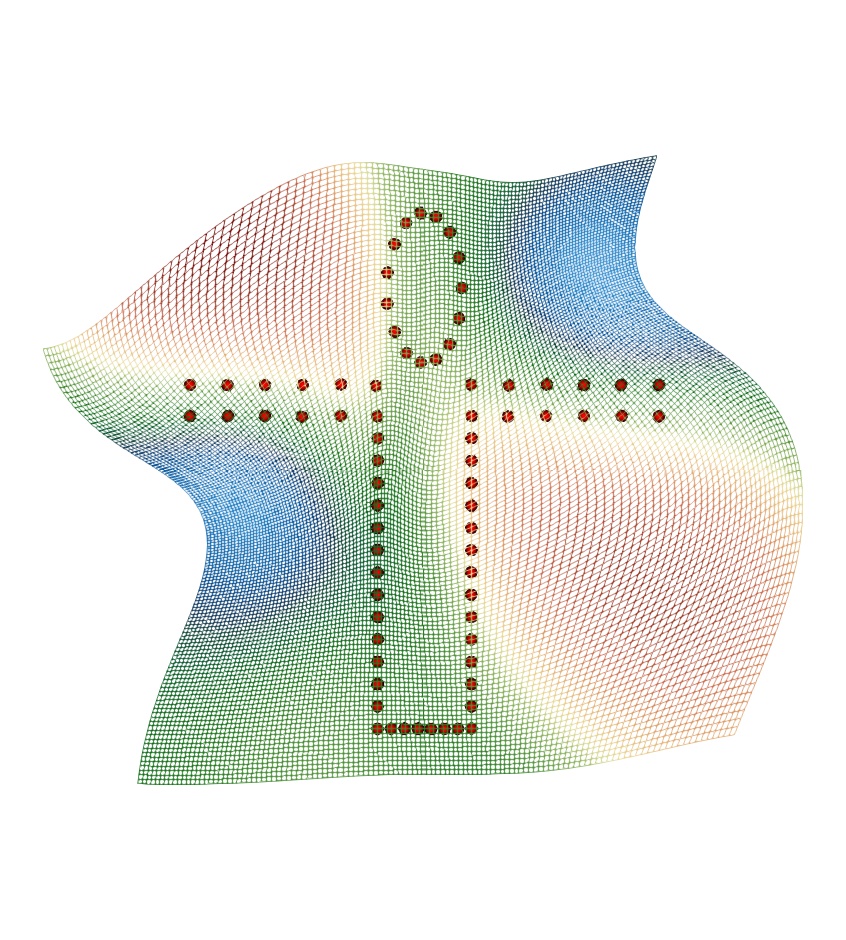}
\subcaption{$r_{20}$}
\end{subfigure}
\caption{Schematic human example with different templates and targets. First row: First template (with penalty constraint  at scale $r_1$) transformed by the multiscale transformation Second row: Same transformations applied to the second template (with penalty constraint at scale $r_{20}$). }
\label{fig: lmkset2 deformation}
\end{figure}

\begin{figure}
\centering
\begin{subfigure}{.13\textwidth}
\centering
\includegraphics[width=.8\linewidth]{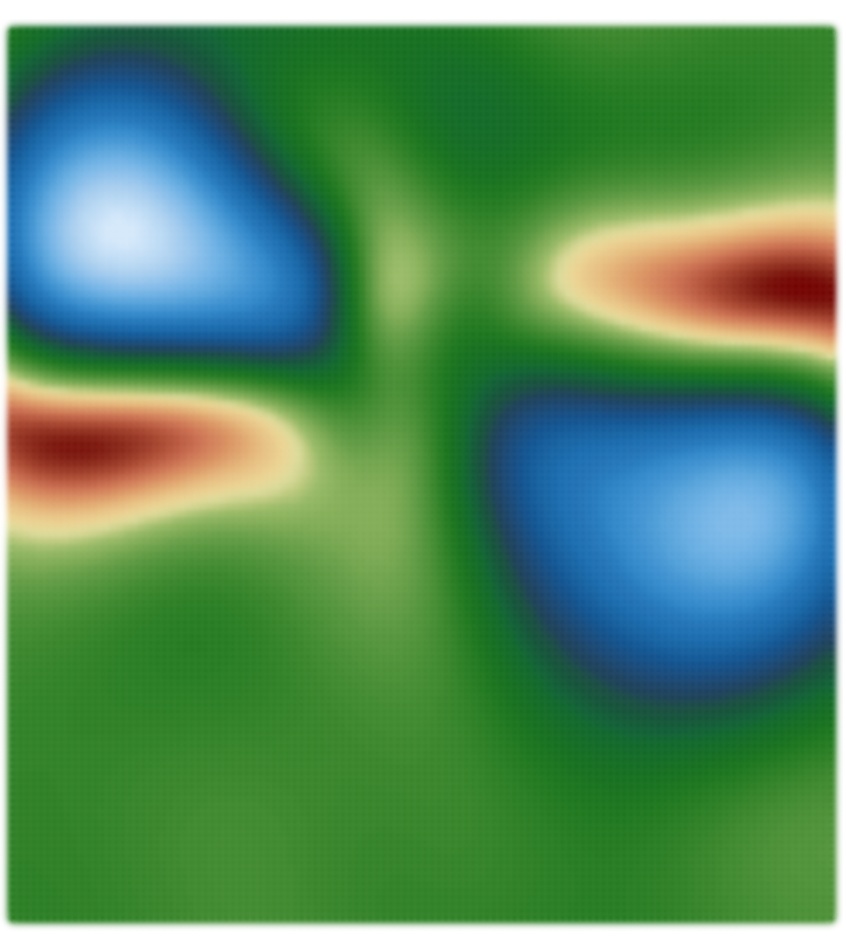}
\subcaption{$r_{1}$}
\end{subfigure}
\begin{subfigure}{.13\textwidth}
\centering
\includegraphics[width=.8\linewidth]{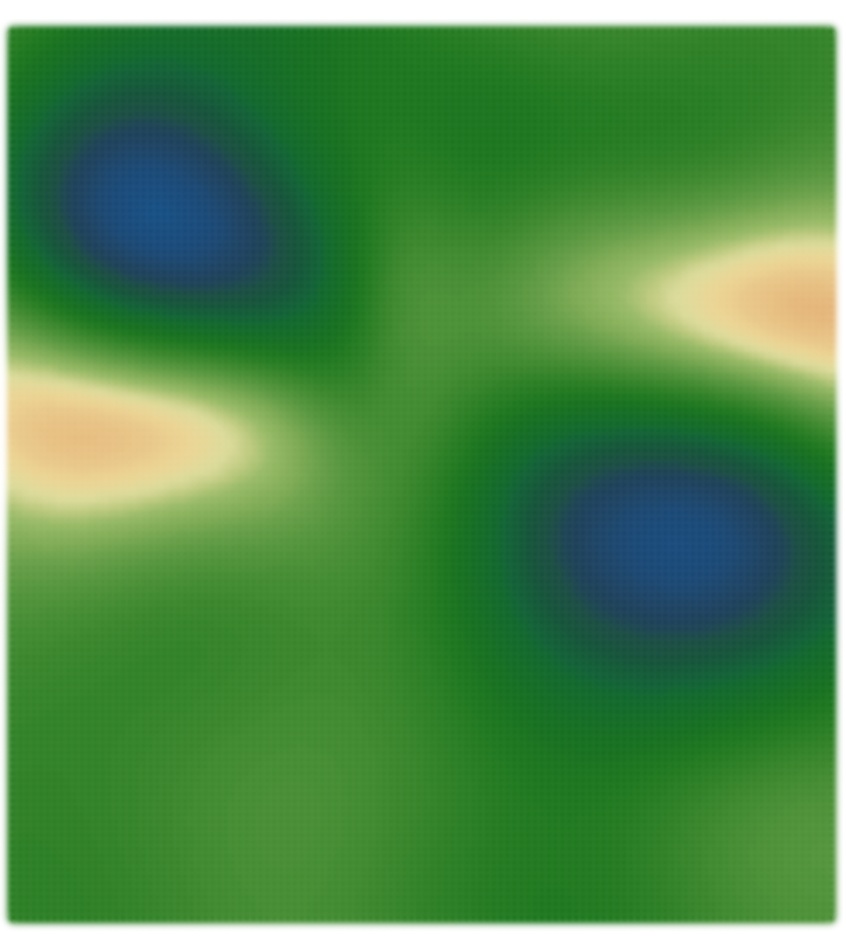}
\subcaption{$r_{5}$}
\end{subfigure}
\begin{subfigure}{.13\textwidth}
\centering
\includegraphics[width=.8\linewidth]{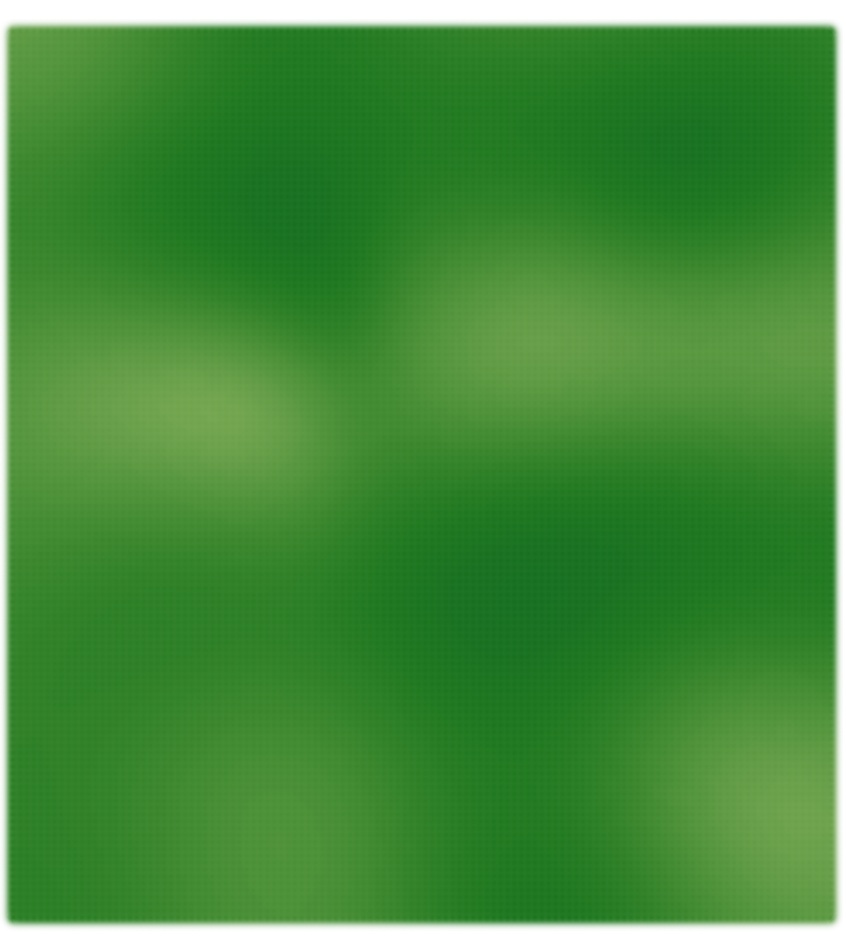}
\subcaption{$r_{9}$}
\end{subfigure}
\begin{subfigure}{.13\textwidth}
\centering
\includegraphics[width=.8\linewidth]{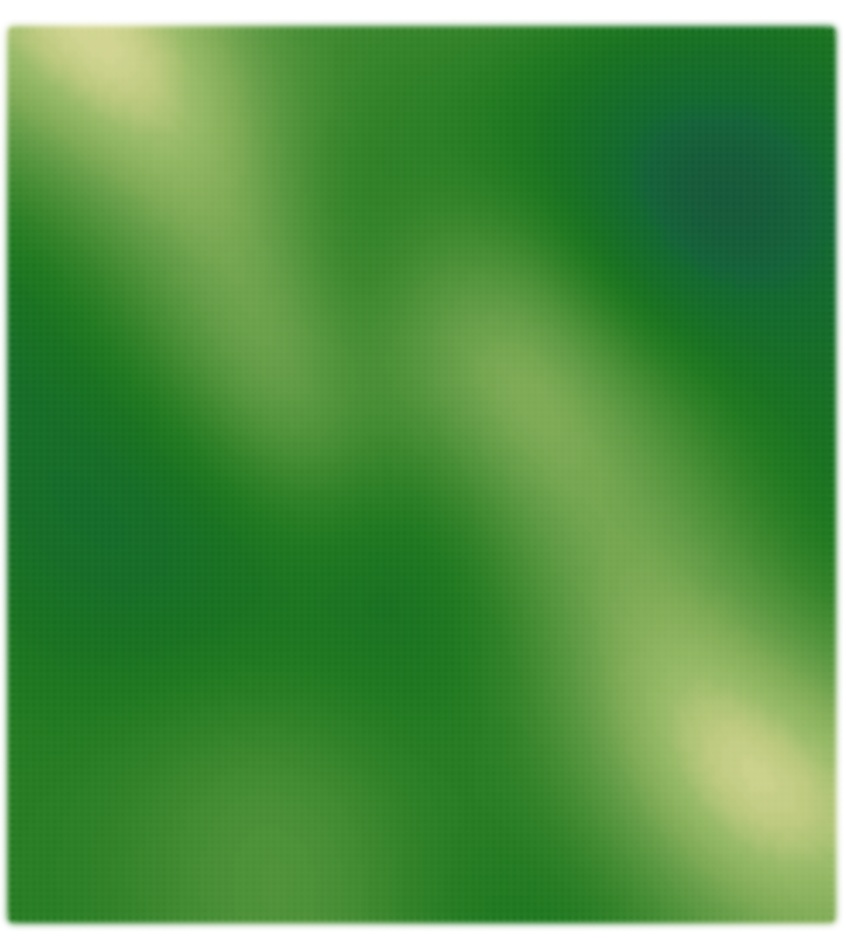}
\subcaption{$r_{13}$}
\end{subfigure}
\begin{subfigure}{.13\textwidth}
\centering
\includegraphics[width=.8\linewidth]{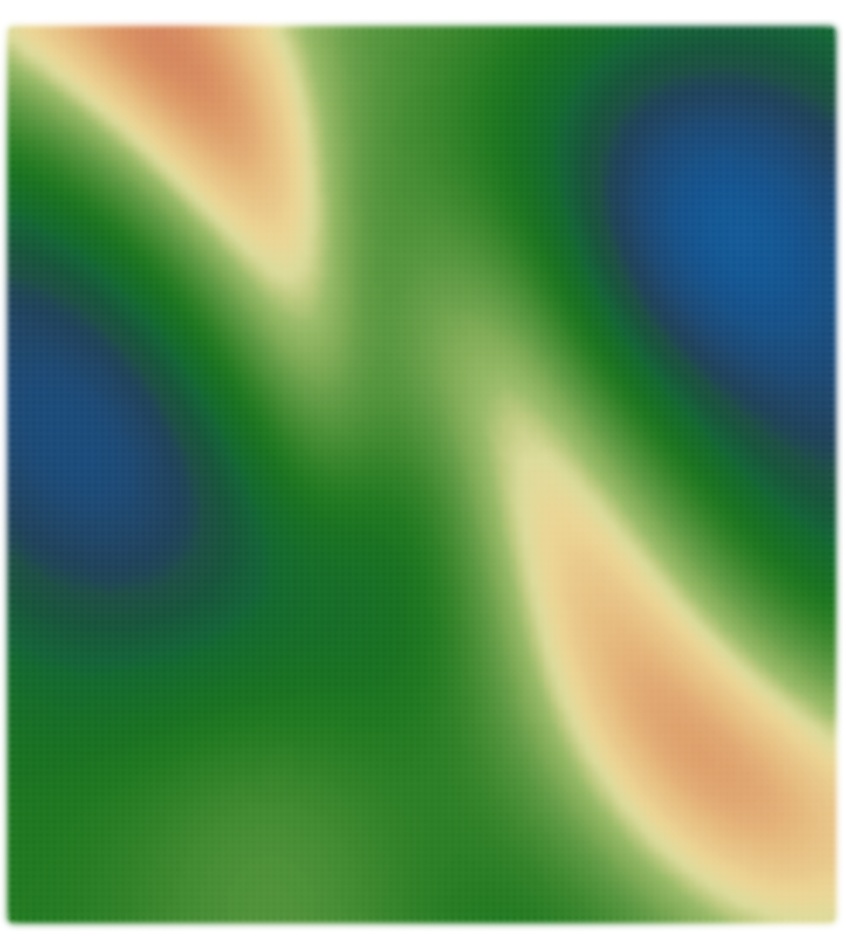}
\subcaption{$r_{17}$}
\end{subfigure}
\begin{subfigure}{.13\textwidth}
\centering
\includegraphics[width=.8\linewidth]{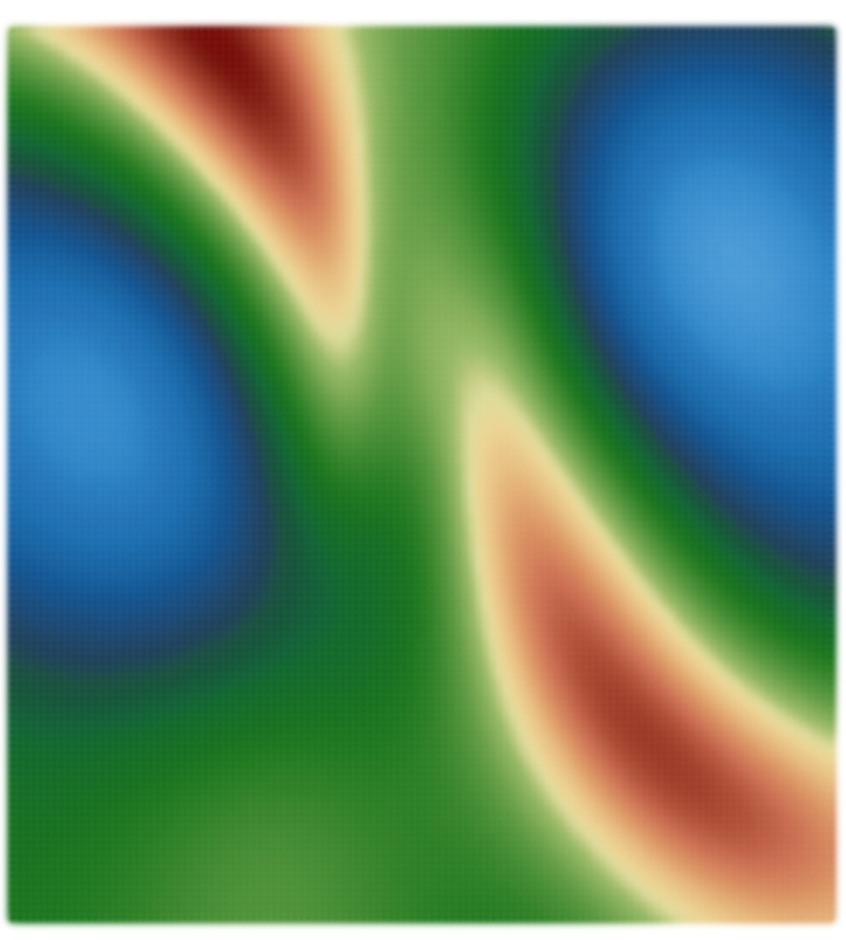}
\subcaption{$r_{20}$}
\end{subfigure}

\setcounter{subfigure}{0}
\begin{subfigure}{.13\textwidth}
\centering
\includegraphics[width=.8\linewidth]{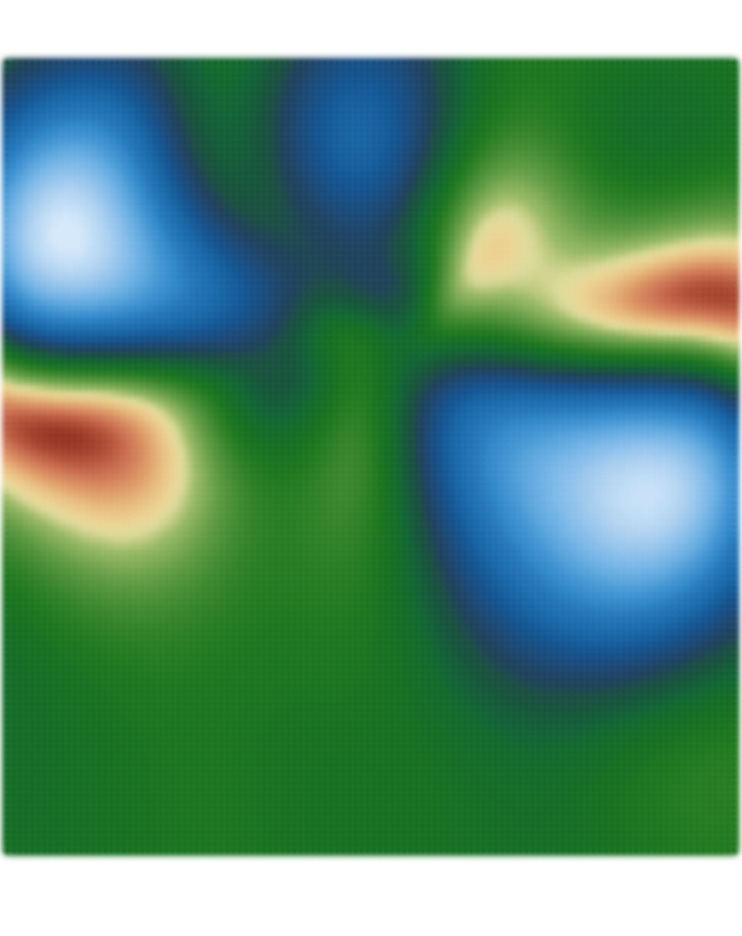}
\subcaption{$\rho^v_{r_{1}}$}
\end{subfigure}
\begin{subfigure}{.13\textwidth}
\centering
\includegraphics[width=.8\linewidth]{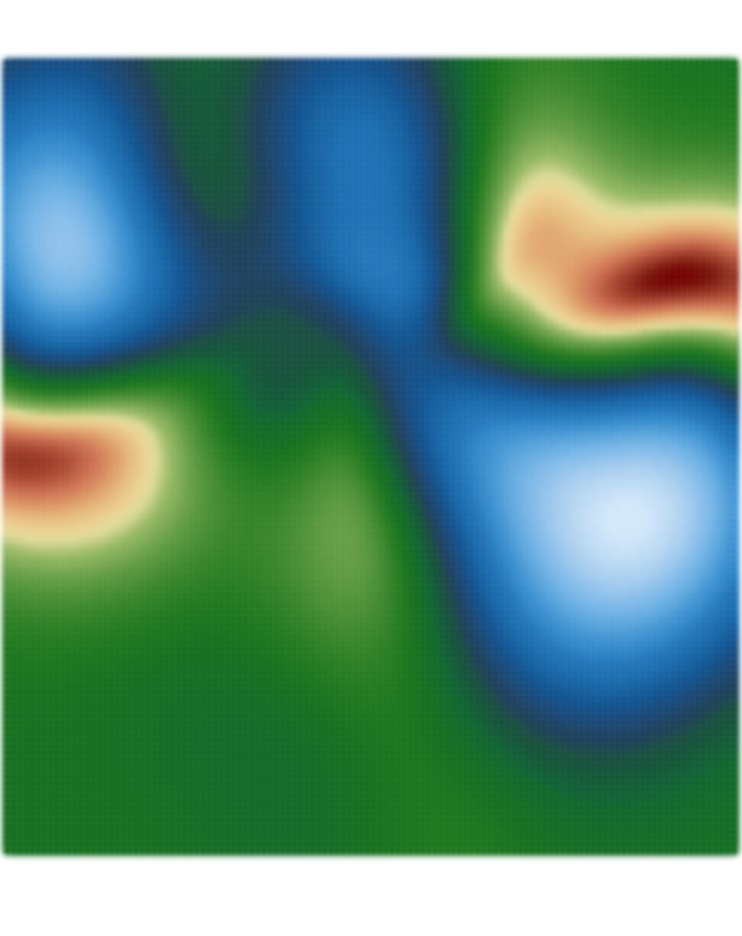}
\subcaption{$\rho^v_{r_{2}}$}
\end{subfigure}
\begin{subfigure}{.13\textwidth}
\centering
\includegraphics[width=.8\linewidth]{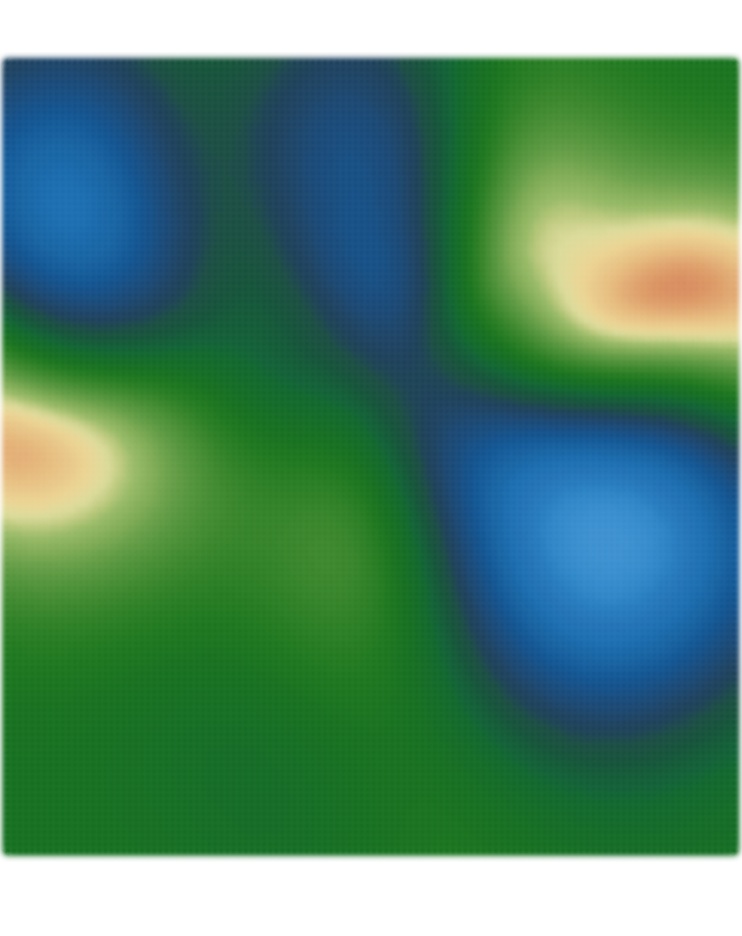}
\subcaption{$\rho^v_{r_{5}}$}
\end{subfigure}
\begin{subfigure}{.13\textwidth}
\centering
\includegraphics[width=.8\linewidth]{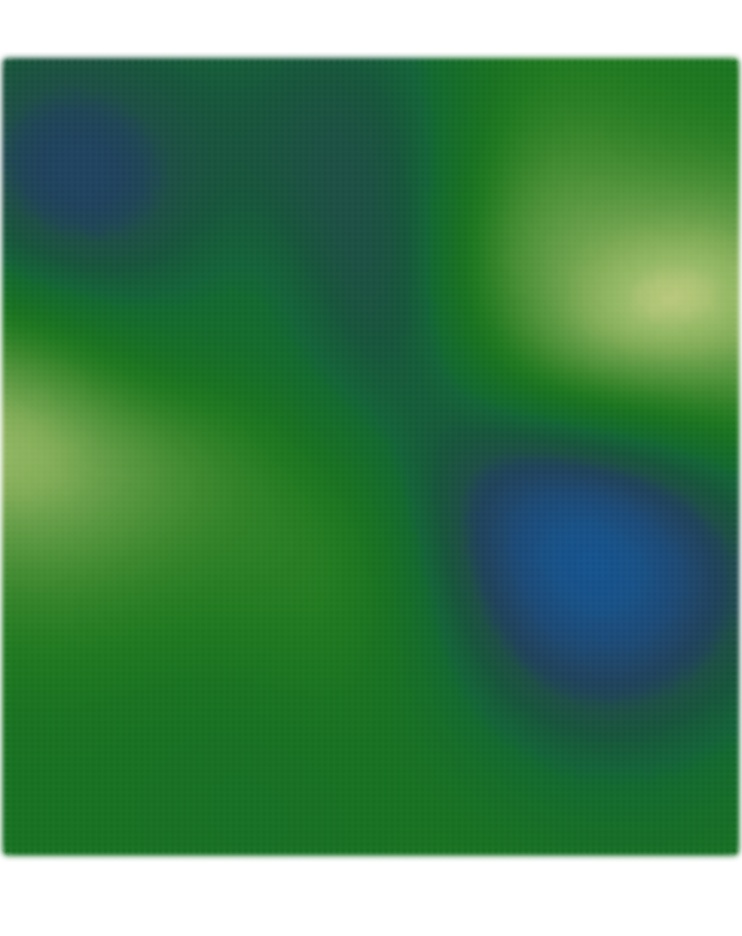}
\subcaption{$\rho^v_{r_{9}}$}
\end{subfigure}
\begin{subfigure}{.13\textwidth}
\centering
\includegraphics[width=.8\linewidth]{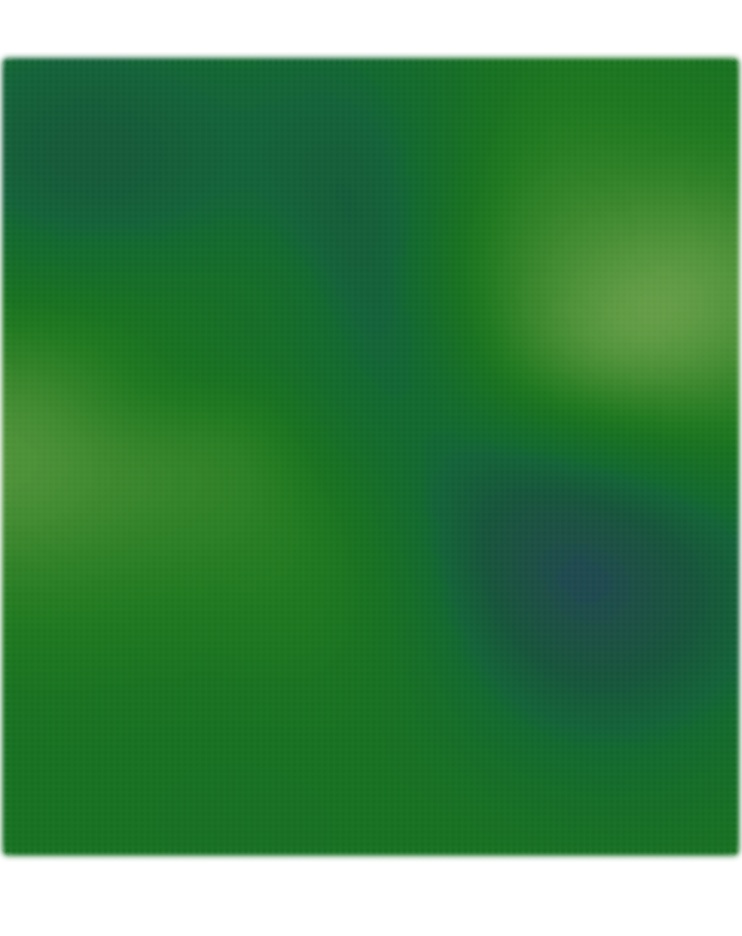}
\subcaption{$\rho^v_{r_{13}}$}
\end{subfigure}
\begin{subfigure}{.13\textwidth}
\centering
\includegraphics[width=.8\linewidth]{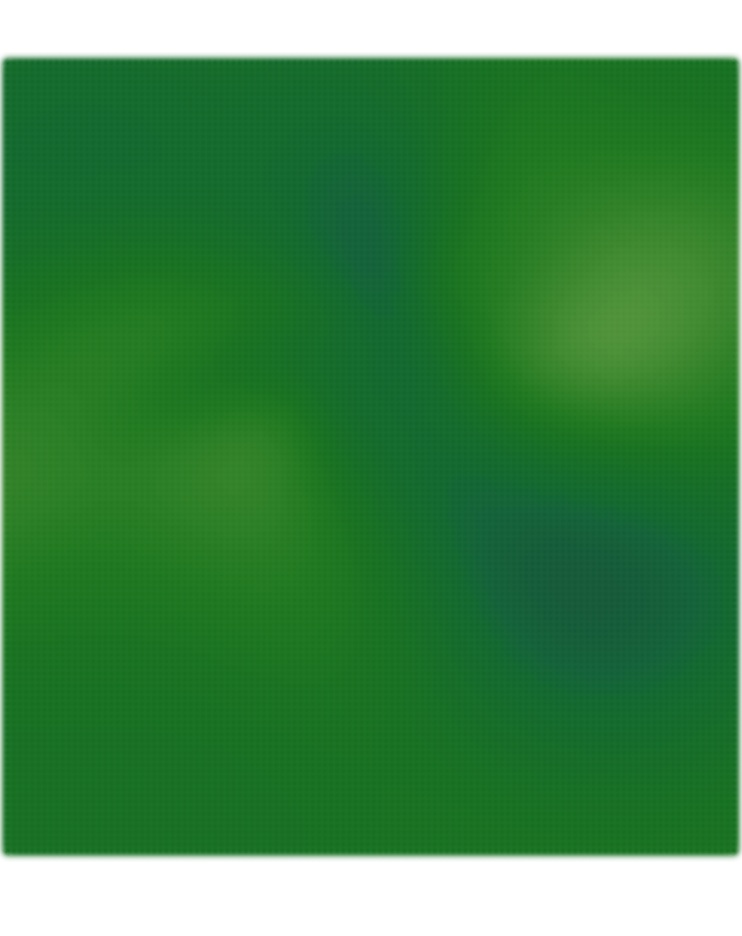}
\subcaption{$\rho^v_{r_{17}}$}
\end{subfigure}
\begin{subfigure}{.13\textwidth}
\centering
\includegraphics[width=.8\linewidth]{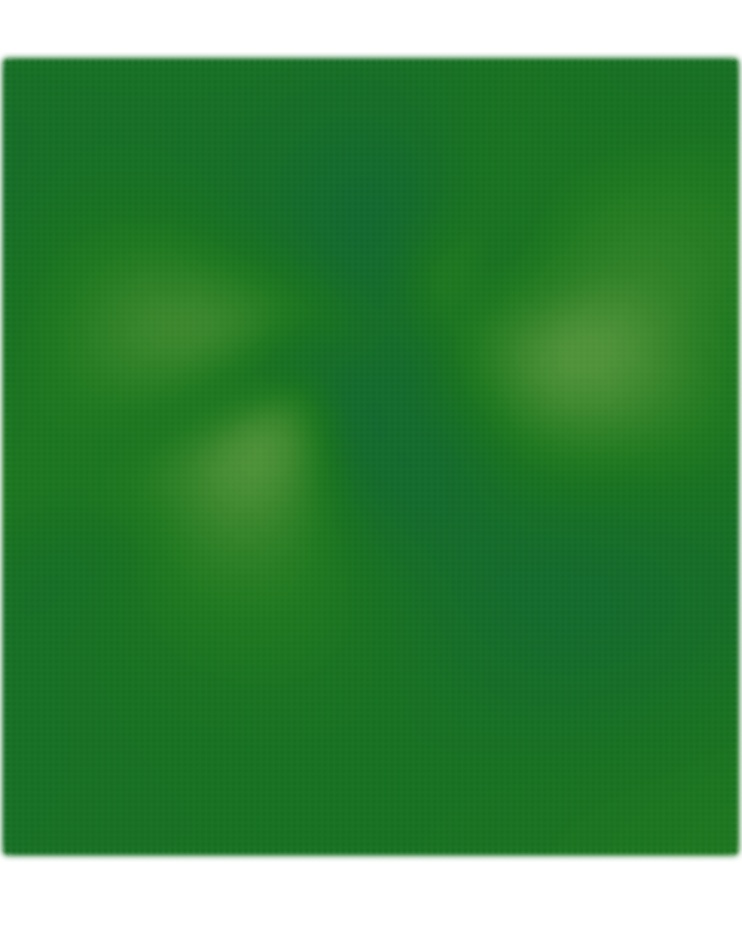}
\subcaption{$\rho^v_{r_{20}}$}
\end{subfigure}
\caption{Schematic human example with different templates and targets. First row: log Jacobian; Second row: residuals. Base scales: $r_1$, $r_{20}$.}
\label{fig: waving residual}
\end{figure}
\end{example}

\section{Conclusion}
In this paper, we have introduced a framework for diffeomorphic registration across scales varying on a continuous interval. This led us to study RKHS's of scale-dependent vector fields, for which one our main contribution was to provide an efficient numerical approximation of the reproducing kernel, which had no closed-form expression. This allowed us to derive a numerical scheme for the solution of a multi-scale LDDMM problem, for which we have provided experimental illustrations. From these examples, we see that our proposed MS-LDDMM model is able to perform a good landmark matching at base scales and at the same time provide an interpolation at the intermediate scales. Although we provided only 2D landmark matching examples, in fact this model can be applied to higher dimensional cases. On the other hand, the interpretation of interpolation heavily depends on the choice of multiscale kernels. Using the Gaussian kernels, we observe a rather fast decay of deformation across scales when there is only one base scale. But as in the schematic human example in  \cref{fig: human ms results}, the decay is ameliorated by imposing an other scale as a base scale, compared to the fast decay to an identity-like diffeomorphism in  \cref{fig: human fine scale} as well as  \cref{fig: human coarse scale}. 

And for the same landmark matching at two base scales, as in  \cref{fig: bumpy ms deformation}, we see that the deformations vary a lot, with the matching at $r_1$ stretches only two small portions on the left and right, while the matching at $r_{20}$ stretches a larger proportion in the middle, which can be distinguished by simple observation. This can be explained by the fact that for coarser scales, the kernels have a heavier tail, and thus a larger neighborhood will be affected by this landmark matching than a finer scale which concentrates more around its center and less effect on points further away. An important feature of the approach is to provide a scale-dependent residual decomposition of the transformation, which provides additional features, and possibly additional power, to morphometric shape analysis. 

\bibliography{ref}

\section{Appendix}\label{appendix}


\subsection{Proof of \texorpdfstring{\cref{lemma: ilambda: V to V}}{}}
\label{sec:lemma2.1} We restate \cref{lemma: ilambda: V to V} here:     

{\em For all $\lambda\in [s_1,s_2]$, $\iota_\lambda: \mathbb V\to V$ is a linear and bounded operator. 
}

\begin{proof}

Linearity is clear as this is just fixing one of the variables. It remains to check the continuity.

Take $\lambda,\mu\in [s_1,s_2]$ and use \cref{eqn: IBP} to obtain:
\begin{align*}
    \norm{\iota_\lambda v}_V \leq \norm{\iota_\mu v}_V + \int_{\min(\lambda,\mu)}^{\max(\lambda, \mu)} \norm{\partial_\lambda v(s)}_Vds\leq \norm{\iota_\mu v}_V + \int_{s_1}^{s_2}\norm{\partial_\lambda v(s)}_Vds.
\end{align*}

Integrating in $\mu$ with respect to Lebesgue's measure yields \begin{align*}
    (s_2-s_1)\norm{\iota_\lambda v}_V\leq \int_{s_1}^{s_2}\norm{\iota_s v}_Vds +(s_2-s_1)\int_{s_1}^{s_2}\norm{\partial_\lambda v(s)}_Vd\mu.
\end{align*}

Taking squares gives\begin{align*}
    (s_2-s_1)^2\norm{\iota_\lambda v}^2_V\leq 2\left(\int_{s_1}^{s_2}\norm{\iota_s v}_Vd\mu\right)^2+2(s_2-s_1)^2\left(\int_{s_1}^{s_2}\norm{\partial_\lambda v(s)}_Vds\right)^2\\
    \leq 2(s_2-s_1)\int_{s_1}^{s_2}\norm{\iota_s v}_V^2d\mu+2(s_2-s_1)^3\int_{s_1}^{s_2}\norm{\partial_\lambda v(s)}_V^2ds.
\end{align*}

This implies \begin{align}
\nonumber
    \norm{\iota_\lambda v}^2_V&\leq \frac{2}{s_2-s_1}\int_{s_1}^{s_2}\norm{\iota_s v}_V^2ds+2(s_2-s_1)\int_{s_1}^{s_2}\norm{\partial_\lambda v(s)}_V^2ds \\
    \label{eq:embedding.1}
    \norm{v(\lambda)}_V&\leq \sqrt{2}\max\left(\frac{1}{\sqrt{s_2-s_1}},\sqrt{s_2-s_1} \right)\norm{v}_{\mathbb V}.
\end{align}
\end{proof}

\subsection{Proof of \texorpdfstring{\cref{prop:V.equiv}}{}}
\label{sec:prop.V.equip} The statement runs as follows: 

{\em
    For any positive measure $\rho$ on $[s_1,s_2]$, $\|\cdot\|_{\mathbb V, \rho}\sim\|\cdot\|_{\mathbb V}$.  
}

\begin{proof}
Given such a measure $\rho$,
 we define $|\rho|=\rho([s_1,s_2])$ and \begin{align*}
    q_{\rho}(\lambda, s) = \left\{ \begin{aligned}
\frac{\rho([s_1, s))}{|\rho|}: s< \lambda\\
- \frac{\rho([s, s_2])}{|\rho|}: s\geq \lambda .
\end{aligned}
\right.
\end{align*} 

Let $v\in\mathbb V$. Then for all $s,\lambda\in[s_1,s_2]$, 
\[v(\lambda) = v(s)+\int_s^\lambda \partial_\lambda v(u)du.\] Integrate both sides with respect to  $\rho$ over $[s_1,\lambda)$, yielding (letting $\mathbf 1_A$ denote the indicator function of a set $A$):
\begin{align*}
    \rho([s_1,\lambda))v(\lambda)&=\int_{[s_1, \lambda)}v(s)d\rho(s)+\int_{s_1}^{s_2}\int_{[s_1,s_2]}\mathbf{1}_{[s_1,\lambda)}(s)\mathbf{1}_{(s,\lambda]}(u)\partial_\lambda v(u)dud\rho(s)\\
    &=\int_{[s_1,\lambda)}v(s)d\rho(s)+\int_{s_1}^{s_2}\int_{[s_1, s_2]}\mathbf{1}_{[s_1, u)}(s)\mathbf{1}_{(s_1,\lambda]}(u)\partial_\lambda v(u)d\rho(s)du\\ &=\int_{[s_1,\lambda)}v(s)d\rho(s)+\int_{s_1}^\lambda\rho([s_1,u))\partial_{\lambda }v(u)du. 
\end{align*}
Similarly, one can integrate over $[\lambda, s_2]$ to get \begin{align*}
    \rho([\lambda, s_2])v(\lambda) &= \int_{[\lambda, s_2]}v(s)d\rho(s)-\int_{[s_1, s_2]}\int_{s_1}^{s_2}\mathbf{1}_{[\lambda, s_2]}(s)\mathbf{1}_{[\lambda, s]}(u)\partial_\lambda v(u)dud\rho(s) \\
    &=\int_{[\lambda, s_2]}v(s)d\rho(s)-\int_{s_1}^{s_2}\int_{[s_1,s_2]}\mathbf{1}_{[u,s_2]}(s)\mathbf{1}_{[\lambda, s_2]}(u)\partial_\lambda v(u)d\rho(s)du \\
    &=\int_{[\lambda, s_2]}v(s)d\rho(s)-\int_{\lambda}^{s_2}\rho([u,s_2])\partial_\lambda v(u)du.
\end{align*}
Then summing these two equations yields \begin{align}
    \rho([s_1,s_2])v(\lambda) &= \int_{[s_1,s_2]}v(s)d\rho(s)+\int_{s_1}^{s_2}\left(\mathbf{1}_{[s_1,\lambda)}(u)\rho([s_1,u))-\mathbf{1}_{[\lambda, s_2]}\rho([u, s_2])\right)\partial_\lambda v(u)du\nonumber\\ \hspace{12pt} v(\lambda)&=\frac{1}{|\rho|}\int_{[s_1,s_2]}v(s)d\rho(s) +\int_{s_1}^{s_2}q_\rho(\lambda, u)\partial_\lambda v(u)du.
    \label{eqn: sum pos meas}
\end{align}

Therefore, by \cref{eqn: sum pos meas} and the fact that $|q_\rho(\lambda, s)|\leq 1$, one has \begin{align}
\nonumber
    \norm{v(\lambda)}_V&\leq \frac{1}{\sqrt{|\rho|}}\left(\int_{[s_1,s_2]}\norm{v(s)}_{V}^2d\rho(s)\right)+\sqrt{s_2-s_1}\left(\int_{s_1}^{s_2}\norm{\partial_\lambda v(s)}_V^2ds\right)^{1/2} \\ 
    \label{eq:embedding.2}
    \sup\limits_{\lambda\in[s_1,s_2]}\norm{v(\lambda)}_V&\leq M\left(\int_{[s_1,s_2]}\norm{v(s)}_V^2d\rho(s)+\int_{s_1}^{s_2}\norm{\partial_\lambda v(s)}_V^2ds\right)^{1/2},
\end{align}
where $M=\sqrt{2}\max\left(\frac{1}{\sqrt{|\rho|}}, \sqrt{s_2-s_1}\right)$.
We retrieve \cref{eq:embedding.1} when $\rho=\mathcal L$, the Lebesgue measure. Define $\norm{v}':=\norm{v}_\infty+\norm{\partial_\lambda v}_2$, where  $\|v\|_{\infty}:=\sup\limits_{\lambda\in[s_1,s_2]}\norm{v(\lambda)}_V$ and $\|\partial_\lambda v\|_2^2:=\int_{s_1}^{s_2}\norm{\partial_\lambda v(s)}_V^2ds$.
\Cref{eq:embedding.1}, combined with the fact that $\|v\|_2 \leq \sqrt{s_2-s_1} \|v\|_\infty$,  shows that the $\|\cdot\|'$ is equivalent to $\|\cdot \|_{\mathbb V}$, so that $(\mathbb V, \|\cdot\|')$ is a Banach space. 

\Cref{eq:embedding.2} now implies that $\norm{v}_{\mathbb V, \rho}$ is also equivalent to $\|\cdot\|'$ and therefore also to $\|\cdot\|_{\mathbb V}$.
\end{proof}

\subsection{Proof of \texorpdfstring{\cref{Prop W lambda}}{}}
\label{sec:Prop W lambda} 
We repeat the statement again:

{\em
    If $\lambda\in [s_1, s_2]$ is such that $\rho([\lambda, s_2])>0$, then  $\iota_\lambda(v)\in W_\lambda$ and $w\in \mathbb W$. Moreover, this operator is bounded from  $(\mathbb W, \norm{\cdot}_{\mathbb W})$ to $(W_\lambda, \norm{\cdot}_\lambda)$. 

    Let $\lambda_0\in [s_1, s_2]$ be such that $\rho([\lambda_0, s_2]) = 0$ and $\rho([\lambda, s_2]) > 0$ for all $\lambda < \lambda_0$. 
    If $W_{\lambda_0}=\cap_{\lambda\in [s_1,\lambda_0)}W_\lambda$, then $v$ (extended by continuity to $[s_1, \lambda_0]$) is such that $\iota_{\lambda_0}(v)\in W_{\lambda_0}$. 
    
    Furthermore, if one assumes, in addition, that the mapping $\lambda\in[s_1,\lambda_0]\mapsto \norm{w}_\lambda$ is continuous for all $w\in W_{\lambda_0}$, then $\iota_{\lambda_0}:\mathbb W\to W_{\lambda_0}$ is bounded. 
}

\begin{proof}

    Fix $\lambda \in [s_1,s_2)$, take $s\in [s_1,s_2]$ with $s\geq \lambda$. As $W_{\mu}\hookrightarrow W_{\nu}$ for $\mu\geq \nu$, $\norm{v(\mu)}_{\nu}$ is well defined for any $\mu\geq \nu$. In particular, one has 
    \[
    \norm{v(\lambda)}_\lambda\leq \norm{v(s)}_\lambda+\int^s_\lambda\norm{\partial_\lambda v(s')}_{\lambda}ds'.
    \]
    
    Integrating both sides with respect to $s$ in measure $\rho$ gives: 
    \begin{align*}
        \int_{[\lambda, s_2]}\norm{v(\lambda)}_\lambda d\rho(s)&\leq \int_{[\lambda, s_2]}\norm{v(s)}_\lambda d\rho(s)+\int_{[\lambda, s_2]}\int^{s}_{\lambda}\norm{\partial_\lambda v(s')}_\lambda ds'd\rho(s).
        \end{align*}
        As a consequence,
        \begin{align*}
        \norm{v(\lambda)}_\lambda&\leq \frac{1}{\rho([\lambda, s_2])}\int_{[\lambda, s_2]}\norm{v(s)}_\lambda d\rho(s)+\frac{1}{\rho([\lambda, s_2])}\int_{[\lambda, s_2]}\int_{\lambda}^{s_2}\norm{\partial_\lambda v(s')}_\lambda ds'd\rho(s) \\
        &\leq \frac{C}{\rho([\lambda, s_2])}\int_{[\lambda, s_2]}\norm{v(s)}_sd\rho(s)+C\int_{\lambda}^{s_2}\norm{\partial_\lambda v(s')}_{s'}ds'\\
        &\leq \frac{C}{\rho([\lambda, s_2])}\int_{[s_1, s_2]}\norm{v(s)}_sd\rho(s)+C\int_{s_1}^{s_2}\norm{\partial_\lambda v(s)}_{s}ds.
    \end{align*} This implies that \begin{align}
        \norm{v(\lambda)}_\lambda^2&\leq \frac{2C^2|\rho|}{\rho([\lambda, s_2])^2}\int_{[s_1, s_2]}\norm{v(s)}_s^2d\rho(s)+2C^2(s_2-s_1)\int_{s_1}^{s_2}\norm{\partial_\lambda v(s)}_s^2ds\nonumber\\
        &\leq \max\left(\frac{2C^2|\rho|}{\rho([\lambda, s_2])^2}, 2C^2(s_2-s_1)\right)\norm{v}_{\mathbb W}^2.
    \end{align}
    Therefore, one can see that $\norm{\iota_\lambda(v)}_\lambda<\infty$ and $\iota_\lambda: \mathbb W\to W_\lambda$ is continuous. 

    We now pass to the second statement of the theorem, first assuming that left-continuity of $W_\lambda$ at $\lambda=\lambda_0$ holds, i.e., $W_{\lambda_0}=\cap_{s\in [s_1,\lambda_0)} W_s$. Let $(s^k)_{k=1}^\infty$ be a increasing sequence in $[s_1,\lambda_0)$ that converges to $\lambda_0$ as $k\to\infty$. Then $\iota_{s^k}(v)\in W_{\mu}$ for all $k\in \mathbb N$ and $\mu\leq s^k$. 
    We have, for $l\geq k \geq k_0$,
    \[
    \|v(s^k) - v(s^l)\|_{s^{k_0}} \leq \int_{s^k}^{s^l} \|\partial_\lambda v(\lambda)\|_{s^{k_0}} d\lambda \leq C \int_{s^k}^{s^l} \|\partial_\lambda v(\lambda)\|_{\lambda} d\lambda,   
    \]
    which implies that, for all $k_0$, $(v(s^k), k\geq k_0)$ is a Cauchy sequence in $W_{s^{k_0}}$. Thus  $v(s^k)$ converges to a limit that belongs to $\cap_{k\geq 1} W_{s^k} = W_{\lambda_0}$.

    We now assume that, for $w\in W_{\lambda_0}$, one has
    $\left|\norm{w}_\lambda - \norm{w}_{\lambda_0}\right|\to 0$ as $\lambda\to \lambda_0$. Note that, for $\lambda < \lambda_0$, 
    $\rho([\lambda, \lambda_0)) = \rho([\lambda, s_2]) > 0$.

    Notice that \begin{align*}
        &\frac{1}{\rho([\lambda, \lambda_0])}\left|\int_{[\lambda, \lambda_0]}\norm{v(s)}_sd\rho(s) - \norm{v(\lambda_0)}_{\lambda_0}\right|
        \\
        &\qquad \leq \frac{1}{\rho([\lambda, \lambda_0])}\int_{[\lambda, \lambda_0]}\left|\norm{v(s)}_s-\norm{v(\lambda_0)}_{\lambda_0}\right|d\rho(s)\\
        &\qquad \leq \max_{s\in [\lambda, \lambda_0]}\left|\norm{v(s)}_s - \norm{v(\lambda_0)}_{\lambda_0}\right|\\
        &\qquad \leq \max_{s\in[\lambda, \lambda_0]}\norm{v(s)-v(\lambda_0)}_s+\max_{s\in[\lambda, \lambda_0]}\left|\norm{v(\lambda_0)}_s-\norm{v(\lambda_0)}_{\lambda_0}\right|
    \end{align*}
    By assumption, the second term goes to zero as $s\to 0$. For the first term, let $\mu\in [s, \lambda_0]$, then \begin{align*}
        \norm{v(\mu)-v(\lambda_0)}_\mu\leq \int_\mu^{\lambda_0}\norm{\partial_\lambda v(\mu)}_\mu d\mu\leq C\int_s^{\lambda_0}\norm{\partial_\lambda v(\mu)}_Vd\mu.
    \end{align*}
    As the right hand side goes to zero as $s\to \lambda_0$, for any $\epsilon>0$, one can take $s$ close enough to $\lambda_0$ such that for any $\mu \in [s, \lambda_0]$,  $\norm{v(\mu)-v(\lambda_0)}_{\mu}<\epsilon$. Therefore, one concludes that $\max\limits_{s\in[\lambda, \lambda_0]}\norm{v(s)-v(\lambda_0)}_s\to 0$ as $\lambda\to \lambda_0$. And thus 
    \begin{align*}
        |\norm{v(\lambda)}_\lambda-\norm{v(\lambda_0)}_{\lambda_0}|\leq \max_{s\in[\lambda, \lambda_0]}\norm{v(s)-v(\lambda_0)}_s+\max_{s\in[\lambda, \lambda_0]}\left|\norm{v(\lambda_0)}_s-\norm{v(\lambda_0)}_{\lambda_0}\right|.
    \end{align*}
    Since the right hand side vanishes as $\lambda\to \lambda_0$, the boundedness of $\iota_{\lambda_0}:\mathbb W\to W_{\lambda_0}$ has been established. 
    \end{proof}

\subsection{Proof of \texorpdfstring{\cref{th:bochner.plus}}{}}

For completeness, as we haven't been able to find a proof of this theorem in the literature, we provide a proof of  \cref{th:bochner.plus}, of which we repeat the statement.

{\em
    Let $S$ be any set and $\Gamma: S\times S \times \mathbb R^d\to \mathbb R$. Then the followings are equivalent:
    \begin{enumerate}[label=(\arabic*)]
        \item For all $s_1, s_2\in S$, $z\mapsto \Gamma(s_1, s_2, z) $ is continuous and the function
    \[
    K: ((s_1, x_1), (s_2, x_2)) \mapsto \Gamma(s_1, s_2, x_1-x_2)
    \]
    is a positive semi-definite kernel. 
    \item There exists a family of finite complex Radon measures $\mathcal M = (\mu_{s_1,s_2}, s_1, s_2 \in S)$ such that, for all $s_1, s_2\in S$ 
    \[
    \Gamma(s_1, s_2, z) = \int_{\mathbb R^d} e^{-i2\pi \xi^Tz} d\mu_{s_1, s_2}(\xi),
    \]
    and $\mathcal M$ is a measure-valued positive kernel in the sense that, for all $s_1, \ldots, s_n\in S$ and all $a_1, \ldots, a_n\in \mathbb C$, 
    \[
    \sum_{k,l=1}^n a_k\bar a_l \mu_{s_k, s_l}
    \]
    is a positive measure.
  \end{enumerate}
}
\subsection{(1) \texorpdfstring{$\Rightarrow$}{Lg} (2)} 
The positivity of $K$ implies that, for all $\boldsymbol{\alpha} = (\alpha_1, \ldots, \alpha_N)\in \mathbb C^N$ and $\boldsymbol{t} = (t_1,\ldots, t_N)\in S^N$ and $\boldsymbol{\xi}=(\xi_1, \ldots, \xi_N)\in (\mathbb R^d)^N$, one has
\[
\sum_{k, l=1}^N \bar\alpha_k \alpha_l \Gamma(t_k, t_l, \xi_k - \xi_l) \geq 0.
\]
Taking $a_1, \ldots, a_n\in \mathbb C$, $s_1, \ldots, s_n\in S$, $b_1, \ldots, b_m\in \mathbb C$, $x_1, \ldots, x_m\in \mathbb R^d$, and letting
\begin{align*}
    \boldsymbol\alpha &= (a_1b_1, \ldots, a_1b_m, \ldots, a_nb_1, \ldots, a_n b_m)\\
    \boldsymbol{t} &= (\underbrace{s_1, \ldots, s_1}_{m\ \text{times}}, \ldots, \underbrace{s_n, \ldots, s_n}_{m \ \text{times}})\\
    \boldsymbol{\xi} &= (\underbrace{x_1, \ldots, x_m, \ldots, x_1, \ldots, x_m}_{n\ \text{times}}),
\end{align*}
this gives
\[
\sum_{k,l=1}^n \sum_{q,r=1}^m \bar a_k a_l \bar b_q b_r \Gamma(s_k, s_l, x_q - x_r) \geq 0.
\]

This shows that, for all $\boldsymbol{a} = (a_1, \ldots, a_n)$, $\boldsymbol{s} = (s_1, \ldots, s_n)$, the function
\[
\Gamma_{\boldsymbol{s}, \boldsymbol{a}}(z) = \sum_{k,l=1}^n \bar a_k a_l \Gamma(s_k, s_l, ,z)
\]
has positive type. Using the scalar Bochner theorem, this shows that there exists a finite positive measure $\tilde\mu_{\boldsymbol{s}, \boldsymbol{a}}$ of which $\Gamma_{\boldsymbol{s}, \boldsymbol{a}}$ is the Fourier transform. Define $\mu_{s,s} = \tilde \mu_{(s), (1)}$ and
\[
\mu_{s_1, s_2} = \frac12 \left(\tilde \mu_{(s_1, s_2), (1,1)} - \iota \tilde \mu_{(s_1, s_2), (1, \iota)} - (1-\iota) \mu_{s_1, s_1} - (1+\iota) \mu_{s_2, s_2} \right).
\]
One can be check by direct computation that $\Gamma(s_1, s_2, \cdot)$ is the Fourier transform of $\mu_{s_1, s_2}$, and identifying inverse Fourier transforms, we then obtain the fact that
\[
\tilde \mu_{\boldsymbol{s}, \boldsymbol{a}} = \sum_{k,l=1}^n \bar a_k a_l \mu_{s_1, s_2},
\]
showing that the right-hand side is a positive measure and proving (2).

\subsection{ (2) \texorpdfstring{$\Rightarrow$}{Lg} (1)}

Let $s_1, \ldots, s_n\in S$, $x_1, \ldots, x_n\in \mathbb R^d$ and $a_1, \ldots, a_n\in \mathbb C$. Then
\[
\sum_{k,l=1}^n \bar a_k a_l \Gamma(s_k, s_l, x_k - x_l) = 
\int_{\mathbb R^d} \sum_{k,l=1}^n \bar a_k a_l e^{-\iota 2\pi \xi^T (x_k-x_l)} d\mu_{s_k, s_l}(\xi).
\]
Letting $f_k(\xi) = a_k e^{\iota 2\pi\xi^Tx_k}$, the r.h.s. takes the form 
\[
\sum_{k,l=1}^n\int_{\mathbb R^d}  \bar f_k(\xi) f_l(\xi)d\mu_{s_k,s_l}(\xi),
\]

For all $\delta>0$, consider the regular grid $G_\delta = \{n\delta: n\in \mathbb Z^d\}$ and associate to each $p \in G_\delta$ the semi-closed set $C_p^\delta = \{\xi\in\mathbb R^d: p_i-\delta/2 \leq \xi_i<p_i+\delta/2, i=1, \ldots, d\}$ so that the sets $C_p^\delta, p\in G_\delta$ are non-intersecting and cover $\mathbb R^d$. Let
\[
f_k^\delta(\xi) = \sum_{p\in G_\delta} f_k(p) \mathbf 1_{C_p^\delta}(\xi).
\]
Fix $\epsilon >0$. Because all $f_k$'s are uniformly continuous, there exists $\delta_0>0$ such that $\|f_k - f_k^\delta\|_\infty \leq \epsilon$ for all $\delta\leq \delta_0$, which implies that 
\[
\left| \sum_{k,l=1}^n\int_{\mathbb R^d}  (\bar f_k(\xi) f_l(\xi) - \bar f^\delta_k(\xi) f^\delta_l(\xi))d\mu_{s_k,s_l}(\xi)\right|\leq \epsilon \sum_{k,l=1}^n (|a_k| + |a_l|) |\mu_{s_k,s_l}|(\mathbb R^d) = C\epsilon.
\]
However, 
\begin{align*}
\sum_{k,l=1}^n\int_{\mathbb R^d}  \bar f^\delta_k(\xi) f^\delta_l(\xi)d\mu_{s_k,s_l}(\xi)
&= \sum_{k,l=1}^n \sum_{p, p'\in G_\delta} \bar f_k(p) f_l(p') \mu_{s_k, s_l}(C_p^\delta \cap C_{p'}^\delta)\\
&= \sum_{p\in G_\delta} \sum_{k,l=1}^n  \bar f_k(p) f_l(p) \mu_{s_k, s_l}(C_p^\delta)\geq 0
\end{align*}
since each term in the inner double sum is non-negative by (2). So, we find
\[
\sum_{k,l=1}^n\int_{\mathbb R^d}  \bar f_k(\xi) f_l(\xi)d\mu_{s_k,s_l}(\xi) \geq - C\epsilon
\]
with $\epsilon$ arbitrary, which gives (1).

\end{document}